\documentclass[reqno,11pt]{amsart}

\usepackage[top=2.0cm,bottom=2.0cm,left=3cm,right=3cm]{geometry}
\usepackage{amsthm,amsmath,amssymb,dsfont}
\usepackage{mathrsfs,amsfonts,functan,extarrows,mathtools}
\usepackage[colorlinks]{hyperref}

\usepackage{marginnote}
\usepackage{xcolor}
\usepackage{stmaryrd}
\usepackage{esint}
\usepackage{graphicx}
\usepackage{bm}

\newtheorem{theorem}{Theorem}[section]
\newtheorem{proposition}{Proposition}[section]

\newtheorem{remark}{Remark}[section]
\newtheorem{lemma}{Lemma}[section]

\numberwithin{equation}{section}
\allowdisplaybreaks

\def\d{\mathrm{d}}
\def\no{\nonumber}
\def\R{\mathbb{R}}
\def\T{\mathbb{T}}
\def\P{\mathbb{P}}
\def\eps{\varepsilon}
\def\div{\mathrm{div}}

\def\l{\left\langle}
\def\r{\right\rangle}

\newcounter{wronumber}

\begin{document}
\title[From VMB to incompressible NSFM with Ohm's law]
			{From Vlasov-Maxwell-Boltzmann system to two-fluid incompressible Navier-Stokes-Fourier-Maxwell system with Ohm's law: convergence for classical solutions}

\author[Ning Jiang]{Ning Jiang}
\address[Ning Jiang]{\newline School of Mathematics and Statistics, Wuhan University, Wuhan, 430072, P. R. China}
\email{njiang@whu.edu.cn}

\author[Yi-Long Luo]{Yi-Long Luo}
\address[Yi-Long Luo]
		{\newline School of Mathematics and Statistics, Wuhan University, Wuhan, 430072, P. R. China}
\email{yl-luo@whu.edu.cn}

\thanks{ \today}

\maketitle

\begin{abstract}
  For the two-species Vlasov-Maxwell-Boltzmann (VMB) system with the scaling under which the moments of the fluctuations to the global Maxwellians formally converge to the two-fluid incompressible Navier-Stokes-Fourier-Maxwell (NSFM) system with Ohm's law, we prove the uniform estimates with respect to Knudsen number $\eps$ for the fluctuations. As consequences, the existence of the global in time classical solutions of VMB with all $\eps \in (0,1]$ is established. Furthermore, the convergence of the fluctuations of the solutions of VMB to the classical solutions of NSFM with Ohm's law is rigorously justified. 
  
This limit was justified in the recent breakthrough of Ars\'enio and Saint-Raymond \cite{Arsenio-SRM-2016} from renormalized solutions of VMB to dissipative solutions of incompressible viscous electro-magneto-hydrodynamics under the corresponding scaling. In this sense, our result gives a classical solution analogue of the corresponding limit in \cite{Arsenio-SRM-2016}.\\

\noindent\textsc{Keywords.} two-species Vlasov-Maxwell-Boltzmann system; two-fluid incompressible Navier-Stokes-Fourier-Maxwell system; Ohm's law; global classical solutions; uniform energy bounds; convergence for classical solutions. %\\

%\noindent\textsc{AMS subject classifications.} 35B45, 35B65, 35Q35, 76D03, 76D09, 76D10
\end{abstract}

%\vspace*{10pt}

%\phantomsection
%\addcontentsline{toc}{section}{\contentsname}

%\tableofcontents

%%%%%%%%%%%%%%%%%%%%%%%%%%%%%%%%%%%%%%（正文）%%%%%%%%%%%%%%%%%%%%%%%%%%%%%
%%%%%%%%%%%%%%%%%%%%%%%%%%%%%%%%%%%%%%%%%%%%%%%%%%%%%%%%%%%%%%%%%%%%%%%%%%%

\section{Introduction.}\label{Sec:Introduction}

\subsection{Vlasov-Maxwell-Boltzmann system}
Two-species Vlasov-Maxwell-Boltzmann system (in brief, VMB) describes the evolution of a gas of two species of oppositely charged particles (cations of charge $q^+ > 0$ and mass $m^+>0$, and anions of charge $-q^- <0$ and mass $m^->0$), subject to auto-induced electromagnetic forces. Such a gas of charged particles, under a global neutrality condition, is called a plasma. The particle number densities $F^+(t,x,v) \geq 0 $ and $F^-(t,x,v) \geq 0 $ represent the distributions of the positively charged ions (i.e. cations), and the negatively charged ions (i.e. anions) at time $t\geq 0$, position $x\in \mathbb{T}^3$, with velocity $v\in \mathbb{R}^3$, respectively. Precisely,  VMB system consists the following equations:
\begin{equation}\label{VMB-0}
    \left\{
    \begin{array}{c}
      \partial_t F^+ + v\cdot \nabla_x F^+ + \tfrac{q^+}{m^+}(E + v \times B)\cdot\nabla_v F^+ = \mathcal{B}(F^+, F^+) + \mathcal{B}(F^+, F^-)\,,\\[2mm]
            \partial_t F^- + v\cdot \nabla_x F^- - \tfrac{q^-}{m^-}(E + v \times B)\cdot\nabla_v F^- = \mathcal{B}(F^-, F^-) + \mathcal{B}(F^-, F^+)\,,\\[2mm]
      \mu_0\eps_0\partial_t E - \nabla_x \times B = -\mu_0\int_{\mathbb{R}^3}(q^+F^+- q^- F^-)v\,\mathrm{d}v\,,\\[2mm]
      \partial_t B + \nabla_{\!x} \times E = 0\,,\\[2mm]
      \div_x E = \tfrac{1}{\eps_0}\int_{\mathbb{R}^3}(q^+F^+- q^- F^-)\,\mathrm{d}v\,,\quad\!\mbox{and}\quad\!
      \div_x B = 0\,.
        \end{array}
  \right.
\end{equation}

The evolutions of the densities  $F^\pm$ are governed by the Vlasov-Boltzmann equations, which are the first two lines in \eqref{VMB-0}. They tell that the variations of the densities $F^\pm$ along the trajectories of the particles are subject to the influence of a Lorentz force and inter-particel collisions in the gas. The Lorentz force acting on the gas is auto-induced. That is, the electric field $E(t,x)$ and the magnetic field $B(t,x)$ are generated by the motion of the particles in the plasma itself. Their motion is governed by the Maxwell's equations, which are the remaining equations in \eqref{VMB-0}, namely Amp\`ere equation, Faraday's equation and Gauss' laws respectively. In \eqref{VMB-0}, the physical constants $\mu_0, \eps_0 >0$ are, respectively, the vacuum permeability (or magnetic constant) and the vacuum permittivity (or electric constant). Note that their relation to the speed of light is the formula $c=\frac{1}{\sqrt{\mu_0\eps_0}}$. For the sake of mathematical convenience, we make the simplification that both kinds of particles have the same mass $m^\pm=m>0$ and charge $q^\pm = q >0$. 

The Boltzmann collision operator, presented in the right-hand sides of the Vlasov-Boltzmann equations in \eqref{VMB-0}, is the quadratic form, acting on the velocity variable, associated to the bilinear operator, 
\begin{equation*}
  \mathcal{B}(F,H)(v) = \int_{\R^3} \int_{\mathbb{S}^2} ( F'H_*' - F H_* ) b(v-v_*, \cos \theta) \d \omega \d v_* \,,
\end{equation*}
where we have used the standard abbreviations 
\begin{equation}\nonumber
F=F(v)\,,\quad\! F'=F(v')\,,\quad\!H_*=H(v_*)\,,\quad\! H'_*=H(v'_*)\,,
\end{equation}
with $(v',v'_*)$ given by
\begin{equation*}
 v' = v - [ (v - v_*) \cdot \omega ] \omega \,,\quad\!
      v_*' = v_* + [(v-v_*) \cdot \omega] \omega\,.
\end{equation*}
for $\omega \in \mathbb{S}^2$. In this paper, we will assume that the Boltzmann collision kernel is of the following {\em hard sphere} form
\begin{equation}
b(v-v_*, \cos \theta) = |(v-v_*) \cdot \omega| = |v-v_*| | \cos \theta | \,.
\end{equation}
This hypothesis is satisfied for all physical model and is more convenient to work with but do not impede the generality of our results. Then the collisional frequency can be defined as
\begin{equation}\label{Collision-Frecency}
\nu(v) =  \int_{\R^3} | v - v_* | M(v_*) \d v_* \,.
\end{equation}

There have been extensive research on the well-posedness of the VMB.  DiPerna-Lions developed a theory of global-in-time renormalized solutions with large initial data, in particular to the Boltzmann equation \cite{DL-Annals1989}, Vlasov-Maxwell equations \cite{DL-CPAM1989} and Vlasov-Poisson-Boltzmann equation \cite{Lions-Kyoto1994, Lions-Kyoto1994-2}. But for VMB there are severe difficulties, among which the major one is that the a priori bounds coming from physical laws are not enough to prove the existence of global solutions, even in the renormalized sense. Recently, Ars\`enio and Saint-Raymond \cite{Arsenio-SRM, Arsenio-SRM-2016} eventually established global-in-time renormalized solutions with large initial data for VMB, both cut-off and non-cutoff collision kernels. We emphasize that by far renormalized solutions are still the only existing theory for solutions without any smallness requirements on initial data. On the other line, in the context of classical solutions, through a so-called nonlinear energy method, Guo \cite{Guo-2003-Invent} constructed a classical solution of VMB near the global Maxwellian. Guo's work inspired many results on VMB with more general collision kernels among which we only mention results for the most general collision kernels with or without angular cutoff assumptions, see \cite{DLYZ-KRM2013, DLYZ-CMP2017, FLLZ-2018}.

\subsection{Hydrodynamic limits}
One of the most important features of the Boltzmann equations (or more generally, kinetic equations) is its connection to the fluid equations. The so-called fluid regimes of the Boltzmann equation are those of asymptotic dynamics of the scaled Boltzmann equations when the Knudsen number $\eps$ is very small. Justifying these limiting processes rigorously has been an active research field from late 70's. Among many results obtained, the main contributions are the incompressible Navier-Stokes and Euler limits. There are two types of results in this field:
\begin{enumerate}
  \item First obtaining the solutions of the scaled Boltzmann equation {\em uniform} in the Knudsen number $\eps$, then extracting a convergent (at least weakly) subsequence converging to the solutions of the fluid equations as $\eps\rightarrow 0\,;$.
  \item First obtaining the solutions for the limiting fluid equations, then constructing a sequence of special solutions (around the Maxwellian) of the scaled Boltzmann equations for small Knudsen number $\eps$.
\end{enumerate}
The key difference between the results of type (1) and (2) are: in type (1), the solutions of the fluid equations are {\em not} known a priori, and are completely obtained from taking limits from the Boltzmann equation. In short, it is ``from kinetic to fluid"; In type (2), the solutions of the fluid equations are {\em known} first. In short, it is ``from fluid to kinetic".

The most successful program in type (1) is the so-called BGL program. As mentioned above, the DiPerna-Lions's renormalized solutions for cutoff kernel \cite{DL-Annals1989} (also the non-cutoff kernels in \cite{al-3}) are the only solutions known to exist globally without any restriction on the size of the initial data so far. From late 80's, Bardos-Golse-Levermore initialized the program (BGL program in brief) to justify Leray's solutions to the incompressible Navier-Stokes equations from DiPerna-Lions' renormalized solutions \cite{BGL-1991-JSP}, \cite{BGL-CPAM1993}. They proved the first convergence result with 5 additional technical assumptions. After 10 years effects by Bardos, Golse, Levermore, Lions and Saint-Raymond, see for example \cite{BGL3, LM3, LM4, GL}, the first complete convergence result without any additional compactness assumption was proved by Golse and Saint-Raymond in \cite{G-SRM-Invent2004} for cutoff Maxwell collision kernel, and in \cite{G-SRM2009} for hard cutoff potentials. Later on, it was extended by Levermore-Masmoudi \cite{LM} to include soft potentials. Recently Arsenio got the similar results for non-cutoff case \cite{Arsenio}. Furthermore, by Jiang, Levermore, Masmoudi and Saint-Raymond, these results were extended to bounded domain where the Boltzmann equation was endowed with the Maxwell reflection boundary condition \cite{Masmoudi-SRM-CPAM2003, JLM-CPDE2010, JM-CPAM2017}, based on the solutions obtained by Mischler \cite{mischler2010asens}.

The BGL program says that, given any $L^2\mbox{-}$bounded functions $(\rho_0, \mathrm{u}_0, \theta_0)$, and for any physically bounded initial data (as required in DiPerna-Lions solutions) $F_{\eps,0}= \mu + \eps \sqrt{\mu}g_{\eps,0}$, such that suitable moments of the fluctuation $g_{\eps,0}$, say, $(\mathcal{P}(g_{\eps,0}, v\sqrt{\mu})_{L^2(\mathbb{R}^3_v)}, (g_{\eps,0}, (\tfrac{|v|^2}{5}-1)\sqrt{\mu})_{L^2(\mathbb{R}^3_v)})$ converges in the sense of distributions to $(\mathrm{u}_0, \theta_0)$, the corresponding DiPerna-Lions solutions are $F_\eps(t,x,v)$. Then the fluctuations $g_{\eps}$ (defined by $F_\eps= \mu + \eps \sqrt{\mu}g_{\eps}$) has weak compactness, such that the corresponding moments of $g_\eps$ converge weakly in $L^1$ to $(\mathrm{u},\theta)$ which is a Leray solution of the incompressible Navier-Stokes equation whose viscosity and heat conductivity coefficients are determined by microscopic information, with initial data $(\mathrm{u}_0, \theta_0)$. Under some situations, for example the well-prepared initial data or in bounded domain with suitable boundary condition, the convergence could be strong $L^1$.

We emphasize that the BGL program indeed gave a new proof of Leray's solutions to the incompressible Navier-Stokes equation, in particular the energy inequality which can be derived from the entropy inequality of the Boltzmann equation. Any a priori information of the Navier-Stokes equation is {\em not} needed, and completely derived from the microscopic Boltzmann equation. In this sense, BGL program is spiritually a part of Hilbert's 6th problem: derive and justify the macroscopic fluid equations from the microscopic kinetic equations (see \cite{SRM2010}).

Another direction in type (1) is in the context of classical solutions. The first work in this type is Bardos-Ukai \cite{b-u}. They started from the scaled Boltzmann equation for cut-off hard potentials, and proved the global existence of classical solutions $g_\eps$ uniformly in $0< \eps <1$. The key feature of Bardos-Ukai's work is that they only need the smallness of the initial data, and did not assume the smallness of the Knudsen number $\eps$. After having the uniform in $\eps$ solutions $g_\eps$, taking limits can provide a classical solution of the incompressible Navier-Stokes equations with small initial data. Bardos-Ukai's approach heavily depends on the sharp estimate especially the spectral analysis on the linearized Boltzmann operator $\mathcal{L}$, and the semigroup method (the semigroup generated by the scaled linear operator $\eps^{-2}\mathcal{L}+\eps^{-1}v\cdot\nabla_{\!x}$). It seems that it is hardly extended to soft potential cutoff, and even harder for the non-cutoff cases, since it is well-known that the operator $\mathcal{L}$ has continuous spectrum in those cases. On the torus, semigroup approach has been employed by Briant \cite{Briant-JDE-2015} and Briant, Merino-Aceituno and Mouhot \cite{BMM-arXiv-2014} to prove incompressible Navier-Stokes limit by employing the functional analysis breakthrough of Gualdani-Mischler-Mouhot \cite{GMM}. Again, their results are for cut-off kernels with hard potentials. Recently, there is type (1) convergence result on the incompressible Navier-Stokes limit of the Boltzmann equation. In \cite{JXZ-Indiana2018}, the uniform in $\eps$ global existence of the Boltzmann equation with or without cutoff assumption was obtained and the global energy estimates were established. Then taking limit as $\eps\rightarrow 0$, it was proved the incompressible Navier-Stokes limit.

Most of the type (2) results are based on the Hilbert expansion  and obtained in the context of classical solutions. It was started from Nishida and Caflisch's work on the compressible Euler limit \cite{Nishida, Caflisch, KMN}. Their approach was revisitied by Guo, Jang and Jiang, combining with nonlinear energy method to apply to the acoustic limit \cite{GJJ-KRM2009, GJJ-CPAM2010, JJ-DCDS2009}. After then this process was used for the incompressible limits, for examples, \cite{DEL-89} and \cite{Guo-2006-CPAM}. In \cite{DEL-89}, De Masi-Esposito-Lebowitz considered  Navier-Stokes limit in dimension 2. More recently, using the nonlinear energy method, in \cite{Guo-2006-CPAM} Guo justified the Navier-Stokes limit (and beyond, i.e. higher order terms in Hilbert expansion). This result was extended in \cite{JX-SIMA2015} to more general initial data which allow the fast acoustic waves.  These results basically say that, given the initial data which is needed in the classical solutions of the Navier-Stokes equation, it can be constructed the solutions of the Boltzmann equation of the form $F_\eps = \mu + \eps \sqrt{\mu}(g_1+ \eps g_2 + \cdots + \eps^n g_\eps)$, where $g_1, g_2, \cdots $ can be determined by the Hilbert expansion, and $g_\eps$ is the error term. In particular, the first order fluctuation $g_1 = \rho_1 + \mathrm{u}_1\!\cdot\! v + \theta_1(\frac{|v|^2}{2}-\frac{3}{2})$, where $(\rho_1, \mathrm{u}_1, \theta_1)$ is the solutions to the incompressible Navier-Stokes equations.

\subsection{Hydrodynamic limits of Vlasov-Maxwell-Boltzmann system}
However, for the VMB, the corresponding hydrodynamic limits are much harder, even at the formal level, since it is coupled with Maxwell equations which are essentially hyperbolic. In a recent remarkable breakthrough \cite{Arsenio-SRM-2016}, Ars\'enio and Saint-Raymond not only proved the existence of renormalized solutions of VMB, as mentioned above, more importantly, also justified various limits (depending on the scalings) towards incompressible viscous electro-magneto-hydrodynamics. Among these limits, the most singular one is from renormalized solutions of two-species VMB to dissipative solutions of the two-fluid incompressible Navier-Stokes-Fourier-Maxwell (in brief, NSFM) system with Ohm's law. 

The proofs in \cite{Arsenio-SRM-2016} for justifying the weak limit from a sequence of solutions of VMB  to a dissipative solution of incompressible NSFM  are extremely hard. Part of the reasons are, besides many difficulties of the existence of renormalized solutions of VMB itself, our current understanding for the incompressible NSFM with Ohm's law is far from complete. From the point view of mathematical analysis, NSFM have a behavior which is more similar to the much less understood incompressible Euler equations than to the Navier-Stokes equations. That is the reason in \cite{Arsenio-SRM-2016}, they consider the so-called dissipative solutions of NSFM rather than the usual weak solutions. The dissipative solutions are were introduced by Lions for 3-dimensional incompressible Euler equations (see section 4.4 of \cite{Lions-1996}).

The studies of incompressible NSFM just started in recent years (for the introduction of physical background, see \cite{Biskamp, Davidson}).  For weak solutions, the existence of global in time Leray type weak solutions are completely open, even in 2-dimension.  A first breakthrough comes from Masmoudi \cite{Masmoudi-JMPA2010}, who in 2-dimensional case proved the existence and uniqueness of global strong solutions of incompressible NSFM (in fact, the system he considered in \cite{Masmoudi-JMPA2010} is little different with the NSFM in this paper, but the analytic analysis are basically the same) for the initial data $(v^{in}, E^{in}, B^{in})\in L^2(\mathbb{R}^2)\times (H^s(\mathbb{R}^2))^2$ with $s>0$. It is notable that in \cite{Masmoudi-JMPA2010}, the divergence-free condition of the magnetic field $B$ or the decay property of the linear part coming from Maxwell's equations is {\em not} used.  Ibrahim and Keraani \cite{Ibrahim-Keraani-2011-SIMA} considered the data $(u^{in}, E^{in}, B^{in}) \in \dot B^{1/2}_{2,1}(\mathbb{R}^3)\times (\dot H^{1/2}(\mathbb{R}^3))^2$ for 3-dimension, and $(v_0, E_0, B_0)\in \dot B^0_{2,1}(\mathbb{R}^2)\times (L^2_{log}(\mathbb{R}^2))^2$ for 2-dimensional case. Later on, German, Ibrahim and Masmoudi \cite{GIM2014} refines the previous results by running a fixed-point argument to obtain mild solutions, but taking the initial velocity field in the natural Navier-Stokes space $H^{1/2}$. In their results the regularity of the initial velocity and electromagnetic fields is lowered. Furthermore, they employed an $L^2L^\infty$-estimate on the velocity field, which significantly simplifies the fixed-point arguments used in \cite{Ibrahim-Keraani-2011-SIMA}. For some other asymptotic problems related, say, the derivation of the MHD from the Navier-Stokes-Maxwell system in the context of weak solutions, see Ars\'enio-Ibrahim-Masmoudi \cite{AIM-ARMA-2015}. Recently, in \cite{JL-CMS-2018} the authors of the current paper proved the global classical solutions of the incompressible NSFM with small intial data, by employing the decay properties of both the electric field and  the wave equation with linear damping of the divergence free magnetic field. This key idea was already used in \cite{GIM2014}. 

The only previous hydrodynamic limit result for the VMB for classical solutions belong to Jang \cite{Jang-ARMA2008}. In fact, in \cite{Jang-ARMA2008}, it was taken a very special scaling that the magnetic effect appears only at a higher order. As a consequence, it vanishes in the limit as the Knudsen number $\eps\rightarrow 0$. So in the limiting equations derived in \cite{Jang-ARMA2008}, there is no equation for the magnetic field at all. We emphasize that in \cite{Jang-ARMA2008}, the Hilbert expansion approach was employed, and the classical solutions to the VMB were constructed on those of the limiting equations. So the convergence results in \cite{Jang-ARMA2008} belong to the type (2) results, as we defined in the last subsection. 

The main concern of the current paper is to prove a type (1) convergence result in the context of classical solutions for two-species VMB. In the scaling which will be specified in the below, we will establish a uniform in $\eps$ estimates for the VMB. As a consequence, we can establish the global in time existence of solutions of VMB near the Maxwellians uniform in $\eps$. More over, as $\eps\rightarrow 0$, the moments of the fluctuations converge to the classical solutions of the two-fluid incompressible Navier-Stokes-Fourier-Maxwell (NSFM) system with Ohm's law. We emphasize that our result belongs to type (1) result. So we do not need any a priori information of NSFM. In fact, our convergence result automatically provides a global existence of NSFM with Ohm's law, of course, with small initial data.

We denote the Knudsen number by $\eps$ and the scaled two-species VMB reads
\begin{equation}\label{VMB-F}
  \left\{
    \begin{array}{l}
      \partial_t F_\eps^\pm + \tfrac{1}{\eps} v \cdot \nabla_x F_\eps^\pm \pm \tfrac{1}{\eps} ( \eps E_\eps + v \times B_\eps ) \cdot \nabla_v F_\eps^\pm = \tfrac{1}{\eps^2} \mathcal{B} (F_\eps^\pm, F_\eps^\pm) + \tfrac{1}{\eps^2} \mathcal{B} (F_\eps^\pm , F_\eps^\mp) \,,\\[2mm]
      \partial_t E_\eps - \nabla_x \times B_\eps = - \tfrac{1}{\eps^2} \int_{\R^3} (F_\eps^+ - F_\eps^-) v \d v \,,\\[2mm]
      \partial_t B_\eps + \nabla_x \times E_\eps = 0 \,,\\[2mm]
      \div_x E_\eps = \tfrac{1}{\eps} \int_{\R^3} ( F_\eps^+ - F_\eps^- ) \d v \,,\\[2mm]
      \div_x B_\eps = 0
    \end{array}
  \right.
\end{equation}
on $\mathbb{T}^3 \times \R^3$,  
Moreover, the initial data of \eqref{VMB-F} are imposed on
\begin{equation}\label{IC-VMB-F}
  \begin{aligned}
    F_\eps^\pm (0, x, v) = F_\eps^{\pm, in} (x,v) \in \R \,, \quad E_\eps (0,x) = E_\eps^{in} (x) \in \R^3 \,, \quad B_\eps (0, x) = B_\eps^{in} (x) \in \R^3 \,,
  \end{aligned}
\end{equation}

A physically relevant requirement for solutions to the Vlasov-Maxwell-Boltzmann system are their mass, momentum and energy are preserved with time. This is also an {\em a priori} property of the Vlasov-Maxwell-Boltzmann system on the torus (see \cite{Arsenio-SRM-2016} for instance) which reads
\begin{align}\label{Conservatn-Law-F}
\no & \frac{\d }{\d t}  \int_{\T^3 \times \R^3} F_\eps^\pm  \d x \d v = 0 \,, \\[2mm]
\no & \frac{\d}{\d t} \left( \int_{\T^3 \times \R^3} v ( F_\eps^+  + F_\eps^-  ) \d x \d v + \eps \int_{\T^3} E_\eps  \times B_\eps \d x \right) = 0 \,, \\[2mm]
& \frac{\d }{\d t} \left( \int_{\T^3 \times \R^3} |v|^2 ( F_\eps^+  + F_\eps^-  ) \d v \d x + \eps^2 \int_{\T^3} ( |E_\eps |^2 + |B_\eps  |^2 ) \d x \right) = 0 \,.
\end{align}
Notice from the Maxwell system and the periodic boundary condition of $E_\eps (t,x)$,
\begin{equation}
\frac{\d}{\d t} \int_{\T^3} B_\eps (t,x) \d x = 0 \,.
\end{equation}
We thus can define a constant vector 
\begin{equation}\label{B-average}
\tfrac{1}{|\T^3|} \int_{\T^3} B_\eps (t,x) = \bar{B}= \tfrac{1}{|\T^3|} \int_{\T^3} B_\eps^{in} (x) \d x \,.
\end{equation}

It is well-known that the global equilibrium for the two-species Vlasov-Maxwell-Boltzmann system is $( [M(v), M(v)],0,\bar{B})$, where the normalized global {\em Maxwellian} $M(v)$ is
\begin{equation*}
  M(v) = \tfrac{1}{(2 \pi)^\frac{3}{2}} \exp(- \tfrac{|v|^2}{2}) \,.
\end{equation*}
and $\bar{B} \in \R^3$ is a constant vector. Our goal is to approximate the two-fluid incompressible Navier-Stokes-Fourier-Maxwell system with Ohm's law by the system \eqref{VMB-F} under the linearization $F_\eps^\pm (t,x,v) = M(v)  + \eps \sqrt{M(v)} G_\eps^\pm (t,x,v) $ when the Knudsen number $\eps$ tends to zero. This leads to the perturbed two-species Vlasov-Maxwell-Boltzmann system
\begin{equation}\label{VMB-G}
  \left\{
    \begin{array}{l}
      \partial_t G_\eps + \tfrac{1}{\eps} \big[ v \cdot  \nabla_x G_\eps + \textsf{q} ( \eps E_\eps + v \times B_\eps ) \cdot \nabla_v \big]  G_\eps + \tfrac{1}{\eps^2} \mathscr{L} G_\eps - \tfrac{1}{\eps}  (E_\eps \cdot v) \sqrt{M}  \textsf{q}_1\\[2mm]
       \qquad \qquad = \tfrac{1}{2} \mathsf{q} (E_\eps \cdot v) G_\eps + \tfrac{1}{\eps} \Gamma (G_\eps, G_\eps) \,,\\[2mm]
      \partial_t E_\eps - \nabla_x \times B_\eps = - \tfrac{1}{\eps} \int_{\R^3} G_\eps \cdot \textsf{q}_1 v \sqrt{M} \d v \,,\\[2mm]
      \partial_t B_\eps + \nabla_x \times E_\eps = 0 \,,\\[2mm]
      \div_x E_\eps = \int_{\R^3} G_\eps \cdot \textsf{q}_1 \sqrt{M} \d v \,, \ \div_x B_\eps = 0 \,,
    \end{array}
  \right.
\end{equation}
where $G_\eps = [ G_\eps^+ , G_\eps^- ]$ represents the column vector in $\R^2$ with the components $G_\eps^\pm$, the $2 \times 2$ diagonal matrix $\textsf{q} = \textit{diag} (1, -1)$, the column vector $\textsf{q}_1 = [1, -1]$, the two species linearized collision operator $\mathscr{L}$ is given as
\begin{equation}\label{Linear-Oprt-VMB}
  \begin{aligned}
    \mathscr{L} G_\eps = \big[ \mathcal{L} G_\eps^+ + \mathcal{L} (G_\eps^+, G_\eps^-) \,, \ \mathcal{L} G_\eps^- + \mathcal{L} ( G_\eps^-, G_\eps^+ ) \big] \,,
  \end{aligned}
\end{equation}
where 
\begin{equation}\label{Linear-Oprt-B}
  \begin{aligned}
    \mathcal{L} g & = - \tfrac{1}{\sqrt{M}} [ \mathcal{B} ( \sqrt{M} g, M ) + \mathcal{B} (M, \sqrt{M} g) ] \\
    & = - \mathcal{Q} (g, \sqrt{M}) - \mathcal{Q} ( \sqrt{M} , g ) \\
    & = \sqrt{M} \int_{\R^3} \Big( \tfrac{g}{\sqrt{M}} + \tfrac{g_*}{\sqrt{M_*}} - \tfrac{g'}{\sqrt{M'}} - \tfrac{g_*'}{\sqrt{M_*'}} \Big) |v - v_*| M_* \d v_* 
  \end{aligned}
\end{equation}
is the usual linearized Boltzmann collision operator, and 
\begin{equation}
  \begin{aligned}
    \mathcal{L} (g, h) & =- \tfrac{1}{\sqrt{M}} [ \mathcal{B} ( \sqrt{M} g, M ) + \mathcal{B} (M, \sqrt{M} h) ] \\
    & = - \mathcal{Q} (g, \sqrt{M}) - \mathcal{Q} ( \sqrt{M} , h ) \\
    & = \sqrt{M} \int_{\R^3} \Big( \tfrac{g}{\sqrt{M}} + \tfrac{h_*}{\sqrt{M_*}} - \tfrac{g'}{\sqrt{M'}} - \tfrac{h_*'}{\sqrt{M_*'}} \Big) |v - v_*| M_* \d v_* \,.
  \end{aligned}
\end{equation}
Here we denote by $ \mathcal{Q} (g,h) = \tfrac{1}{\sqrt{M}} \mathcal{B} (\sqrt{M} g, \sqrt{M} h ) $. More specifically,
\begin{equation}
  \mathcal{Q} (g,h) =  \int_{\R^3} \int_{\mathbb{S}^2} ( g'h_*' - g h_* ) b ( |v - v_*|, \cos \theta ) \d \sigma \d v_* = \int_{\R^3} ( g'h_*' - g h_* ) |v - v_*| \sqrt{M_*} \d v_* \,.
\end{equation}
We then define a bilinear symmetric operator $\Gamma (G, H)$ as
\begin{equation}\label{Gamma}
 \begin{aligned}
   \Gamma (G , H ) = \tfrac{1}{2} [ \mathcal{Q} ( G^+, H^+ ) & + \mathcal{Q} ( H^+, G^+ ) + \mathcal{Q} ( G^+, H^- ) + \mathcal{Q} ( H^+, G^- ) \,, \\
   & \mathcal{Q} (G^-, H^-) + \mathcal{Q} (H^-, G^-) + \mathcal{Q} (G^- , H^+) + \mathcal{Q} ( H^-, G^+ ) ]
 \end{aligned}
\end{equation}
for vector-valued functions $G(v) = [ G^+ (v) , G^- (v) ]$ and $ H(v) = [ H^+ (v) , H^- (v) ] $.

The two species linearized collision operator $\mathscr{L}$ has several properties (see \cite{Arsenio-SRM-2016}, for instance), which will be used throughout this paper. The linear operator $\mathscr{L}$ is a closed self-adjoint operator in $L^2_v$ with kernel
\begin{equation*}
\textrm{Ker} (\mathscr{L}) = \textrm{Span} \{ \phi_1 (v) \,, \cdots \phi_6 (v) \} \,,
\end{equation*}
where $\phi_1 (v) = [1 , 0]  \sqrt{M} = \tfrac{1}{2} ( \mathsf{q}_1 + \mathsf{q}_2 ) \sqrt{M} $, $\phi_2 (v) = [0,1]  \sqrt{M} = \tfrac{1}{2} ( \mathsf{q}_2 - \mathsf{q}_1 ) \sqrt{M} $, $\phi_{i+2} (v) = [v_i , v_i]  \sqrt{M} = v_i \mathsf{q}_2 \sqrt{M} $ for $i = 1,2,3$ and $\phi_6 (v) = \tfrac{1}{2} [|v|^2 - 3 , |v|^2 - 3 ]  \sqrt{M} = ( \tfrac{|v|^2}{2} - \tfrac{3}{2} ) \mathsf{q}_2 \sqrt{M}$. Here the column vectors $\mathsf{q}_2 = [1,1]$. The family $(\phi_i (v))_{1 \leq i \leq 6}$ is an orthonormal basis of $\textrm{Ker} (\mathscr{L})$ in $L^2_v $ and we denote $\P$ by the orthogonal projection onto $\textrm{Ker}(\mathscr{L})$ in $L^2_v $:
\begin{equation}\label{VMB-Proj}
\P G (v) = \sum_{i=1}^2 \langle G,   \phi_i \rangle_{L^2_v} \phi_i (v) + \sum_{i=1}^3 \tfrac{1}{2} \langle G,   \phi_{i+2} \rangle_{L^2_v} \phi_{i+2} (v) +\tfrac{1}{3} \l G , \phi_6 \r_{L^2_v} \phi_6 (v)
\end{equation}
for all $G : \R^3 \mapsto \R^2$ in $L^2_v$. For notational simplicity, we denote by $\rho^+ =\langle G, \phi_1 (v) \rangle_{L^2_v} $, $\rho^- =\langle G, \phi_2 (v) \rangle_{L^2_v} $, $ u_i = \tfrac{1}{2} \langle G,   \phi_{i+2} \rangle_{L^2_v} $ for $i = 1,2,3$ and $\theta = \tfrac{1}{3} \l G , \phi_6 \r_{L^2_v} $, so that we can rewrite
\begin{equation}\label{VMB-Proj-Smp}
  \begin{aligned}
    \mathbb{P} G = & \rho^+ \phi_1 (v) + \rho^- \phi_2 (v) + \sum_{i=1}^3 u_i \phi_{i+2} (v) + \theta \phi_6 (v) \\
    = & \rho^+ \tfrac{\mathsf{q}_1 + \mathsf{q}_2}{2} \sqrt{M} + \rho^- \tfrac{\mathsf{q}_2 - \mathsf{q}_1}{2} \sqrt{M} + u \cdot v \mathsf{q}_2 \sqrt{M} + \theta ( \tfrac{|v|^2}{2} - \tfrac{3}{2} ) \mathsf{q}_2 \sqrt{M} \\
    = & ( \rho + \tfrac{1}{2} n ) \tfrac{\mathsf{q}_1 + \mathsf{q}_2}{2} \sqrt{M} + ( \rho - \tfrac{1}{2} n ) \tfrac{\mathsf{q}_2 - \mathsf{q}_1}{2} \sqrt{M} + u \cdot v \mathsf{q}_2 \sqrt{M} + \theta ( \tfrac{|v|^2}{2} - \tfrac{3}{2} ) \mathsf{q}_2 \sqrt{M} \,,
  \end{aligned}
\end{equation}
where $\rho = \tfrac{1}{2} (\rho^+ + \rho^-)$ and $n = \rho^+ - \rho^-$. We also define $\P^\perp = \mathbb{I} - \P $, where $\mathbb{I}$ is the identical mapping.

For the usual linearized Boltzmann collision operator $\mathcal{L}$ defined in \eqref{Linear-Oprt-B}, it is well known that the kernel of $\mathcal{L}$ is
\begin{equation*}
\textrm{Ker} ( \mathcal{L} ) = \textrm{Span} \big\{ \chi_1 (v) \,, \ \chi_2 (v) \,, \ \chi_3 (v) \,, \ \chi_4 (v) \,, \ \chi_5 (v) \big\} \,,
\end{equation*}
where $\chi_1 (v) = \sqrt{M(v)}$, $\chi_{1+i} (v) = v_i \sqrt{M(v)}$ for $i = 1, 2, 3$, and $\chi_5 (v) = \big(\tfrac{|v|^2}{2} - \tfrac{3}{2}\big) \sqrt{M(v)}$. The collection $\big( \chi_i (v) \big)_{1 \leq i \leq 5}$ consists of an orthonormal basis of $ \textrm{Ker} ( \mathcal{L} ) $ in $L^2_v$. One then can define the orthogonal projection $\Pi_{\mathcal{L}}$ on to $ \textrm{Ker} ( \mathcal{L} ) $ in $L^2_v$ as
\begin{equation}\label{Boltzmann-Proj}
\Pi_{\mathcal{L}} g (v) = \sum_{i=1}^4 \langle g , \chi_i \rangle_{L^2_v} \chi_i (v) + \tfrac{2}{3} \l g , \chi_5 \r_{L^2_v} \chi_5 (v)
\end{equation}
for $g : \R^3 \mapsto \R$ in $L^2_v$. We denote by $\Pi^\perp_{\mathcal{L}} = \mathbb{I} - \Pi_{\mathcal{L}}$. By direct verification, we know that for all $G = [ G^+, G^- ] : \R^3 \mapsto \R^2$ in $L^2_v$
\begin{equation}\label{Relation-VMN-Boltzmann-Fluid}
  \begin{aligned}
    \langle G^+ , \chi_1 (v) \rangle_{L^2_v} = & \langle G, \phi_1 (v) \rangle_{L^2_v} = \tfrac{1}{2} \langle G , (\mathsf{q}_1 + \mathsf{q}_2) \sqrt{M} \rangle_{L^2_v} \,, \\[2mm]
     \ \langle G^- , \chi_1 (v) \rangle_{L^2_v} = & \langle G, \phi_2 (v) \rangle_{L^2_v} = \tfrac{1}{2} \langle G, (\mathsf{q}_2 - \mathsf{q}_1) \sqrt{M} \rangle_{L^2_v} \,, \\[2mm]
    \langle G^\pm , \chi_{1+i} (v) \rangle_{L^2_v} = & \tfrac{1}{2} \langle G , \phi_{2+i} (v) \rangle_{L^2_v} \pm \tfrac{1}{2} \langle G, \textsf{q} \phi_{2+i} (v) \rangle_{L^2_v} \\[2mm]
    = & \tfrac{1}{2} \langle G , \mathsf{q}_2 v_i \sqrt{M} \rangle_{L^2_v} \pm \tfrac{1}{2} \langle G , \mathsf{q}_1 v_i \sqrt{M} \rangle_{L^2_v} \,, \ (i= 1, 2, 3) \\[2mm]
    \langle G^\pm , \chi_5 (v) \rangle_{L^2_v} = & \tfrac{1}{3} \langle G , \phi_6 (v) \rangle_{L^2_v} \pm \tfrac{1}{2} \langle G, \textsf{q} \phi_6 (v) \rangle_{L^2_v} \\[2mm]
    = & \tfrac{1}{2} \langle G , \mathsf{q}_2 ( \tfrac{|v|^2}{3} - 1 ) \sqrt{M} \rangle_{L^2_v} \pm \tfrac{1}{2} \langle G, \mathsf{q}_1 ( \tfrac{|v|^2}{3} - 1 ) \sqrt{M} \rangle_{L^2_v} \,.
  \end{aligned}
\end{equation}
We denote briefly by
\begin{equation}\label{rho-u-theta}
  \begin{aligned}
    & \rho^+ = \l G^+, \chi_1 \r_{L^2_v} = \l G , \phi_1 \r_{L^2_v} \,, \ \rho^- = \l G^-, \chi_1 \r_{L^2_v} = \l G , \phi_2 \r_{L^2_v} \,,\\
    & u^\pm_i = \l G^\pm , \chi_{i+1} \r_{L^2_v} \ \textrm{for } i = 1, 2, 3 , \ \theta^\pm = \l G^\pm , \tfrac{2}{3} \chi_5 \r_{L^2_v} \,,\\
    & u_i = \frac{1}{2} (u^+_i + u^-_i) = \tfrac{1}{2} \langle G , \phi_{2+i} (v) \rangle_{L^2_v} \,, \ \theta = \frac{1}{2} (\theta^+ + \theta^-) = \tfrac{1}{3} \langle G , \phi_6 (v) \rangle_{L^2_v}
  \end{aligned}
\end{equation}
for all $G(x,v) = [G^+(x,v), G^-(x,v)]$ belonging to $ L^2_{x,v}$. 

By assuming that the initial data $ ( F_\eps^{in} = [ F_\eps^{+, in} , F_\eps^{-, in} ] , E_\eps^{in}, B_\eps^{in} )$ has the same mass, momentum and energy as the steady state $( M \mathsf{q}_2 , 0 , \bar{B} )$, we can then rewrite the the conservation laws \eqref{Conservatn-Law-F} in terms of the perturbation $ ( G_\eps = [ G_\eps^+ , G_\eps^- ] , E_\eps , B_\eps ) $ as
\begin{equation}\label{Conservatn-Law-G}
\left\{
\begin{array}{l}
\int_{\T^3 \times \R^3} G_\eps (t,x,v) \cdot \phi_i (v) \d x \d v = 0 \quad \textrm{for } i = 1, 2\,, \\[4mm]
\int_{\T^3 \times \R^3} v G_\eps (t,x,v) \cdot \ \textsf{q}_2 \sqrt{M(v)} \d x \d v +  \int_{\T^3} E_\eps (t,x) \times B_\eps (t,x) \d x = 0 \,, \\[4mm]
\int_{\T^3 \times \R^3}  G_\eps (t,x,v) \cdot \phi_6 (v) \d v \d x + \eps \int_{\T^3} ( |E_\eps (t,x)|^2 + |B_\eps (t,x) - \bar{B} |^2 ) \d x = 0 
\end{array}
\right.
\end{equation}
for all $t \geq 0$. Then, without loss of generality, the initial conditions of \eqref{VMB-G} shall be imposed on
\begin{equation}\label{IC-VMB-G}
  \begin{aligned}
    G_\eps (0,x,v) = G_\eps^{in} (x,v) \in \R^2 \,, \quad E_\eps (0,x) = E_\eps^{in} (x) \in \R^3 \,, \quad B_\eps (0,x) = B_\eps^{in} (x) \in \R^3 \,,
  \end{aligned}
\end{equation}
which satisfy the conservation laws \eqref{Conservatn-Law-G} and the constraint $\eqref{VMB-G}_4$ initially.

\subsection{Notations.} We gather here the notations we will use throughout this paper. We first define the following shorthand notation,
\begin{equation*}
  \langle \cdot \rangle = \sqrt{1 + |\cdot|^2} \,.
\end{equation*}
We denote the symbol $(f)_{\T^3}$ by the average of the function $f(x)$ over $x \in \T^3$, namely,
\begin{equation*}
  (f)_{\T^3} \overset{\Delta}{=} \frac{1}{|\T^3|} \int_{\T^3} f(x) \d x \,.
\end{equation*}
For convention, we index the usual $L^p$ space by the name of the concerned variable. So we have, for $p \in [1, + \infty]$,
\begin{equation*}
  L^p_{[0,T]} = L^p ([0,T]) \,, \ L^p_x = L^p (\T^3) \,, \ L^p_v = L^p (\R^3) \,, \ L^p_{x,q} = L^p (\mathbb{T}^3 \times \mathbb{R}^3) \,.
\end{equation*}
For $p = 2$, we use the notations $ \langle \cdot \,, \cdot \rangle_{L^2_x} $, $ \langle \cdot \,, \cdot \rangle_{L^2_v} $ and $ \langle \cdot \,, \cdot \rangle_{L^2_{x,v}} $ to represent the inner product on the Hilbert spaces $L^2_x$, $L^2_v$ and $L^2_{x,v}$, respectively. 

Let $ w : \R^3_v \rightarrow \R^+$ be a strictly positive measurable function. For $p \in [1, + \infty )$, we denote the $w$-weighted $L^p$ spaces by
\begin{equation*}
  L^p_v(w) = L^p (\mathbb{R}^3; w \d v) \,, \ L^p_{x,v} (w) = L^p ( \mathbb{T}^3 \times \mathbb{R}^3; w \d v \d x )
\end{equation*}
with the norms
\begin{equation*}
  \begin{aligned}
    & \| f \|_{L^p_v (w)} = \left( \int_{\R^3} | f(v) |^p w(v) \d v \right)^\frac{1}{p} < + \infty \,, \\
    & \| g \|_{L^p_{x,q} (w)} = \left( \int_{\T^3 \times \R^3} | g(x,v) |^p w(v) \d v \d x \right)^\frac{1}{p} < + \infty \,.
  \end{aligned}
\end{equation*}

For any multi-indexes $\alpha = (\alpha_1, \alpha_2, \alpha_3)$ and $ m = ( m_1, m_2, m_3 )$ in $\mathbb{N}^3$ we denote the $(m,\alpha)^{th}$ partial derivative by
\begin{equation*}
  \partial^m_\alpha = \partial^m_x \partial^\alpha_v = \partial^{m_1}_{x_1} \partial^{m_2}_{x_2} \partial^{m_3}_{x_3} \partial^{\alpha_1}_{v_1} \partial^{\alpha_2}_{v_2} \partial^{\alpha_3}_{v_3} \,.
\end{equation*}
If each component of $m \in \mathbb{N}^3$ is not greater than that of $\tilde{m}$'s, we denote by $m \leq \tilde{m}$. The symbol $m < \tilde{m}$ means $m \leq \tilde{m}$ and $|m| < | \tilde{m} |$, where $|m|= m_1 + m_2 + m_3$.

We define the spaces $H^s_{x,v} $ and $H^s_x L^2_v$ by the norms
\begin{align*}
  \| f \|_{H^s_{x,q}} = \bigg(  \sum_{|m| + |\alpha| \leq s} \| \partial_\alpha^m f \|^2_{L^2_{x,v}} \bigg)^\frac{1}{2} \,, \  \| f \|_{H^s_x L^2_v} = \bigg(  \sum_{|m| \leq s} \| \partial^m_x f \|^2_{L^2_{x,v}} \bigg)^\frac{1}{2} \,.
\end{align*}
Furthermore, we give the spaces $H^s_{x,v} (w)$ and $H^s_x L^2_v (w)$ with norms
\begin{equation*}
  \| f \|_{H^s_{x,v} (w)} = \bigg( \sum_{|m| + |\alpha| \leq s} \|  \partial^m_\alpha f  \|_{L^2_{x,v} (w) } \bigg)^\frac{1}{2} \,, \ \| f \|_{H^s_x L^2_v (w) } = \bigg(  \sum_{|m| \leq s} \| \partial^m_x f \|^2_{L^2_{x,v} (w) } \bigg)^\frac{1}{2} \,.
\end{equation*}
We also introduce the spaces $ \widetilde{H}^s_{x,v} $ and $\widetilde{H}^s_{x,v} (\nu)$ endowed with the norms
\begin{equation*}
  \| f \|_{\widetilde{H}^s_{x,v}} = \Bigg( \sum_{\substack{|m| + |\alpha| \leq s \\ |\alpha| \geq 1}} \| \partial^m_\alpha f \|^2_{L^2_{x,v}} \Bigg)^\frac{1}{2} \,, \ \| f \|_{\widetilde{H}^s_{x,v} (\nu)} = \Bigg( \sum_{\substack{|m| + |\alpha| \leq s \\ |\alpha| \geq 1}} \| \partial^m_\alpha f \|^2_{L^2_{x,v} (\nu)} \Bigg)^\frac{1}{2} \,.
\end{equation*}

\subsection{Main results.} 
There are two theorems built in this paper. The first theorem is about the global existence of the two-species VMB system uniform with respect to the Knudsen number $0 < \eps \leq 1$. The second is on the two-fluid incompressible Navier-Stokes-Fourier-Maxwell limit with Ohm's law as $\eps \rightarrow 0$, taken from the solutions $(G_\eps , E_\eps , B_\eps)$ of the two-species VMB system \eqref{VMB-G} which are constructed in the first theorem.

To state our main theorems, we introduce the following energy functional and dissipation rate functional respectively
\begin{equation}\label{Energy-E-D}
  \begin{aligned}
    \mathbb{E}_s (G,E,B) = & \| G \|^2_{H^s_{x,v}} + \| E \|^2_{H^s_x} + \| B \|^2_{H^s_x} \,, \\
    \mathbb{D}_s (G,E,B) = & \tfrac{1}{\eps^2} \| \mathbb{P}^\perp G \|^2_{H^s_{x,v}(\nu)} + \| \nabla_x \mathbb{P} G \|^2_{H^{s-1}_x L^2_v} + \| E \|^2_{H^{s-1}_x} + \| \nabla_x B \|^2_{H^{s-2}_x} \,.
  \end{aligned}
\end{equation}

\begin{theorem}\label{Main-Thm-1}
	For the integer $s \geq 3$ and $0 < \eps \leq 1$, there are constants $\ell_0 > 0$, $c_0 > 0$ and $c_1 > 0$, independent of $\eps$, such that if $\mathbb{E}_s (G_\eps^{in} , E_\eps^{in}, B_\eps^{in}) \leq \ell_0$, then the Cauchy problem \eqref{VMB-G}-\eqref{IC-VMB-G} admits a global solution
	\begin{equation}
	  \begin{aligned}
	    & G_\eps (t,x,v) \in L^\infty_t ( \R^+ ; H^s_{x,v} ) \,, \mathbb{P}^\perp G_\eps (t,x,v) \in L^2_t (\R^+ ; H^s_{x,v}(\nu)) \,, \\
	    & E_\eps (t,x), B_\eps (t,x) \in L^\infty_t ( \R^+ ; H^s_x )
	  \end{aligned}
	\end{equation}
	with the global uniform energy estimate
	\begin{equation}\label{Uniform-Bnd}
	  \begin{aligned}
	    \sup_{t \geq 0} \mathbb{E}_s ( G_\eps , E_\eps , B_\eps ) (t) + c_0 \int_0^\infty \mathbb{D}_s (G_\eps , E_\eps , B_\eps) (t) \d t \leq c_1 \mathbb{E}_s (G_\eps^{in} , E_\eps^{in}, B_\eps^{in}) \,.
	  \end{aligned}
	\end{equation}
\end{theorem}

The next theorem is about the limit to the two fluid incompressible Navier-Stokes-Fourier-Maxwell system with Ohm's law:
\begin{equation}\label{INSFM-Ohm}
  \left\{
    \begin{array}{l}
      \partial_t u + u \cdot \nabla_x u - \mu \Delta_x u + \nabla_x p = \tfrac{1}{2} ( n E + j \times B ) \,, \qquad \div_x \, u = 0 \,, \\ [2mm]
      \partial_t \theta + u \cdot \nabla_x \theta - \kappa \Delta_x \theta = 0 \,, \qquad\qquad\qquad\qquad\qquad\quad\ \, \rho + \theta = 0 \,, \\ [2mm]
      \partial_t E - \nabla_x \times B = - j \,, \qquad\qquad\qquad\qquad\qquad\qquad\ \ \ \, \div_x \, E = n \,, \\ [2mm]
      \partial_t B + \nabla_x \times E = 0 \,, \qquad\qquad\qquad\qquad\qquad\qquad\qquad \div_x \, B = 0 \,, \\ [2mm]
      \qquad \qquad j - nu = \sigma \big( - \tfrac{1}{2} \nabla_x n + E + u \times B \big) \,, \qquad\quad\,\ w = \tfrac{3}{2} n \theta \,,
    \end{array}
  \right.
\end{equation}
where the viscosity $\mu$, the heat conductivity $\kappa$ and the electrical conductivity $\sigma$ are given by
\begin{equation}\label{VHE-Coefficients}
  \begin{aligned}
    \mu = \tfrac{1}{10} \int_{\R^3} A : \widehat{A} M \d v \,, \quad \kappa = \tfrac{2}{15} \int_{\R^3} B \cdot \widehat{B} M \d v \quad \textrm{and} \quad \sigma = \tfrac{2}{3} \int_{\R^3} \Phi \cdot \widetilde{\Phi} M \d v \,.
  \end{aligned}
\end{equation}
For the derivation of \eqref{VHE-Coefficients}, i.e. the relation of  $\mu$, $\kappa$, $\sigma$ with $A$, $\widehat{A}$, $B$, $\widehat{B}$, $\Phi$ and $\widetilde{\Phi}$, see \cite{Arsenio-SRM-2016}.

\begin{theorem}\label{Main-Thm-2}
	Let $0 < \eps \leq 1$, $s \geq 3$ and $\ell_0 > 0$ be as in Theorem \ref{Main-Thm-1}. Assume that the initial data $( G_\eps^{in} , E_\eps^{in}, B_\eps^{in} )$ in \eqref{IC-VMB-G} satisfy
	\begin{enumerate}
		\item $G_\eps^{in} \in H^s_{x,v}$, $E_\eps^{in}$, $B_\eps^{in} \in H^s_x$;
		\item $\mathbb{E}_s (G_\eps^{in}, E_\eps^{in}, B_\eps^{in}) \leq \ell_0$;
		\item there exist scalar functions $\rho^{in} (x)$, $ \theta^{in} (x) $,  $n^{in}(x) \in H^s_x$ and vector-valued functions $u^{in}(x) $,  $E^{in} (x)$, $B^{in} (x) \in H^s_x$ such that 
		\begin{equation}
		  \begin{array}{l}
			G_\eps^{in} \rightarrow G^{in} \quad \textrm{strongly in} \ H^s_{x,v} \,, \\
			E_\eps^{in} \rightarrow E^{in} \quad \textrm{strongly in} \ H^s_x \,, \\
			B_\eps^{in} \rightarrow B^{in} \quad \textrm{strongly in} \ H^s_x
		  \end{array}
		\end{equation}
		as $\eps \rightarrow 0$, where $G^{in} (x,v)$ is of the form
		\begin{equation}
		  \begin{aligned}
		    G^{in} (x,v) = & ( \rho^{in} (x) + \tfrac{1}{2} n^{in} (x) ) \tfrac{\mathsf{q}_1 + \mathsf{q}_2}{2} \sqrt{M} + ( \rho^{in} (x) - \tfrac{1}{2} n^{in} (x) ) \tfrac{\mathsf{q}_2 - \mathsf{q}_1}{2} \sqrt{M} \\
		    & + u^{in} \cdot v \mathsf{q}_2 \sqrt{M} + \theta^{in} ( \tfrac{|v|^2}{2} - \tfrac{3}{2} ) \mathsf{q}_2 \sqrt{M} \,.
		  \end{aligned}
		\end{equation}
	\end{enumerate}
  Let $(G_\eps, E_\eps, B_\eps)$ be the family of solutions to the perturbed two-species Vlasov-Maxwell-Boltzmann \eqref{VMB-G} with the initial conditions \eqref{IC-VMB-G} constructed in Theorem \ref{Main-Thm-1}. Then, as $\eps \rightarrow 0$,
  \begin{equation}
    \begin{aligned}
      G_\eps \rightarrow ( \rho + \tfrac{1}{2} n ) \tfrac{\mathsf{q}_1 + \mathsf{q}_2}{2} \sqrt{M} + ( \rho - \tfrac{1}{2} n ) \tfrac{\mathsf{q}_2 - \mathsf{q}_1}{2} \sqrt{M} + u \cdot v \mathsf{q}_2 \sqrt{M} + \theta ( \tfrac{|v|^2}{2} - \tfrac{3}{2} ) \sqrt{M}
    \end{aligned}
  \end{equation}
  weakly-$\star$ in $t \geq 0$, strongly in $H^{s-1}_{x,v}$ and weakly in $H^s_{x,v}$, and
  \begin{equation}
    \begin{aligned}
      E_\eps \rightarrow E \quad \textrm{and } \quad B_\eps \rightarrow B
    \end{aligned}
  \end{equation}
  strongly in $C( \R^+; H^{s-1}_x )$, weakly-$\star$ in $t \geq 0$ and weakly in $H^s_x$. Here 
  $$( u, \theta , n, E, B ) \in C(\R^+; H^{s-1}_x) \cap L^\infty (\R^+ ; H^s_x)$$ 
  is the solution to the incompressible Navier-Stokes-Fourier-Maxwell equations \eqref{INSFM-Ohm} with Ohm's law, which has the initial data
  \begin{equation}
    \begin{aligned}
      u|_{t=0} = \mathcal{P} u^{in} (x) \,, \ \theta |_{t=0} = \tfrac{3}{5} \theta^{in} (x) - \tfrac{2}{5} \rho^{in} (x) \,, \ E|_{t=0} = E^{in} (x) \,, \ B|_{t=0} = B^{in} (x) \,,
    \end{aligned}
  \end{equation}
  where $\mathcal{P}$ is the Leray projection. Moreover, the convergence of the moments holds: 
  \begin{equation}
    \begin{aligned}
      & \mathcal{P} \langle G_\eps , \tfrac{1}{2} \mathsf{q}_2 v \sqrt{M} \rangle_{L^2_v} \rightarrow u \,, \\
      & \langle G_\eps , \tfrac{1}{2} \mathsf{q}_2 ( \tfrac{|v|^2}{5} - 1 ) \sqrt{M} \rangle_{L^2_v} \rightarrow \theta \,,
    \end{aligned}
  \end{equation}
  strongly in $C(\R^+ ; H^{s-1}_x)$, weakly-$\star$ in $t \geq 0$ and weakly in $H^s_x$ as $\eps \rightarrow 0$.
\end{theorem}

The organization of this paper is as follows: in the next section, we give some basic properties of the linear collision operator and bilinear symmetric operator. In Section \ref{Sec:Unif-Spatial-Bnd}, the spatial derivative estimates are derived, which are not closed. Then we derive the mixed derivative estimate and obtain a closed uniform energy inequality in Section \ref{Sec:Unif-Mix-Bnd-Global}. Moreover, we also prove the global well-posedness under the small initial data for all $\eps \in (0,1]$. In Section \ref{Sec:Limits}, based on the uniform global in time energy bound, we take the limit to derive the incompressible NSFM system with Ohm's law. Finally, we construct the local classical solutions for all $0 < \eps \leq 1$ under the small size of the initial data.

\section{Preliminaries}\label{Sec:Preliminaries}

In this section we focus on some basic properties of the two species linearized collision operator $\mathscr{L}$  and the bilinear symmetric operator $\Gamma$, which will be frequently used in the estimating of uniform energy bounds of the perturbed VMB system \eqref{VMB-G}.

\begin{lemma}[Collisional frequency]\label{Lmm-CF-nu}
	The collision frequency $\nu(v)$ defined in \eqref{Collision-Frecency} has the following properties:
	\begin{enumerate}
		\item $\nu(v)$ is smooth and there are positive constants $C_1$ and $C_2$ such that
		\begin{equation}\label{nu-Propty-1}
		  C_1 ( 1 + |v| ) \leq  \nu(v) \leq C_2 ( 1 + |v| )
		\end{equation}
		for every $v \in \R^3$.
		
		\item For any $\alpha \in \mathbb{N}^3$, $\alpha \neq 0$,
		\begin{equation}\label{nu-Propty-2}
		  \sup_{v \in \R^3} | \partial^\alpha_v \nu(v) | < + \infty \,.
		\end{equation}
		
		\item If the velocities $v$, $v_*$, $v'$, $v_*' \in \R^3$ satisfy $ v + v_* = v' + v_*' $ and $ |v|^2 + |v_*|^2 = |v'|^2 + |v_*'|^2 $, then
		\begin{equation}\label{nu-Propty-3}
		  \nu(v) + \nu(v_*) \leq C_3 ( \nu(v') + \nu(v_*') )
		\end{equation}
		holds for some positive constant $C_3$.
	\end{enumerate}
\end{lemma}

\begin{proof}[Proof of Lemma \ref{Lmm-CF-nu}]
	We indeed can prove more general conclusions corresponding to Lemma \ref{Lmm-CF-nu}. More precisely, we consider $\nu (v) = \int_{\R^3} |v - v_*|^\gamma M(v_*) \d v_*$ for any $\gamma \in [0,1]$, which is actually the collision frequency with respect to the hard potential ($\gamma \in ( 0, 1 ]$) and Maxwellian ($\gamma = 0$) collision kernel. In this case, the inequality \eqref{nu-Propty-1} will be
	\begin{equation}\label{nu-Propty-1*}
	  C_1 ( 1 + |v| )^\gamma \leq  \nu(v) \leq C_2 ( 1 + |v| )^\gamma \,,
	\end{equation}
	and the last two conclusions in Lemma \ref{Lmm-CF-nu} are still valid.
	
	We first prove the inequality \eqref{nu-Propty-1*}. From the elementary bounds
	\begin{equation}\label{Element-Inq}
	|v - v_*|^2 \leq ( 1 + |v_*|^2 ) ( 1 + |v|^2) \leq ( 1 + |v_*| )^2 (1 + |v|)^2
	\end{equation}
	for every $v \in \mathbb{R}^3$ and the fact $ M(v_*) = \tfrac{1}{\sqrt{2 \pi}^3} e^{- \tfrac{|v_*|^2}{2}} > 0 $, we directly obtain the upper bound
	\begin{equation*}
	\nu(v) = \int_{\mathbb{R}^3} |v - v_*|^\gamma M(v_*) \d v_* \leq \int_{\mathbb{R}^3} ( 1 + |v_*| )^\gamma M(v_*) \d v_* (1 + |v|)^\gamma \,,
	\end{equation*}
	which yields the upper bound of \eqref{nu-Propty-1*}.
	
	Next, the bound \eqref{Element-Inq} and the fact $M(v_*) > 0$ imply that
	\begin{equation*}
	(1+|v|)^{- \gamma} |v-v_*|^\gamma M(v_*) \leq (1 + |v_*|)^\gamma M(v_*) \,.
	\end{equation*}
	Then the Lebesgue Dominated Convergence Theorem implies that the positive function
	\begin{equation}\label{Functn-v}
	v \mapsto \frac{\nu(v)}{(1+|v|)^\gamma} = \frac{1}{(1+|v|)^\gamma} \int_{\mathbb{R}^3} |v-v_*|^\gamma M(v_*) \d v_*
	\end{equation}
	is continuous over $\mathbb{R}^3$ and satisfies
	\begin{equation*}
	\begin{aligned}
	\lim_{|v| \rightarrow \infty} \frac{\nu(v)}{(1+|v|)^\gamma} = & \lim_{|v| \rightarrow \infty} \frac{1}{(1+|v|)^\gamma} \int_{\mathbb{R}^3} |v-v_*|^\gamma M(v_*) \d v_* \\
	= & \int_{\mathbb{R}^3} M(v_*) \d v_* = 1 > 0 \,.
	\end{aligned}
	\end{equation*}
	The function \eqref{Functn-v} is thereby bounded away from zero, thereby the lower bound of \eqref{nu-Propty-1*} follows.
	
	We next derive the bound \eqref{nu-Propty-2}. One notices that
	\begin{equation*}
	\begin{aligned}
	\nabla_v \nu(v) = & \gamma \int_{\R^3} \tfrac{v-v_*}{|v-v_*|^{2-\gamma}} M(v_*) \d v_* \\
	=& \gamma \int_{\R^3} \tfrac{u}{|u|^{2-\gamma}} M(v-u) \d u \,,
	\end{aligned}
	\end{equation*}
	where the variables change $v_* \rightarrow u = v - v_*$ is utilized. Then for any $\beta' = [\beta_1', \beta_2', \beta_3']$,
	\begin{equation}\label{nu-1}
	\begin{aligned}
	|\partial_{\beta'} \nabla_v \nu(v)| = & \gamma \big| \int_{\R^3} \tfrac{u}{|u|^{2 - \gamma}} \partial_{\beta'} M(v-u) \d u \big| \\
	\leq & \gamma \int_{\R^3} \tfrac{1}{|u|^{1-\gamma}} |\partial_{\beta'} M(v-u)| \d u \,.
	\end{aligned}
	\end{equation}
	By direct calculations, we know that there is a constant $C_{\beta'} > 0$ such that
	\begin{equation*}
	|\partial_{\beta'} M(v - u)| \leq C_{\beta'} e^{- |v - u|^2 / 2} \,,
	\end{equation*}
	which implies that by \eqref{nu-1}
	\begin{align*}
	|\partial_{\beta'} \nabla_v \nu(v)| \leq &  C_{\beta'} \gamma \int_{\R^3} \tfrac{1}{|u|^{1-\gamma}} e^{- \tfrac{|v-u|^2}{4}} \d u \\
	= & \gamma C_{\beta'} \int_{\R^3} \tfrac{1}{|v - v_*|^{1-\gamma}} e^{- \tfrac{|v_*|^2}{4}} \d v_* \\
	= & \gamma C_{\beta'} \Big\{ \int_{|v-v_*| \geq 1} + \int_{|v-v_*| < 1} \Big\} \tfrac{1}{|v - v_*|^{1-\gamma}} e^{- \tfrac{|v_*|^2}{4}} \d v_* \\
	\leq & \gamma C_{\beta'} \int_{|v-v_*| \geq 1} e^{- \tfrac{|v_*|^2}{4}} \d v_* + \gamma C_{\beta'} \int_0^1 \int_{\mathbb{S}^2} \tfrac{1}{r^{1-\gamma}} r^2 \d \omega \d r \\
	\leq & \gamma C_{\beta'} \int_{\R^3} e^{- \tfrac{|v_*|^2}{4}} \d v_* + \tfrac{4 \pi \gamma}{2 + \gamma} C_{\beta'} < \infty \,.
	\end{align*}
	Then the bound \eqref{nu-Propty-2} holds.
	
	Finally, we verify the inequality \eqref{nu-Propty-3}. By the elementary inequality
	\begin{equation*}
	a^2 + b^2 \leq 2 (a+b)^2 \leq 4 (a^2 + b^2)
	\end{equation*}
	for $a,b \geq 0$, we derive from the conditions of $v$, $v'$, $v_*$ and $v_*'$ in Lemma \ref{Lmm-CF-nu} (3) that
	\begin{equation}\label{nu-2}
	|v| + |v_*| \leq C (|v'| + |v_*'|) \,.
	\end{equation}
	From the inequality \eqref{nu-Propty-1}, we have
	\begin{equation*}
	\begin{aligned}
	\nu(v) + \nu(v_*) \leq & \overline{C} \big[ (1+|v|)^\gamma + (1+|v_*|)^\gamma \big] \\
	\leq & 2 \overline{C} \big[ (1+|v|) + (1+|v_*|) \big]^\gamma \\
	\leq & C(\gamma) \big[ (1 + |v'|) + (1+|v_*'|) \big]^\gamma \\
	\leq & C(\gamma) \big[ (1+|v'|)^\gamma + (1+|v_*'|)^\gamma \big] \\
	\leq & \tfrac{C(\gamma)}{\underline{C}} ( \nu(v') + \nu(v_*') ) \,,
	\end{aligned}
	\end{equation*}
	where the elementary inequality
	\begin{equation*}
	(a+b)^\gamma \leq a^\gamma + b^\gamma \ (a,b \geq 0 \,, \gamma \in [0,1])
	\end{equation*}
	is used. Then the proof of Lemma \ref{Lmm-CF-nu} is finished.
\end{proof}

\begin{lemma}[Control of $\nu$-weighted norms]\label{Lmm-nu-norm}
	The related $\nu$-weighted norms have the following properties:
	\begin{enumerate}
		\item There is a constant $C_4 > 0$, such that for all $h \in L^2_v (\nu)$, 
		\begin{equation}\label{L2v-nu-norm}
		  \| h \|^2_{L^2_v} \leq C_4 \| h \|^2_{L^2_v (\nu)} \,.
		\end{equation}
		
		\item Let the integer $s \geq 1$. Then there exist positive constants $C_5$ and $C_6$, such that for all $h \in H^s_{x,v} (\nu)$
		\begin{equation}
		  \Big\langle \partial^m_\alpha ( \nu(v) h ) \,, \partial^m_\alpha h \Big\rangle_{L^2_{x,v}} \geq C_5 \| \partial^m_\alpha h \|^2_{L^2_{x,v} (\nu)} - C_6 \sum_{\alpha' < \alpha} \| \partial^m_{\alpha'} h \|^2_{L^2_{x,v}} 
		\end{equation}
		holds for $|m| + |\alpha| = s$ with $|\alpha| \geq 1$.
	\end{enumerate}
\end{lemma}

\begin{proof}[Proof of Lemma \ref{Lmm-nu-norm}]
	(1) From the inequality \eqref{nu-Propty-1}, we derive that
	\begin{equation}
	  \begin{aligned}
	     C_1 \leq \nu (v) \,,
	  \end{aligned}
	\end{equation}
	which immediately implies the inequality \eqref{L2v-nu-norm} holds.
	
	(2) Via direct calculation, we obtain
	\begin{equation}\label{K-=}
	  \begin{aligned}
	    \big\langle \partial^m_\alpha ( \nu (v) h ) , \partial^m_\alpha h \big\rangle_{L^2_{x,v}} = & \| \partial^m_\alpha h \|^2_{L^2_{x,v}(\nu)}  + \underset{K}{\underbrace{ \sum_{\alpha' < \alpha} C_\alpha^{\alpha'} \big\langle \partial^{\alpha-\alpha'}_v \nu (v) \partial^m_{\alpha'} h , \partial^m_\alpha h \big\rangle_{L^2_{x,v}} }} \,.
	  \end{aligned}
	\end{equation}
	Then, we employ the H\"older inequality, the part (2) of Lemma \ref{Lmm-CF-nu}, the part (1) of Lemma \ref{Lmm-nu-norm} and the Young's inequality to estimate the term $K$ in the previous equality \eqref{K-=}. More precisely, we have
	\begin{equation}\label{K-bnd}
	  \begin{aligned}
	    |K| \leq & C \sum_{\alpha' < \alpha} \| \partial^{\alpha-\alpha'}_v \nu (v) \|_{L^\infty_v} \| \partial^m_{\alpha'} h \|_{L^2_{x,v}} \| \partial^m_\alpha h \|_{L^2_{x,v}} \\
	    \leq & C \sum_{\alpha' < \alpha} \| \partial^m_{\alpha'} h \|_{L^2_{x,v}} \| \partial^m_\alpha h \|_{L^2_{x,v}(\nu)} \\
	    \leq & C \sum_{\alpha' < \alpha} \| \partial^m_{\alpha'} h \|^2_{L^2_{x,v}} + \tfrac{1}{2} \| \partial^m_\alpha h \|^2_{L^2_{x,v}(\nu)} \,.
	  \end{aligned}
	\end{equation}
	Substituting the inequality \eqref{K-bnd} into \eqref{K-=} implies that
	\begin{equation}
	  \begin{aligned}
	    \big\langle \partial^m_\alpha ( \nu (v) h ) , \partial^m_\alpha h \big\rangle_{L^2_{x,v}} \geq \tfrac{1}{2} \| \partial^m_\alpha h \|^2_{L^2_{x,v}(\nu)} - C \sum_{\alpha' < \alpha} \| \partial^m_{\alpha'} h \|^2_{L^2_{x,v}} \,.
	  \end{aligned}
	\end{equation} 
	Thus the proof of Lemma \ref{Lmm-nu-norm} is finished.
\end{proof}

\begin{lemma}[Coercivity on $\mathscr{L}$]\label{Lmm-L}
	The two species linearized collision operator $\mathscr{L} : L^2_v \rightarrow L^2_v$ has the following properties:
	\begin{enumerate}
		\item $\mathscr{L}$ is closed, self-adjoint and can be decomposed as
		\begin{equation*}
		\mathscr{L} = 2 \nu (v) \mathbb{I} - \mathscr{K},
		\end{equation*}
		where $\mathscr{K}$ is a compact operator in $L^2_v$, and $\nu(v)$ is the collisional frequency.
		
		\item Let the integer $s \geq 1$. Then for any $\delta > 0$, there is a $C(\delta) > 0$ such that for all $h \in H^s_{x,v} (\nu)$
		\begin{equation}
		\Big\langle \partial^m_\alpha \mathscr{K} (h), \partial^m_\alpha  h \Big\rangle_{L^2_{x,v}} \leq C(\delta) \| \partial^m_x h \|^2_{L^2_{x,v}} + \delta \| \partial^m_\alpha h \|^2_{L^2_{x,v}(\nu)}
		\end{equation}
		holds for $|m| + |\alpha| = s$, $|\alpha| \geq 1$.
		
		\item There is a $\lambda > 0$, such that for all $h \in L^2_v$
		\begin{equation}
		  \langle \mathscr{L} h, h \rangle_{L^2_v} \geq \lambda \| \P^\perp h \|^2_{L^2_v (\nu)} \,.
		\end{equation}
	\end{enumerate}
\end{lemma}

\begin{proof}[Proof of Lemma \ref{Lmm-L}]
	(1) The proof can be referred to Proposition 5.6 of \cite{Arsenio-SRM-2016} or \cite{Levermore-Sun-2010-KRM} for details.
	
	(2) We refer to Lemma 2 of \cite{Guo-2003-Invent} for details fo proof, and we omit the details here.
	
	(3) The coercivity of $\mathscr{L}$ has been proved in Proposition 5.7 of \cite{Arsenio-SRM-2016} or Lemma 1 of \cite{Guo-2003-Invent}. We omit the details here.
\end{proof}

\begin{lemma}[Control of $\Gamma$: Torus version]\label{Lmm-Gamma-Torus}
	Let $\Gamma : L^2_v \times L^2_v \rightarrow L^2_v$ be the bilinear symmetric operator defined in \eqref{Gamma}.
	\begin{enumerate}
		\item For any $G, H \in L^2_v$, we have
		\begin{equation}
		\Gamma (G,H) \in \mathrm{Ker}^\perp ( \mathscr{L} ) \,.
		\end{equation}
		
		\item Let any integer $s \geq 3$ and $F \in L^2_{x,v} (\nu)$, $G, H \in H^s_{x,v} (\nu)$. Then for all $m, \alpha \in \mathbb{N}^3$, $|m| + |\alpha| \leq s$,
		\begin{equation}
		\bigg| \Big\langle \partial^m_\alpha \Gamma ( G, H ), F \Big\rangle_{L^2_{x,v}} \bigg| \leq \left\{  
		\begin{array}{l}
		\mathcal{G}^s_{x,v} (G,H) \| F \|_{L^2_{x,v} (\nu)} \,, \quad \textrm{if } \alpha  \neq 0 \,, \\[2mm]
		\mathcal{G}^s_x (G,H) \| F \|_{L^2_{x,v} (\nu)} \,, \quad\ \textrm{if } \alpha = 0 \,,
		\end{array}
		\right.
		\end{equation}
		where $\mathcal{G}^s_{x,v}$ and $\mathcal{G}^s_x$ satisfy $\mathcal{G}^s_{x,v} \leq \mathcal{G}^{s+1}_{x,v}$, $\mathcal{G}^s_x \leq \mathcal{G}^{s+1}_x$ and there exists a positive constant $C_\Gamma > 0$ such that 
		\begin{equation}
		\begin{aligned}
		\mathcal{G}^s_{x,v} (G,H) & = C_\Gamma \left( \| G \|_{H^s_{x,v}} \| H \|_{H^s_{x,v}(\nu)} + \| G \|_{H^s_{x,v}(\nu)} \| H \|_{H^s_{x,v}} \right) \,,
		\end{aligned}
		\end{equation}
		and
		\begin{equation}
		\begin{aligned}
		\mathcal{G}^s_x (G,H) & = C_\Gamma \left( \| G \|_{H^s_x L^2_v} \| H \|_{H^s_x L^2_v (\nu)} + \| G \|_{H^s_x L^2_v (\nu)} \| H \|_{H^s_x L^2_v} \right)  \,.
		\end{aligned}
		\end{equation}
	\end{enumerate}
\end{lemma}

\begin{proof}[Proof of Lemma \ref{Lmm-Gamma-Torus}]
	(1) Since $\mathcal{Q} (g,h) \in \textrm{Ker}^\perp (\mathcal{L})$, we easily verify that $ \Gamma ( G, H ) \in \textrm{Ker}^\perp ( \mathscr{L} ) $.
	
	(2) We first estimate the term $ \big\langle \partial^m_\alpha \mathcal{Q} ( g_1, g_2 ) , g_3 \big\rangle_{L^2_{x,v}} $ for all $g_1, g_2: \T^3 \times \R^3 \mapsto \R $ in $H^s_{x,v} (\nu)$ and $g_3: \T^3 \times \R^3 \mapsto \R $ in $L^2_{x,v} (\nu)$. Straightforward calculations give us
	\begin{equation}\label{Q-g1-g2}
	  \begin{aligned}
	    \partial^m_\alpha & \mathcal{Q} (g_1, g_2) =  \partial^m_\alpha \int_{\R^3} \big( g_1 (v') g_2 (v_*') - g_1 (v) g_2 (v_*) \big) |v- v_*| \sqrt{M(v_*)} \d v_* \\
	    & =  \partial^m_\alpha \int_{\R^3} \big[  g_1 ( v - \sigma \cdot (v-v_*) \sigma ) g_2 ( v_* + \sigma \cdot ( v - v_* ) \sigma ) \\
	    & \qquad \qquad \qquad \qquad \qquad - g_1 (v) g_2 (v_*) \big] |v- v_*| M^\frac{1}{2} (v_*) \d v_* \\
	    &  \xlongequal[]{u=v-v_*} \partial^m_\alpha \int_{\R^3} \big[ g_1 ( v - \sigma \cdot u \sigma ) g_2 ( v - u + \sigma \cdot u \sigma ) \\
	    & \qquad \qquad \qquad \qquad \qquad - g_1 (v) g_2 (v- u) \big] |u| M^\frac{1}{2} (v - u) \d u \\
	    & = \underset{I_{gain}}{ \underbrace{\sum_{\substack{m_1 + m_2 = m \\ \alpha_1 + \alpha_2 + \alpha_3 = \alpha }} \int_{\R^3} \partial^{m_1}_{\alpha_1} g_1 ( v - \sigma \cdot u \sigma ) \partial^{m_2}_{\alpha_2} g_2 ( v - u + \sigma \cdot u \sigma ) |u| \partial^{\alpha_3}_v M^\frac{1}{2} (v-u) \d u }} \\
	    & \ \ \underset{I_{loss}}{ \underbrace{-   \sum_{\substack{m_1 + m_2 = m \\ \alpha_1 + \alpha_2 + \alpha_3 = \alpha}} \partial^{m_1}_{\alpha_1} g_1 (v) \int_{\R^3} \partial^{m_2}_{\alpha_2} g_2 (v - u) |u| \partial^{\alpha_3}_v M^\frac{1}{2} (v - u) \d u } }\,.
	  \end{aligned}
	\end{equation}
	
	{\em Estimates on $I_{loss}$.} It is easy to know that for any $\eta \in (0,1)$ and $\beta \in \mathbb{N}^3$
	\begin{equation}
	  \partial^\beta_v M^\frac{1}{2} (v - u) \leq C M^\frac{\eta}{2} (v-u)
	\end{equation}
	holds for some positive constant $C$. Then the second term $I_{loss}$ is bounded by
	\begin{equation}
	  \begin{aligned}
	    | I_{loss} |&  \leq  C \sum_{\substack{ m_1 + m_2 = m \\ \alpha_1 + \alpha_2 \leq \alpha }} \big| \partial^{m_1}_{\alpha_1} g_v (v) \big| \int_{\R^3} |u| M^\frac{\eta}{2} (v - u) \big| \partial^{m_2}_{\alpha_2} g_2 (v - u) \big| \d u \\
	    \leq & C \sum_{\substack{ m_1 + m_2 = m \\ \alpha_1 + \alpha_2 \leq \alpha }} \big| \partial^{m_1}_{\alpha_1} g_v (v) \big| \left( \int_{\R^3} |u|^{2} M^\eta (v -u) \d u \right)^\frac{1}{2} \left( \int_{\R^3} \big| \partial^{m_2}_{\alpha_2} g_2 (v-u)  \big|^2 \d u \right)^\frac{1}{2} \,.
	  \end{aligned}
	\end{equation}
	The following elementary inequalities 
	\begin{equation*}
	  |v-v_*|^2 \leq ( 1 + |v_*|^2 ) ( 1 + |v|^2 ) \leq (  1 + |v_*| )^2 ( 1 + |v| )^2
	\end{equation*}
	and Lemma \ref{Lmm-CF-nu} (1) yield that
	\begin{equation*}
	  \begin{aligned}
	    \int_{\R^3} |u|^{2} M^\eta (v - u) \d u  = & \int_{\R^3} |v - v_*|^{2} M^\eta (v_*) \d v_* \\
	    \leq & ( 1 + |v| )^{2} \int_{\R^3} ( 1 + |v_*| )^{2} M^\eta (v_*) \d v_* \leq C \nu^2 (v) \,.
	  \end{aligned}
	\end{equation*}
	Then we have 
	\begin{equation*}
	  \begin{aligned}
	    |I_{loss}| \leq C \sum_{\substack{ m_1 + m_2 = m \\ \alpha_1 + \alpha_2 \leq \alpha }} \nu(v) \big| \partial^{m_1}_{\alpha_1} g_1 (v) \big| \left( \int_{\R^3} \big| \partial^{m_2}_{\alpha_2} g_2 (v-u)  \big|^2 \d u \right)^\frac{1}{2} \,,
	  \end{aligned}
	\end{equation*}
	which immediately derive from the H\"older inequality that
	\begin{equation}
	  \begin{aligned}
	    \Big| \big\langle I_{loss} , g_3 \big\rangle_{L^2_{x,v}} \Big|  \leq & C \sum_{\substack{ m_1 + m_2 = m \\ \alpha_1 + \alpha_2 \leq \alpha}} \Big\langle | \partial^{m_1}_{\alpha_1} g_1 (v) | \nu(v) \| \partial^{m_2}_{\alpha_2} g_2 \|_{L^2_v} , g_3 \Big\rangle_{L^2_{x,v}} \\
	    \leq & C \sum_{\substack{ m_1 + m_2 = m \\ \alpha_1 + \alpha_2 \leq \alpha}}  \underset{{\bf I}_{\alpha_1, \alpha_2}^{m_1 , m_2}} { \underbrace{ \Big\langle \| \partial^{m_1}_{\alpha_1} g_1 \|_{L^2_v (\nu)} \| \partial^{m_2}_{\alpha_2} g_2 \|_{L^2_v} , \| g_3 \|_{L^2_v(\nu)} \Big\rangle_{L^2_x} }} \,.
	  \end{aligned}
	\end{equation}
	
	We next control the terms ${\bf I}^{m_1, m_2}_{\alpha_1, \alpha_2}$ for all multi-indexes $m_1$, $m_2$, $\alpha_1$ and $\alpha_2$ in $\mathbb{N}^3$ satisfying $m_1| + |m_2 = m$ and $\alpha_1 + \alpha_2 \leq \alpha$. If $\alpha = 0$ and $m_1 = 0$ or $m_2 = 0$, via using the H\"older inequality and the Sobolev embedding $H^2_x (\T^3) \hookrightarrow L^\infty_x (\T^3)$, we have
	\begin{equation}
	  \begin{aligned}
	    \big| {\bf I}^{m, 0}_{0, 0} \big| + \big| {\bf I}^{0, m}_{0, 0} \big| \leq & \big( \| \partial^m_x g_1 \|_{L^2_{x,v}(\nu)} \| g_2 \|_{L^\infty_x L^2_v} + \| g_1 \|_{L^\infty_x L^2_v (\nu)} \| \partial^m_x g_2 \|_{L^2_{x,v}} \big) \| g_3 \|_{L^2_{x,v}(\nu)} \\
	    \leq & C \big( \| \partial^m_x g_1 \|_{L^2_{x,v}(\nu)} \| g_2 \|_{H^2_x L^2_v} + \| g_1 \|_{H^2_x L^2_v (\nu)} \| \partial^m_x g_2 \|_{L^2_{x,v}} \big) \| g_3 \|_{L^2_{x,v}(\nu)} \\
	    \leq & C \| g_1 \|_{H^s_x L^2_v(\nu)} \| g_2 \|_{H^s_x L^2_v} \| g_3 \|_{L^2_{x,v}(\nu)} \\
	    \leq & C  \| g_1 \|_{H^s_x L^2_v(\nu)} \| g_2 \|_{H^s_x L^2_v} \| g_3 \|_{L^2_{x,v}(\nu)} \,.
	  \end{aligned}
	\end{equation}
	If $\alpha = 0$ and $m_1 , m_2 \neq 0$, the H\"older inequality and the Sobolev embedding $H^1_x (\T^3) \hookrightarrow L^4_x (\T^3)$ imply that
	\begin{equation}
	  \begin{aligned}
	    {\bf I}^{m_1, m_2}_{0, 0} \leq & \big\| \| \partial^{m_1}_x g_1 \|_{L^2_v(\nu)} \big\|_{L^4_x} \big\| \| \partial^{m_2}_x g_2 \|_{L^2_v} \big\|_{L^4_x} \| g_3 \|_{L^2_{x,v}(\nu)} \\
	    \leq &  C \big\| \| \partial^{m_1}_x g_1 \|_{L^2_v(\nu)} \big\|_{H^1_x} \big\| \| \partial^{m_2}_x g_2 \|_{L^2_v} \big\|_{H^1_x} \| g_3 \|_{L^2_{x,v}(\nu)} \\
	    \leq & C \| g_1 \|_{H^s_x L^2_v (\nu)} \| g_2 \|_{H^s_x L^2_v} \| g_3 \|_{L^2_{x,v}(\nu)} \,,
	  \end{aligned}
	\end{equation}
	where the last inequality is derived from the inequality 
	\begin{equation}\label{Inq-xv-mixed-Holder}
	\big\| \nabla_x \| g \|_{L^2_v (w)} \big\|_{L^2_x} \leq \| \nabla_x g \|_{L^2_{x,v} (w)}
	\end{equation}
	for $w(v) = 1$ or $\nu(v)$. This inequality is derived from the H\"older inequality as follows:
	\begin{equation*}
	\begin{aligned}
	\left\| \nabla_x \| g \|_{L^2_v (w)} \right\|_{L^2_x} = & \left( \int_{\T^3} \Big| \nabla_x \| g \|_{L^2_v (w)} \Big|^2 \d x \right)^\frac{1}{2} = \left( \int_{\T^3} \Big| \tfrac{\nabla_x \| g \|^2_{L^2_v (w)}}{2 \| g \|_{L^2_v (w)}} \Big|^2 \right)^\frac{1}{2} \\
	= & \left( \int_{\T^3} \Big| \tfrac{\nabla_x \int_{\R^3} | g |^2 w \d v}{2 \| g \|_{L^2_v (w)}} \Big|^2 \right)^\frac{1}{2} = \left( \int_{\T^3} \Big| \tfrac{\int_{\R^3}  g \nabla_x g w \d v}{\| g \|_{L^2_v (w)}} \Big|^2 \right)^\frac{1}{2} \\
	\leq & \left( \int_{\T^3} \int_{\R^3} | \nabla_x g |^2 w \d v d x \right)^\frac{1}{2} = \| \nabla_x g \|_{L^2_{x,v} (w)} \,.
	\end{aligned}
	\end{equation*}
	If $\alpha \neq 0$ and $m_1 + \alpha_1 = 0$ or $m_2 + \alpha_2 = 0$, then one derives from the Sobolev embedding $H^2_x (\T^3) \hookrightarrow L^\infty_x (\T^3)$ that
	\begin{equation}
	  \begin{aligned}
	    \big| {\bf I}^{0,m}_{0, \alpha_2}  \big| + \big| {\bf I}^{m, 0}_{\alpha_1, 0}  \big| \leq & \Big( \| g_1 \|_{L^\infty_x L^2_v (\nu)} \| \partial^{m}_{\alpha_2} g_2 \|_{L^2_{x,v}} + \| \partial^m_{\alpha_1} g_1 \|_{L^2_{x,v}(\nu)} \| g_2 \|_{L^\infty_x L^2_v} \Big) \| g_3 \|_{L^2_{x,v}(\nu)} \\
	    \leq & C \Big( \| g_1 \|_{H^2_x L^2_v (\nu)} \| \partial^{m}_{\alpha_2} g_2 \|_{L^2_{x,v}} + \| \partial^m_{\alpha_1} g_1 \|_{L^2_{x,v}(\nu)} \| g_2 \|_{H^2_x L^2_v} \Big) \| g_3 \|_{L^2_{x,v}(\nu)} \\
	    \leq & C \| g_1 \|_{H^s_{x,v}(\nu)} \| g_2 \|_{H^s_{x,v}} \| g_3 \|_{L^2_{x,v}(\nu)} \\
	    \leq & C \| g_1 \|_{H^s_{x,v}(\nu)} \| g_2 \|_{H^s_{x,v}} \| g_3 \|_{L^2_{x,v}(\nu)} \,.
	  \end{aligned}
	\end{equation}
	Here we require $s \geq 2$. If $\alpha \neq 0$ and $m_1 + \alpha_1 \neq 0$, $m_2 + \alpha_2 \neq 0$, then the terms $ {\bf I}^{m_1, m_2}_{\alpha_1, \alpha_2} $ are bounded by
	\begin{equation}
	\begin{aligned}
	{\bf I}^{m_1, m_2}_{\alpha_1, \alpha_2} \leq & \big\| \| \partial^{m_1}_{\alpha_1} g_1 \|_{L^2_v (\nu)} \big\|_{L^4_x} \big\| \| \partial^{m_2}_{\alpha_2} g_2 \|_{L^2_v} \big\|_{L^4_x} \| g_3 \|_{L^2_{x,v}(\nu)} \\
	\leq & C \big\| \| \partial^{m_1}_{\alpha_1} g_1 \|_{L^2_v (\nu)} \big\|_{H^1_x} \big\| \| \partial^{m_2}_{\alpha_2} g_2 \|_{L^2_v} \big\|_{H^1_x} \| g_3 \|_{L^2_{x,v}(\nu)} \\
	\leq & C \big\| g_1 \big\|_{H^s_{x,v} (\nu)} \big\| g_2 \big\|_{H^s_{x,v}} \| g_3 \|_{L^2_{x,v}(\nu)} \\
	\leq & C \big\| g_1 \big\|_{H^s_{x,v} (\nu)} \big\| g_2 \big\|_{H^s_{x,v}} \| g_3 \|_{L^2_{x,v}(\nu)} \,.
	\end{aligned}
	\end{equation}
	Here the H\"older inequality, the Sobolev embedding $H^1_x (\T^3) \hookrightarrow L^4_x (\T^3)$ and the inequality \eqref{Inq-xv-mixed-Holder} are utilized. We summarize the all above bounds on the terms ${\bf I}^{m_1, m_2}_{\alpha_1, \alpha_2}$ and obtain
	\begin{equation}\label{Loss}
	  \begin{aligned}
	    \Big| \big\langle I_{loss} , g_3 \big\rangle_{L^2_{x,v}} \Big|  \leq & \sum_{\substack{ m_1 + m_2 = m \\ \alpha_1 + \alpha_2 \leq \alpha }} {\bf I}^{m_1, m_2}_{\alpha_1, \alpha_2} \\[2mm]
	    \leq & \left\{  
	      \begin{array}{l}
	        C \| g_1 \|_{H^s_{x,v}(\nu)} \| g_2 \|_{H^s_{x,v}} \| g_3 \|_{L^2_{x,v}(\nu)} \qquad \textrm{if } \alpha \neq 0 \,, \\[4mm]
	        C \| g_1 \|_{H^s_x L^2_v (\nu)} \| g_2 \|_{H^s_x L^2_v} \| g_3 \|_{L^2_{x,v}(\nu)} \quad\ \textrm{if } \alpha = 0 \,.
	      \end{array}
	    \right.
	  \end{aligned}
	\end{equation}
	
	{\em Estimates on $I_{gain}$.} We next deal with the term $I_{gain}$. By the H\"older inequality and the part (3) of Lemma \ref{Lmm-CF-nu}, we have
	\begin{align*}
	| I_{gain} | \leq & \sum_{\substack{ m_1 + m_2 = m \\ \alpha_1 + \alpha_2 + \alpha_3 = \alpha }} \left( |u|^{2 \gamma} | \partial^{\alpha_3}_v M^\frac{1}{2} (v-u) |^2 \d u \right)^\frac{1}{2} \\
	& \qquad \times \left( \int_{\R^3} | \partial^{m_1}_{\alpha_1} g_1 ( v - \sigma \cdot u \sigma ) |^2 | \partial^{m_2}_{\alpha_2} g_2 ( v - u + \sigma \cdot u \sigma ) |^2 \d u \right)^\frac{1}{2} \\
	\leq & C \sum_{\substack{ m_1 + m_2 = m \\ \alpha_1 + \alpha_2 \leq \alpha }} \nu(v) \left( \int_{\R^3} | \partial^{m_1}_{\alpha_1} g_1 ( v - \sigma \cdot (v - v_*) \sigma ) |^2 | \partial^{m_2}_{\alpha_2} g_2 ( v_* + \sigma \cdot (v - v_*) \sigma ) |^2 \d u \right)^\frac{1}{2} \\
	\leq & C \sum_{\substack{ m_1 + m_2 = m \\ \alpha_1 + \alpha_2 \leq \alpha }} \nu^\frac{1}{2} (v) \Bigg( \int_{\R^3} \big[ \nu( v - \sigma \cdot (v - v_*) \sigma ) + \nu( v_* + \sigma \cdot (v - v_*) \sigma ) \big] \\
	& \qquad \times | \partial^{m_1}_{\alpha_1} g_1 ( v - \sigma \cdot (v - v_*) \sigma ) |^2 | \partial^{m_2}_{\alpha_2} g_2 ( v_* + \sigma \cdot (v - v_*) \sigma ) |^2 \d v_* \Bigg)^\frac{1}{2} \,,
	\end{align*}
	which implies that
	\begin{align}\label{Bnd-gain-II+III}
	\no \Big| & \langle I_{gain}, g_3 \rangle_{L^2_{x,v}} \Big| \leq C \sum_{\substack{ m_1 + m_2 = m \\ \alpha_1 + \alpha_2 \leq \alpha }} \Bigg\langle \bigg( \int_{\R^3} \big[ \nu( v - \sigma \cdot (v - v_*) \sigma ) + \nu( v_* + \sigma \cdot (v - v_*) \sigma ) \big] \\
	\no & \qquad \times | \partial^{m_1}_{\alpha_1} g_1 ( v - \sigma \cdot (v - v_*) \sigma ) |^2 | \partial^{m_2}_{\alpha_2} g_2 ( v_* + \sigma \cdot (v - v_*) \sigma ) |^2 \d v_* \bigg)^\frac{1}{2} \,, g_3 \nu^\frac{1}{2} (v) \Bigg\rangle_{L^2_{x,v}} \\
	\no \leq & C \| g_3 \|_{L^2_{x,v}(\nu)} \sum_{\substack{ m_1 + m_2 = m \\ \alpha_1 + \alpha_2 \leq \alpha }} \Bigg( \int_{\T^3} \int_{\R^3} \int_{\R^3} \big[ \nu( v - \sigma \cdot (v - v_*) \sigma ) + \nu( v_* + \sigma \cdot (v - v_*) \sigma ) \big] \\
	\no & \qquad \quad \times | \partial^{m_1}_{\alpha_1} g_1 ( v - \sigma \cdot (v - v_*) \sigma ) |^2 | \partial^{m_2}_{\alpha_2} g_2 ( v_* + \sigma \cdot (v - v_*) \sigma ) |^2 \d v_* \d v \d x \Bigg)^\frac{1}{2} \\
	\no & \xlongequal[w_* = v_* + \sigma \cdot (v - v_*) \sigma]{w = v - \sigma \cdot (v - v_* ) \sigma} C \| g_3 \|_{L^2_{x,v} (\nu)} \sum_{\substack{ m_1 + m_2 = m \\ \alpha_1 + \alpha_2 \leq \alpha }} \Bigg( \int_{\T^3} \int_{\R^3} \int_{\R^3} [ \nu(w) + \nu(w_*) ] \\
	\no & \qquad \qquad \qquad \qquad \qquad \qquad \qquad  \times |\partial^{m_1}_{\alpha_1} g_1 (w)|^2 | \partial^{m_2}_{\alpha_2} g_2 (w_*) |^2 \d w_* \d w \d x \Bigg)^\frac{1}{2} \\
	\no & =  C \| g_3 \|_{L^2_{x,v} (\nu)} \sum_{\substack{ m_1 + m_2 = m \\ \alpha_1 + \alpha_2 \leq \alpha }} \Bigg( \underset{\textbf{II}^{m_1 , m_2}_{\alpha_1 , \alpha_2}}{\underbrace{ \int_{\T^3} \| \partial^{m_1}_{\alpha_1} g_1 \|^2_{L^2_v (\nu)} \| \partial^{m_2}_{\alpha_2} g_2 \|^2_{L^2_v} \d x }} \\
	& \qquad \qquad \qquad \qquad 
	\qquad \qquad \qquad \qquad \qquad  + \underset{\textbf{III}^{m_1 , m_2}_{\alpha_1 , \alpha_2}}{\underbrace{ \int_{\T^3} \| \partial^{m_1}_{\alpha_1} g_1 \|^2_{L^2_v} \| \partial^{m_2}_{\alpha_2} g_2 \|^2_{L^2_v (\nu)} \d x }} \Bigg)^\frac{1}{2} \,.
	\end{align}
	
	We will estimate term $ \textbf{II}^{m_1 , m_2}_{\alpha_1 , \alpha_2} $ in the previous equality by employing the H\"older inequality, the Sobolev embeddings $ H^1_x (\T^3) \hookrightarrow L^4_x (\T^3) $ and $ H^2_x (\T^3) \hookrightarrow L^\infty_x ( \T^3 ) $. 
	
	If $|m| + |\alpha| \leq s$ with $\alpha \neq 0$ and the multi-indexes $m_1$, $m_2$, $\alpha_1$ and $\alpha_2$ in $\mathbb{N}^3$ satisfy $|m_1| + |\alpha_1| \leq s -1$, $|m_2| + |\alpha_2| \leq s -1$ and $m_1 + m_2 = m$, $\alpha_1 + \alpha_2 \leq \alpha$, we deduce that
	\begin{equation}
	  \begin{aligned}
	    \textbf{II}^{m_1 , m_2}_{\alpha_1 , \alpha_2} \leq & \big\| \| \partial^{m_1}_{\alpha_1} g_1 \|_{L^2_v (\nu)} \big\|^2_{L^4_x} \big\| \| \partial^{m_2}_{\alpha_2} g_2 \|_{L^2_v} \big\|^2_{L^4_x}  \\
	    \leq & C \| g_1 \|^2_{H^s_{x,v}(\nu)} \| g_2 \|^2_{H^s_{x,v}} \,,
	  \end{aligned}
	\end{equation}
	where the inequality \eqref{Inq-xv-mixed-Holder} is utilized. If $|m| + |\alpha| \leq s$ and $\alpha \neq 0$ satisfying $|m_1| + |\alpha_1| = s$ or $|m_2| + |\alpha_2| = s $, then $m_2 = \alpha_2 = 0$ or  $m_1 = \alpha_1 = 0$, respectively. Then we can estimate that
	\begin{equation}
	  \begin{aligned}
	    \textbf{II}^{0 , m_2}_{0 , \alpha_2} + \textbf{II}^{m_1 , 0}_{\alpha_1 , 0} \leq & \int_{\T^3} \| \partial^{m_1}_{\alpha_1} g_1 \|^2_{L^2_v(\nu)} \| g_2 \|^2_{L^2_v} \d x + \int_{\T^3} \| g_1 \|^2_{L^2_v(\nu)} \| \partial^{m_2}_{\alpha_2} g_2 \|^2_{L^2_v} \d x \\
	    \leq & \| \partial^{m_1}_{\alpha_1} g_1 \|^2_{L^2_{x,v}(\nu)} \| g_2 \|^2_{L^\infty_x L^2_v} + \| g_1 \|^2_{L^\infty_x L^2_v(\nu)} \| \partial^{m_2}_{\alpha_2} g_2 \|^2_{L^2_{x,v}} \\
	    \leq & C \| g_1 \|^2_{H^s_{x,v}(\nu)} \| g_2 \|^2_{H^s_{x,v}} \,.
	  \end{aligned}
	\end{equation}
	Here the condition $s \geq 2$ is required. In summary, we obtain that for $m_1 + m_2 = m $ and $\alpha_1 + \alpha_2 \leq \alpha$ with $|m|+|\alpha| \leq s \, (\alpha \neq 0)$
	\begin{equation}\label{II-bnd}
	  \textbf{II}^{m_1 , m_2}_{\alpha_1 , \alpha_2} \leq C  \| g_1 \|^2_{H^s_{x,v}(\nu)} \| g_2 \|^2_{H^s_{x,v}} \,.
	\end{equation}
    Furthermore, the term $ \textbf{III}^{m_1 , m_2}_{\alpha_1 , \alpha_2} $ can be controlled by borrowing the analogous arguments of the estimates on the term $ \textbf{II}^{m_1 , m_2}_{\alpha_1 , \alpha_2} $ in \eqref{II-bnd}. More precisely,
	\begin{equation}\label{III-bnd}
	  \begin{aligned}
	    \textbf{III}^{m_1 , m_2}_{\alpha_1 , \alpha_2} \leq C  \| g_1 \|^2_{H^s_{x,v}} \| g_2 \|^2_{H^s_{x,v}(\nu)} 
	  \end{aligned}
	\end{equation}
	holds for all $m_1 + m_2 = m $ and $\alpha_1 + \alpha_2 \leq \alpha$ with $|m|+|\alpha| \leq s \, (\alpha \neq 0)$. We thereby derive from plugging the bounds \eqref{II-bnd} and \eqref{III-bnd} into the relation \eqref{Bnd-gain-II+III} that
	\begin{equation}\label{Gain-alpha-neq0}
	  \begin{aligned}
	    \big| \langle I_{gain}, g_3 \rangle_{L^2_{x,v}} \big| \leq &  C \sum_{\substack{ m_1 + m_2 = m \\ \alpha_1 + \alpha_2 \leq \alpha }} \big( \textbf{II}^{m_1 , m_2}_{\alpha_1 , \alpha_2} + \textbf{III}^{m_1 , m_2}_{\alpha_1 , \alpha_2} \big)^\frac{1}{2} \| g_3 \|_{L^2_{x,v}(\nu)}  \\
	    \leq & C \big( \| g_1 \|_{H^s_{x,v}(\nu)} \| g_2 \|_{H^s_{x,v}} + \| g_1 \|_{H^s_{x,v}} \| g_2 \|_{H^s_{x,v}(\nu)} \big) \| g_3 \|_{L^2_{x,v}(\nu)} 
	  \end{aligned}
	\end{equation}
	holds for all $|m| + |\alpha| \leq s$ with $\alpha \neq 0$. Via the similar arguments of \eqref{Gain-alpha-neq0}, one can also yield that
	\begin{equation}\label{Gain-alpha-=0}
	  \begin{aligned}
	    \big| \langle I_{gain}, g_3 \rangle_{L^2_{x,v}} \big| \leq C \big( \| g_1 \|_{H^s_x L^2_v(\nu)} \| g_2 \|_{H^s_x L^2_v} + \| g_1 \|_{H^s_x L^2_v} \| g_2 \|_{H^s_x L^2_v (\nu)} \big) \| g_3 \|_{L^2_{x,v}(\nu)} 
	  \end{aligned}
	\end{equation}
	hold for all $|m| + |\alpha| \leq s$ with $\alpha = 0$.
	
	Combining the bounds \eqref{Loss}, \eqref{Gain-alpha-neq0} and \eqref{Gain-alpha-=0}, we derive from \eqref{Q-g1-g2} that 
	\begin{equation}\label{Q-g1-g2-g3}
	  \begin{aligned}
	    & \big| \langle \partial^m_\alpha \mathcal{Q} (g_1 , g_2) , g_3 \rangle_{L^2_{x,v}} \big| \leq \big| \langle I_{loss}, g_3 \rangle_{L^2_{x,v}} \big| + \big| \langle I_{gain}, g_3 \rangle_{L^2_{x,v}} \big| \\[1mm]
	    \leq & \left\{
	           \begin{array}{l}
	             C \big( \| g_1 \|_{H^s_x L^2_v(\nu)} \| g_2 \|_{H^s_x L^2_v} + \| g_1 \|_{H^s_x L^2_v} \| g_2 \|_{H^s_x L^2_v (\nu)} \big) \| g_3 \|_{L^2_{x,v}(\nu)} \,, \quad \textrm{if } \ \alpha = 0  \\[2mm]
	             C \big( \| g_1 \|_{H^s_{x,v}(\nu)} \| g_2 \|_{H^s_{x,v}} + \| g_1 \|_{H^s_{x,v}} \| g_2 \|_{H^s_{x,v}(\nu)} \big) \| g_3 \|_{L^2_{x,v}(\nu)} \,, \quad\quad \ \, \textrm{if } \ \alpha \neq 0 \,.
	           \end{array}
	         \right.
	  \end{aligned}
	\end{equation}
	holds for all $|m| + |\alpha| \leq s$. Here $s \geq 2$ is required.
	
	We finally estimate the quantity $\l \partial^m_\alpha \Gamma (G, H) , F \r_{L^2_{x,v}} $ for $|m| + |\alpha| \leq s$. Recalling the definition of $\Gamma (G, H)$ in \eqref{Gamma}, i.e.,
	\begin{equation*}
	\begin{aligned}
	\Gamma (G , H ) = & \tfrac{1}{2} [ \mathcal{Q} ( G^+, H^+ ) + \mathcal{Q} ( H^+, G^+ ) + \mathcal{Q} ( G^+, H^- ) + \mathcal{Q} ( H^+, G^- ) \,, \\
	& \qquad  \mathcal{Q} (G^-, H^-) + \mathcal{Q} (H^-, G^-) + \mathcal{Q} (G^- , H^+) + \mathcal{Q} ( H^-, G^+ ) ] \,,
	\end{aligned}
	\end{equation*}
	we have
	\begin{equation}\label{Relt-Gamma}
	\begin{aligned}
	\l \partial^m_\alpha \Gamma (G, H), F \r_{L^2_{x,v}} = & \tfrac{1}{2} \l \partial^m_\alpha \mathcal{Q} ( G^+ , H^+ ) , F^+ \r_{L^2_{x,v}} + \tfrac{1}{2} \l \partial^m_\alpha \mathcal{Q} ( H^+ , G^+ ) , F^+ \r_{L^2_{x,v}} \\
	+ & \tfrac{1}{2} \l \partial^m_\alpha \mathcal{Q} ( G^- , H^- ) , F^- \r_{L^2_{x,v}} + \tfrac{1}{2} \l \partial^m_\alpha \mathcal{Q} ( H^- , G^- ) , F^- \r_{L^2_{x,v}} \\
	+ & \tfrac{1}{2} \l \partial^m_\alpha \mathcal{Q} ( G^+ , H^- ) , F^+ \r_{L^2_{x,v}} + \tfrac{1}{2} \l \partial^m_\alpha \mathcal{Q} ( H^+ , G^- ) , F^+ \r_{L^2_{x,v}} \\
	+ & \tfrac{1}{2} \l \partial^m_\alpha \mathcal{Q} ( G^- , H^+ ) , F^- \r_{L^2_{x,v}} + \tfrac{1}{2} \l \partial^m_\alpha \mathcal{Q} ( H^- , G^+ ) , F^- \r_{L^2_{x,v}} \\
	= & \tfrac{1}{2} \sum_{\tau = \pm} \sum_{\gamma = \pm} \big[ \langle \partial^m_\alpha \mathcal{Q} ( G^\tau , H^\gamma ) , F^\tau \rangle_{L^2_{x,v}} + \langle \partial^m_\alpha \mathcal{Q} ( H^\tau, G^\gamma ) , F^\tau \rangle_{L^2_{x,v}} \big] \,.
	\end{aligned}
	\end{equation}
	Then the inequalities \eqref{Q-g1-g2-g3} imply that
	\begin{equation}
	  \begin{aligned}
	    \langle \partial^m_\alpha \Gamma (G, & H), F \rangle_{L^2_{x,v}} \\
	    \leq & C \sum_{\tau = \pm} \sum_{\gamma = \pm} \big( \| G^\tau \|_{X_\alpha (\nu)} \| H^\gamma \|_{X_\alpha} + \| G^\tau \|_{X_\alpha } \| H^\gamma \|_{X_\alpha (\nu)} \big) \| F^\tau \|_{L^2_{x,v}(\nu)} \\
	    & + C \sum_{\tau = \pm} \sum_{\gamma = \pm} \big( \| H^\tau \|_{X_\alpha (\nu)} \| G^\gamma \|_{X_\alpha} + \| H^\tau \|_{X_\alpha } \| G^\gamma \|_{X_\alpha (\nu)} \big) \| F^\tau \|_{L^2_{x,v}(\nu)} \\
	    \leq & C \big( \| G \|_{X_\alpha (\nu)} \| H \|_{X_\alpha} + \| G \|_{X_\alpha } \| H \|_{X_\alpha (\nu)} \big) \| F \|_{L^2_{x,v}(\nu)} 
	  \end{aligned}
	\end{equation}
	holds for all $|m| + |\alpha| \leq s$, where $X_\alpha = H^s_x L^2_v$ if $\alpha = 0$, while $X_\alpha = H^s_{x,v}$ if $\alpha \neq 0$. Here the facts $\| G^\pm \|_Y \leq C \| G \|_Y$ $(Y = X_\alpha , X_\alpha (\nu), L^2_{x,v}(\nu))$ is also utilized. We thus complete the proof of Lemma \ref{Lmm-Gamma-Torus}.
\end{proof}

\section{Uniform energy estimates on the spatial derivatives.}\label{Sec:Unif-Spatial-Bnd}

At beginning of this section, we give the following local well-posedness of the perturbed VMB system \eqref{VMB-G} with small initial data:

\begin{proposition}\label{Prop-Local-Solutn}
	Let $0 < \eps \leq 1$ and $s \geq 3$. There are $\ell > 0$ and $T^* > 0$, independent of $\eps$, such that if $\mathbb{E}_s (G_\eps^{in}, E_\eps^{in} , B_\eps^{in}) \leq \ell$ and $T^* \leq \sqrt{\ell}$, the Cauchy problem \eqref{VMB-G}-\eqref{IC-VMB-G} admits a unique solution $(G_\eps, E_\eps , B_\eps ) $ satisfying
	\begin{equation}\no
	\begin{aligned}
	G_\eps \in L^\infty (0,T^*; H^s_{x,v}), E_\eps , B_\eps \in L^\infty (0,T^* ; H^s_x)
	\end{aligned}
	\end{equation}
	with the energy bound
	\begin{equation}\label{Local-Energy-Bnd}
	\begin{aligned}
	\sup_{t \in [0, T^*]} \mathbb{E}_s ( G_\eps , E_\eps , B_\eps ) + \tfrac{1}{\eps^2} \int_0^{T^*} \| \mathbb{P}^\perp G_\eps \|^2_{H^s_{x,v}(\nu)} \d t \leq C \ell  \,,
	\end{aligned}
	\end{equation}
	where the constant $C > 0$ is independent of $\eps$.
\end{proposition}
The proof of Proposition \ref{Prop-Local-Solutn} will be given in Section Appendix \ref{Sec:Appendix-Local-Solutn}. Our goal of this section is to  derive the uniform energy estimates on the spatial derivatives to the perturbed VMB system \eqref{VMB-G}. For notational simplicity, we drop the lower index $\eps$, i.e., 
\begin{equation}\label{VMB-G-drop-eps}
  \left\{
    \begin{array}{l}
      \partial_t G + \tfrac{1}{\eps} \big[ v \cdot  \nabla_x G + \textsf{q} ( \eps E + v \times B ) \cdot \nabla_v \big]  G + \tfrac{1}{\eps^2} \mathscr{L} G - \tfrac{1}{\eps}  (E \cdot v) \sqrt{M}  \textsf{q}_1 \\[2mm]
      \qquad \qquad = \tfrac{1}{2} \mathsf{q} (E \cdot v) G + \tfrac{1}{\eps} \Gamma (G, G) \,,\\[2mm]
      \partial_t E - \nabla_x \times B = - \tfrac{1}{\eps} \int_{\R^3} G \cdot \textsf{q}_1 v \sqrt{M} \d v \,,\\[2mm]
      \partial_t B + \nabla_x \times E = 0 \,,\\[2mm]
      \div_x E = \int_{\R^3} G \cdot \textsf{q}_1 \sqrt{M} \d v \,, \ \div_x B = 0 \,.
    \end{array}
  \right.
\end{equation}
The {\em key points} are the following three aspects: First, the two species linearized collision operator part $ \tfrac{1}{\eps^2} \mathscr{L} G $ will give us kinetic dissipation term $ \tfrac{1}{\eps^2} \| \mathbb{P}^\perp G \|^2_{H^s_x L^2_v (\nu)} $ with singularity $\tfrac{1}{\eps^2}$. Secondly, by the {\em micro-macro decomposition}, we can obtain a fluid dissipation term $\| \nabla_x \mathbb{P} G \|^2_{H^{s-1}_x L^2_v} $. Finally, we find the construction of Ohm's law, which will gives us a damping term $\partial_t B$ of the Farady equation, so that we can get the global energy estimate on the electric field $E$ and magnetic field $B$.

\subsection{Energy estimates with kinetic dissipation.} In this subsection, we will give the energy estimates of the spatial derivatives with kinetic dissipation by direct energy methods. More precisely, we prove the following proposition.
\begin{proposition}\label{Prop-Spatial}
	Assume that $(G, E, B)$ is the solution to the perturbed VMB system \eqref{VMB-G} constructed in Proposition \ref{Prop-Local-Solutn}. Let the integer $s \geq 3$. Then there are constants $\lambda > 0$ and $C > 0$, independent of $\eps > 0$, such that
	\begin{equation}\label{Spatial-Bnd}
	\begin{aligned}
	& \tfrac{1}{2} \tfrac{\d }{\d t} \left( \| G \|^2_{H^s_x L^2_v} + \| E \|^2_{H^s_x} + \| B \|^2_{H^s_x} \right) + \tfrac{\lambda}{\eps^2} \| \mathbb{P}^\perp G \|^2_{H^s_x L^2_v(\nu)} \\
	\leq & C \| E \|_{H^s_x} \left( \| \mathbb{P} G \|^2_{H^s_x L^2_v} + \| \mathbb{P}^\perp G \|^2_{H^s_x L^2_v (\nu)} \right) \\
	+ & C \| E \|_{H^s_x} \| \mathbb{P}^\perp G \|_{H^s_x L^2_v (\nu)} \Big( \| \mathbb{P}^\perp G \|_{H^s_x L^2_v (\nu)} + \sum_{|m| \leq s - 1} \| \nabla_v \partial^m_x \mathbb{P}^\perp G \|_{L^2_{x,v}(\nu)} \Big) \\
	+ & \frac{C}{\eps} \left( \| G \|_{H^s_x L^2_v} + \| B \|_{H^s_x} \right) \| \mathbb{P}^\perp G \|_{H^s_x L^2_v (\nu)} \\
	& \quad \times \Big( \| \mathbb{P} G \|_{H^s_x L^2_v} + \| \mathbb{P}^\perp G \|_{H^s_x L^2_v (\nu)} +  \sum_{|m| \leq s - 1} \| \nabla_v \partial^m_x \mathbb{P}^\perp G \|_{L^2_{x,v}(\nu)} \Big) \,.
	\end{aligned}
	\end{equation}
\end{proposition}

\begin{proof}[Proof of Proposition \ref{Prop-Spatial}]
For all multi-index $m \in \mathbb{N}^3$ satisfying $|m| \leq s$ ($s \geq 3$), we act $\partial^m_x$ on the first equation of \eqref{VMB-G-drop-eps} and take the $L^2_{x,v}$-inner product with $\partial^m_x G$. Then we have
\begin{equation}\label{Spatial-0}
  \begin{aligned}
    & \tfrac{1}{2} \tfrac{\d}{\d t} \| \partial^m_x G \|^2_{L^2_{x,v}} + \tfrac{1}{\eps^2} \l \mathscr{L} \partial^m_x G , \partial^m_x G \r_{L^2_{x,v}} - \tfrac{1}{\eps} \l ( \partial^m_x E \cdot v ) \sqrt{M} \mathsf{q}_1, \partial^m_x G \r_{L^2_{x,v}} \\
    = & \tfrac{1}{2} \l \mathsf{q} \partial^m_x [ (E \cdot v) G ] , \partial^m_x G \r_{L^2_{x,v}} + \tfrac{1}{\eps} \l \partial^m_x \Gamma (G, G) , \partial^m_x G \r_{L^2_{x,v}} \\
    & - \tfrac{1}{\eps} \sum_{0 \neq m' \leq m} \l \mathsf{q} \partial^{m'}_x ( \eps E + v \times B ) \cdot \nabla_v \partial^{m-m'}_x G , \partial^m_x G \r_{L^2_{x,v}} \,.
  \end{aligned}
\end{equation}
First, from Part (3) of Lemma \ref{Lmm-L}, we know that there is a $\lambda > 0$ such that
\begin{equation}\label{Spatial-1}
  \tfrac{1}{\eps^2} \l \mathscr{L} \partial^m_x G , \partial^m_x G \r_{L^2_{x,v}} \geq \tfrac{\lambda}{\eps^2} \| \partial^m_x \mathbb{P}^\perp G \|^2_{L^2_{x,v}(\nu)} \,,
\end{equation}
which gives us the kinetic dissipation. Next, by the second Amp\`ere's equation and the third Faraday's equation in \eqref{VMB-G-drop-eps}, we compute that
\begin{equation}\label{Spatial-2}
  \begin{aligned}
    - & \tfrac{1}{\eps} \l ( \partial^m_x E \cdot v ) \sqrt{M} \mathsf{q}_1, \partial^m_x G \r_{L^2_{x,v}} = - \tfrac{1}{\eps} \int_{\T^3} \left( \int_{\R^3} \partial^m_x G \cdot \mathsf{q}_1 v \sqrt{M} \d v \right) \cdot \partial^m_x E \d x \\
    = & \int_{\T^3} \left( \partial_t \partial^m_x E - \nabla_x \times \partial^m_x B \right) \cdot \partial^m_x E \d x = \tfrac{1}{2} \tfrac{\d}{\d t} \| \partial^m_x E \|^2_{L^2_x} - \int_{\T^3} ( \nabla_x \times \partial^m_x B ) \cdot \partial^m_x E \d x \\
    = & \tfrac{1}{2} \tfrac{\d}{\d t} \| \partial^m_x E \|^2_{L^2_x} - \int_{\T^3} ( \nabla_x \times \partial^m_x E ) \cdot \partial^m_x B \d x = \tfrac{1}{2} \tfrac{\d}{\d t} \| \partial^m_x E \|^2_{L^2_x} + \int_{\T^3} \partial_t \partial^m_x B \cdot \partial^m_x B \d x \\
    = &  \tfrac{1}{2} \tfrac{\d}{\d t} \left( \| \partial^m_x E \|^2_{L^2_x} + \| \partial^m_x B \|^2_{L^2_x} \right) \,,
  \end{aligned}
\end{equation}
where we utilize the relation $ \int_{\T^3} ( \nabla_x \times \partial^m_x B ) \cdot \partial^m_x E \d x = \int_{\T^3} ( \nabla_x \times \partial^m_x E ) \cdot \partial^m_x B \d x $.

We now deal with the term $\tfrac{1}{2} \l \mathsf{q} \partial^m_x [ (E \cdot v) G ] , \partial^m_x G \r_{L^2_{x,v}}$ for all $|m| \leq s$. We make use of the decomposition $G = \mathbb{P} G + \mathbb{P}^\perp G$ and then obtain 
\begin{equation}
  \begin{aligned}
    & \tfrac{1}{2} \l \mathsf{q} \partial^m_x [ (E \cdot v) G ] , \partial^m_x G \r_{L^2_{x,v}} \\
    = & \tfrac{1}{2} \sum_{m' \leq m} C_m^{m'} \l ( \partial^{m'}_x E \cdot v ) \left( \mathsf{q} \partial^{m-m'}_x \mathbb{P} G + \mathsf{q} \partial^{m-m'}_x \mathbb{P}^\perp G \right) , \partial^m_x \mathbb{P} G + \partial^m_x \mathbb{P} G \r_{L^2_{x,v}}  \\
    = & \underset{\textbf{A}_1}{ \underbrace{ \tfrac{1}{2} \sum_{m' \leq m} C_m^{m'} \l ( \partial^{m'}_x E \cdot v ) \mathsf{q} \partial^{m-m'}_x \mathbb{P} G , \partial^m_x \mathbb{P} G \r_{L^2_{x,v}} } } \\
    & + \underset{\textbf{A}_2}{ \underbrace{ \tfrac{1}{2} \sum_{m' \leq m} C_m^{m'} \l ( \partial^{m'}_x E \cdot v ) \mathsf{q} \partial^{m-m'}_x \mathbb{P} G , \partial^m_x \mathbb{P}^\perp G \r_{L^2_{x,v}} } } \\
    & + \underset{\textbf{A}_3}{ \underbrace{ \tfrac{1}{2} \sum_{m' \leq m} C_m^{m'} \l ( \partial^{m'}_x E \cdot v ) \mathsf{q} \partial^{m-m'}_x \mathbb{P}^\perp G , \partial^m_x \mathbb{P} G \r_{L^2_{x,v}} } } \\
    & + \underset{\textbf{A}_4}{ \underbrace{ \tfrac{1}{2} \sum_{m' \leq m} C_m^{m'} \l ( \partial^{m'}_x E \cdot v ) \mathsf{q} \partial^{m-m'}_x \mathbb{P}^\perp G , \partial^m_x \mathbb{P}^\perp G \r_{L^2_{x,v}} } } \,.
  \end{aligned}
\end{equation}
For the term $\textbf{A}_1$, we derive from the definition of \eqref{Boltzmann-Proj}, the H\"older inequality, the Sobolev embeddings $ H^2_x (\T^3) \hookrightarrow L^\infty_x (\T^3) $, $ H^1_x (\T^3) \hookrightarrow L^4_x (\T^3) $ that for all $|m| \leq s \, (s \geq 2)$
\begin{equation}
  \begin{aligned}
    \textbf{A}_1 \leq & C \sum_{0 \neq m' < m} \l |\partial^{m'}_x E| \, | v \mathsf{q} \partial^{m-m'}_x \mathbb{P} G | \,, | \partial^m_x \mathbb{P} G | \r_{L^2_{x,v}} \\
    & + C \l |\partial^m_x E | \, |v \mathsf{q} \mathbb{P} G| + |E| \, |v \mathsf{q} \partial^m_x \mathbb{P} G| \,, | \partial^m_x \mathbb{P} G | \r_{L^2_{x,v}} \\
    \leq & C \sum_{0 \neq m' < m} \| \partial^m_x E \|_{L^4_x} \| \partial^{m-m'}_x \mathbb{P} G \|_{L^4_x L^2_v} \| \partial^m_x \mathbb{P} G \|_{L^2_{x,v}} \\
    & + C \left( \| \partial^m_x E \|_{L^2_x} \| \mathbb{P} G \|_{L^\infty_x L^2_v} + \| E \|_{L^\infty_x} \| \partial^m_x \mathbb{P} G \|_{L^2_{x,v}} \right) \| \partial^m_x \mathbb{P} G \|_{L^2_{x,v}} \\
    \leq & C \| E \|_{H^s_x} \| \mathbb{P} G \|^2_{H^s_x L^2_v} \,.
  \end{aligned}
\end{equation}
For the terms $\textbf{A}_2$ and $\textbf{A}_3$, we similarly have
\begin{equation}
  \textbf{A}_2 + \textbf{A}_3 \leq C \| E \|_{H^s_x} \| \mathbb{P} G \|_{H^s_x L^2_v} \| \mathbb{P}^\perp G \|_{H^s_x L^2_v (\nu)} \,.
\end{equation}
For the term $\textbf{A}_4$,
\begin{equation}
  \begin{aligned}
    \textbf{A}_4 \leq & C \sum_{0 \neq m' < m} \l | \partial^{m'}_x E | \, | \nu \mathsf{q} \partial^{m-m'}_x \mathbb{P}^\perp G | \,, | \partial^m_x \mathbb{P}^\perp G | \r_{L^2_{x,v}} \\
    & + C \l |\partial^m_x E| \, | \nu \mathsf{q} \mathbb{P}^\perp G | + | E | \, | \nu(v) \mathsf{q} \partial^m_x \mathbb{P}^\perp G | \,, | \partial^m_x \mathbb{P}^\perp G | \r_{L^2_{x,v}} \\
    \leq & C \sum_{0 \neq m' < m} \| \partial^{m'}_x E \|_{L^4_x} \| \partial^{m-m'}_x \mathbb{P}^\perp G \|_{L^4_x L^2_v (\nu)} \| \partial^m_x \mathbb{P}^\perp G \|_{L^2_{x,v}(\nu)} \\
    + & C \left( \| \partial^m_x E \|_{L^2_x} \| \mathbb{P}^\perp G \|_{L^\infty_x L^2_v (\nu)} + \| E \|_{L^\infty_x} \| \partial^m_x \mathbb{P}^\perp G \|_{L^2_{x,v} (\nu)} \right) \| \partial^m_x \mathbb{P}^\perp G \|_{L^2_{x,v}(\nu)} \\
    \leq & C \| E \|_{H^s_x} \| \mathbb{P}^\perp G \|_{H^s_x L^2_v (\nu)} \| \mathbb{P}^\perp G \|_{H^s_x L^2_v (\nu)} \,,
  \end{aligned}
\end{equation}
holds for all $|m| \leq s \, (s \geq 2)$, which is derived from the fact $|v| \leq \nu(v)$ (implied by the part (1) of Lemma \ref{Lmm-CF-nu}), the H\"older inequality, the relation \eqref{Inq-xv-mixed-Holder}, i.e., $ \big\| \nabla_x \| g \|_{L^2_v (w)} \big\|_{L^2_x} \leq \| \nabla_x g \|_{L^2_{x,v} (w)} $ and the Sobolev embeddings $ H^2_x (\T^3) \hookrightarrow L^\infty_x (\T^3) $, $ H^1_x (\T^3) \hookrightarrow L^4_x (\T^3) $. In summary, we have
\begin{equation}\label{Spatial-3}
  \begin{aligned}
     \tfrac{1}{2} & \l \mathsf{q} \partial^m_x [ (E \cdot v) G ] , \partial^m_x G \r_{L^2_{x,v}} \\
     & \leq C \left( \| \mathbb{P} G \|^2_{H^s_x L^2_v} + \| \mathbb{P}^\perp G \|^2_{H^s_x L^2_v (\nu)} + \| \mathbb{P}^\perp G \|_{H^s_x L^2_v (\nu)} \| \mathbb{P}^\perp G \|_{H^s_x L^2_v (\nu)} \right) \| E \|_{H^s_x} \,.
  \end{aligned}
\end{equation}

We now estimate the term $\frac{1}{\eps} \l \partial^m_x \Gamma (G, H) \,, \partial^m_x G \r_{L^2_{x,v}}$ by employing Lemma \ref{Lmm-Gamma-Torus}. More precisely, 
\begin{equation}\label{Spatial-4}
  \begin{aligned}
    & \frac{1}{\eps} \l \partial^m_x \Gamma (G, G) \,, \partial^m_x G \r_{L^2_{x,v}} = \frac{1}{\eps} \l \partial^m_x \Gamma (G, G) \,, \partial^m_x \mathbb{P}^\perp G \r_{L^2_{x,v}} \\
    \leq & \frac{C_\Gamma}{\eps} \| G \|_{H^s_x L^2_v} \| G \|_{H^s_x L^2_v(\nu)} \| \mathbb{P}^\perp G \|_{H^s_x L^2_v (\nu)} \\
    \leq & \frac{C_\Gamma}{\eps} \| G \|_{H^s_x L^2_v} \left( \| \mathbb{P} G \|_{H^s_x L^2_v(\nu)} + \| \mathbb{P}^\perp G \|_{H^s_x L^2_v(\nu)} \right) \| \mathbb{P}^\perp G \|_{H^s_x L^2_v (\nu)} \\
    \leq & \frac{C_\Gamma}{\eps} \| G \|_{H^s_x L^2_v} \left( \| \mathbb{P} G \|_{H^s_x L^2_v} + \| \mathbb{P}^\perp G \|_{H^s_x L^2_v(\nu)} \right) \| \mathbb{P}^\perp G \|_{H^s_x L^2_v (\nu)} \,.
  \end{aligned}
\end{equation}

It remains to estimate the term $ - \frac{1}{\eps} \sum\limits_{0 \neq m' \leq m} C_m^{m'} \l \mathsf{q} \partial^{m'}_x ( \eps E + v \times B ) \cdot \nabla_v \partial^{m-m'}_x G , \partial^m_x G \r_{L^2_{x,v}} $ carefully. By using the relation $G = \mathbb{P} G + \mathbb{P}^\perp G$, it can easily be decomposed as four parts:
\begin{equation}
  \begin{aligned}
     & - \tfrac{1}{\eps} \sum_{0 \neq m' \leq m} C_m^{m'} \l \mathsf{q} \partial^{m'}_x ( \eps E + v \times B ) \cdot \nabla_v \partial^{m-m'}_x G , \partial^m_x G \r_{L^2_{x,v}} \\
     = & \underset{\textbf{B}_1}{ \underbrace{ - \tfrac{1}{\eps} \sum_{0 \neq m' \leq m} C_m^{m'} \l \mathsf{q} \partial^{m'}_x ( \eps E + v \times B ) \cdot \nabla_v \partial^{m-m'}_x \mathbb{P} G , \partial^m_x \mathbb{P} G \r_{L^2_{x,v}}  } } \\
     &  \underset{\textbf{B}_2}{ \underbrace{ - \tfrac{1}{\eps} \sum_{0 \neq m' \leq m} C_m^{m'} \l \mathsf{q} \partial^{m'}_x ( \eps E + v \times B ) \cdot \nabla_v \partial^{m-m'}_x \mathbb{P}^\perp G , \partial^m_x \mathbb{P} G \r_{L^2_{x,v}}  } } \\
     &  \underset{\textbf{B}_3}{ \underbrace{ - \tfrac{1}{\eps} \sum_{0 \neq m' \leq m} C_m^{m'} \l \mathsf{q} \partial^{m'}_x ( \eps E + v \times B ) \cdot \nabla_v \partial^{m-m'}_x \mathbb{P} G , \partial^m_x \mathbb{P}^\perp G \r_{L^2_{x,v}}  } } \\
     &  \underset{\textbf{B}_4}{ \underbrace{ - \tfrac{1}{\eps} \sum_{0 \neq m' \leq m} C_m^{m'} \l \mathsf{q} \partial^{m'}_x ( \eps E + v \times B ) \cdot \nabla_v \partial^{m-m'}_x \mathbb{P}^\perp G , \partial^m_x \mathbb{P}^\perp G \r_{L^2_{x,v}}  } } \,.
  \end{aligned}
\end{equation}
The {\em key point} is to deal with the singularity $\frac{1}{\eps}$ occurring in the terms $\textbf{B}_1$, $\textbf{B}_2$, $\textbf{B}_3$ and $\textbf{B}_4$. Thanks to the kinetic dissipation term $\tfrac{1}{\eps^2} \| \mathbb{P}^\perp G \|^2_{H^s_x L^2_v (\nu)}$ derived from the non-negativity of $\mathscr{L}$ in Lemma \ref{Lmm-L}, the singular term $\frac{1}{\eps} \mathbb{P}^\perp G$ in $\textbf{B}_2$, $\textbf{B}_3$ and $\textbf{B}_4$ will be absorbed after some subtle calculations. The singular part in $\textbf{B}_1$ is actuary the term 
$$ \tfrac{1}{\eps} \sum\limits_{0 \neq m' \leq m} C_m^{m'} \l \mathsf{q} \partial^{m'}_x (v \times B ) \cdot \nabla_v \partial^{m-m'}_x \mathbb{P} G , \partial^m_x \mathbb{P} G \r_{L^2_{x,v}} \,.$$
Recalling the definition of $\mathbb{P} G$ in \eqref{VMB-Proj}, we have
\begin{equation}
  \begin{aligned}
    \nabla_v \mathbb{P} G = (u + v \theta ) \mathsf{q}_2 \sqrt{M} - \tfrac{1}{2} v \mathbb{P} G \,,
  \end{aligned}
\end{equation}
where $\mathsf{q}_2 = [1,1] \in \mathbb{R}^2$, $\theta = \frac{1}{2} \l G , \frac{2}{3} \phi_6 \r_{L^2_v}$ and the vector field $u = [ u_1, u_2, u_3 ] \in \mathbb{R}^3$ with the components $u_i = \frac{1}{2} \l G , \phi_{2+i} \r_{L^2_v}$ ($i = 1,2,3$). Then we have
\begin{equation}\label{Cancellation-B-G}
  \begin{aligned}
    & \tfrac{1}{\eps} \sum\limits_{0 \neq m' \leq m} C_m^{m'} \l \mathsf{q} \partial^{m'}_x (v \times B ) \cdot \nabla_v \partial^{m-m'}_x \mathbb{P} G , \partial^m_x \mathbb{P} G \r_{L^2_{x,v}} \\
    = & \tfrac{1}{\eps} \sum\limits_{0 \neq m' \leq m} C_m^{m'} \l v \times \partial^{m'}_x B \cdot \left[ (\partial^{m-m'}_x u + v \partial^{m-m'}_x \theta ) \mathsf{q} \mathsf{q}_2 \sqrt{M} - \tfrac{1}{2} v \mathsf{q} \partial^{m-m'}_x \mathbb{P} G \right] , \partial^m_x \mathbb{P} G \r_{L^2_{x,v}} \\
    = & \tfrac{1}{\eps} \sum\limits_{0 \neq m' \leq m} C_m^{m'} \l ( v \times \partial^{m'}_x B ) \cdot \partial^{m-m'}_x u , \mathsf{q}_1 \sqrt{M} \cdot \partial^m_x \mathbb{P} G \r_{L^2_{x,v}} \\
    & + \tfrac{1}{\eps} \sum\limits_{0 \neq m' \leq m} C_m^{m'} \l ( v \times \partial^{m'}_x B ) \cdot v , \left( \partial^{m-m'}_x \theta \mathsf{q}_1 \sqrt{M} - \frac{1}{2} \mathsf{q} \partial^{m-m'}_x \mathbb{P} G  \right) \cdot \partial^m_x \mathbb{P} G \r_{L^2_{x,v}} \\
    = & \tfrac{1}{\eps} \sum\limits_{0 \neq m' \leq m} C_m^{m'} \l ( v \times \partial^{m'}_x B ) \cdot \partial^{m-m'}_x u , \partial^m_x \l G, \phi_1 - \phi_2 \r_{L^2_v} M \r_{L^2_{x,v}} \\
    = & \tfrac{1}{\eps} \sum\limits_{0 \neq m' \leq m} C_m^{m'} \l ( \l v , M \r_{L^2_v} \times \partial^{m'}_x B ) \cdot \partial^{m-m'}_x u , \partial^m_x \l G , \phi_1 - \phi_2 \r_{L^2_v} \r_{L^2_x} \\
    = & 0 \,,
  \end{aligned}
\end{equation}
where we make use of the cancellations $( v \times \partial^{m'}_x B ) \cdot v = 0$, $ \mathsf{q}_1 \sqrt{M} \cdot \partial^m_x \mathbb{P} G = \partial^m_x \l G , \phi_1 - \phi_2 \r_{L^2_v} M $ and $ \l v , M \r_{L^2_v} = 0 $. We thereby know that the term $\textbf{B}_1$ does not involve the singularity.

Based on the above statements, we now estimate the terms $\textbf{B}_1$, $\textbf{B}_2$, $\textbf{B}_3$ and $\textbf{B}_4$ one by one. For the term $\textbf{B}_1$, we have
\begin{equation}
  \begin{aligned}
    \textbf{B}_1 = & - \sum_{0 \neq m' < m} C_m^{m'} \l \mathsf{q} \partial^{m'}_x E \cdot \nabla_v \partial^{m-m'}_x \mathbb{P} G , \partial^m_x \mathbb{P} G \r_{L^2_{x,v}} - \l \mathsf{q} \partial^m_x E \cdot \nabla_v \mathbb{P} G , \partial^m_x \mathbb{P} G \r_{L^2_{x,v}} \\
    \leq & C \sum_{0 \neq m' < m} \| \partial^{m'}_x E \|_{L^4_x} \| \partial^{m-m'}_x \mathbb{P} G \|_{L^4_x L^2_v } \| \partial^m_x \mathbb{P} G \|_{L^2_{x,v}} + C \| \partial^m_x E \|_{L^2_x} \| \mathbb{P} G \|_{L^\infty_x L^2_v } \| \partial^m_x \mathbb{P} G \|_{L^2_{x,v}} \\
    \leq & C \| E \|_{H^s_x} \| \mathbb{P} G \|_{H^s_x L^2_v} \| \nabla_x \mathbb{P} G \|_{H^{s-1}_x L^2_v}
  \end{aligned}
\end{equation}
for all $|m| \leq s$ ($s \geq 2$), where we make use of the H\"older inequality, the Sobolev embeddings $H^1_x (\T^3) \hookrightarrow L^4_x (\T^3)$ and $H^2_x (\T^3) \hookrightarrow L^\infty_x (\T^3)$, the definition \eqref{Boltzmann-Proj} of $\mathbb{P} G$ and the cancellation \eqref{Cancellation-B-G}. For the term $\textbf{B}_2$, we derive from the similar arguments in estimating the term $\textbf{B}_1$ that
\begin{equation}
  \begin{aligned}
    \textbf{B}_2 = & \tfrac{1}{\eps} \sum_{0 \neq m' \leq m} C_m^{m'} \l \mathsf{q} \partial^{m'}_x ( \eps E + v \times B ) \cdot \partial^m_x  \nabla_v \mathbb{P} G  , \partial^{m-m'}_x \mathbb{P}^\perp G \r_{L^2_{x,v}} \\
    \leq & \frac{C}{\eps} \sum_{0 \neq m' < m} \l \eps | \partial^{m'}_x E | + | \partial^{m'}_x B | , \| \partial^m_x \mathbb{P} G \|_{L^2_v} \| \partial^{m-m'}_x \mathbb{P}^\perp G \|_{L^2_v} \r_{L^2_x} \\
    & + \frac{C}{\eps} \l \eps | \partial^m_x E | + | \partial^m_x B | , \| \partial^m_x \mathbb{P} G \|_{L^2_v} \| \mathbb{P}^\perp G \|_{L^2_v} \r_{L^2_x} \\
    \leq & \frac{C}{\eps} \sum_{0 \neq m' < m} \left( \eps \| \partial^{m'}_x E \|_{L^4_x} + \| \partial^{m'}_x B \|_{L^4_x} \right) \| \partial^m_x \mathbb{P} G \|_{L^2_{x,v}} \| \partial^{m-m'}_x \mathbb{P}^\perp G \|_{L^4_x L^2_v} \\
    & + \frac{C}{\eps} \left( \eps \| \partial^m_x E \|_{L^2_x} + \| \partial^m_x B \|_{L^2_x} \right) \| \partial^m_x \mathbb{P} G \|_{L^2_{x,v}} \| \mathbb{P}^\perp G \|_{L^\infty_x L^2_v} \\
    \leq & \frac{C}{\eps}  \left( \eps \| E \|_{H^s_x} + \| B \|_{H^s_x} \right) \| \mathbb{P} G \|_{H^s_x L^2_v} \| \mathbb{P}^\perp G \|_{H^s_x L^2_v} \,.
  \end{aligned}
\end{equation}
The term $\textbf{B}_3$ can be similarly estimated as
\begin{equation}
  \begin{aligned}
     \textbf{B}_3 \leq \frac{C}{\eps}  \left( \eps \| E \|_{H^s_x} + \| B \|_{H^s_x} \right) \| \mathbb{P} G \|_{H^s_x L^2_v} \| \mathbb{P}^\perp G \|_{H^s_x L^2_v} \,.
  \end{aligned}
\end{equation}
We next control the term $\textbf{B}_4$. By the H\"older inequality, the relation $\max \{ 1, |v| \} \leq C \nu (v)$ derived from the part (1) of Lemma \ref{Lmm-CF-nu} and the Sobolev embedding theories, we have
\begin{equation}
  \begin{aligned}
     \textbf{B}_4 \leq & \frac{C}{\eps} \sum_{|m'|=1} \l \eps |\partial^{m'}_x E| + \nu (v) | \partial^{m'}_x B | \,, | \partial^{m-m'}_x \nabla_v \mathbb{P}^\perp G | \, | \partial^m_x \mathbb{P}^\perp G | \r_{L^2_{x,v}} \\
     + & \frac{C}{\eps} \sum_{2 \leq |m'| < |m|} \l \eps |\partial^{m'}_x E| + \nu (v) | \partial^{m'}_x B | \,, | \partial^{m-m'}_x \nabla_v \mathbb{P}^\perp G | \, | \partial^m_x \mathbb{P}^\perp G | \r_{L^2_{x,v}} \\
     + & \frac{C}{\eps} \l \eps |\partial^m_x E| + \nu (v) | \partial^m_x B | \,, | \nabla_v \mathbb{P}^\perp G | \, | \partial^m_x \mathbb{P}^\perp G | \r_{L^2_{x,v}} \\
     \leq & \frac{C}{\eps} \sum_{|m'|=1} \left( \eps \| \partial^{m'}_x E \|_{L^\infty_x} + \| \partial^{m'}_x B \|_{L^\infty_x} \right) \| \partial^{m-m'}_x \nabla_v \mathbb{P}^\perp G \|_{L^2_{x,v}(\nu)} \| \partial^m_x \mathbb{P}^\perp G \|_{L^2_{x,v}(\nu)} \\
     + & \tfrac{C}{\eps} \sum_{2 \leq |m'| < |m|} \left( \eps \| \partial^{m'}_x E \|_{L^4_x} + \| \partial^{m'}_x B \|_{L^4_x} \right) \| \partial^{m-m'}_x \nabla_v \mathbb{P}^\perp G \|_{L^4_x L^2_v(\nu)} \| \partial^m_x \mathbb{P}^\perp G \|_{L^2_{x,v}(\nu)} \\
     + & \tfrac{C}{\eps}  \left( \eps \| \partial^m_x E \|_{L^2_x} + \| \partial^m_x B \|_{L^2_x} \right) \| \nabla_v \mathbb{P}^\perp G \|_{L^\infty_x L^2_v(\nu)} \| \partial^m_x \mathbb{P}^\perp G \|_{L^2_{x,v}(\nu)} \\
     \leq & \tfrac{C}{\eps}  \left( \eps \| E \|_{H^s_x} + \| B \|_{H^s_x} \right) \| \mathbb{P}^\perp G \|_{H^s_x L^2_v(\nu)} \sum_{|m| \leq s - 1} \| \nabla_v \partial^m_x \mathbb{P}^\perp G \|_{L^2_{x,v}(\nu)} 
  \end{aligned}
\end{equation}
for all $|m| \leq s$ and $s \geq 3$. We summarize the all estimates on terms $\textbf{B}_1$, $\textbf{B}_2$, $\textbf{B}_3$ and $\textbf{B}_4$ above and then we obtain
\begin{equation}\label{Spatial-5}
  \begin{aligned}
    & - \frac{1}{\eps} \sum\limits_{0 \neq m' \leq m} C_m^{m'} \l \mathsf{q} \partial^{m'}_x ( \eps E + v \times B ) \cdot \nabla_v \partial^{m-m'}_x G , \partial^m_x G \r_{L^2_{x,v}} \\
    \leq & \tfrac{C}{\eps}  \left( \eps \| E \|_{H^s_x} + \| B \|_{H^s_x} \right) \left( \| \nabla_x \mathbb{P} G \|_{H^{s-1}_x L^2_v} + \| \mathbb{P}^\perp G \|_{H^s_x L^2_v(\nu)} \right) \\ 
    & \quad \times \sum_{|m| \leq s - 1} \| \nabla_v \partial^m_x \mathbb{P}^\perp G \|_{L^2_{x,v}(\nu)}  + C \| E \|_{H^s_x} \| \mathbb{P} G \|^2_{H^s_x L^2_v} 
  \end{aligned}
\end{equation}
for all $|m| \leq s$. Here we require the integer $s \geq 3$. Plugging the relations \eqref{Spatial-1}, \eqref{Spatial-2}, \eqref{Spatial-3}, \eqref{Spatial-4} and  \eqref{Spatial-5} into the equality \eqref{Spatial-0} reduces to
\begin{equation}
  \begin{aligned}
     & \tfrac{1}{2} \tfrac{\d }{\d t} \left( \| \partial^m_x G \|^2_{L^2_{x,v}} + \| \partial^m_x E \|^2_{L^2_x} + \| \partial^m_x B \|^2_{L^2_x} \right) + \tfrac{\lambda}{\eps^2} \| \partial^m_x \mathbb{P}^\perp G \|^2_{L^2_{x,v}(\nu)} \\
     \leq & C \| E \|_{H^s_x} \left( \| \mathbb{P} G \|^2_{H^s_x L^2_v} + \| \mathbb{P}^\perp G \|^2_{H^s_x L^2_v (\nu)} \right) \\
     + & C \| E \|_{H^s_x} \| \mathbb{P}^\perp G \|_{H^s_x L^2_v (\nu)} \Big( \mathbb{P}^\perp G \|_{H^s_x L^2_v (\nu)} + \sum_{|m| \leq s - 1} \| \nabla_v \partial^m_x \mathbb{P}^\perp G \|_{L^2_{x,v}(\nu)} \Big) \\
     + & \frac{C}{\eps} \left( \| G \|_{H^s_x L^2_v} + \| B \|_{H^s_x} \right) \| \mathbb{P}^\perp G \|_{H^s_x L^2_v (\nu)} \\
     & \quad \times \Big( \| \mathbb{P} G \|_{H^s_x L^2_v} + \| \mathbb{P}^\perp G \|_{H^s_x L^2_v (\nu)} +  \sum_{|m| \leq s - 1} \| \nabla_v \partial^m_x \mathbb{P}^\perp G \|_{L^2_{x,v}(\nu)} \Big)
  \end{aligned}
\end{equation}
for all multi-indexes $m \in \mathbb{N}^3$ with $|m| \leq s$ and $s \geq 3$. Summing up for all $|m| \leq s$ in the above inequality implies the bound \eqref{Spatial-Bnd}. Then the proof of Proposition \ref{Prop-Spatial} is finished.
\end{proof}

\subsection{Micro-macro decomposition on the equation of $G$} In this subsection, we will find a dissipation of the fluid part $\mathbb{P} G$ by using the so-called micro-macro decomposition method for the VMB version, which is inspired by that for the Boltzmann version. It is well-known that the micro-macro decomposition method for the Boltzmann version is actually depended on the so-called {\em thirteen moments}, see \cite{Guo-2006-CPAM} for instance. 

However, in order to obtain the dissipative term of the fluid part of the perturbed two species VMB system, we will introduce the following linear independent basis in $L^2_v$
\begin{equation}\label{Basis-B}
  \begin{aligned}
    \mathfrak{B} = \Big\{ [1,0]\sqrt{M}, \ & [0,1]\sqrt{M}, \ [v_i, 0] \sqrt{M}, \ [0,v_i] \sqrt{M}, \ [v_i^2, v_i^2] \sqrt{M}, \\
    & [v_i |v|^2, v_i |v|^2] \sqrt{M}, \ [v_j v_k , v_j v_k] \sqrt{M} ; \ 1 \leq i \leq 3, 1 \leq j < k \leq 3 \Big\} \,,
  \end{aligned}
\end{equation}
which consists of seventeen linear independent moments. We call the basis $\mathfrak{B}$ the {\em seventeen moments}. This can be seen in \cite{Guo-2003-Invent}, for instance. For notational simplicity, we denote by
\begin{equation}
\begin{aligned}
\mathfrak{B} = \Big\{ \beta^{\pm} (v), \ \beta_i^\pm (v), \ \beta_i (v), \ \widetilde{\beta}_i (v), \ \beta_{jk} (v) ; \ 1 \leq i \leq 3, 1 \leq j < k \leq 3 \Big\} \,,
\end{aligned}
\end{equation}
where
\begin{equation}
  \begin{aligned}
    \beta^+ (v) = &\, [1,0] \sqrt{M(v)} \,, \qquad\ \ \beta^- (v) = [0,1] \sqrt{M(v)} \,,\\
    \beta_i^+ (v) = &\, [v_i, 0 ] \sqrt{M(v)} \,, \qquad\ \beta_i^+ (v) = [0, v_i] \sqrt{M(v)} \,, \\
    \beta_i (v) = &\, [ v_i^2 , v_i^2 ] \sqrt{M(v)} \,, \qquad \widetilde{\beta}_i (v) = [ v_i |v|^2 , v_i |v|^2 ] \sqrt{M(v)} \,, \\
    \beta_{jk} (v) = &\, [v_j v_k , v_j v_k] \sqrt{M(v)} \,.
  \end{aligned}
\end{equation}
One can easily justify that $\mathfrak{B}$ is linearly independent in $L^2_v$. Indeed, if 
\begin{equation}\label{Linear-Indep}
  \begin{aligned}
    \sum_\pm k_\pm \beta^\pm (v) + \sum_\pm \sum_{i=1}^3 k_{i \pm} \beta_i^\pm (v) + \sum_{i=1}^3 k_i \beta_i (v) + \sum_{i=1}^3 \widetilde{k}_i \widetilde{\beta}_i (v) + \sum_{1 \leq j < k \leq 3} k_{jk} \beta_{jk} (v) = 0 \,,
  \end{aligned}
\end{equation}
we take $L^2_v$-inner product in \eqref{Linear-Indep} by multiplying each element in the set $\mathfrak{B}$, and then we obtain
\begin{equation}\label{Linear-Indep-1}
  \left\{
    \begin{array}{l}
      k_\pm + \sum\limits_{i=1}^3 k_i = 0 \,,\\
      k_{i\pm} + 5 k_i = 0 \ (1 \leq i \leq 3) \,, \\
      k_+ + k_- + 6 k_i = 0 \ (1 \leq i \leq 3) \,, \\
      k_{i+} + k_{i-} + 2 \widetilde{k}_i = 0 \ (1 \leq i \leq 3) \,, \\
      k_{jk} = 0 \ (1 \leq j < k \leq 3) \,,
    \end{array}
  \right.
\end{equation}
where we make use of the relations $\l 1, M \r_{L^2_v} = 1$, $\l |v|^2 , M \r_{L^2_v} = 3$ and $\l |v|^4 , M \r_{L^2_v} = 15$. Straightforward calculations imply that the linear system \eqref{Linear-Indep-1} admits only zero solution, namely, $k_\pm = k_{i \pm} = k_i = \widetilde{k}_i = k_{jk} = 0$ for $1 \leq i \leq 3$ and $1 \leq j < k \leq 3$. Consequently, we know that $\mathfrak{B}$ is linearly independent.

We now assume that $\{ e_j \}_{1 \leq j \leq 17}$ is a orthonormal basis of the linear space $\textrm{Span} \{ \mathfrak{B} \}$ with dimensions 17. Then each $e_j$ is a certain linear combination of $\mathfrak{B}$ and $ \textrm{Span} \{ \mathfrak{B} \} = \textrm{Span} \{ e_j ; 1 \leq j \leq 17 \} \subset L^2_v $. We thereby define a projection $ \mathcal{P}_{\mathfrak{B}} : L^2_v \rightarrow \textrm{Span} \{ \mathfrak{B} \} \subset L^2_v $ by
\begin{equation}
  \begin{aligned}
    \mathcal{P}_{\mathfrak{B}} f = \sum_{j=1}^{17} \l f, e_j \r_{L^2_v} e_j
  \end{aligned}
\end{equation}
for any $f \in L^2_v$. Since each $e_j$ can be represented as a certain linear combination of $\mathfrak{B}$, the projection  $ \mathcal{P}_{\mathfrak{B}} $ admits an equivalent form
\begin{equation}\label{PB-proj}
    \mathcal{P}_{\mathfrak{B}} f = \sum_\pm f^\pm \beta^\pm (v) + \sum_\pm \sum_{i=1}^3 f_i^\pm \beta_i^\pm (v) + \sum_{i=1}^3 f_i \beta_i (v) + \sum_{i=1}^3 \widetilde{f}_i \widetilde{\beta}_i (v) + \sum_{1 \leq i < j \leq 3} f_{ij} \beta_{ij} (v) \,,
\end{equation}
where the coefficients $f^\pm$, $f_i^\pm$, $f_i$, $\widetilde{f}_i$ and $f_{ij}$ are only depended on $f$, $\mathfrak{B}$ and $\{ e \}_{1 \leq j \leq 17}$.

We now decompose the first kinetic equation of $G$ in \eqref{VMB-G-drop-eps} from two aspects: 

(I) We first substitute the identity $G = \mathbb{P} G + \mathbb{P}^\perp G$ into the first kinetic equation on $G$, thus rewriting as
\begin{equation}\label{Fluid-G}
  \begin{aligned}
    \eps \partial_t \mathbb{P} G + v \cdot \nabla_x \mathbb{P} G - ( E \cdot v ) \sqrt{M} \mathsf{q}_1 = \Theta (\mathbb{P}^\perp G) + \Psi ( \mathbb{P} G ) + \Gamma ( G, G ) \,,
  \end{aligned}
\end{equation}
where
\begin{equation}\label{Theta-Pperp-G}
  \begin{aligned}
    \Theta(\mathbb{P}^\perp G) = - \left( \eps \partial_t + v \cdot \nabla_x + \tfrac{1}{\eps} \mathscr{L} + \mathsf{q} ( \eps E + v \times B ) \cdot \nabla_v - \tfrac{1}{2} \eps \mathsf{q} ( E \cdot v ) \right) \mathbb{P}^\perp G \,,
  \end{aligned}
\end{equation}
and
\begin{equation}\label{Psi-P-G}
  \begin{aligned}
    \Psi (\mathbb{P} G) = - \mathsf{q} ( \eps E + v \times B ) \cdot \nabla_v \mathbb{P} G + \tfrac{1}{2} \eps \mathsf{q} ( E \cdot v ) \mathbb{P} G \,.
  \end{aligned}
\end{equation}
Based on the definition \eqref{VMB-Proj} of $\mathbb{P} G$, direct calculation implies that the left terms in the equation \eqref{Fluid-G} is
\begin{equation}
  \begin{aligned}
    & \eps \partial_t \mathbb{P} G + v \cdot \nabla_x \mathbb{P} G - ( E \cdot v ) \sqrt{M} \mathsf{q}_1 \\
    = & \sum_\pm \eps \partial_t ( \rho^\pm - \tfrac{3}{2} \theta ) \beta^\pm (v) + \sum_\pm \sum_{i=1}^3\big[ \eps \partial_t u_i + \partial_i ( \rho^\pm - \tfrac{3}{2} \theta ) - (\pm E_i ) \big] \beta_i^\pm (v) \\
    + & \sum_{i=1}^3 \big( \tfrac{1}{2} \eps \partial_t \theta + \partial_i u_i \big) \beta_i (v) + \sum_{i=1}^3 \tfrac{1}{2} \partial_i \theta \widetilde{\beta}_i (v) + \sum_{1 \leq i < j \leq 3} ( \partial_i u_j + \partial_j u_i ) \beta_{ij} (v) \,,
  \end{aligned}
\end{equation}
thus belonging to the space $\textrm{Span} \{ \mathfrak{B} \} \subset L^2_v$. We thereby project the equation \eqref{Fluid-G} into $\textrm{Span} \{ \mathfrak{B} \} $ and obtain
\begin{equation}\label{Deco-17-moments}
  \left\{
     \begin{array}{l}
        \beta^\pm (v) : \ \eps \partial_t ( \rho^\pm - \tfrac{3}{2} \theta ) = \Theta^\pm + \Psi^\pm + \Gamma^\pm \,, \\[2mm]
        \beta_i^\pm (v) : \ \eps \partial_t u_i + \partial_i ( \rho^\pm - \tfrac{3}{2} \theta ) - ( \pm E_i ) = \Theta_i^\pm + \Psi_i^\pm + \Gamma_i^\pm \,, \ 1 \leq i \leq 3 \,, \\[2mm]
        \beta_i (v) \ : \ \tfrac{1}{2} \eps \partial_t \theta + \partial_i u_i = \Theta_i + \Psi_i + \Gamma_i \,, \qquad\qquad\qquad\qquad\ 1 \leq i \leq 3 \,, \\[2mm]
        \widetilde{\beta}_i (v) \ : \ \tfrac{1}{2} \partial_i \theta = \widetilde{\Theta}_i + \widetilde{\Psi}_i + \widetilde{\Gamma}_i \,, \qquad\qquad\qquad\qquad\qquad\quad\ \ 1 \leq i \leq 3 \,, \\[2mm]
        \beta_{ij} (v) : \partial_i u_j + \partial_j u_i = \Theta_{ij} + \Psi_{ij} + \Gamma_{ij} \,, \qquad\qquad\qquad\qquad\! 1 \leq i \neq j \leq 3 \,,
     \end{array}
  \right.
\end{equation}
where all the symbols $\Theta$, $\Psi$ and $\Gamma$ with various indexes are the coefficients of $\mathcal{P}_{\mathfrak{B}} \Theta (\mathbb{P}^\perp G)$, $\mathcal{P}_{\mathfrak{B}} \Psi (\mathbb{P} G)$ and $\mathcal{P}_{\mathfrak{B}} \Gamma (G,G)$, respectively.

(II) We project the first kinetic equation of $G$ in \eqref{VMB-G-drop-eps} into $\textrm{Ker} (\mathscr{L})$ by multiplying the vectors $\phi_1 (v)$, $\phi_2 (v)$, $\tfrac{1}{2} \phi_3(v)$, $\tfrac{1}{2} \phi_4(v)$, $\tfrac{1}{2} \phi_5(v)$ and $\tfrac{2}{3} \phi_6(v)$, respectively, and integrating over $v \in \R^3$. Thanks to the fact $\Gamma (G, G) \in \textrm{Ker}^\perp (\mathscr{L})$ shown in Lemma \ref{Lmm-Gamma-Torus}, the careful calculation reduces to
\begin{equation}\label{Deco-6-moments}
  \left\{
     \begin{array}{l}
        \eps \partial_t \rho^+ + \div_x u = \l \Theta (\mathbb{P}^\perp G) + \Psi (\mathbb{P} G) , \phi_1 (v) \r_{L^2_v} \,, \\[2mm]
        \eps \partial_t \rho^- + \div_x u = \l \Theta (\mathbb{P}^\perp G) + \Psi (\mathbb{P} G) , \phi_2 (v) \r_{L^2_v} \,, \\[2mm]
        \eps \partial_t u_i + \partial_i \big( \tfrac{\rho^+ + \rho^-}{2} + \theta \big) = \tfrac{1}{2} \l \Theta (\mathbb{P}^\perp G) + \Psi (\mathbb{P} G) , \phi_{i+2} (v) \r_{L^2_v} \ \textrm{for } 1 \leq i \leq 3 \,, \\[2mm]
        \eps \partial_t \theta + \tfrac{2}{3} \div_x u = \tfrac{1}{3} \l \Theta (\mathbb{P}^\perp G) + \Psi (\mathbb{P} G) , \phi_6 (v) \r_{L^2_v} \,.
     \end{array}
  \right.
\end{equation}

Based on the decompositions \eqref{Deco-17-moments} and \eqref{Deco-6-moments}, we can derive the dissipation $\| \nabla_x \mathbb{P} G \|^2_{H^{s-1}_x L^2_v}$ of the fluid part $\mathbb{P} G$ for the integer $s \geq 3$. More precisely, we will give the following proposition.
\begin{proposition}\label{Prop-MM-Est}
	Assume that $(G, E, B)$ is the solution to the perturbed VMB system \eqref{VMB-G} constructed in Proposition \ref{Prop-Local-Solutn}. Let integer $s \geq 3$. Then there is positive constant $C > 0$, independent of $\eps > 0$, such that
	\begin{equation}\label{Fluid-Dissip}
	  \begin{aligned}
	     & \| \nabla_x \mathbb{P} G \|^2_{H^{s-1}_x L^2_v} + \| \div_x E \|^2_{H^{s-1}_x} \leq - \eps \tfrac{\d}{\d t} \mathscr{A}_s (G)（(t) + \tfrac{C}{\eps^2} \| \mathbb{P}^\perp G \|^2_{H^{s}_x L^2_v (\nu)} \\
	     + & C \big( \| G \|^2_{H^s_x L^2_v} + \| E \|^2_{H^s_x} + \| B \|^2_{H^s_x} \big) \big( \| \nabla_x \mathbb{P} G \|^2_{H^{s-1}_x L^2_v} + \tfrac{1}{\eps^2} \| \mathbb{P}^\perp G \|^2_{H^s_x L^2_v (\nu)} \big) \\
	     + &  C \big( \| G \|^2_{H^s_x L^2_v} + \| E \|^2_{H^s_x} + \| B \|^2_{H^s_x} \big) \big( (\rho^+)^2_{\T^3} + (\rho^-)^2_{\T^3} + (u)^2_{\T^3} + (\theta)^2_{\T^3} \big)
	  \end{aligned}
	\end{equation}
	for all $0 < \eps \leq 1$, where the quantity $\mathscr{A}_s (G) (t)$ is defined as
	\begin{equation}\label{Quantity-A}
	  \begin{aligned}
	    \mathscr{A}_s (G) (t) = & \sum_{|m| \leq s -1} \sum_{i=1}^3 \Big[ \l 4 \partial^m_x u_i , \partial_i \partial^m_x \rho^+ + \partial_i \partial^m_x \rho^- \r_{L^2_x} + \sum_{j=1}^3 \l 32 \partial^m_x \mathbb{P}^\perp G , \partial_j \partial^m_x u_i \zeta_{ij} (v) \r_{L^2_{x,v}} \\
	    & + \l 32 \partial^m_x \mathbb{P}^\perp G , \partial_i \partial^m_x \theta \widetilde{\zeta}_i (v) \r_{L^2_{x,v}} + \l 4 \partial^m_x \mathbb{P}^\perp G , \partial_i \partial^m_x \zeta_i^+ (v) + \partial_i \partial^m_x \rho^- \zeta_i^- (v) \r_{L^2_{x,v}}  \Big] \,.
	  \end{aligned}
	\end{equation}
	Here the functions $\zeta_i^\pm (v)$, $\widetilde{\zeta}_i (v)$ and $\zeta_{ij} (v)$ are some fixed linear combinations of the basis $\mathfrak{B}$ defined in \eqref{Basis-B}.
\end{proposition}

\begin{proof}[Proof of Proposition \ref{Prop-MM-Est}] 
	We prove this conclusion by three steps: we first derive the three fluid dissipative terms $\| \nabla_x u \|_{H^{s-1}_x}^2$, $\| \nabla_x \theta \|^2_{H^{s-1}_x}$ and $\| \nabla_x \rho^+ \|^2_{H^{s-1}_x} + \| \nabla_x \rho^- \|^2_{H^{s-1}_x}$ from the hydrodynamics relations \eqref{Deco-17-moments} and \eqref{Deco-6-moments}, where $\rho^\pm$, $u$ and $\theta$ are the three parts of the fluid part $\mathbb{P} G$ by three steps. Finally, we combine the estimates obtained in the previous three steps.\\
	
	{\em Step 1. Bounds on $\| \nabla_x u \|^2_{H^{s-1}_x}$ for the integer $s \geq 3$.} For all multi-indexes $m \in \mathbb{N}^3$ satisfying $|m| \leq s - 1$, we derive from the last $u$-equation and the third $\theta$-equation of \eqref{Deco-17-moments} that
	\begin{equation}
	  \begin{aligned}
	    - \Delta_x \partial^m_x u_i = & - \sum_{j=1}^3 \partial_j \partial_j \partial^m_x u_i = - \sum_{j \neq i} \partial_j \partial_j \partial^m_x u_i - \partial_i \partial_i \partial^m_x u_i \\
	    = & - \sum_{j \neq i} \partial^m_x \partial_j \big( - \partial_i u_j + \Theta_{ij} + \Psi_{ij} + \Gamma_{ij} \big) - \partial_i \partial^m_x \big( - \tfrac{1}{2} \eps \partial_t \theta + \Theta_i + \Psi_i + \Gamma_i \big) \\
	    = & \sum_{j \neq i} \partial_i \partial^m_x ( - \tfrac{1}{2} \eps \partial_t \theta + \Theta_j + \Psi_j + \Gamma_j ) - \sum_{j \neq  i} \partial_j \partial^m_x (\Theta_{ij} + \Psi_{ij} + \Gamma_{ij}  ) \\
	    & - \partial_i \partial^m_x ( - \tfrac{1}{2} \eps \partial_t \theta + \Theta_i + \Psi_i + \Gamma_i ) \\
	    = & - \tfrac{1}{2} \eps \partial_t \partial_i \partial^m_x \theta + \sum_{\bm{\Lambda} \in \{ \Theta, \Psi, \Gamma \} } \Big[ \sum_{j \neq i} ( \partial_i \partial^m_x \bm{\Lambda}_j - \partial_j \partial^m_x \bm{\Lambda}_{ij}  ) - \partial_i \partial^m_x \bm{\Lambda}_i \Big] \,.
	  \end{aligned}
	\end{equation}
	Then the definition \eqref{PB-proj} of $\mathcal{P}_{\mathfrak{B}}$ tells us that there is a certain linear combinations $\zeta_{ij} (v)$ of $\mathcal{B}$ such that
	\begin{equation}
	   \begin{aligned}
	      & \sum_{\bm{\Lambda} \in \{ \Theta, \Psi, \Gamma \} } \Big[ \sum_{j \neq i} ( \partial_i \partial^m_x \bm{\Lambda}_j - \partial_j \partial^m_x \bm{\Lambda}_{ij}  ) - \partial_i \partial^m_x \bm{\Lambda}_i \Big] \\
	      = & \sum_{j=1}^3 \partial_j \partial^m_x \l \Theta ( \mathbb{P}^\perp G ) + \Psi ( \mathbb{P} G ) + \Gamma ( G, G ) , \zeta_{ij} (v) \r_{L^2_v} \,.
	   \end{aligned}
	\end{equation}
	So, we have
	\begin{equation}
	  \begin{aligned}
	    - \Delta_x \partial^m_x u_i = - \tfrac{1}{2} \eps \partial_t \partial_i \partial^m_x \theta + \sum_{j=1}^3 \partial_j \partial^m_x \l \Theta ( \mathbb{P}^\perp G ) + \Psi ( \mathbb{P} G ) + \Gamma ( G, G ) , \zeta_{ij} (v) \r_{L^2_v} \,.
	  \end{aligned}
	\end{equation}
	Furthermore, the forth $\theta$-equation of \eqref{Deco-6-moments} gives us
	\begin{equation}
	  \begin{aligned}
	    - \tfrac{1}{2} \eps \partial_t \partial_i \partial^m_x \theta = \tfrac{1}{3} \partial_i \partial^m_x \div_x u - \tfrac{1}{6} \partial_i \partial^m_x \l \Theta (\mathbb{P}^\perp G) + \Psi (\mathbb{P} G) , \phi_6 (v) \r_{L^2_v} \,.
	  \end{aligned}
	\end{equation}
	Then we deduce that
	\begin{equation}
	   \begin{aligned}
	      - \Delta_x \partial^m_x u_i - \tfrac{1}{3} \partial_i \partial^m_x \div_x u = & - \tfrac{1}{6} \partial_i \partial^m_x \l \Theta (\mathbb{P}^\perp G) + \Psi (\mathbb{P} G) , \phi_6 (v) \r_{L^2_v} \\
	      + & \sum_{j=1}^3 \partial_j \partial^m_x \l \Theta ( \mathbb{P}^\perp G ) + \Psi ( \mathbb{P} G ) + \Gamma ( G, G ) , \zeta_{ij} (v) \r_{L^2_v} \,,
	   \end{aligned}
	\end{equation}
	which implies that by multiplying $\partial^m_x u_i$, integrating by parts over $x \in \T^3$ and summing up for $1 \leq i \leq 3$
	\begin{equation}\label{u-Dissip-0}
	  \begin{aligned}
	    & \| \nabla_x \partial^m_x u \|^2_{L^2_x} + \tfrac{1}{3} \| \partial^m_x \div_x u \|^2_{L^2_x} \\
	    = &  \underset{\textbf{C}_1}{ \underbrace{ - \sum_{i = 1}^3 \l \sum_{j=1}^3 \partial^m_x \l \Theta (\mathbb{P}^\perp G) + \partial_t \mathbb{P}^\perp G + \Psi (\mathbb{P} G ) + \Gamma (G, G) , \zeta_{ij} (v) \r_{L^2_v} , \partial^m_x \partial_j u_i \r_{L^2_x} } } \\
	    + & \underset{\textbf{C}_2}{ \underbrace{  \tfrac{1}{6} \sum_{i=1}^3 \l \partial^m_x \l \Theta (\mathbb{P}^\perp G) + \partial_t \mathbb{P}^\perp G + \Psi (\mathbb{P} G ) , \phi_6 (v) \r_{L^2_v} , \partial^m_x \partial_i u_i \r_{L^2_x}  } } \\
	    + & \underset{\textbf{C}_3}{ \underbrace{ \sum_{i,j = 1}^3 \l \eps \l \partial_t \partial^m_x \mathbb{P}^\perp G , \zeta_{ij} (v) \r_{L^2_v} , \partial^m_x \partial_j u_i \r_{L^2_x}  } } \,,
	  \end{aligned}
	\end{equation}
	where we utilize the cancellation $ \l \partial_t \partial^m_x \mathbb{P}^\perp G , \phi_6 (v) \r_{L^2_v} = 0 $ since $\phi_6 (v) \in \textrm{Ker} (\mathscr{L} ) $.
	
	We then estimate the terms $\textbf{C}_1$, $\textbf{C}_2$ and $\textbf{C}_3$ one by one. For the term $\textbf{C}_1$, we decompose as
	\begin{equation}
	  \begin{aligned}
	    \textbf{C}_1 = & \underset{\textbf{C}_1 (\Theta)}{ \underbrace{ - \sum_{i,j = 1}^3 \l \partial^m_x \l \Theta ( \mathbb{P}^\perp G ) + \eps \partial_t \mathbb{P}^\perp G , \zeta_{ij}(v) \r_{L^2_v} , \partial^m_x \partial_j u_i \r_{L^2_x} } } \\
	    & \underset{\textbf{C}_1 (\Psi)}{ \underbrace{ - \sum_{i,j = 1}^3 \l \partial^m_x \l \Psi ( \mathbb{P} G ) , \zeta_{ij}(v) \r_{L^2_v} , \partial^m_x \partial_j u_i \r_{L^2_x} } } \ \underset{\textbf{C}_1 (\Gamma)}{ \underbrace{ - \sum_{i,j = 1}^3 \l \partial^m_x \l \Gamma (G,G) , \zeta_{ij}(v) \r_{L^2_v} , \partial^m_x \partial_j u_i \r_{L^2_x} } } \,.
	  \end{aligned}
	\end{equation}
	Recalling the expression \eqref{Theta-Pperp-G} of $\Theta (\mathbb{P}^\perp G)$, we derive from the integration by parts over $v \in \R^3$, the H\"older inequality and the Sobolev embedding theory that
	\begin{equation}
	  \begin{aligned}
	    \textbf{C}_1 ( \Theta ) = & \l \nabla_x \partial^m_x \mathbb{P}^\perp G , \sum_{i,j=1}^3 v \zeta_{ij}(v) \partial^m_x \partial_j u_i \r_{L^2_{x,v}} + \tfrac{1}{\eps} \l \partial^m_x \mathbb{P}^\perp G , \sum_{i,j = 1}^3 \mathscr{L} \zeta_{ij}(v) \partial^m_x \partial_j u_i \r_{L^2_{x,v}} \\
	    - & \eps \l \partial^m_x ( E \mathbb{P}^\perp G ) , \sum_{i,j = 1}^3 \big( \mathsf{q} \nabla_v \zeta_{ij} (v) - \tfrac{1}{2} v \mathsf{q} \zeta_{ij}(v) \big) \partial^m_x \partial_j u_i \r_{L^2_{x,v}} \\
	    - & \l \partial^m_x ( B \mathbb{P}^\perp G ) , \sum_{i,j = 1}^3 v \times \nabla_v \mathsf{q} \zeta_{ij} (v) \partial^m_x \partial_j u_i \r_{L^2_{x,v}} \\
	    \leq & C \Big( \| \nabla_x \partial^m_x \mathbb{P}^\perp G \|_{L^2_{x,v}} + \tfrac{1}{\eps} \| \partial^m_x \mathbb{P}^\perp G \|_{L^2_{x,v}}  \\
	    & \qquad + \eps \| \partial^m_x ( E \mathbb{P}^\perp G ) \|_{L^2_{x,v}} +  \| \partial^m_x ( B \mathbb{P}^\perp G ) \|_{L^2_{x,v}} \Big) \| \nabla_x \partial^m_x u \|_{L^2_x} \\
	    \leq & \tfrac{C }{\eps} \big( 1 + \| E \|_{H^s_x} + \| B \|_{H^s_x} \big) \| \mathbb{P}^\perp G \|_{H^s_x L^2_v} \| \nabla_x \partial^m_x u \|_{L^2_x} 
	  \end{aligned}
	\end{equation}
	for $0 < \eps \leq 1 $, where the last second inequality is implied by the fact that all functions appearing in the previous estimations involving the function $\zeta_{ij} (v)$ and depending only on the variable $v \in \R^3$ are in $L^2_v$ (since $\zeta_{ij} (v)$ includes a factor $\exp ( - \frac{|v|^2}{4} )$). Similarly, by \eqref{Psi-P-G} and \eqref{VMB-Proj} we have
	\begin{equation}
	  \begin{aligned}
	    \textbf{C}_1 ( \Psi ) = & \l \partial^m_x \big[ ( \eps E + v \times B ) \cdot \nabla_v \mathbb{P} G - \tfrac{1}{2} ( E \cdot v ) \mathbb{P} G \big] , \mathsf{q} \zeta_{ij} (v) \partial^m_x \partial_j u_i \r_{L^2_{x,v}} \\
	    \leq & C ( \| E \|_{H^s_x} + \| B \|_{H^s_x} ) \| \mathbb{P} G \|_{H^s_x L^2_v} \| \nabla_x \partial^m_x u \|_{L^2_x} \,.
	  \end{aligned}
	\end{equation}
	We also derive from the H\"older inequality and Lemma  \ref{Lmm-Gamma-Torus} that
	\begin{equation}
	  \begin{aligned}
	    \textbf{C}_1 ( \Gamma ) \leq & C \| G \|_{H^{s-1}_x L^2_v} \| G \|_{H^{s-1}_x L^2_v (\nu)} \sum_{i,j = 1}^3 \| \partial^m_x \partial_j u_i \zeta_{ij} (v) \|_{L^2_{x,v}(\nu)} \\
	    \leq & C \| G \|_{H^s_x L^2_v} \left( \| \mathbb{P} G \|_{H^s_x L^2_v} + \| \mathbb{P}^\perp G \|_{H^s_x L^2_v (\nu)} \right) \| \nabla_x \partial^m_x u \|_{L^2_x} \,.
	  \end{aligned}
	\end{equation}
	Consequently, we obtain
	\begin{equation}\label{Bnd-C1}
	  \begin{aligned}
	    \textbf{C}_1 = & \textbf{C}_1 ( \Theta ) + \textbf{C}_1 ( \Psi ) + \textbf{C}_1 ( \Gamma ) \\
	    \leq & C \left( \| G \|_{H^s_x L^2_v} + \| E \|_{H^s_x} + \| B \|_{H^s_x} \right) \| \mathbb{P} G \|_{H^s_x L^2_v} \| \nabla_x \partial^m_x u \|_{L^2_x} \\
	    + & \tfrac{C}{\eps} \left( 1 + \| G \|_{H^s_x L^2_v} + \| E \|_{H^s_x} + \| B \|_{H^s_x} \right) \| \mathbb{P}^\perp G \|_{H^s_x L^2_v} \| \nabla_x \partial^m_x u \|_{L^2_x} 
	  \end{aligned}
	\end{equation}
	for $0 < \eps \leq 1$. We derive from the analogous procedures in estimations of $\textbf{C}_1$ that
	\begin{equation}\label{Bnd-C2}
	  \begin{aligned}
	    \textbf{C}_2 \leq & C \left( \| G \|_{H^s_x L^2_v} + \| E \|_{H^s_x} + \| B \|_{H^s_x} \right) \| \mathbb{P} G \|_{H^s_x L^2_v} \| \nabla_x \partial^m_x u \|_{L^2_x} \\
	    + & \tfrac{C}{\eps} \left( 1 + \| G \|_{H^s_x L^2_v} + \| E \|_{H^s_x} + \| B \|_{H^s_x} \right) \| \mathbb{P}^\perp G \|_{H^s_x L^2_v} \| \nabla_x \partial^m_x u \|_{L^2_x} \,.
	  \end{aligned}
	\end{equation}
	
	For the term $\textbf{C}_3$, making use of the third equation of \eqref{Deco-6-moments} reduces to
	\begin{equation}
	  \begin{aligned}
	    \textbf{C}_3 = & - \eps \tfrac{\d}{\d t} \sum_{i,j=1}^3 \l \partial_j \partial^m_x \mathbb{P}^\perp G , \partial^m_x u_i \zeta_{ij} (v) \r_{L^2_{x,v}} + \sum_{i,j=1}^3 \l \partial_j \partial^m_x \mathbb{P}^\perp G , \eps \partial_t  \partial^m_x u_i \zeta_{ij} (v) \r_{L^2_{x,v}} \\
	    = & - \eps \tfrac{\d}{\d t} \sum_{i,j=1}^3 \l \partial_j \partial^m_x \mathbb{P}^\perp G , \partial^m_x u_i \zeta_{ij} (v) \r_{L^2_{x,v}}  \underset{\textbf{C}_{31}}{ \underbrace{ -  \sum_{i,j = 1}^3 \l \partial_j \partial^m_x \mathbb{P}^\perp G , \partial_i \partial^m_x \big( \tfrac{\rho^+ + \rho^-}{2} + \theta \big) \zeta_{ij} (v) \r_{L^2_{x,v}} } } \\
	    + &  \underset{\textbf{C}_{32}}{ \underbrace{ \tfrac{1}{2} \sum_{i,j = 1}^3 \l \partial_j \partial^m_x \mathbb{P}^\perp G , \partial^m_x \l \Theta (\mathbb{P}^\perp G) + \Psi (\mathbb{P} G) , \phi_i (v) \r_{L^2_v} \zeta_{ij} (v) \r_{L^2_{x,v}} }} \,.
	  \end{aligned}
	\end{equation}
	The H\"older inequality implies that the term $\textbf{C}_{31}$ is bounded by
	\begin{equation}
	  \begin{aligned}
	    \textbf{C}_{31} \leq & \sum_{i,j = 1}^3 \| \partial_j \partial^m_x \mathbb{P}^\perp G \|_{L^2_{x,v}(\nu)} \big\| \partial_i \partial^m_x \big( \tfrac{\rho^+ + \rho^-}{2} + \theta \big) \big\|_{L^2_x} \| \zeta_{ij} (v) \|_{L^2_v} \\
	    \leq & C \| \mathbb{P}^\perp G \|_{H^s_x L^2_v (\nu)} \left( \| \nabla_x \partial^m_x \rho^+ \|_{L^2_x} + \| \nabla_x \partial^m_x \rho^- \|_{L^2_x} + \| \nabla_x \partial^m_x \theta \|_{L^2_x} \right) \,,
	  \end{aligned}
	\end{equation}
	where the part (1) of Lemma \ref{Lmm-CF-nu} is also utilized in the last inequality. One employs the similar arguments in estimating the term $\textbf{C}_1$ in \eqref{Bnd-C1} and then yields
	\begin{equation}
	  \begin{aligned}
	    \textbf{C}_{32} \leq & C ( \| E \|_{H^s_x} + \| B \|_{H^s_x} ) \| \mathbb{P} G \|_{H^s_x L^2_v} \| \mathbb{P}^\perp G \|_{H^s_x L^2_v} \\
	    & + C ( 1 + \| E \|_{H^s_x} + \| B \|_{H^s_x} ) \| \mathbb{P}^\perp G \|^2_{H^s_x L^2_v} \,.
	  \end{aligned}
	\end{equation}
	We summarize the above estimates and know
	\begin{equation}\label{Bnd-C3}
	  \begin{aligned}
	    & \textbf{C}_3  \leq - \eps \tfrac{\d}{\d t} \sum_{i,j=1}^3 \l \partial_j \partial^m_x \mathbb{P}^\perp G , \partial^m_x u_i \zeta_{ij} (v) \r_{L^2_{x,v}} \\
	    & + C \| \mathbb{P}^\perp G \|_{H^s_x L^2_v (\nu)} \left( \| \nabla_x \partial^m_x \rho^+ \|_{L^2_x} + \| \nabla_x \partial^m_x \rho^- \|_{L^2_x} + \| \nabla_x \partial^m_x \theta \|_{L^2_x} \right) \\
	    & + C ( \| E \|_{H^s_x} + \| B \|_{H^s_x} ) \| \mathbb{P} G \|_{H^s_x L^2_v} \| \mathbb{P}^\perp G \|_{H^s_x L^2_v} + C ( 1 + \| E \|_{H^s_x} + \| B \|_{H^s_x} ) \| \mathbb{P}^\perp G \|^2_{H^s_x L^2_v} \,.
	  \end{aligned}
	\end{equation}
	As a result, substituting the bounds \eqref{Bnd-C1}, \eqref{Bnd-C2} and \eqref{Bnd-C3} into the equality \eqref{u-Dissip-0} yields that
	\begin{equation*}
	  \begin{aligned}
	     & \| \nabla_x \partial^m_x u \|^2_{L^2_x} + \tfrac{1}{3} \| \partial^m_x \div_x u \|^2_{L^2_x} \leq - \eps \tfrac{\d}{\d t} \sum_{i,j=1}^3 \l \partial_j \partial^m_x \mathbb{P}^\perp G , \partial^m_x u_i \zeta_{ij} (v) \r_{L^2_{x,v}} \\
	     + & C \| \mathbb{P}^\perp G \|_{H^s_x L^2_v (\nu)} \left( \| \nabla_x \partial^m_x \rho^+ \|_{L^2_x} + \| \nabla_x \partial^m_x \rho^- \|_{L^2_x} + \| \nabla_x \partial^m_x \theta \|_{L^2_x} \right) \\
	     + & C ( \| E \|_{H^s_x} + \| B \|_{H^s_x} ) \| \mathbb{P} G \|_{H^s_x L^2_v} \| \mathbb{P}^\perp G \|_{H^s_x L^2_v} + C ( 1 + \| E \|_{H^s_x} + \| B \|_{H^s_x} ) \| \mathbb{P}^\perp G \|^2_{H^s_x L^2_v} \\
	     + & C \left( \| G \|_{H^s_x L^2_v} + \| E \|_{H^s_x} + \| B \|_{H^s_x} \right) \| \mathbb{P} G \|_{H^s_x L^2_v} \| \nabla_x \partial^m_x u \|_{L^2_x} \\
	     + & \tfrac{C}{\eps} \left( 1 + \| G \|_{H^s_x L^2_v} + \| E \|_{H^s_x} + \| B \|_{H^s_x} \right) \| \mathbb{P}^\perp G \|_{H^s_x L^2_v} \| \nabla_x \partial^m_x u \|_{L^2_x} \,,
	  \end{aligned}
	\end{equation*}
	which implies by Young's inequality
	\begin{equation}\label{u-Dissip}
	  \begin{aligned}
	    \| \nabla_x \partial^m_x u \|^2_{L^2_x} + \| \partial^m_x \div_x u \|^2_{L^2_x} \leq & - 3 \eps \frac{\d}{\d t} \sum_{i,j=1}^3 \l \partial_j \partial^m_x \mathbb{P}^\perp G , \partial^m_x u_i \zeta_{ij} (v) \r_{L^2_{x,v}} \\
	    + & C ( \| G \|^2_{H^s_x L^2_v} + \| E \|^2_{H^s_x} + \| B \|^2_{H^s_x} ) \| \mathbb{P} G \|^2_{H^s_x L^2_v} \\
	    + & \frac{C}{\eps} ( 1 + \| G \|^2_{H^s_x L^2_v} + \| E \|^2_{H^s_x} + \| B \|^2_{H^s_x} ) \| \mathbb{P}^\perp G \|^2_{H^s_x L^2_v (\nu)} \\
	    + & \delta ( \| \nabla_x \partial^m_x \rho^+ \|^2_{L^2_x} + \| \nabla_x \partial^m_x \rho^- \|^2_{L^2_x} + \| \nabla_x \partial^m_x \theta \|^2_{L^2_x} )
	  \end{aligned}
	\end{equation}
	for any small $\delta > 0$ to be determined and for all $0 < \eps \leq 1$.\\
	
	{\em Step 2. Bounds on $\| \nabla_x \theta \|^2_{H^{s-1}_x}$ for the integer $s \geq 3$.} For any multi-index $m \in \mathbb{N}^3$ with $|m| \leq s - 1$, we derive from the last second $\partial_i \theta$-equation of \eqref{Deco-17-moments} that
	\begin{equation}
	  \begin{aligned}
	    - \Delta_x \partial^m_x \theta = - \sum_{i=1}^3 \partial_i \partial_i \partial^m_x \theta = - 2 \sum_{i=1}^3 \partial_i \partial^m_x ( \tfrac{1}{2} \partial_i \theta ) = - 2 \sum_{i=1}^3 \partial_i \partial^m_x ( \widetilde{\Theta}_i + \widetilde{\Psi}_i + \widetilde{\Gamma}_i ) \,.
	  \end{aligned}
	\end{equation}
	By the definition \eqref{PB-proj} of $\mathcal{P}_{\mathfrak{B}}$ we know that there are some certain linear combinations $\widetilde{\zeta}_i (v) $ $(i=1,2,3)$ of $\mathfrak{B}$ such that
	\begin{equation*}
	  \begin{aligned}
	     - 2 \sum_{i=1}^3 \partial_i \partial^m_x ( \widetilde{\Theta}_i + \widetilde{\Psi}_i + \widetilde{\Gamma}_i ) \\
	     = \sum_{i=1}^3 \partial_i \partial^m_x \left[ \l \Theta ( \mathbb{P}^\perp G ) , \widetilde{\zeta}_i (v) \r_{L^2_v} + \l \Psi ( \mathbb{P} G ) , \widetilde{\zeta}_i (v) \r_{L^2_v} + \l \Gamma ( G , G ) , \widetilde{\zeta}_i (v) \r_{L^2_v} \right] \,,
	  \end{aligned}
	\end{equation*}
	which tells us
	\begin{equation}
	  \begin{aligned}
	    - \Delta_x \partial^m_x \theta = \sum_{i=1}^3 \partial_i \partial^m_x \left[ \l \Theta ( \mathbb{P}^\perp G ) , \widetilde{\zeta}_i (v) \r_{L^2_v} + \l \Psi ( \mathbb{P} G ) , \widetilde{\zeta}_i (v) \r_{L^2_v} + \l \Gamma ( G , G ) , \widetilde{\zeta}_i (v) \r_{L^2_v} \right] \,.
	  \end{aligned}
	\end{equation}
	We now take $L^2_x$-inner product by multiplying $\partial^m_x \theta$ in the above equality and integrating by parts over $x \in \T^3$. Consequently, we have
	\begin{equation}\label{theta-Dissip-0}
	  \begin{aligned}
	     & \| \nabla_x \partial^m_x \theta \|^2_{L^2_x} =  \underset{\textbf{D}_1}{\underbrace{ \sum_{i=1}^3 \l \partial^m_x \l \Theta ( \mathbb{P}^\perp G ) , \widetilde{\zeta}_i (v) \r_{L^2_v} , \partial_i \partial^m_x \theta \r_{L^2_x} }} \\
	     & + \underset{\textbf{D}_2}{\underbrace{ \sum_{i=1}^3 \l \partial^m_x \l \Psi ( \mathbb{P} G ) , \widetilde{\zeta}_i (v) \r_{L^2_v} , \partial_i \partial^m_x \theta \r_{L^2_x} }} + \underset{\textbf{D}_3}{\underbrace{ \sum_{i=1}^3 \l \partial^m_x \l \Gamma ( G , G ) , \widetilde{\zeta}_i (v) \r_{L^2_v} , \partial_i \partial^m_x \theta \r_{L^2_x} }} \,.
	  \end{aligned}
	\end{equation}
	We next deal with the terms $\textbf{D}_1$, $\textbf{D}_2$ and $\textbf{D}_3$ term by term. The definition \eqref{Theta-Pperp-G} of $\Theta (\mathbb{P}^\perp G)$ yields that the quantity $\textbf{D}_1$ can be decomposed as 
	\begin{equation}
	  \begin{aligned}
	     \textbf{D}_1 = & - \sum_{i=1}^3 \Big\langle \eps \partial^m_x \partial_t \mathbb{P}^\perp G + v \cdot \nabla_x \partial^m_x \mathbb{P}^\perp G + \tfrac{1}{\eps} \mathscr{L} \partial^m_x \mathbb{P}^\perp G \\
	     & \qquad + \partial^m_x [ \mathsf{q} ( \eps E + v \times B ) \cdot \nabla_v \mathbb{P}^\perp G ]  - \tfrac{1}{2} \eps \partial^m_x [ \mathsf{q} ( E \cdot v ) \mathbb{P}^\perp G ] , \partial_i \partial^m_x \theta \widetilde{\zeta}_i (v) \Big\rangle_{L^2_{x,v}} \\
	     = & \eps \frac{\d}{\d t} \sum_{i=1}^3 \l \partial_i \partial^m_x \mathbb{P}^\perp G , \partial^m_x \theta \widetilde{\zeta}_i (v) \r_{L^2_{x,v}} \underset{\textbf{D}_{11}}{\underbrace{ - \sum_{i=1}^3 \l \partial_i \partial^m_x \mathbb{P}^\perp G , \eps \partial_t \partial^m_x \theta \widetilde{\zeta}_i (v) \r_{L^2_{x,v}} }} \\
	     & \underset{\textbf{D}_{12}}{\underbrace{ - \sum_{i=1}^3 \l \tfrac{1}{\eps} \partial^m_x \mathbb{P}^\perp G , \partial_i \partial^m_x \theta \mathscr{L} \widetilde{\zeta}_i (v) \r_{L^2_{x,v}} }} + \underset{\textbf{D}_{13}}{ \underbrace{ \sum_{i=1}^3 \l \eps \partial^m_x ( E \mathsf{q} \mathbb{P}^\perp G ) , \partial_i \partial^m_x \theta \nabla_v ( \tfrac{\widetilde{\zeta}_i (v)}{\sqrt{M}} ) \sqrt{M} \r_{L^2_{x,v}} } } \\
	     & \underset{\textbf{D}_{14}}{\underbrace{ - \sum_{i=1}^3 \l v \cdot \nabla_x \partial^m_x \mathbb{P}^\perp G , \partial_i \partial^m_x \theta \widetilde{\zeta}_i (v) \r_{L^2_{x,v}} }} + \underset{\textbf{D}_{15}}{ \underbrace{ \sum_{i=1}^3 \l \partial^m_x [ (v \times B) \mathsf{q} \mathbb{P}^\perp G ] , \partial_i \partial^m_x \theta \nabla_v \widetilde{\zeta}_i (v) \r_{L^2_{x,v}} }} \,,
	  \end{aligned}
	\end{equation}
	where we make use of integration by parts over $v \in \R^3$ in the terms $\textbf{D}_{13}$ and $\textbf{D}_{15}$. Noticing that the last $\theta$-equation in \eqref{Deco-6-moments}
	\begin{equation*}
	  \eps \partial_t \theta + \tfrac{2}{3} \div_x u = \tfrac{1}{3} \l \Theta (\mathbb{P}^\perp G) + \Psi (\mathbb{P} G) , \phi_6 (v) \r_{L^2_v} \,,
	\end{equation*}
	the quantity $\textbf{D}_{11}$ can be rewritten as
	\begin{equation}
	  \begin{aligned}
	    \textbf{D}_{11} = & \underset{\textbf{D}_{111}}{ \underbrace{ \tfrac{2}{3} \sum_{i=1}^3 \l \partial_i \partial^m_x \mathbb{P}^\perp \cdot \widetilde{\zeta}_i (v) , \partial^m_x \div_x u \r_{L^2_{x,v}}  } } \\
	    & \underset{\textbf{D}_{112}}{ \underbrace{ - \tfrac{1}{3} \sum_{i=1}^3 \l \partial_i \partial^m_x \mathbb{P}^\perp G \cdot \widetilde{\zeta}_i (v) , \partial^m_x \l \Theta ( \mathbb{P}^\perp G ) , \phi_6 (v) \r_{L^2_v} \r_{L^2_{x,v}}  } } \\
	    & \underset{\textbf{D}_{113}}{ \underbrace{ - \tfrac{1}{3} \sum_{i=1}^3 \l \partial_i \partial^m_x \mathbb{P}^\perp G \cdot \widetilde{\zeta}_i (v) , \partial^m_x \l \Psi ( \mathbb{P} G ) , \phi_6 (v) \r_{L^2_v} \r_{L^2_{x,v}}  } } \,.
	  \end{aligned}
	\end{equation}
	One observes that $\widetilde{\zeta}_i (v)$ and $\phi_6 (v)$ are both in $L^2_v$. By utilizing the H\"older inequality and the part (1) of Lemma \ref{Lmm-CF-nu}, the quantity $ \textbf{D}_{111} $ is bounded by
	\begin{equation}
	  \begin{aligned}
	    \textbf{D}_{111} \leq C \| \mathbb{P}^\perp G \|_{H^s_x L^2_v (\nu)} \| \partial^m_x \div_x u \|_{L^2_x} \,.
	  \end{aligned}
	\end{equation}
	Since $\phi_6 (v) \in \textrm{Ker} (\mathscr{L})$, we derive from \eqref{Theta-Pperp-G}, the H\"older inequality, the Sobolev embedding theory and the part (1) of Lemma \ref{Lmm-CF-nu} that
	  \begin{align}
	    \no \textbf{D}_{112} = & \tfrac{1}{3} \sum_{i=1}^3 \Big\langle \partial_i \partial^m_x \mathbb{P}^\perp G \cdot \widetilde{\zeta}_i (v) , \partial^m_x \big\langle v \cdot \nabla_x \mathbb{P}^\perp G \\
	    \no & \qquad + \mathsf{q} ( \eps E + v \times B ) \cdot \nabla_v \mathbb{P}^\perp G - \tfrac{1}{2} \eps \mathsf{q} ( E \cdot v ) \mathbb{P}^\perp G , \phi_6 (v) \big\rangle_{L^2_v} \Big\rangle_{L^2_{x,v}} \\
	    \no = & \tfrac{1}{3} \sum_{i=1}^3 \Big\langle \partial_i \partial^m_x \mathbb{P}^\perp G \cdot \widetilde{\zeta}_i (v) , \partial^m_x \l v \cdot \nabla_x \mathbb{P}^\perp G , \phi_6 (v) \r_{L^2_v} \\
	    \no & \qquad \qquad \qquad  - \partial^m_x \l \eps E \cdot v , \mathbb{P}^\perp G \cdot \mathsf{q}_1 \sqrt{M(v)} \r_{L^2_v} \Big\rangle_{L^2_{x,v}} \\
	    \leq & C ( 1 + \| E \|_{H^s_x} ) \| \mathbb{P}^\perp G \|^2_{H^s_x L^2_v (\nu)}
	  \end{align}
	for $0 < \eps \leq 1$, where the second equality is implied by making use of the relation 
	$$ \l \mathsf{q} ( \eps E + v \times B ) \cdot \nabla_v \mathbb{P}^\perp G - \tfrac{1}{2} \eps \mathsf{q} ( E \cdot v ) \mathbb{P}^\perp G , \phi_6 (v) \r_{L^2_v} = - \l \eps E \cdot v , \mathbb{P}^\perp G \cdot \mathsf{q}_1 \sqrt{M(v)} \r_{L^2_v} \,. $$
	Indeed, recalling that $\phi_6 (v) = ( \frac{|v|^2}{2} - \frac{3}{2} ) \sqrt{M(v)} \mathsf{q}_2$ for $\mathsf{q}_2 = [1,1]$, by utilizing the integration by parts over $v \in \R^3$ and the fact $(v \times B) \cdot v = 0$ we have 
	\begin{equation*}
	  \begin{aligned}
	    & \l \mathsf{q} ( \eps E + v \times B ) \cdot \nabla_v \mathbb{P}^\perp G - \tfrac{1}{2} \eps \mathsf{q} ( E \cdot v ) \mathbb{P}^\perp G , \phi_6 (v) \r_{L^2_v} \\
	    = & \l \mathsf{q} ( \eps E + v \times B ) \cdot \nabla_v \mathbb{P}^\perp G - \tfrac{1}{2} \eps \mathsf{q} ( E \cdot v ) \mathbb{P}^\perp G , ( \frac{|v|^2}{2} - \frac{3}{2} ) \sqrt{M(v)} \mathsf{q}_2 \r_{L^2_v} \\
	    = & \l ( \eps E + v \times B ) \mathbb{P}^\perp G , - v \sqrt{M(v)} \mathsf{q}_1 + ( \tfrac{|v|^2}{2} - \tfrac{3}{2} ) \sqrt{M(v)} \mathsf{q}_1 \tfrac{v}{2} \r_{L^2_v} \\
	    & - \l \tfrac{\eps}{2} ( E \cdot v ) \mathbb{P}^\perp G , ( \frac{|v|^2}{2} - \frac{3}{2} ) \sqrt{M(v)} \mathsf{q}_1 \r_{L^2_v} \\
	    = & - \l \eps E \cdot v , \mathbb{P}^\perp G \cdot \mathsf{q}_1 \sqrt{M(v)} \r_{L^2_v} \,.
	  \end{aligned}
	\end{equation*}
	By employing the similar arguments in estimation of the term $\textbf{D}_{112}$, we deduce from \eqref{Psi-P-G}, the H\"older inequality, the Sobolev embedding theory and the part (1) of Lemma \ref{Lmm-CF-nu} that
	\begin{equation}
	   \textbf{D}_{113} \leq C \| E \|_{H^s_x} \| \mathbb{P} G \|_{H^s_x L^2_v} \| \mathbb{P}^\perp G \|_{H^s_x L^2_v (\nu)} \,.
	\end{equation}
	In summary, we have
	\begin{equation}
	  \begin{aligned}
	    \textbf{D}_{11} = & \textbf{D}_{111} + \textbf{D}_{112} + \textbf{D}_{113}  \leq C \| \mathbb{P}^\perp G \|_{H^s_x L^2_v (\nu)} \| \partial^m_x \div_x u \|_{L^2_x} \\
	    + & C ( 1 + \| E \|_{H^s_x} ) \| \mathbb{P}^\perp G \|^2_{H^s_x L^2_v (\nu)} + \| E \|_{H^s_x} \| \mathbb{P} G \|_{H^s_x L^2_v} \| \mathbb{P}^\perp G \|_{H^s_x L^2_v (\nu)} \\
	    \leq & \delta \| \partial^m_x \div_x u \|^2_{L^2_x} + C \| \mathbb{P}^\perp G \|^2_{H^s_x L^2_v (\nu)} + C \| E \|^2_{H^s_x} \left( \| \mathbb{P} G \|^2_{H^s_x L^2_v} + \| \mathbb{P}^\perp G \|^2_{H^s_x L^2_v (\nu)} \right)
	  \end{aligned}
	\end{equation}
	for any small $\delta > 0$ to be determined and for $0 < \eps \leq 1$. For the term $\textbf{D}_{12}$, we derive from the H\"older inequality and the part (1) of Lemma \ref{Lmm-CF-nu} that
	\begin{equation}
	  \begin{aligned}
	    \textbf{D}_{12} \leq & \frac{1}{\eps} \sum_{i=1}^3 \| \partial^m_x \mathbb{P}^\perp G \|_{L^2_{x,v}} \| \partial_i \partial^m_x \theta \|_{L^2_x} \| \mathscr{L} \widetilde{\zeta}_i (v) \|_{L^2_v} \\
	    \leq & \frac{C}{\eps} \| \mathbb{P}^\perp G \|_{H^s_x L^2_v (\nu)} \| \nabla_x \partial^m_x \theta \|_{L^2_x} \,.
	  \end{aligned}
	\end{equation}
	Moreover, the term $\textbf{D}_{13}$ is bounded by
	\begin{equation}
	  \begin{aligned}
	    \textbf{D}_{13} \leq & \eps \sum_{i=1}^3 \| \partial^m_x ( E \mathbb{P}^\perp G ) \|_{L^2_{x,v}} \| \partial_i \partial^m_x \theta \|_{L^2_x} \| \nabla_v ( \tfrac{\mathsf{q} \widetilde{\zeta}_i (v)}{\sqrt{M}} ) \sqrt{M} \|_{L^2_v} \\
	    \leq & C \eps \| \nabla_x \partial^m_x \theta \|_{L^2_x} \| \partial^m_x ( E \mathbb{P}^\perp G ) \|_{L^2_{x,v}} \\
	    \leq & C  \eps \| \nabla_x \partial^m_x \theta \|_{L^2_x} \| E \|_{H^s_x} \| \mathbb{P}^\perp G \|_{H^s_x L^2_v (\nu)} \,,
	  \end{aligned}
	\end{equation}
	where we make use of the H\"older inequality, the Sobolev embedding theory and the part (1) of Lemma \ref{Lmm-CF-nu}. For the term $\textbf{D}_{14}$, we have
	\begin{equation}
	  \begin{aligned}
	    \textbf{D}_{14} \leq & \sum_{i=1}^3 \| \nabla_x \partial^m_x \mathbb{P}^\perp G \|_{L^2_{x,v}} \| \partial_i \partial^m_x \theta \|_{L^2_x} \| v \widetilde{\zeta}_i (v) \|_{L^2_v} \\
	    \leq & C \| \nabla_x \partial^m_x \theta \|_{L^2_x} \| \mathbb{P}^\perp G \|_{H^s_x L^2_v (\nu)} \,.
	  \end{aligned}
	\end{equation}
	By the similar estimation of the term $\textbf{D}_{13}$, the quantity $\textbf{D}_{15}$ is bounded by
	\begin{equation}
	\begin{aligned}
	\textbf{D}_{15} \leq & \sum_{i=1}^3 \| \partial^m_x ( B \mathbb{P}^\perp G ) \|_{L^2_{x,v}} \| \partial_i \partial^m_x \theta \|_{L^2_x} \| v \nabla_v ( \mathsf{q} \widetilde{\zeta}_i (v) ) \|_{L^2_v} \\
	\leq & C \| \nabla_x \partial^m_x \theta \|_{L^2_x} \| \partial^m_x ( B \mathbb{P}^\perp G ) \|_{L^2_{x,v}} \\
	\leq & C  \| \nabla_x \partial^m_x \theta \|_{L^2_x} \| B \|_{H^s_x} \| \mathbb{P}^\perp G \|_{H^s_x L^2_v (\nu)} \,.
	\end{aligned}
	\end{equation}
	Consequently, collecting the above bounds yields
	\begin{equation}
	  \begin{aligned}
	    \textbf{D}_1 = & \eps \frac{\d}{\d t} \sum_{i=1}^3 \l \partial_i \partial^m_x \mathbb{P}^\perp G , \partial^m_x \theta \widetilde{\zeta}_i (v) \r_{L^2_{x,v}} + \textbf{D}_{11} + \textbf{D}_{12} + \textbf{D}_{13} + \textbf{D}_{14} + \textbf{D}_{15} \\
	    \leq & \eps \frac{\d}{\d t} \sum_{i=1}^3 \l \partial_i \partial^m_x \mathbb{P}^\perp G , \partial^m_x \theta \widetilde{\zeta}_i (v) \r_{L^2_{x,v}}  + C \| E \|^2_{H^s_x} \left( \| \mathbb{P} G \|^2_{H^s_x L^2_v} + \| \mathbb{P}^\perp G \|^2_{H^s_x L^2_v (\nu)} \right) \\
	    + & C \| \nabla_x \partial^m_x \theta \|_{L^2_x} \| \mathbb{P}^\perp G \|_{H^s_x L^2_v (\nu)}  ( \tfrac{1}{\eps} + \| E \|_{H^s_x} + \| B \|_{H^s_x}  ) + C \| \mathbb{P}^\perp G \|^2_{H^s_x L^2_v (\nu)} + \delta \| \partial^m_x \div_x u \|^2_{L^2_x}
	  \end{aligned}
	\end{equation}
	for any small $\delta > 0$ to be determined and for $0 < \eps \leq 1$.
	
	For the term $\textbf{D}_2$, it is derived from \eqref{Psi-P-G}, the definition \eqref{VMB-Proj} of $\mathbb{P} G$, the H\"older inequality and the Sobolev embedding theory that
	\begin{equation}
	  \begin{aligned}
	    \textbf{D}_2 = & \sum_{i=1}^3 \l \partial^m_x \big[ - \mathsf{q} ( \eps E + v \times B ) \cdot \nabla_v \mathbb{P} G + \tfrac{1}{2} \mathsf{q} \eps ( E \cdot v ) \mathbb{P} G \big] , \partial_i \partial^m_x \theta \widetilde{\zeta}_i (v) \r_{L^2_{x,v}} \\
	    \leq & C \| \partial^m_x [ ( \eps E + B ) \mathbb{P} G ] \|_{L^2_{x,v}} \| \nabla_x \partial^m_x \theta \|_{L^2_x} \\
	    \leq & C ( \eps \| E \|_{H^s_x} + \| B \|_{H^s_x} ) \| \mathbb{P} G \|_{H^s_x L^2_v} \| \nabla_x \partial^m_x \theta \|_{L^2_x} \,.
	  \end{aligned}
	\end{equation} 
	For the quantity $\textbf{D}_3$, we derive from Lemma \ref{Lmm-Gamma-Torus}, the composition $G = \mathbb{P} G + \mathbb{P}^\perp G$ and the definition \eqref{VMB-Proj} of $\mathbb{P} G$ that
	\begin{equation}
	  \begin{aligned}
	    \textbf{D}_3 = & \sum_{i=1}^3 \l \partial^m_x \Gamma (G , G) , \partial_i \partial^m_x \theta \widetilde{\zeta}_i (v) \r_{L^2_{x,v}} \\
	    \leq & C \sum_{i=1}^3 \| G \|_{H^s_x L^2_v} \| G \|_{H^s_x L^2_v (\nu)} \| \partial_i \partial^m_x \theta \|_{L^2_x} \| \widetilde{\zeta}_i (v) \|_{L^2_v (\nu)} \\
	    \leq & C\| \nabla_x \partial^m_x \theta \|_{L^2_x} \| G \|_{H^s_x L^2_v} ( \| \mathbb{P} G \|_{H^s_x L^2_v} + \| \mathbb{P}^\perp G \|_{H^s_x L^2_v (\nu)} ) \,.
	  \end{aligned}
	\end{equation}
	We finally plug the bounds of $\textbf{D}_1$, $\textbf{D}_2$ and $\textbf{D}_3$ into the relation \eqref{theta-Dissip-0} and obtain the bound
	\begin{equation*}
	  \begin{aligned}
	    \| \nabla_x \partial^m_x \theta \|^2_{L^2_x} \leq & \eps \frac{\d}{\d t} \sum_{i=1}^3 \l \partial_i \partial^m_x \mathbb{P}^\perp G , \partial^m_x \theta \widetilde{\zeta}_i (v) \r_{L^2_{x,v}} \\
	    + & \delta \| \partial^m_x \div_x u \|^2_{L^2_x}  + C \| E \|^2_{H^s_x} \left( \| \mathbb{P} G \|^2_{H^s_x L^2_v} + \| \mathbb{P}^\perp G \|^2_{H^s_x L^2_v (\nu)} \right) \\
	    + & C \| \nabla_x \partial^m_x \theta \|_{L^2_x} \| \mathbb{P}^\perp G \|_{H^s_x L^2_v (\nu)}  ( \tfrac{1}{\eps} + \| E \|_{H^s_x} + \| B \|_{H^s_x}  ) + C \| \mathbb{P}^\perp G \|^2_{H^s_x L^2_v (\nu)} \\
	    + & C \| \nabla_x \partial^m_x \theta \|_{L^2_x} \big( \| \mathbb{P} G \|_{H^s_x L^2_v} + \| \mathbb{P}^\perp G \|_{H^s_x L^2_v (\nu)} \big) \big( \| G \|_{H^s_x L^2_v} + \| E \|_{H^s_x} + \| B \|_{H^s_x} \big) \,,
	  \end{aligned}
	\end{equation*}
	which immediately yields by Young's inequality that for all mulit-indexes $m \in \mathbb{N}^3$ with $|m| \leq s - 1$
	\begin{equation}\label{theta-Dissip}
	  \begin{aligned}
	    \| \nabla_x \partial^m_x \theta \|^2_{L^2_x} \leq & \eps \frac{\d}{\d t} \sum_{i=1}^3 \l \partial_i \partial^m_x \mathbb{P}^\perp G , \partial^m_x \theta \widetilde{\zeta}_i (v) \r_{L^2_{x,v}} + 2 \delta \| \partial^m_x \div_x u \|^2_{L^2_x} + \tfrac{C}{\eps^2} \| \mathbb{P}^\perp G \|^2_{H^s_x L^2_v (\nu)} \\
	    + & C ( \| G \|^2_{H^s_x L^2_v} + \| E \|^2_{H^s_x} + \| B \|^2_{H^s_x} ) ( \| \mathbb{P} G \|^2_{H^s_x L^2_v} + \| \mathbb{P}^\perp G \|^2_{H^s_x L^2_v (\nu)} ) 
	  \end{aligned}
	\end{equation}
	for any small $\delta > 0$ to be determined and for all $0 < \eps \leq 1$.\\
	
	{\em Step 3. Bounds on $\| \nabla_x \rho^\pm \|^2_{H^{s-1}_x}$ for the integer $s \geq 3$.} From the second $u_i$-equation of \eqref{Deco-17-moments}, we deduce
	\begin{equation}
	  \begin{aligned}
	    & - \Delta_x \partial^m_x \rho^\pm \pm \partial^m_x \div_x E = \sum_{i=1}^3 \partial_i \partial^m_x ( - \partial_i \rho^\pm \pm E_i ) \\
	    = & - \sum_{i=1}^3 \partial_i \partial^m_x  ( - \eps \partial_t u_i + \tfrac{3}{2} \partial_i \theta + \Theta_i^\pm + \Psi_i^\pm + \Gamma_i^\pm ) \\
	    = & - \sum_{i=1}^3 \partial_i \partial^m_x  \Big[ - \eps \partial_t u_i + \tfrac{3}{2} \partial_i \theta + \l \Theta (\mathbb{P}^\perp G) , \zeta^\pm_i (v) \r_{L^2_v} \\
	    & \qquad \qquad  + \l \Psi (\mathbb{P} G) , \zeta^\pm_i (v) \r_{L^2_v} + \l \Gamma (G,G) , \zeta^\pm_i (v) \r_{L^2_v} \Big]
	  \end{aligned}
	\end{equation}
	for some certain linear combinations $\zeta^\pm_i (v)$ of the basis $\mathfrak{B}$. Multiplying by $\partial^m_x \rho^\pm$, taking $L^2_x$-inner and integrating by parts over $x \in \R^3$ give us
	\begin{equation}\label{rho-Dissip-0}
	  \begin{aligned}
	    & \| \nabla_x \partial^m_x \rho^\pm \|^2_{L^2_x} \pm \l \partial^m_x \div_x E , \partial^m_x \rho^\pm \r_{L^2_x} \\
	    = & - \sum_{i=1}^3 \l \partial_i \partial^m_x ( - \eps \partial_t u_i + \tfrac{3}{2} \partial_i \theta + \Theta_i^\pm + \Psi_i^\pm + \Gamma_i^\pm ) , \partial^m_x \rho^\pm \r_{L^2_x} \,.
	  \end{aligned}
	\end{equation}
	Noticing that the equation $\div_x E = \int_{\R^3} G \cdot \mathsf{q}_1 \sqrt{M} \d v$ in the perturbed VMB system \eqref{VMB-G-drop-eps} reduces to
	\begin{equation}
	  \div_x E = \rho^+ - \rho^- \,,
	\end{equation}
	we then sum up for the indexes $\pm$ in \eqref{rho-Dissip-0} and obtain
	\begin{equation}
	  \begin{aligned}
	    & \| \nabla_x \partial^m_x \rho^+ \|^2_{L^2_x} + \| \nabla_x \partial^m_x \rho^- \|^2_{L^2_x} + \| \partial^m_x \div_x E \|^2_{L^2_x} \\
	    = & - \sum_{\gamma = \pm} \sum_{i=1}^3 \l \partial_i \partial^m_x ( - \eps \partial_t u_i + \tfrac{3}{2} \partial_i \theta + \Theta_i^\gamma + \Psi_i^\gamma + \Gamma_i^\gamma ) , \partial^m_x \rho^\gamma \r_{L^2_x} \\
	    = & \underset{\textbf{E}_1}{ \underbrace{ \sum_{i=1}^3 \l \eps \partial_t \partial_i \partial^m_x u_i , \partial^m_x \rho^+ + \partial^m_x \rho^- \r_{L^2_x} }} \ \underset{\textbf{E}_2}{ \underbrace{ - \sum_{i=1}^3 \l \partial_i \partial^m_x \Theta (\mathbb{P}^\perp G) , \partial^m_x \rho^+ \zeta^+_i (v) + \partial^m_x \rho^- \zeta^-_i (v) \r_{L^2_{x,v}} }} \\
	     &  \underset{\textbf{E}_3}{ \underbrace{ - \tfrac{3}{2} \sum_{i=1}^3 \l \partial_i \partial^m_x \partial_i \theta , \partial^m_x \rho^+ + \partial^m_x \rho^- \r_{L^2_x} }} \ \underset{\textbf{E}_4}{ \underbrace{ - \sum_{i=1}^3 \l \partial_i \partial^m_x \Psi (\mathbb{P} G) , \partial^m_x \rho^+ \zeta^+_i (v) + \partial^m_x \rho^- \zeta^-_i (v) \r_{L^2_{x,v}} }}  \\
	     & \underset{\textbf{E}_5}{ \underbrace{ - \sum_{i=1}^3 \l \partial_i \partial^m_x \Gamma ( G , G ) , \partial^m_x \rho^+ \zeta^+_i (v) + \partial^m_x \rho^- \zeta^-_i (v) \r_{L^2_{x,v}} }} \,.
	  \end{aligned}
	\end{equation}
	Then we will estimate the quantities $\textbf{E}_i$ for $i = 1,2 , \cdots, 5$. For the term $\textbf{E}_1$, we deduce from the first two $\rho^\pm$-equations of \eqref{Deco-6-moments} that
	\begin{equation}
	  \begin{aligned}
	    \textbf{E}_1 = & \eps \frac{\d}{\d t} \sum_{i = 1}^3 \l \partial_i \partial^m_x u_i , \partial^m_x \rho^+ + \partial^m_x \rho^- \r_{L^2_x} - \sum_{i=1}^3 \l \partial_i \partial^m_x u_i , \eps \partial_t \partial^m_x \rho^+ + \eps \partial_t \partial^m_x \rho^- \r_{L^2_x} \\
	    = & \eps \frac{\d}{\d t} \sum_{i = 1}^3 \l \partial_i \partial^m_x u_i , \partial^m_x \rho^+ + \partial^m_x \rho^- \r_{L^2_x} + 2 \| \partial^m_x \div_x u \|^2_{L^2_x} \\
	    & \underset{ \textbf{E}_{11} }{ \underbrace{ - \sum_{i=1}^3 \l \partial_i \partial^m_x u_i , \partial^m_x \Theta ( \mathbb{P}^\perp G ) ( \phi_1 (v) + \phi_2 (v) ) \r_{L^2_{x,v}} } } \\
	    & \underset{ \textbf{E}_{12} }{ \underbrace{ - \sum_{i=1}^3 \l \partial_i \partial^m_x u_i , \partial^m_x \Psi ( \mathbb{P} G ) ( \phi_1 (v) + \phi_2 (v) ) \r_{L^2_{x,v}} } } \,.
	  \end{aligned}
	\end{equation}
	Since $\nabla_v ( \phi_1 (v) + \phi_2 (v) ) = - \frac{1}{2} v ( \phi_1 (v) + \phi_2 (v) )$, we have
	\begin{equation}\label{E-B-cancellation}
	  \begin{aligned}
	    & \l \mathsf{q} ( \eps E + v \times B ) \cdot \nabla_v \mathbb{P}^\perp G , \phi_1 (v) + \phi_2 (v) \r_{L^2_v} \\
	    = & - \l \mathsf{q} ( \eps E + v \times B ) \mathbb{P}^\perp G , \nabla_v ( \phi_1 (v) + \phi_2 (v) ) \r_{L^2_v} \\
	    = & \tfrac{1}{2} \l \mathsf{q} ( \eps E + v \times B ) \mathbb{P}^\perp G , v ( \phi_1 (v) + \phi_2 (v) ) \r_{L^2_v} \\
	    = & \l \tfrac{1}{2} \eps \mathsf{q} ( E \cdot v ) \mathbb{P}^\perp G , \phi_1 (v) + \phi_2 (v) \r_{L^2_v} \,,
	  \end{aligned}
	\end{equation}
	where the fact $(v \times B) \cdot v = 0$ is utilized. Since $\phi_1 (v)$, $\phi_2 (v)$ are both in $\textrm{Ker} (\mathscr{L})$, the previous identity and \eqref{Theta-Pperp-G} reduce to
	\begin{equation}
	  \begin{aligned}
	    \textbf{E}_{11} = & \sum_{i=1}^3 \Big\langle \partial_i \partial^m_x u_i \partial^m_x [ v \cdot \nabla_x \mathbb{P}^\perp G + \mathsf{q} ( \eps E + v \times B ) \cdot \nabla_v \mathbb{P}^\perp G \\
	    & \qquad \qquad - \tfrac{1}{2} \eps \mathsf{q} (E \cdot v) \mathbb{P}^\perp G ] ( \phi_1 (v) + \phi_2 (v) ) \Big\rangle_{L^2_{x,v}} \\
	    = & \sum_{i=1}^3 \l \partial_i \partial^m_x u_i , ( v \cdot \nabla_x \partial^m_x \mathbb{P}^\perp G ) ( \phi_1 (v) + \phi_2 (v) ) \r_{L^2_{x,v}} \\
	    \leq & \| \partial^m_x \div_x u \|_{L^2_x} \| \nabla_x \partial^m_x \mathbb{P}^\perp G \|_{L^2_{x,v}} \| v ( \phi_1 (v) + \phi_2 (v) ) \|_{L^2_v} \\
	    \leq & C \| \mathbb{P}^\perp G \|_{H^s_x L^2_v} \| \partial^m_x \div_x u \|_{L^2_x} \,.
	  \end{aligned}
	\end{equation}
	The cancellation \eqref{E-B-cancellation} still holds when the function $\mathbb{P}^\perp G$ is replaced by $\mathbb{P} G$. Namely,
	\begin{equation}
	  \begin{aligned}
	    \l \mathsf{q} ( \eps E + v \times B ) \cdot \nabla_v \mathbb{P} G , \phi_1 (v) + \phi_2 (v) \r_{L^2_v} = \l \tfrac{1}{2} \eps \mathsf{q} ( E \cdot v ) \mathbb{P} G , \phi_1 (v) + \phi_2 (v) \r_{L^2_v} \,.
	  \end{aligned}
	\end{equation}
	Then by the relation \eqref{Psi-P-G} of $\Psi (\mathbb{P} G)$ we compute the term $\textbf{E}_{12}$ that
	  \begin{align}
	    \no \textbf{E}_{12} = & \sum_{i=1}^3 \l \partial_i \partial^m_x u_i , \partial^m_x [ \mathsf{q} ( \eps E + v \times B ) \cdot \nabla_v \mathbb{P} G -  \tfrac{1}{2} \eps \mathsf{q} ( E \cdot v ) \mathbb{P} G ] ( \phi_1 (v) + \phi_2 (v) ) \r_{L^2_{x,v}} \\
	    \no = &  \sum_{i=1}^3 \l \partial_i \partial^m_x u_i , \partial^m_x \l \mathsf{q} ( \eps E + v \times B ) \cdot \nabla_v \mathbb{P} G , \phi_1 (v) + \phi_2 (v) \r_{L^2_v} \r_{L^2_x} \\
	    & -  \sum_{i=1}^3 \l \partial_i \partial^m_x u_i , \partial^m_x \l \tfrac{1}{2} \eps \mathsf{q} ( E \cdot v ) \mathbb{P} G , \phi_1 (v) + \phi_2 (v) \r_{L^2_v} \r_{L^2_x} = 0 \,.
	  \end{align}
	Then we, consequently, obtain 
	\begin{equation}
	  \begin{aligned}
	    \textbf{E}_1 \leq & \eps \frac{\d}{\d t} \sum_{i = 1}^3 \l \partial_i \partial^m_x u_i , \partial^m_x \rho^+ + \partial^m_x \rho^- \r_{L^2_x} + 2 \| \partial^m_x \div_x u \|^2_{L^2_x} + C \| \mathbb{P}^\perp G \|_{H^s_x L^2_v} \| \partial^m_x \div_x u \|_{L^2_x} \\
	    \leq & \eps \frac{\d}{\d t} \sum_{i = 1}^3 \l \partial_i \partial^m_x u_i , \partial^m_x \rho^+ + \partial^m_x \rho^- \r_{L^2_x} + 3 \| \partial^m_x \div_x u \|^2_{L^2_x} + C \| \mathbb{P}^\perp G \|^2_{H^s_x L^2_v (\nu)} \,,
	  \end{aligned}
	\end{equation}
	where the last inequality is derived from the Young's inequality and the part (1) of Lemma \ref{Lmm-CF-nu}.
	
	Next, we estimate the term $\textbf{E}_2$. From the relation \eqref{Theta-Pperp-G}, we deduce 
	\begin{equation}
	  \begin{aligned}
	    \textbf{E}_2 & =  \underset{\textbf{E}_{21}}{ \underbrace{ \sum_{i=1}^3 \l \eps \partial_i \partial^m_x \partial_t \mathbb{P}^\perp G , \partial^m_x \rho^+ \zeta^+_i (v) + \partial^m_x \rho^- \zeta^-_i (v) \r_{L^2_{x,v}} }} \\
	    &  \underset{\textbf{E}_{22}}{ \underbrace{ - \sum_{i=1}^3 \l v \cdot \nabla_x \partial^m_x  \mathbb{P}^\perp G ,  \partial_i \partial^m_x \rho^+ \zeta^+_i (v) +  \partial_i \partial^m_x \rho^- \zeta^-_i (v) \r_{L^2_{x,v}} }} \\
	    &  \underset{\textbf{E}_{23}}{ \underbrace{ -  \sum_{i=1}^3 \l \tfrac{1}{\eps} \partial^m_x \mathbb{P}^\perp G , \partial_i \partial^m_x \rho^+ \mathscr{L} \zeta^+_i (v) + \partial_i \partial^m_x \rho^- \mathscr{L} \zeta^-_i (v) \r_{L^2_{x,v}} }} \\
	    &  \underset{\textbf{E}_{24}}{ \underbrace{ - \sum_{i=1}^3 \l \partial^m_x ( \eps E \mathbb{P}^\perp G ) , \partial_i \partial^m_x \rho^+ \nabla_v ( \tfrac{\mathsf{q} \zeta^+_i (v)}{\sqrt{M}} ) \sqrt{M} + \partial_i \partial^m_x \rho^- \nabla_v ( \tfrac{\mathsf{q} \zeta^-_i (v)}{\sqrt{M}} ) \sqrt{M} \r_{L^2_{x,v}} }} \\
	    &  \underset{\textbf{E}_{25}}{ \underbrace{ - \sum_{i=1}^3 \l  \partial^m_x ( B \mathbb{P}^\perp G ) , \partial_i \partial^m_x \rho^+ v \times \nabla_v ( \tfrac{\mathsf{q} \zeta^+_i (v)}{\sqrt{M}} ) \sqrt{M} + \partial_i \partial^m_x \rho^- v \times \nabla_v ( \tfrac{\mathsf{q} \zeta^-_i (v)}{\sqrt{M}} ) \sqrt{M} \r_{L^2_{x,v}} }} \,,
	  \end{aligned}
	\end{equation}
	where the last two terms is derived from the integration by parts over $v \in \R^3$. For the term $\textbf{E}_{21}$, we derive from the first two $\rho^\pm$-equations of \eqref{Deco-6-moments} that
	  \begin{align}
	    \no \textbf{E}_{21} = & \eps \frac{\d}{\d t} \sum_{i=1}^3 \l \partial_i \partial^m_x \mathbb{P}^\perp G , \partial^m_x \rho^+ \zeta_i^+ (v) + \partial^m_x \rho^- \zeta^-_i (v) \r_{L^2_{x,v}} \\
	    \no & - \sum_{i=1}^3 \l \partial_i \partial^m_x \mathbb{P}^\perp G , \eps \partial_t \partial^m_x \rho^+ \zeta^+_i (v) + \eps \partial_t \partial^m_x \rho^- \zeta^-_i (v) \r_{L^2_{x,v}} \\
	    \no = & \eps \frac{\d}{\d t} \sum_{i=1}^3 \l \partial_i \partial^m_x \mathbb{P}^\perp G , \partial^m_x \rho^+ \zeta_i^+ (v) + \partial^m_x \rho^- \zeta^-_i (v) \r_{L^2_{x,v}} \\
	    \no & + \underset{\textbf{E}_{211}}{ \underbrace{  \sum_{i=1}^3 \l \partial_i \partial^m_x \mathbb{P}^\perp G , \partial^m_x \div_x u ( \zeta_i^+ (v) + \zeta_i^- (v) ) \r_{L^2_{x,v}} }} \\
	    \no & \underset{\textbf{E}_{212}}{ \underbrace{ - \sum_{i=1}^3 \l \partial_i \partial^m_x \mathbb{P}^\perp G , \partial^m_x \l \Theta (\mathbb{P}^\perp G) + \Psi (\mathbb{P} G) , \phi_1 (v) \r_{L^2_v} \zeta^+_i (v) \r_{L^2_{x,v}} }} \\
	    & \underset{\textbf{E}_{213}}{ \underbrace{ - \sum_{i=1}^3 \l \partial_i \partial^m_x \mathbb{P}^\perp G , \partial^m_x \l \Theta (\mathbb{P}^\perp G) + \Psi (\mathbb{P} G) , \phi_2 (v) \r_{L^2_v} \zeta^-_i (v) \r_{L^2_{x,v}} }} \,.
	  \end{align}
	The H\"older inequality implies that the term $\textbf{E}_{211}$ is bounded by
	\begin{equation}
	  \begin{aligned}
	    \textbf{E}_{211} \leq & \sum_{i=1}^3 \| \partial_i \partial^m_x \mathbb{P}^\perp G \|_{L^2_{x,v}} \| \partial^m_x \div_x u \|_{L^2_x} \| \zeta^+_i (v) + \zeta^-_i (v) \|_{L^2_v} \\
	    \leq & C \| \nabla_x \partial^m_x \mathbb{P}^\perp G \|_{L^2_{x,v}} \| \partial^m_x \div_x u \|_{L^2_x} \\
	    \leq & C \| \mathbb{P}^\perp G \|_{H^s_x L^2_v} \| \partial^m_x \div_x u \|_{L^2_x} \,.
	  \end{aligned}
	\end{equation}
	Since $\nabla_v \phi_i (v) = - \tfrac{1}{2} v \phi_i (v)$ for $i = 1 , 2$, we have 
	\begin{equation}
	  \begin{aligned}
	    & \l - \mathsf{q} ( \eps E + v \times B ) \cdot \nabla_v H , \phi_i (v) \r_{L^2_v} = \l  \mathsf{q} ( \eps E + v \times B ) H , \nabla_v \phi_i (v) \r_{L^2_v} \\
	    = & - \l \tfrac{1}{2} \eps \mathsf{q} ( E \cdot v ) H , \phi_i (v) \r_{L^2_v} + \l \mathsf{q} [ (v \times B) \cdot \tfrac{1}{2} v ] H , \phi_i (v) \r_{L^2_v} \\
	    = & - \l \tfrac{1}{2} \eps \mathsf{q} ( E \cdot v ) H , \phi_i (v) \r_{L^2_v}
	  \end{aligned}
	\end{equation}
	for any $H = [H^+ , H^-]$, where the last equality is derived from the relation $(v \times B) \cdot v = 0$. If taking $H = \mathbb{P} G$ or $\mathbb{P}^\perp G$, we obtain
	\begin{equation}
	  \begin{aligned}
	     \l - \mathsf{q} ( \eps E + v \times B ) \cdot \nabla_v \mathbb{P} G + \tfrac{1}{2} \eps \mathsf{q} ( E \cdot v ) \mathbb{P} G , \phi_i (v) \r_{L^2_v} = 0 \,, \\
	     \l - \mathsf{q} ( \eps E + v \times B ) \cdot \nabla_v \mathbb{P}^\perp G + \tfrac{1}{2} \eps \mathsf{q} ( E \cdot v ) \mathbb{P}^\perp G , \phi_i (v) \r_{L^2_v} = 0 \,.
	  \end{aligned}
	\end{equation}
	Additionally, $\phi_1 (v)$ is in $\textrm{Ker} (\mathscr{L})$, we deduce from \eqref{Theta-Pperp-G} and \eqref{Psi-P-G} that
	\begin{equation}
	  \begin{aligned}
	    \l \Theta (\mathbb{P}^\perp G) + \Psi (\mathbb{P} G) , \phi_i (v) \r_{L^2_v} = - \l v \cdot \nabla_x \mathbb{P}^\perp G , \phi_i (v) \r_{L^2_v} 
	  \end{aligned}
	\end{equation}
	for $i = 1 ,2$. Thus we have
	\begin{equation}
	  \begin{aligned}
	    \textbf{E}_{212} + \textbf{E}_{213} = & \sum_{i=1}^3 \l \partial_i \partial^m_x \mathbb{P}^\perp G , \partial^m_x \l v \cdot \nabla_x \mathbb{P}^\perp G, \phi_1 (v) \r_{L^2_v} \zeta^+_i (v) \r_{L^2_{x,v}} \\
	    & + \sum_{i=1}^3 \l \partial_i \partial^m_x \mathbb{P}^\perp G , \partial^m_x \l v \cdot \nabla_x \mathbb{P}^\perp G, \phi_2 (v) \r_{L^2_v} \zeta^-_i (v) \r_{L^2_{x,v}} \\
	    = & \sum_{i=1}^3 \| \partial_i \partial^m_x \mathbb{P}^\perp G \|_{L^2_{x,v}} \| \nabla_x \partial^m_x \mathbb{P}^\perp G \|_{L^2_{x,v}} \\
	    & \qquad \qquad \times  \left( \| v \phi_1 (v) \|_{L^2_v} \| \zeta^+_i (v) \|_{L^2_v} + \| v \phi_2 (v) \|_{L^2_v} \| \zeta^-_i (v) \|_{L^2_v} \right) \\
	    \leq & C \| \mathbb{P}^\perp G \|^2_{H^s_x L^2_v} \,.
	  \end{aligned}
	\end{equation}
	Consequently, the previous bounds give us
	\begin{equation}
	  \begin{aligned}
	    \textbf{E}_{21} \leq & \eps \frac{\d}{\d t} \sum_{i=1}^3 \l \partial_i \partial^m_x \mathbb{P}^\perp G , \partial^m_x \rho^+ \zeta_i^+ (v) + \partial^m_x \rho^- \zeta^-_i (v) \r_{L^2_{x,v}} \\
	    + & C \| \mathbb{P}^\perp G \|_{H^s_x L^2_v} \| \partial^m_x \div_x u \|_{L^2_x} + C \| \mathbb{P}^\perp G \|^2_{H^s_x L^2_v} \,.
	  \end{aligned}
	\end{equation}
	For the quantity $\textbf{E}_{22}$, it is derived from the H\"older inequality that
	\begin{equation}
	  \begin{aligned}
	    \textbf{E}_{22} \leq & \sum_{i=1}^3 \| \nabla_x \partial^m_x \mathbb{P}^\perp G \|_{L^2_{x,v}} ( \| \partial_i \partial^m_x \rho^+ \|_{L^2_x} + \| \partial_i \partial^m_x \rho^- \|_{L^2_x} ) \\
	    & \qquad \times ( \| v \zeta^+_i (v) \|_{L^2_v} + \| v \zeta^-_i (v) \|_{L^2_v} ) \\
	    \leq & C \| \mathbb{P}^\perp G \|_{H^s_x L^2_v} ( \| \nabla_x \partial^m_x \rho^+ \|_{L^2_x} + \| \nabla_x \partial^m_x \rho^- \|_{L^2_x} ) \,.
	  \end{aligned}
	\end{equation}
	Similarly, for the term $\textbf{E}_{23}$ we have
	\begin{equation}
	  \begin{aligned}
	     \textbf{E}_{23} \leq & \frac{1}{\eps} \sum_{i=1}^3 \| \partial^m_x \mathbb{P}^\perp G \|_{L^2_{x,v}} ( \| \partial_i \partial^m_x \rho^+ \|_{L^2_x} + \| \partial_i \partial^m_x \rho^- \|_{L^2_x} ) \\
	     & \qquad \times ( \| \mathscr{L} \zeta^+_i (v) \|_{L^2_v} + \| \mathscr{L} \zeta^-_i (v) \|_{L^2_v} ) \\
	     \leq & \frac{C}{\eps} \| \mathbb{P}^\perp G \|_{H^s_x L^2_v} ( \| \nabla_x \partial^m_x \rho^+ \|_{L^2_x} + \| \nabla_x \partial^m_x \rho^- \|_{L^2_x} ) \,.
	  \end{aligned}
	\end{equation}
	The term $\textbf{E}_{24}$ can be estimated as
	\begin{equation}
	  \begin{aligned}
	    \textbf{E}_{24} \leq & \sum_{i=1}^3 \| \partial^m_x ( \eps E \mathbb{P}^\perp G ) \|_{L^2_{x,v}} ( \| \partial_i \partial^m_x \rho^+ \|_{L^2_x} + \| \partial_i \partial^m_x \rho^- \|_{L^2_c} ) \\
	    & \qquad \times \big( \| \nabla_v ( \tfrac{\mathsf{q} \zeta_i^+ (v)}{\sqrt{M}} ) \sqrt{M} \|_{L^2_v} + \| \nabla_v ( \tfrac{\mathsf{q} \zeta_i^- (v)}{\sqrt{M}} ) \sqrt{M} \|_{L^2_v} \big) \\
	    \leq & C \eps \| \partial^m_x (E \mathbb{P}^\perp G) \|_{L^2_{x,v}} \big( \| \nabla_x \partial^m_x \rho^+ \|_{L^2_x} + \| \nabla_x \partial^m_x \rho^- \|_{L^2_x} \big) \\
	    \leq & C \eps \| E \|_{H^s_x} \| \mathbb{P}^\perp G \|_{H^s_x L^2_v} \big( \| \nabla_x \partial^m_x \rho^+ \|_{L^2_x} + \| \nabla_x \partial^m_x \rho^- \|_{L^2_x} \big)
	  \end{aligned} 
	\end{equation}
	by making use of the H\"older inequality and the Sobolev embedding theory. Similarly, for the term $\textbf{E}_{25}$ we have
	\begin{equation}
	\begin{aligned}
	\textbf{E}_{25} \leq & \sum_{i=1}^3 \| \partial^m_x ( B \mathbb{P}^\perp G ) \|_{L^2_{x,v}} ( \| \partial_i \partial^m_x \rho^+ \|_{L^2_x} + \| \partial_i \partial^m_x \rho^- \|_{L^2_c} ) \\
	& \quad \times \big( \| v \times \nabla_v ( \tfrac{\mathsf{q} \zeta_i^+ (v)}{\sqrt{M}} ) \sqrt{M} \|_{L^2_v} + \| v \times \nabla_v ( \tfrac{\mathsf{q} \zeta_i^- (v)}{\sqrt{M}} ) \sqrt{M} \|_{L^2_v} \big) \\
	\leq & C \| \partial^m_x (B \mathbb{P}^\perp G) \|_{L^2_{x,v}} \big( \| \nabla_x \partial^m_x \rho^+ \|_{L^2_x} + \| \nabla_x \partial^m_x \rho^- \|_{L^2_x} \big) \\
	\leq & C \| B \|_{H^s_x} \| \mathbb{P}^\perp G \|_{H^s_x L^2_v} \big( \| \nabla_x \partial^m_x \rho^+ \|_{L^2_x} + \| \nabla_x \partial^m_x \rho^- \|_{L^2_x} \big) \,.
	\end{aligned} 
	\end{equation}
	Therefore, by the bounds on quantities $\textbf{E}_{2i}$ for $1 \leq i \leq 5$ in the previous, we obtain the bound of $\textbf{E}_2$
	\begin{equation}
	  \begin{aligned}
	    \textbf{E}_2 = & \textbf{E}_{21} + \textbf{E}_{22} + \textbf{E}_{23} + \textbf{E}_{24} + \textbf{E}_{25} \\
	    \leq & \eps \frac{\d}{\d t} \sum_{i=1}^3 \l \partial_i \partial^m_x \mathbb{P}^\perp G , \partial^m_x \rho^+ \zeta_i^+ (v) + \partial^m_x \rho^- \zeta^-_i (v) \r_{L^2_{x,v}} \\
	    + & C \| \mathbb{P}^\perp G \|_{H^s_x L^2_v} \| \partial^m_x \div_x u \|_{L^2_x} + C \| \mathbb{P}^\perp G \|^2_{H^s_x L^2_v} \\
	    + & C ( \tfrac{1}{\eps} + \| E \|_{H^s_x} + \| B \|_{H^s_x} ) \| \mathbb{P}^\perp G \|_{H^s_x L^2_v} \big( \| \nabla_x \partial^m_x \rho^+ \|_{L^2_x} + \| \nabla_x \partial^m_x \rho^- \|_{L^2_x} \big)
	  \end{aligned}
	\end{equation}
	for $0 < \eps \leq 1$.
	
	For the term $\textbf{E}_3$, we just employ the H\"older inequality to estimate
	\begin{equation}
	  \textbf{E}_3 \leq \tfrac{3}{2} \| \nabla_x \partial^m_x \theta \|_{L^2_x} \big( \| \nabla_x \partial^m_x \rho^+ \|_{L^2_x} + \| \nabla_x \partial^m_x \rho^- \|_{L^2_x} \big) \,.
	\end{equation}
	We can deal with the quantity $\textbf{E}_4$ by utilizing the analogous arguments in estimating the term $ \textbf{E}_2 $ (in fact, it is simpler than the estimation of $ \textbf{E}_2 $). More precisely,
	\begin{equation}
	  \begin{aligned}
	    \textbf{E}_4 = & - \sum_{i=1}^3 \l \partial^m_x ( \eps E \mathbb{P} G ) , \partial_i \partial^m_x \rho^+ \nabla_v ( \tfrac{\mathsf{q} \zeta^+_i (v)}{\sqrt{M}} ) \sqrt{M} + \partial_i \partial^m_x \rho^- \nabla_v ( \tfrac{\mathsf{q} \zeta^-_i (v)}{\sqrt{M}} ) \sqrt{M} \r_{L^2_{x,v}} \\
	    & - \sum_{i=1}^3 \l  \partial^m_x ( B \mathbb{P} G ) , \partial_i \partial^m_x \rho^+ v \times \nabla_v ( \tfrac{\mathsf{q} \zeta^+_i (v)}{\sqrt{M}} ) \sqrt{M} + \partial_i \partial^m_x \rho^- v \times \nabla_v ( \tfrac{\mathsf{q} \zeta^-_i (v)}{\sqrt{M}} ) \sqrt{M} \r_{L^2_{x,v}} \\
	    \leq & \sum_{i=1}^3 \eps \| \partial^m_x (E \mathbb{P} G) \|_{L^2_{x,v}} \big( \| \partial_i \partial^m_x \rho^+ \|_{L^2_x} + \| \partial_i \partial^m_x \rho^- \|_{L^2_x} \big) \\
	    & \qquad \times \big( \| \nabla_v ( \tfrac{\mathsf{q} \zeta_i^+ (v)}{\sqrt{M}} ) \sqrt{M} \|_{L^2_v} + \| \nabla_v ( \tfrac{\mathsf{q} \zeta_i^- (v)}{\sqrt{M}} ) \sqrt{M} \|_{L^2_v} \big) \\
	    + & \sum_{i=1}^3 \| \partial^m_x (B \mathbb{P} G) \|_{L^2_{x,v}} \big( \| \partial_i \partial^m_x \rho^+ \|_{L^2_x} + \| \partial_i \partial^m_x \rho^- \|_{L^2_x} \big) \\
	    & \qquad \times \big( \| v \times \nabla_v ( \tfrac{\mathsf{q} \zeta_i^+ (v)}{\sqrt{M}} ) \sqrt{M} \|_{L^2_v} + \| v \times \nabla_v ( \tfrac{\mathsf{q} \zeta_i^- (v)}{\sqrt{M}} ) \sqrt{M} \|_{L^2_v} \big) \\
	    \leq & C \big( \eps \| \partial^m_x ( E \mathbb{P} G ) \|_{L^2_{x,v}} + \| \partial^m_x (B \mathbb{P} G) \|_{L^2_{x,v}} \big) \big( \| \nabla_x \partial^m_x \rho^+ \|_{L^2_x} + \| \nabla_x \partial^m_x \rho^- \|_{L^2_x} \big) \\
	    \leq & C ( \eps \| E \|_{H^s_x} + \| B \|_{H^s_x} ) \| \mathbb{P} G \|_{H^s_x L^2_v} \big( \| \nabla_x \partial^m_x \rho^+ \|_{L^2_x} + \| \nabla_x \partial^m_x \rho^- \|_{L^2_x} \big) \,.
	  \end{aligned}
	\end{equation}
	
	We finally estimate the term $\textbf{E}_5$ by employing Lemma \ref{Lmm-Gamma-Torus}. The details are shown as
	\begin{equation}
	  \begin{aligned}
	    \textbf{E}_5 = & \sum_{i=1}^3 \l \partial^m_x \Gamma ( G , G ) , \partial_i \partial^m_x \rho^+ \zeta_i^+ (v) + \partial_i \partial^m_x \rho^- \zeta_i^- (v) \r_{L^2_{x,v}} \\
	    \leq & C \sum_{i=1}^3 \| G \|_{H^s_x L^2_v} \| G \|_{H^s_x L^2_v (\nu)} \| \partial_i \partial^m_x \rho^+ \zeta_i^+ (v) + \partial_i \partial^m_x \rho^- \zeta_i^- (v) \|_{L^2_{x,v}(\nu)} \\
	    \leq & C \| G \|_{H^s_x L^2_v} ( \| \mathbb{P} G \|_{H^s_x L^2_v} + \| \mathbb{P}^\perp G \|_{H^s_x L^2_v (\nu)} ) ( \| \nabla_x \partial^m_x \rho^+ \|_{L^2_x} + \| \nabla_x \partial^m_x \rho^- \|_{L^2_x} ) \,.
	  \end{aligned}
	\end{equation}
	
	Collecting the previous estimations on the quantities $\textbf{E}_1$, $\textbf{E}_2$, $\textbf{E}_3$, $\textbf{E}_4$, $\textbf{E}_5$, and employing the part (1) of Lemma \ref{Lmm-CF-nu} or Lemma \ref{Lmm-nu-norm}, we have
	\begin{equation*}
	  \begin{aligned}
	    & \| \nabla_x \partial^m_x \rho^+ \|^2_{L^2_x} + \| \nabla_x \partial^m_x \partial^m_x \rho^- \|^2_{L^2_x} + \| \partial^m_x \div_x E \|^2_{L^2_x} \\
	    \leq & \eps \frac{\d}{\d t} \sum_{i=1}^3 \Big( \l \partial_i \partial^m_x u_i , \partial^m_x \rho^+ + \partial^m_x \rho^- \r_{L^2_x} + \l \partial_i \partial^m_x \mathbb{P}^\perp G , \partial^m_x \rho^+ \zeta_i^+ (v) + \partial^m_x \rho^- \zeta_i^- (v) \r_{L^2_{x,v}} \Big) \\
	    & + 3 \| \partial^m_x \div_x u \|^2_{L^2_x} + C \| \mathbb{P}^\perp G \|^2_{H^s_x L^2_v (\nu)} + C \| \mathbb{P}^\perp G \|_{H^s_x L^2_v (\nu)} \| \partial^m_x \div_x u \|_{L^2_x} \\
	    & + \big( \| \nabla_x \partial^m_x \rho^+ \|_{L^2_x} + \| \nabla_x \partial^m_x \rho^- \|_{L^2_x} \big) \Big[ \tfrac{3}{2} \| \nabla_x \partial^m_x \theta \|_{L^2_x} + \tfrac{C}{\eps} \| \mathbb{P}^\perp G \|_{H^s_x L^2_v (\nu)} \\
	    & + C ( \| G \|_{H^s_x L^2_v} + \| E \|_{H^s_x} + \| B \|_{H^s_x} ) ( \| \mathbb{P} G \|_{H^s_x L^2_v} + \| \mathbb{P}^\perp G \|_{H^s_x L^2_v (\nu)} ) \Big] 
	  \end{aligned}
	\end{equation*}
	for $0 < \eps \leq 1$, which immediately implies by the Young's inequality that for all multi-indexes $m \in \mathbb{N}^3$ with $|m| \leq s - 1$ ($s \geq 3$)
	  \begin{align}\label{rho-Dissip}
	    \no & \| \nabla_x \partial^m_x \rho^+ \|^2_{L^2_x} + \| \nabla_x \partial^m_x \partial^m_x \rho^- \|^2_{L^2_x} + \| \partial^m_x \div_x E \|^2_{L^2_x} \\
	    \no \leq & 2 \eps \frac{\d}{\d t} \sum_{i=1}^3 \Big( \l \partial_i \partial^m_x u_i , \partial^m_x \rho^+ + \partial^m_x \rho^- \r_{L^2_x} + \l \partial_i \partial^m_x \mathbb{P}^\perp G , \partial^m_x \rho^+ \zeta_i^+ (v) + \partial^m_x \rho^- \zeta_i^- (v) \r_{L^2_{x,v}} \Big) \\
	    \no &  + 8 \| \partial^m_x \div_x u \|^2_{L^2_x} + 6 \| \nabla_x \partial^m_x \theta \|^2_{L^2_x} + \frac{C}{\eps^2} \| \mathbb{P}^\perp G \|^2_{H^s_x L^2_v (\nu)} \\
	    & + C  ( \| G \|^2_{H^s_x L^2_v} + \| E \|^2_{H^s_x} + \| B \|^2_{H^s_x} ) ( \| \mathbb{P} G \|^2_{H^s_x L^2_v} + \| \mathbb{P}^\perp G \|^2_{H^s_x L^2_v (\nu)} )
	  \end{align}
	for $0 < \eps \leq 1$. \\
	
	{\em Step 4. Summarization for the dissipation of the fluid part $\mathbb{P} G$.} We summarize the estimates derived in the previous three steps. We first add $\frac{1}{16}$ times of the bound \eqref{rho-Dissip} to the summation of bounds \eqref{u-Dissip} and \eqref{theta-Dissip}. Then we deduce that for all multi-indexes $m \in \mathbb{N}^3$ with $|m| \leq s -1$ ($s \geq 3$)
	\begin{equation}
	  \begin{aligned}
	    & \tfrac{1}{16} \| \nabla_x \partial^m_x \rho^+ \|^2_{L^2_x} + \tfrac{1}{16} \| \nabla_x \partial^m_x \rho^- \|^2_{L^2_x} + \tfrac{1}{16} \| \partial^m_x \div_x E \|^2_{L^2_x} \\
	    & + \tfrac{1}{2} \| \nabla_x \partial^m_x \theta \|^2_{L^2_x} + \| \nabla_x \partial^m_x u \|^2_{L^2_x} + \tfrac{1}{2} \| \partial^m_x \div_x u \|^2_{L^2_x} \\
	    \leq & \eps \frac{\d}{\d t} \sum_{i = 1}^3 \Big[ \l \tfrac{1}{8} \partial_i \partial^m_x u_i , \partial^m_x \rho^+ + \partial^m_x \rho^- \r_{L^2_x} + \sum_{j=1}^3 \l \partial_j \partial^m_x \mathbb{P}^\perp G , \partial^m_x u_i \zeta_{ij} (v) \r_{L^2_{x,v}} \\
	    &  + \l \partial_i \partial^m_x \mathbb{P}^\perp G , \partial^m_x \theta \widetilde{\zeta}_i (v) \r_{L^2_{x,v}} + \l \tfrac{1}{8} \partial_i \partial^m_x \mathbb{P}^\perp G , \partial^m_x \rho^+ \zeta_i^+ (v) + \partial^m_x \rho^- \zeta_i^- (v) \r_{L^2_{x,v}} \Big] \\
	    &  + C ( \| G \|^2_{H^s_x L^2_v} + \| E \|^2_{H^s_x} + \| B \|^2_{H^s_x} ) \big( \| \mathbb{P} G \|^2_{H^s_x L^2_v} + \tfrac{1}{\eps^2} \| \mathbb{P}^\perp G \|^2_{H^s_x L^2_v (\nu)} \big) \\
	    & + \frac{C}{\eps^2} \| \mathbb{P}^\perp G \|^2_{H^s_x L^2_v (\nu)} + 2 \delta \| \partial^m_x \div_x u \|^2_{L^2_x} + \delta \big( \| \nabla_x \partial^m_x \rho^+ \|^2_{L^2_x} + \| \nabla_x \partial^m_x \rho^- \|^2_{L^2_x} + \| \nabla_x \partial^m_x \theta \|^2_{L^2_x} \big)
	  \end{aligned}
	\end{equation}
	holds for any small $\delta > 0$ to be determined and for all $0 < \eps \leq 1$. We now take $\delta = \tfrac{1}{32}$, so that
	\begin{equation}\label{Fluid-Dissip-1}
	  \begin{aligned}
	    & 2 \| \nabla_x \partial^m_x u \|^2_{L^2_x} + 3 \| \nabla_x \partial^m_x \theta \|^2_{L^2_x} + \| \partial^m_x \div_x u \|^2_{L^2_x} \\
	    & + \| \nabla_x \partial^m_x \rho^+ \|^2_{L^2_x} + \| \nabla_x \partial^m_x \rho^- \|^2_{L^2_x} + \| \partial^m_x \div_x  E \|^2_{L^2_x} \\
	    \leq & \eps \frac{\d}{\d t} \sum_{i = 1}^3 \Big[ \l 4 \partial_i \partial^m_x u_i , \partial^m_x \rho^+ + \partial^m_x \rho^- \r_{L^2_x} + \sum_{j=1}^3 \l 32 \partial_j \partial^m_x \mathbb{P}^\perp G , \partial^m_x u_i \zeta_{ij} (v) \r_{L^2_{x,v}} \\
	    &  + \l 32 \partial_i \partial^m_x \mathbb{P}^\perp G , \partial^m_x \theta \widetilde{\zeta}_i (v) \r_{L^2_{x,v}} + \l 4 \partial_i \partial^m_x \mathbb{P}^\perp G , \partial^m_x \rho^+ \zeta_i^+ (v) + \partial^m_x \rho^- \zeta_i^- (v) \r_{L^2_{x,v}} \Big] \\
	    + & \frac{C}{\eps^2} \| \mathbb{P}^\perp G \|^2_{H^s_x L^2_v (\nu)} + C ( \| G \|^2_{H^s_x L^2_v} + \| E \|^2_{H^s_x} + \| B \|^2_{H^s_x} ) \big( \| \mathbb{P} G \|^2_{H^s_x L^2_v} + \tfrac{1}{\eps^2} \| \mathbb{P}^\perp G \|^2_{H^s_x L^2_v (\nu)} \big)
	  \end{aligned}
	\end{equation}
	holds for all multi-indexes $m \in \mathbb{N}^3$ with $|m| \leq s - 1$ ($s \geq 3$) and for all $0 < \eps \leq 1$.
	
	Recalling that
	\begin{equation}
	  \mathbb{P} G = \rho^+ \phi_1 (v) + \rho^- \phi_2 (v) + \sum_{i=1}^3 u_i \phi_{i+2} (v) + \theta \phi_6 (v),
	\end{equation}
	we directly compute that for all $ m \in \mathbb{N}^3$
	\begin{equation}
	  \begin{aligned}
	    \| \partial^m_x \mathbb{P} G \|^2_{L^2_v} = & (\partial^m_x \rho^+)^2 \| \phi_1 (v) \|^2_{L^2_v} + (\partial^m_x \rho^-)^2 \| \phi_2 (v) \|^2_{L^2_v} \\
	    & + \sum_{i=1}^3 ( \partial^m_x u_i)^2 \| \phi_{i+2} (v) \|^2_{L^2_v} + (\partial^m_x \theta)^2 \| \phi_6 (v) \|^2_{L^2_v} \\
	    = & (\partial^m_x \rho^+)^2 + (\partial^m_x \rho^-)^2 + 2 | \partial^m_x u |^2 + 3 (\partial^m_x \rho^+)^2 \,,
	  \end{aligned}
	\end{equation}
	where we utilize the facts $\l 1 , M \r_{L^2_v} = \l v_i^2 , M \r_{L^2_v} = 1$ for $i = 1,2,3$ and $\l |v|^4 , M \r_{L^2_v} = 15$. Consequently, we have
	\begin{equation}\label{Fluid-norm-1}
	  \begin{aligned}
	    \| \nabla_x \partial^m_x \mathbb{P} G \|^2_{L^2_{x,v}} = \| \nabla_x \partial^m_x \rho^+ \|^2_{L^2_x} + \| \nabla_x \partial^m_x \rho^- \|^2_{L^2_x} + 2 \| \nabla_x \partial^m_x u \|^2_{L^2_x} + 3 \| \nabla_x \partial^m_x \theta \|^2_{L^2_x} \,.
	  \end{aligned}
	\end{equation}
	Moreover, one also has
	\begin{equation}\label{Fluid-norm-2}
	  \begin{aligned}
	    & \| \mathbb{P} G \|^2_{H^s_x L^2_v} = \| \mathbb{P} G \|^2_{L^2_{x,v}} + \| \nabla_x \mathbb{P} G \|^2_{H^{s-1}_x L^2_v} \\
	    = & \| \rho^+ \|^2_{L^2_x} + \| \rho^- \|^2_{L^2_x} + 2 \| u \|^2_{L^2_x} + 3 \| \theta \|^2_{L^2_x} + \| \nabla_x \mathbb{P} G \|^2_{H^{s-1}_x L^2_v} \\
	    \leq & 2 \| \rho^+ - ( \rho^+ )_{\T^3} \|^2_{L^2_x} + 2 \| \rho^- - ( \rho^- )_{\T^3} \|^2_{L^2_x} + 4 \| u - ( u )_{\T^3} \|^2_{L^2_x} + 6 \| \theta - ( \theta )_{\T^3} \|^2_{L^2_x} \\
	    & + 2 |\T^3| \big( (\rho^+)^2_{\T^3} + (\rho^-)^2_{\T^3} + 2 (u)^2_{\T^3} + 3 (\theta)^2_{\T^3} \big) + \| \nabla_x \mathbb{P} G \|^2_{H^{s-1}_x L^2_v} \\
	    \leq & C ( \| \nabla_x \rho^+ \|^2_{L^2_x} + \| \nabla_x \rho^- \|^2_{L^2_x} + 2 \| \nabla_x u \|^2_{L^2_x} + 3 \| \nabla_x \theta \|^2_{L^2_x} ) +  \| \nabla_x \mathbb{P} G \|^2_{H^{s-1}_x L^2_v} \\
	    & + C \big( (\rho^+)^2_{\T^3} + (\rho^-)^2_{\T^3} + 2 (u)^2_{\T^3} + 3 (\theta)^2_{\T^3} \big) \\
	    \leq & C \| \nabla_x \mathbb{P} G \|^2_{H^{s-1}_x L^2_v} + C \big( (\rho^+)^2_{\T^3} + (\rho^-)^2_{\T^3} + 2 (u)^2_{\T^3} + 3 (\theta)^2_{\T^3} \big)
	  \end{aligned}
	\end{equation}
	for $s \geq 3$, where the last second inequality is implied by the Poincar\'e inequality. From substituting the relations \eqref{Fluid-norm-1} and \eqref{Fluid-norm-2} into the bound \eqref{Fluid-Dissip-1} and summing up for all $|m| \leq s - 1$, we deduce that the inequality \eqref{Fluid-Dissip} holds for all $0 < \eps \leq 1$. The proof of Proposition \ref{Prop-MM-Est} is completed.
\end{proof}

\subsection{Estimations on some average quantities} In this subsection, we will estimate the average quantities $ (\rho^\pm)^2_{\T^3} $, $(u)^2_{\T^3}$ and $(\theta)^2_{\T^3}$ appearing in the right-hand side of \eqref{Fluid-Dissip} by utilizing the conservation laws \eqref{Conservatn-Law-G} of mass, momentum and energy for the VMB system \eqref{VMB-F}. More precisely, we will give the following proposition.

\begin{proposition}\label{Prop-rho-u-theta-average}
	Under the assumptions in Proposition \ref{Prop-MM-Est} and the initial data of $(G, E , B)$ satisfying the conservation laws \eqref{Conservatn-Law-G}, we have
	\begin{equation}\label{Average-Bnd}
	  \begin{aligned}
	    (\rho^+)^2_{\T^3} + (\rho^-)^2_{\T^3} + (u)^2_{\T^3} + (\theta)^2_{\T^3} \leq C \left( \| E \|^2_{L^2_x} + \| B \|^2_{H^1_x} \right) \left( \| E \|^2_{L^2_x} + \| \nabla_x B \|^2_{L^2_x} \right)
	  \end{aligned}
	\end{equation}
	for some positive constant $C > 0$ and for all $0 < \eps \leq 1$.
\end{proposition}

The term $ \| E \|^2_{L^2_x} + \| \nabla_x B \|^2_{L^2_x} $ will be a part of the energy dissipation as shown in the next subsection and the quantity $ \| E \|^2_{L^2_x} + \| B \|^2_{H^1_x} $ can be controlled by the energy term. As a consequence, the average quantity $ (\rho^+)^2_{\T^3} + (\rho^-)^2_{\T^3} + (u)^2_{\T^3} + (\theta)^2_{\T^3} $ will be well dealt in deriving the global in times energy bounds uniformly in $\eps$ with small initial data. We next prove this proposition.

\begin{proof}[Proof of Proposition \ref{Prop-rho-u-theta-average}]
From the conservation laws \eqref{Conservatn-Law-G} and the definition \eqref{rho-u-theta} of $\rho^\pm$, $u$ and $\theta$, we deduce
\begin{equation}\label{Cons-rho-u-theta}
  \left\{
    \begin{array}{l}
      (\rho^\pm)_{\T^3} = 0 \,, \\[3mm]
      (u)_{\T^3} = - \frac{1}{2 |\T^3|} \int_{\T^3} E \times B \d x \,, \\[3mm]
      (\theta)_{\T^3} = - \frac{\eps}{3 |\T^3|} \int_{\T^3} ( |E|^2 + |B - \bar{B}|^2 ) \d x \,.
    \end{array}
  \right.
\end{equation}
Recalling the relation \eqref{B-average}, which means $\int_{\T^3} ( B - \bar{B} ) \d x = 0$, we derive from the Poincar\'e inequality that
\begin{equation}
  \begin{aligned}
    \int_{\T^3} |B - \bar{B}|^2 \d x \leq C \| \nabla_x B \|^2_{L^2_x}
  \end{aligned}
\end{equation}
for some positive constant $C > 0$. Thus, the third relation of \eqref{Cons-rho-u-theta} gives us
\begin{equation}
  \begin{aligned}
    | (\theta)_{\T^3} | \leq C ( \| E \|^2_{L^2_x} + \| \nabla_x B \|^2_{L^2_x} )
  \end{aligned}
\end{equation}
for all $0 < \eps \leq 1$. For the second equality of \eqref{Cons-rho-u-theta}, it is yielded by the H\"older inequality that
\begin{equation}
  \begin{aligned}
    | (u)_{\T^3} | \leq C \| E \|_{L^2_x} \| B \|_{L^2_x} \,.
  \end{aligned}
\end{equation}
Consequently, we obtain
\begin{equation}
  \begin{aligned}
    & (\rho^+)^2_{\T^3} + (\rho^-)^2_{\T^3} + (u)^2_{\T^3} + (\theta)^2_{\T^3} \\
    \leq & C \| E \|^2_{L^2_x} \| B \|^2_{L^2_x} + C \big( \| E \|^2_{L^2_x} + \| \nabla_x B \|^2_{L^2_x} \big)^2 \\
    \leq & C \left( \| E \|^2_{L^2_x} + \| B \|^2_{H^1_x} \right) \left( \| E \|^2_{L^2_x} + \| \nabla_x B \|^2_{L^2_x} \right) \,.
  \end{aligned}
\end{equation}
The proof of Proposition \ref{Prop-rho-u-theta-average} is finished.
\end{proof}

\subsection{Decay structures on the Maxwell system} In this subsection, we will find enough dissipation or decay properties on the electronic field $E$ and the magnetic field $B$ by making use of the Maxwell equations, hence the last four equations in \eqref{VMB-G-drop-eps}
\begin{equation}\label{Mxw}
\left\{
\begin{array}{l}
\partial_t E - \nabla_x \times B = - \frac{1}{\eps} ( u^+ - u^- ) \,, \\[2mm]
\partial_t B + \nabla_x \times E = 0 \,, \\[2mm]
\div_x E = \rho^+ - \rho^- \,, \ \div_x B = 0 \,,
\end{array}
\right.
\end{equation}
where we make use of the definition \eqref{rho-u-theta}. It is noticed that the second Faraday's law equation in \eqref{Mxw} does not have explicit dissipative term. If we take $\partial_t$ on the evolution of the magnetic field $B$ and combine with the evolution of $E$, we have
\begin{equation*}
\partial_{tt} B + \nabla_x \times ( \nabla_x \times B ) = \frac{1}{\eps} \nabla \times ( u^+ - u^- )  \,,
\end{equation*}
which implies that
\begin{equation}\label{Mxw-1}
\partial_{tt} B - \Delta_x B = \frac{1}{\eps} \nabla \times ( u^+ - u^- )
\end{equation}
by the equality $ \nabla_x \times ( \nabla_x \times B ) = - \Delta_x B $ under the divergence-free property $\div_x B = 0$.

However, the dissipation of \eqref{Mxw-1} is remain not enough. We try to derive the Ohm's law from the microscopic equation of $G$ in \eqref{VMB-G-drop-eps}, which will supply a decay term $\partial_t B$. More precisely, we dot with $\mathsf{q}_1$ in the first equation of \eqref{VMB-G-drop-eps}, and then we gain
\begin{align}\label{Mxw-2}
\no & \eps \partial_t ( G \cdot \mathsf{q}_1 ) + v \cdot \nabla_x ( G \cdot \mathsf{q}_1 ) + (\eps E + v \times B ) \cdot \nabla_v ( G \cdot \mathsf{q}_2 ) - 2 E \cdot v \sqrt{M}\\
& + \tfrac{1}{\eps} ( \mathscr{L} G ) \cdot \mathsf{q}_1 = \tfrac{1}{2} \eps ( E \cdot v ) ( G \cdot \mathsf{q}_2 ) + \Gamma ( G , G ) \cdot \mathsf{q}_1 \,,
\end{align}
where we make use of the relation $\mathsf{q} \mathsf{q}_1 = \mathsf{q}_2$. Recalling the definition \eqref{Linear-Oprt-VMB} of $\mathscr{L}$ and \eqref{Linear-Oprt-B} of $\mathcal{L}$, we calculate
\begin{equation}
  \begin{aligned}
    ( \mathscr{L} G ) \cdot \mathsf{q}_1 = & [ \mathcal{L} G^+ + \mathcal{L} ( G^+ , G^- ) , \mathcal{L} G^- + \mathcal{L} ( G^- , G^+ ) ] \cdot \mathsf{q}_1 \\
    = & \mathcal{L} ( G^+ - G^- ) + \mathcal{L} ( G^+ , G^- ) - \mathcal{L} ( G^- , G^+ ) \\
    = & \mathcal{L} ( G \cdot \mathsf{q}_1 ) - \big[ \mathcal{Q} ( G^+ , 1 ) + \mathcal{Q} ( 1 , G^- ) - \mathcal{Q} ( G^- , 1 ) - \mathcal{Q} (1 , G^+) \big] \\
    = & \mathcal{L} ( G \cdot \mathsf{q}_1 ) - \big[ \mathcal{Q} ( G \cdot \mathsf{q}_1 , 1) - \mathcal{Q} ( 1 , G \cdot \mathsf{q}_1 ) \big] \\
    = & \mathcal{L} ( G \cdot \mathsf{q}_1 ) + \mathfrak{L} ( G \cdot \mathsf{q}_1 ) \,,
  \end{aligned}
\end{equation}
where the linear operator $\mathfrak{L}$ is defined as
\begin{equation}\label{Mathfrak-L}
  \begin{aligned}
    \mathfrak{L} g = & \mathcal{L} ( g, - g ) =  - \big[ \mathcal{Q} (g , \sqrt{M}) - \mathcal{Q} (\sqrt{M} , g) \big] \\
    = & \sqrt{M} \int_{\R^3} \Big( \tfrac{g}{\sqrt{M}} - \tfrac{g_*}{\sqrt{M_*}} - \tfrac{g'}{\sqrt{M'}} + \tfrac{g_*'}{\sqrt{M_*'}} \Big) |v - v_*| M_* \d v_* \,.
  \end{aligned}
\end{equation}
Then the equation \eqref{Mxw-2} of $ G \cdot \mathsf{q}_1 $ can be rewritten as
\begin{equation}\label{Ohm-Law-Kinetic}
  \begin{aligned}
    \tfrac{1}{\eps} ( \mathcal{L} + \mathfrak{L} ) ( G \cdot \mathsf{q}_1 ) = & - \eps \partial_t ( G \cdot \mathsf{q}_1 ) - v \cdot \nabla_x ( G \cdot \mathsf{q}_1 ) - ( \eps E + v \times B ) \cdot \nabla_v ( G \cdot \mathsf{q}_2 )  \\
    & + 2 E \cdot v \sqrt{M} + \tfrac{1}{2} \eps ( E \cdot v ) ( G \cdot \mathsf{q}_2 ) + \Gamma ( G , G ) \cdot \mathsf{q}_1 \,.
  \end{aligned}
\end{equation}

We now display the following properties of the operator $\mathfrak{L}$, so that we can derive the corresponding macroscopic form from \eqref{Ohm-Law-Kinetic}.

\begin{lemma}\label{Lmm-mathfrak-L}
	The linear operator $\mathfrak{L}$ has the following properties:
	\begin{enumerate}
		\item Hilbert's decomposition of $\mathfrak{L}$:
		
		\noindent The linear operator $\mathfrak{L}$ can be decomposed as
		\begin{equation}
		  \mathfrak{L} g = \nu (v) g - \mathfrak{K} g \,,
		\end{equation}
		where $\mathfrak{K}$ is a compact integral operator on $L^2_v$.
		
		\item Coercivity of $\mathfrak{L}$:
		
		\noindent The linear operator $\mathfrak{L}$ is a nonnegative self-adjoint operator on $L^2_v$ with null space 
		\begin{equation}
		  \mathrm{Ker} ( \mathfrak{L} ) = \mathrm{Span} \{ \sqrt{M} \} \,.
		\end{equation}
		Moreover, the following coercivity estimate holds: there is $C > 0$ such that, for all $ g \in \mathrm{Ker}^\perp ( \mathfrak{L} ) \subset L^2_v $,
		\begin{equation}
		  \| g \|^2_{L^2_v (\nu)} \leq C \int_{\R^3} g \mathfrak{L} g \d v \,.
		\end{equation}
		In particular, for any $ g \in \mathrm{Ker}^\perp ( \mathfrak{L} ) \subset L^2_v $,
		\begin{equation}
		  \| g \|_{L^2_v (\nu)} \leq C \| \mathfrak{L} g \|_{L^2_v} \,.
		\end{equation}
		
		\item Properties of $\mathcal{L} + \mathfrak{L}$:
		
		\noindent For $\Phi (v) = v \sqrt{M} = [ \chi_2 (v), \chi_3 (v) , \chi_4(v) ] \in L^2_v$ and $\Psi (v) = ( \tfrac{|v|^2}{2} - \tfrac{3}{2} ) \sqrt{M} = \chi_5 (v) \in L^2_v$, there are two functions $\widetilde{\Phi}$, $\widetilde{\Psi} \in \mathrm{Ker}^\perp ( \mathcal{L} + \mathfrak{L} ) = \mathrm{Ker}^\perp ( \mathfrak{L} ) $ such that
		\begin{equation}
		  ( \mathcal{L} + \mathfrak{L} ) \widetilde{\Phi} = \Phi \ \textrm{and } ( \mathcal{L} + \mathfrak{L} ) \widetilde{\Psi} = \Psi \,,
		\end{equation} 
		which are uniquely determined in $ \mathrm{Ker}^\perp ( \mathcal{L} + \mathfrak{L} ) $. Furthermore, there exist two scalar valued functions $\alpha , \beta : \R^+ \rightarrow \R$ such that
		\begin{equation}\label{Tilde-Phi-Phi}
		  \widetilde{\Phi} (v) = \alpha (|v|) \Phi (v) \quad \textrm{and} \quad \widetilde{\Psi} (v) = \beta (|v|) \Psi (v) \,.
		\end{equation}
	\end{enumerate}
\end{lemma} 

\begin{remark}\label{Rmk-mathfrak-L}
	The relations \eqref{Tilde-Phi-Phi} imply that
	\begin{equation}
	  \int_{\R^3} \Phi_i (v) \widetilde{\Phi}_j (v) \d v = \tfrac{1}{2} \sigma \delta_{ij} \,,
	\end{equation}
	where $\sigma = \tfrac{2}{3} \int_{\R^3} \Phi \cdot \widetilde{\Phi} \d v $ defines the electrical conductivity $\sigma > 0$. Moreover, we also define the energy conductivity $\lambda > 0$ by $ \lambda = \int_{\R^3} \Psi \widetilde{\Psi} \d v $.
\end{remark}

\begin{proof}[Proof of Lemma \ref{Lmm-mathfrak-L}]
	The first two parts of Lemma \ref{Lmm-mathfrak-L} can be seen in Proposition 5.8 and 5.9 of Ars\'enio-Saint-Raymond's book \cite{Arsenio-SRM-2016}, while the last part of Lemma \ref{Lmm-mathfrak-L} and Remark \ref{Rmk-mathfrak-L} have been shown on Page 46 in \cite{Arsenio-SRM-2016}. So we omit the details of the proof here.
\end{proof}

Next we multiply by $\widetilde{\Phi}(v) $ in \eqref{Ohm-Law-Kinetic} and integrate over $v \in \R^3$. Then we derive from \eqref{Ohm-Law-Kinetic}, the part (3) of Lemma \ref{Lmm-mathfrak-L} and Remark \ref{Rmk-mathfrak-L} that
\begin{equation}\label{Ohm-Law-fluid}
  \begin{aligned}
    \tfrac{1}{\eps} ( u^+ - u^- ) = & \sigma E - \eps \l \partial_t ( G \cdot \mathsf{q}_1 ) , \widetilde{\Phi} \r_{L^2_v} - \l v \cdot \nabla_x ( G \cdot \mathsf{q}_1 ) , \widetilde{\Phi} \r_{L^2_v} \\
    + & \l - ( \eps E + v \times B ) \cdot \nabla_v ( G \cdot \mathsf{q}_2 ) + \tfrac{1}{2} \eps ( E \cdot v ) ( G \cdot \mathsf{q}_2 ) + \Gamma ( G , G ) \cdot \mathsf{q}_1 , \widetilde{\Phi} \r_{L^2_v} \\
    \overset{\Delta}{=} & \sigma E + \mathcal{K} ( G , E, B ) \,,
  \end{aligned}
\end{equation}
where we utilize the relation
\begin{equation*}
  \begin{aligned}
    \l ( \mathcal{L} + \mathfrak{L} ) ( G \cdot \mathsf{q}_1 ) , \widetilde{\Phi} \r_{L^2_v} = \l G \cdot \mathsf{q}_1 , ( \mathcal{L} + \mathfrak{L} )  \widetilde{\Phi} \r_{L^2_v} = \l G \cdot \mathsf{q}_1 , \Phi \r_{L^2_v} = u^+ - u^-
  \end{aligned}
\end{equation*}
implied by the definition \eqref{rho-u-theta} and the self-adjoint property of $\mathcal{L} + \mathfrak{L}$.
Then, from substituting \eqref{Ohm-Law-fluid} into \eqref{Mxw-1} we deduce that
\begin{equation}\label{Decay-B}
\partial_{tt} B - \Delta_x B + \sigma \partial_t B = \nabla_x \times \mathcal{K} ( G, E, B ) \,,
\end{equation}
where we use the Faraday's law equation $\partial_t B + \nabla_x \times E = 0$. We thereby have found the decay term $\partial_t B$ of $B$-equation. Moreover, by plugging the relation \eqref{Ohm-Law-fluid} into \eqref{Mxw}, we get
\begin{equation}\label{Mxw-E-Decay}
\left\{
\begin{array}{l}
\partial_t E - \nabla_x \times B + \sigma E = - \mathcal{K} ( G, E, B ) \,, \\[2mm]
\partial_t B + \nabla_x \times E = 0 \,, \\[2mm]
\div_x E = \rho^+ - \rho^- \,, \ \div_x B = 0 \,,
\end{array}
\right.
\end{equation}
in which we have the damping structure $\sigma E$ of the electric field $E$.

Based on the equation \eqref{Decay-B} and Maxwell system \eqref{Mxw-E-Decay}, we derive the following proposition, which gives us some energy dissipative structures on the electric field $E$ and magnetic field $B$. 

\begin{proposition}\label{Prop-Decay-E-B}
	Assume that $(G, E, B)$ is the solution to the perturbed VMB system \eqref{VMB-G} constructed in Proposition \ref{Prop-Local-Solutn}. Then there is a constant $C > 0$, independent of $\eps > 0$, such that the energy inequality on $(E, B)$
	\begin{equation}\label{E-B-decay}
	\begin{aligned}
	\tfrac{1}{2} \tfrac{\d}{\d t} \mathscr{E}_1 (E, B) & + \eps \tfrac{\d}{\d t} \mathscr{A}_s (E,B) (t) + \mathscr{D}_1 (E, B) \\
	\leq & C \eps \big( \| \nabla_x \mathbb{P} G \|^2_{H^{s-1}_x L^2_v} + \| \mathbb{P}^\perp G \|^2_{H^s_x L^2_v (\nu)} \big) \\
	+ & C ( \| G \|_{H^s_x L^2_v} + \| B \|_{H^s_x} ) \big( \| E \|^2_{H^{s-1}_x} + \| \partial_t B \|^2_{H^{s-2}_x} + \| \nabla_x B \|^2_{H^{s-2}_x} \big) \\
	+ & C ( \| G \|_{H^s_x L^2_v} + \| B \|_{H^s_x} ) \big( | \nabla_x \mathbb{P} G \|^2_{H^{s-1}_x L^2_v} + \| \mathbb{P}^\perp G \|^2_{H^s_x L^2_v (\nu)} \big) \\
	+ & C ( \eps + \| G \|_{H^s_x L^2_v} + \| B \|_{H^s_x} ) \big( (\rho^+)^2_{\T^3} + (\rho^-)^2_{\T^3} + (u)^2_{\T^3} + (\theta)^2_{\T^3} \big)
	\end{aligned}
	\end{equation}
	holds for any $0  < \eps \leq 1$ and integer $s \geq 3$, where
	\begin{equation}
	  \begin{aligned}
	    \mathscr{E}_1 (E, B) = & \delta_2 \| E \|^2_{H^{s-1}} + \delta_2 \| B \|^2_{H^{s-1}_x} + \delta_1 ( \sigma - 1 ) \| B \|^2_{H^{s-2}_x} \\
	    & + \delta_1 \| \partial_t B + B \|^2_{H^{s-2}_x} + ( 1 - \delta_1 ) \| \partial_t B \|^2_{H^{s-2}_x} + \| \nabla_x B \|^2_{H^{s-2}_x} \,,
	  \end{aligned}
	\end{equation}
	and
	\begin{equation}
	  \begin{aligned}
	    \mathscr{D}_1 (E,B) = & \tfrac{\sigma}{2} \delta_2 \| E \|^2_{H^{s-1}_x} + \tfrac{\sigma}{2} \| \partial_t B \|^2_{H^{s-2}_x} +  \tfrac{\delta_1}{4} \| \nabla_x B \|^2_{H^{s-2}_x}
	  \end{aligned}
	\end{equation}
for some small $\delta_1, \delta_2 > 0$, independent of $\eps \in (0,1]$, and the constant $\sigma > 0$ is mentioned in Remark \ref{Rmk-mathfrak-L}. Here the quantity $\mathscr{A}_s (E,B) (t)$ is defined as
\begin{equation}\label{Quantity-A-EB}
  \begin{aligned}
    & \mathscr{A}_s (E,B) (t) = \sum_{|m| \leq s - 1} \Big[ \tfrac{\eps}{2} \delta_2 \big\| \big\langle \partial^m_x G \cdot \mathsf{q}_1 , \widetilde{\Phi} \big\rangle_{L^2_v} \big\|^2_{L^2_v} - \delta_2 \big\langle \partial^m_x G \cdot \mathsf{q}_1 , \widetilde{\Phi} \cdot \partial^m_x E \big\rangle_{L^2_{x,v}}  \Big] \\
    & + \sum_{|m| \leq s - 2} \Big[ \tfrac{\eps}{2} \big\| \nabla_x \times \big\langle \partial^m_x G \cdot \mathsf{q}_1 , \widetilde{\Phi} \big\rangle_{L^2_v} \big\|^2_{L^2_x}  + \l \nabla_x \times \big\langle \partial^m_x G \cdot \mathsf{q}_1 , \widetilde{\Phi} \big\rangle_{L^2_v} , \partial_t \partial^m_x B + \delta_1 \partial^m_x B \r_{L^2_x} \Big]   \,.
  \end{aligned}
\end{equation}
\end{proposition}

\begin{remark}
	As shown in the proof, the smallness of the constants $\delta_1, \delta_2 > 0$ is such that the energy dissipative parts in the left-hand side of \eqref{E-B-decay-1} have the lower bound $\mathscr{D}_1 (G,B)$ and the coefficient $1 - \delta_1$ of $\| \partial_t B \|^2_{H^{s-2}_x}$ in the energy functional $\mathscr{E}_1 (E, B)$ is lager than $ \tfrac{1}{2} $. However, the energy functional $\mathscr{E}_1 (E,B)$ may not be nonnegative, because the chosen constants $ \delta_1, \delta_2 > 0 $ can not ensure the positivity of $ \delta_2 - \delta_1 + \delta_1 \sigma $, so that the term $ (\delta_2 - \delta_1 + \delta_1 \sigma ) \| B \|^2_{L^2_x}  $ in the following equality
	\begin{equation}
	  \begin{aligned}
	    & \delta_2 \| B \|^2_{H^{s-1}_x} + \delta_1 ( \sigma - 1 ) \| B \|^2_{H^{s-2}_x} + \| \nabla_x B \|^2_{H^{s-2}} \\
	    = & (\delta_2 - \delta_1 + \delta_1 \sigma ) \| B \|^2_{L^2_x} + (\delta_2 + 1) \sum_{|m| = s - 2} \| \nabla_x \partial^m_x B \|^2_{L^2_x} + (  \delta_2 + 1 - \delta_1 + \delta_1 \sigma ) \| \nabla_x B \|^2_{H^{s-3}_x}
	  \end{aligned}
	\end{equation}
	may be negatively valued. Thanks to the energy inequality \eqref{Spatial-Bnd} in Proposition \ref{Prop-Spatial}, we will design a new positive energy after carefully adjusting the coefficients of $\mathscr{E}_1 (E, B)$.
\end{remark}

\begin{proof}[Proof of Proposition \ref{Prop-Decay-E-B}]
	We will complete this proof by two steps: We first derive a energy inequality which involves the dissipative structures $\| \partial_t B \|^2_{H^{s-2}_x}$ and $\| \nabla_x B \|^2_{H^{s-2}_x}$ of the magnetic field $B$ by utilizing the second order wave system \eqref{Decay-B}. Secondly, from the first order wave system \eqref{Mxw-E-Decay}, we can derive another energy inequality involving the dissipative structure $\| E \|^2_{H^{s-1}_x}$ of the electric field $E$.\\
	
	{\em Step 1. Derivations of the dissipative structures $\| \partial_t B \|^2_{H^{s-2}_x}$ and $\| \nabla_x B \|^2_{H^{s-2}_x}$.} By acting the derivative operator $\partial^m_x$ on the equation \eqref{Decay-B} for all multi-indexes $m \in \mathbb{N}^3$ with $ |m| \leq s - 2$ $(s \geq 3)$, taking the $L^2$-inner product by dot with $\partial_t \partial^m_x B $ and integrating by parts over $x \in \T^3$, we deduce that
	\begin{equation}\label{B-energy-equ-1}
	  \tfrac{1}{2} \tfrac{\d}{\d t} ( \| \partial_t \partial^m_x B \|^2_{L^2_x} + \| \nabla_x \partial^m_x B \|^2_{L^2_x} ) + \sigma \| \partial_t \partial^m_x B \|^2_{L^2_x} = \l \nabla_x \times \partial^m_x \mathcal{K} ( G , E, B ) , \partial_t \partial^m_x B \r_{L^2_x} \,.
	\end{equation}
	If we replace the multiplied vector $\partial_t \partial^m_x B$ by the vector $\partial^m_x B$ in the previous process, we have
	\begin{equation}\label{B-energy-equ-2}
	  \begin{aligned}
	    & \tfrac{1}{2} \tfrac{\d}{\d t} \big( \| \partial_t \partial^m_x B + \partial^m_x B \|^2_{L^2_x} - \| \partial_t \partial^m_x B \|^2_{L^2_x} + ( \sigma - 1 ) \| \partial^m_x B \|^2_{L^2_x} \big) \\
	    & - \| \partial_t \partial^m_x B \|^2_{L^2_x} + \| \nabla_x \partial^m_x B \|^2_{L^2_x} \\
	    & =\l \nabla_x \times \partial^m_x \mathcal{K} ( G, E, B ) , \partial^m_x B \r_{L^2_x} 
	  \end{aligned}
	\end{equation}
	for all $|m| \leq s - 2$, where we utilize the equality
	\begin{equation*}
	  \begin{aligned}
	    & \l \partial_{tt} \partial^m_x B , \partial^m_x B \r_{L^2_x} = - \| \partial_t \partial^m_x B \|^2_{L^2_x} \\
	    & + \tfrac{1}{2} \tfrac{\d}{\d t} \big( \| \partial_t \partial^m_x B + \partial^m_x B \|^2_{L^2_x} - \| \partial_t \partial^m_x B \|^2_{L^2_x} + ( \sigma - 1 ) \| \partial^m_x B \|^2_{L^2_x} \big) \,.
	  \end{aligned}
	\end{equation*}
	Then, the equality \eqref{B-energy-equ-1} being added by $\delta_1$ times of \eqref{B-energy-equ-2} yields 
	\begin{equation}\label{B-energy-equ} 
	  \begin{aligned}
	    & \tfrac{1}{2} \tfrac{\d}{\d t} \big( \delta_1 \| \partial_t \partial^m_x B + \partial^m_x B \|^2_{L^2_x} + ( 1 - \delta_1 ) \| \partial_t \partial^m_x B \|^2_{L^2_x} + \| \nabla_x \partial^m_x B \|^2_{L^2_x} + \delta_1 ( \sigma - 1 ) \| \partial^m_x B \|^2_{L^2_x} \big) \\
	    & + ( \sigma - \delta_1 ) \| \partial_t \partial^m_x B \|^2_{L^2_x} + \delta_1 \| \nabla_x \partial^m_x B \|^2_{L^2_x} \\
	    = & \l \nabla_x \times \partial^m_x \mathcal{K} ( G, E, B ) , \partial_t \partial^m_x B \r_{L^2_x} + \l  \partial^m_x \mathcal{K} ( G, E, B ) ,  \delta_1 \nabla_x \times \partial^m_x B \r_{L^2_x}
	  \end{aligned}
	\end{equation}
	holds for all $|m| \leq s - 2 \ (s \geq 3)$, where $\delta_1 \in ( 0, 1 ]$ is a small number to be determined.
	
	Next, we estimate the two terms in the right-hand side of the equality \eqref{B-energy-equ}. Recalling the definition of $\mathcal{K} (G,E,B)$ in \eqref{Ohm-Law-fluid}, we decompose the first term in the right-hand side of \eqref{B-energy-equ} as some parts. More precisely,
	  \begin{align}
	    \no & \l \nabla_x \times \partial^m_x \mathcal{K} ( G, E, B ) , \partial_t \partial^m_x B \r_{L^2_x} \\
	    \no = & \underset{\textbf{F}_1}{ \underbrace{ - \eps \l \nabla_x \times \partial_t \l \partial^m_x G \cdot \mathsf{q}_1 , \widetilde{\Phi} \r_{L^2_v} , \partial_t \partial^m_x B \r_{L^2_x} } } \\
	    \no & \underset{\textbf{F}_2}{ \underbrace{ - \l \nabla_x  \times \l v \cdot \nabla_x ( \partial^m_x G \cdot \mathsf{q}_1  ) , \widetilde{\Phi} \r_{L^2_v} , \partial_t \partial^m_x B \r_{L^2_x} }} \\
	    \no & \underset{\textbf{F}_3}{ \underbrace{ - \eps \l \nabla_x \times \partial^m_x \l E \cdot \nabla_v ( G \cdot \mathsf{q}_2 ) , \widetilde{\Phi} \r_{L^2_v} , \partial_t \partial^m_x B \r_{L^2_x} }} \\
	    \no & + \underset{\textbf{F}_4}{ \underbrace{ \tfrac{1}{2} \eps \l \nabla_x \times \partial^m_x \l (E \cdot v) ( G \cdot \mathsf{q}_2 ) , \widetilde{\Phi} \r_{L^2_v} , \partial_t \partial^m_x B \r_{L^2_x} }} \\
	    \no & \underset{\textbf{F}_5}{ \underbrace{ - \l \nabla_x \times \partial^m_x \l (v \times B) \cdot \nabla_v ( G \cdot \mathsf{q}_2 ) , \widetilde{\Phi} \r_{L^2_v} , \partial_t \partial^m_x B \r_{L^2_x} }} \\
	    & + \underset{\textbf{F}_6}{ \underbrace{ \l \nabla_x \times \partial^m_x \l \Gamma (G,G) \cdot \mathsf{q}_1 , \widetilde{\Phi} \r_{L^2_v} , \partial_t \partial^m_x B \r_{L^2_x} }}
	  \end{align}
	holds for all $|m| \leq s - 2 \ (s \geq 3)$. The terms $\textbf{F}_i$ $(i = 1,2, \cdots, 6)$ will be controlled one by one.
	
	For the term $\textbf{F}_1$, we derive from the equation \eqref{Decay-B} that
	\begin{equation}
	  \begin{aligned}
	    \textbf{F}_1 = & - \eps \frac{\d}{\d t} \l \nabla_x \times \l \partial^m_x G \cdot \mathsf{q}_1 , \widetilde{\Phi} \r_{L^2_v} , \partial_t \partial^m_x B \r_{L^2_x} \\
	    & \underset{\textbf{F}_{11}}{\underbrace{ - \eps \l \nabla_x \big[ \nabla_x \times \l \partial^m_x G \cdot \mathsf{q}_1 , \widetilde{\Phi} \r_{L^2_v} \big] , \nabla_x \partial^m_x B \r_{L^2_x} }} \\
	    & \underset{\textbf{F}_{12}}{\underbrace{ - \eps \sigma \l \nabla_x \times \l \partial^m_x G \cdot \mathsf{q}_1 , \widetilde{\Phi} \r_{L^2_v} , \partial_t \partial^m_x B \r_{L^2_x} }} \\
	    & + \underset{\textbf{F}_{13}}{\underbrace{ \eps \l \nabla_x \times \l \partial^m_x G \cdot \mathsf{q}_1 , \widetilde{\Phi} \r_{L^2_v} , \nabla_x \times \partial^m_x \mathcal{K} ( G, E, B ) \r_{L^2_x} }} \,.
	  \end{aligned}
	\end{equation}
	It is derived from the H\"older inequality and the decomposition $G = \mathbb{P} G + \mathbb{P}^\perp G$ that the quantity $\textbf{F}_{11}$ is bounded by
	\begin{equation}
	  \begin{aligned}
	    \textbf{F}_{11} \leq & \eps \| \nabla_x ( \nabla_x \partial^m_x G  ) \|_{L^2_{x,v}} \| \mathsf{q}_1 \widetilde{\Phi} \|_{L^2_v} \| \nabla_x \partial^m_x B \|_{L^2_x} \\
	    \leq & C \eps \big( \| \nabla_x \mathbb{P} G \|_{H^{s-1}_x L^2_v} + \| \mathbb{P}^\perp G \|_{H^s_x L^2_v} \big) \| \nabla_x \partial^m_x B \|_{L^2_x} \,.
	  \end{aligned}
	\end{equation}
	Similarly, the quantity $\textbf{F}_{12}$ is controlled by
	\begin{equation}
	  \begin{aligned}
	    \textbf{F}_{12} \leq C \eps \big( \| \nabla_x \mathbb{P} G \|_{H^{s-1}_x L^2_v} + \| \mathbb{P}^\perp G \|_{H^s_x L^2_v} \big) \| \partial_t \partial^m_x B \|_{L^2_x} \,.
	  \end{aligned}
	\end{equation}
	The calculations on $\textbf{F}_{13}$ are rather tedious, since the expression $\mathcal{K} (G,E,B)$ includes six parts. We plug the expression $\mathcal{K} (G,E,B)$ defined in \eqref{Ohm-Law-fluid} into the term $\textbf{F}_{13}$. Then we get
	  \begin{align}
	    \no \textbf{F}_{13} = & - \tfrac{\eps^2}{2} \tfrac{\d}{\d t} \big\| \nabla_x \times \l \partial^m_x G \cdot \mathsf{q}_1 , \widetilde{\Phi} \r_{L^2_v} \big\|^2_{L^2_x} \\
	    \no & \underset{\textbf{F}_{131}}{\underbrace{ - \eps \l \nabla_x \times \l \partial^m_x G \cdot \mathsf{q}_1 , \widetilde{\Phi} \r_{L^2_v} , \nabla_x \times \l v \cdot \nabla_x ( \partial^m_x G \cdot \mathsf{q}_1 ) , \widetilde{\Phi} \r_{L^2_v} \r_{L^2_x}  }} \\
	    \no & \underset{\textbf{F}_{132}}{\underbrace{ - \eps^2 \l \nabla_x \times \l \partial^m_x G \cdot \mathsf{q}_1 , \widetilde{\Phi} \r_{L^2_v} , \nabla_x \times  \partial^m_x \l E \cdot \nabla_v ( G \cdot \mathsf{q}_2 ) , \widetilde{\Phi} \r_{L^2_v} \r_{L^2_x}  }} \\
	    \no & + \underset{\textbf{F}_{133}}{\underbrace{ \tfrac{1}{2} \eps^2 \l \nabla_x \times \l \partial^m_x G \cdot \mathsf{q}_1 , \widetilde{\Phi} \r_{L^2_v} , \nabla_x \times  \partial^m_x \l (E \cdot v ) ( G \cdot \mathsf{q}_2 ) , \widetilde{\Phi} \r_{L^2_v} \r_{L^2_x}  }} \\
	    \no & \underset{\textbf{F}_{134}}{\underbrace{ - \eps \l \nabla_x \times \l \partial^m_x G \cdot \mathsf{q}_1 , \widetilde{\Phi} \r_{L^2_v} , \nabla_x \times \partial^m_x \l (v \times B) \cdot \nabla_v ( G \cdot \mathsf{q}_2 ) , \widetilde{\Phi} \r_{L^2_v} \r_{L^2_x}  }} \\
	     & \underset{\textbf{F}_{135}}{\underbrace{ - \eps \l \nabla_x \times \l \partial^m_x G \cdot \mathsf{q}_1 , \widetilde{\Phi} \r_{L^2_v} , \nabla_x \times \l \partial^m_x \Gamma ( G , G ) \cdot \mathsf{q}_1  , \widetilde{\Phi} \r_{L^2_v} \r_{L^2_x}  }} \,.
	  \end{align}
	We estimate 
	\begin{equation}
	  \begin{aligned}
	    \textbf{F}_{131} \leq & \eps \| \nabla_x \partial^m_x G \|_{L^2_{x,v}} \| \nabla_x ( \nabla_x \partial^m_x G ) \|_{L^2_{x,v}} \| \mathsf{q}_1 \widetilde{\Phi} \|_{L^2_v} \| v \mathsf{q}_1 \widetilde{\Phi} \|_{L^2_v} \\
	    \leq & C \eps \| \nabla_x G \|^2_{H^{s-1}_x L^2_v} \\
	    \leq & C \eps ( \| \nabla_x \mathbb{P} G \|^2_{H^{s-1}_x L^2_v} + \| \mathbb{P}^\perp G \|^2_{H^s_x L^2_v} ) \,,
	  \end{aligned}
	\end{equation}
	where the H\"older inequality and the decomposition $G = \mathbb{P} G + \mathbb{P}^\perp G$ are utilized. The quantity $\textbf{F}_{132}$ is bounded by
	\begin{equation}
	  \begin{aligned}
	    \textbf{F}_{132} = & \eps^2 \l \nabla_x \times \l \partial^m_x G \cdot \mathsf{q}_1 , \widetilde{\Phi} \r_{L^2_v} , \nabla_x \times \partial^m_x \l E ( G \cdot \mathsf{q}_2 ), \nabla_v \widetilde{\Phi} \r_{L^2_v} \r_{L^2_x} \\
	    \leq & \eps^2 \| \nabla_x \partial^m_x G \|_{L^2_{x,v}} \| \nabla_x \partial^m_x ( E G ) \|_{L^2_{x,v}} \| \mathsf{q}_1 \widetilde{\Phi} \|_{L^2_v} \| \mathsf{q}_2 \nabla_v \widetilde{\Phi} \|_{L^2_v} \\
	    \leq & C \eps^2 \| \nabla_x G \|_{H^{s-1}_x L^2_v} \Big( \| E \|_{L^4_x} \| \nabla_x \partial^m_x G \|_{L^4_x L^2_v} \\
	    & \qquad + \sum_{ 1 \leq |m'| \leq |m|+1 } \| \partial^{m'}_x E \|_{L^2_x} \sum_{|m'| \leq |m|} \| \partial^{m'}_x G \|_{L^\infty_x L^2_v} \Big) \\
	    \leq & C \eps^2 \| \nabla_x G \|_{H^{s-1}_x L^2_v} \| E \|_{H^{s-1}_x} \| G \|_{H^s_x L^2_v} \\
	    \leq & C \eps^2 \| G \|_{H^s_x L^2_v} \| E \|_{H^{s-1}_x} \big( \| \nabla_x \mathbb{P} G \|_{H^{s-1}_x L^2_v} + \| \mathbb{P}^\perp G \|_{H^s_x L^2_v} \big) \,,
	  \end{aligned}
	\end{equation}
	where we make use of the H\"older inequality, the Sobolev embeddings $H^1_x (\T^3) \hookrightarrow L^4_x (\T^3)$, $H^2_x (\T^3) \hookrightarrow L^\infty_x (\T^3)$ and the decomposition $G = \mathbb{P} G + \mathbb{P}^\perp G$. Similarly, the term $\textbf{F}_{133}$ can be controlled by
	\begin{equation}
	  \begin{aligned}
	    \textbf{F}_{133} \leq C \eps^2 \| G \|_{H^s_x L^2_v} \| E \|_{H^{s-1}_x} \big( \| \nabla_x \mathbb{P} G \|_{H^{s-1}_x L^2_v} + \| \mathbb{P}^\perp G \|_{H^s_x L^2_v} \big) \,.
	  \end{aligned}
	\end{equation}
	Furthermore, we estimate the term $\textbf{F}_{134}$ for $|m| \leq s - 2 \ (s \geq 3)$ as
	\begin{equation}
	  \begin{aligned}
	    \textbf{F}_{134} = & \eps \l \nabla_x \times \l \partial^m_x G \cdot \mathsf{q}_1 , \widetilde{\Phi} \r_{L^2_v} , \nabla_x \times \partial^m_x \l ( v \times B ) ( G \cdot \mathsf{q}_2 ), \nabla_v \widetilde{\Phi} \r_{L^2_v} \r_{L^2_x} \\
	    \leq & \eps \| \nabla_x \partial^m_x G \|_{L^2_{x,v}} \| \nabla_x \partial^m_x (B G) \|_{L^2_{x,v}} \| \mathsf{q}_1 \widetilde{\Phi} \|_{L^2_v} \| v \mathsf{q}_2 \nabla_v \widetilde{\Phi} \|_{L^2_v} \\
	    \leq & C \eps \| \nabla_x G \|_{H^{s-1}_x L^2_v} \Big( \| \nabla_x \partial^m_x B \|_{L^2_x} \| G \|_{L^\infty_x L^2_v} \\
	    & \qquad + \sum_{ 1 \leq |m'| \leq |m|+1 } \| \partial^{m'}_x G \|_{L^2_x L^2_v} \sum_{|m'| \leq |m|} \| \partial^{m'}_x B \|_{L^\infty_x} \Big) \\
	    \leq & C \eps \| \nabla_x G \|_{H^{s-1}_x L^2_v} \big( \| \nabla_x B \|_{H^{s-2}_x } \| G \|_{H^s_x L^2_v} + \| \nabla_x G \|_{H^{s-1}_x L^2_v} \| B \|_{H^s_x} \big) \\
	    \leq & C \eps ( \| B \|_{H^s_x} + \| G \|_{H^s_x L^2_v} ) \big( \| \nabla_x B \|^2_{H^{s-2}_x} + \| \nabla_x \mathbb{P} G \|^2_{H^{s-1}_x L^2_v} + \| \mathbb{P}^\perp G \|^2_{H^s_x L^2_v} \big)
	  \end{aligned}
	\end{equation}
	by utilizing the decomposition $G = \mathbb{P} G + \mathbb{P}^\perp G$, the H\"older inequality and the Sobolev embedding $H^2_x ( \T^3 ) \hookrightarrow L^\infty_x (\T^3)$. In order to control the quantity $\textbf{F}_{135}$, we employ Lemma \ref{Lmm-Gamma-Torus} and the decomposition $G = \mathbb{P} G + \mathbb{P}^\perp G$. We thereby get
	  \begin{align}
	    \no \textbf{F}_{135} = & - \eps \l \l \nabla_x \partial^m_x G \times \widetilde{\Phi} , \mathsf{q}_1 \r_{L^2_v} \mathsf{q}_1 , \nabla_x \partial^m_x \Gamma (G,G) \times \widetilde{\Phi} \r_{L^2_{x,v}} \\
	    \no \leq & C \eps \| G \|_{H^s_x L^2_v} \| G \|_{H^s_x L^2_v (\nu)} \big\| \l \nabla_x \partial^m_x G \times \widetilde{\Phi} , \mathsf{q}_1 \r_{L^2_v} \mathsf{q}_1 \widetilde{\Phi} \big\|_{L^2_{x,v}(\nu)} \\
	    \no \leq & C \eps \| G \|_{H^s_x L^2_v} \| G \|_{H^s_x L^2_v (\nu)} \| \nabla_x G \|_{H^{s-1}_x L^2_v} \\
	    \no \leq & C \eps \| G \|_{H^s_x L^2_v} \big( \| \nabla_x \mathbb{P} G \|_{H^{s-1}_x L^2_v} + \| \mathbb{P}^\perp G \|_{H^s_x L^2_v } \big) \\
	    \no & \times \left( \| \nabla_x \mathbb{P} G \|_{H^{s-1}_x L^2_v} + \| \mathbb{P}^\perp G \|_{H^s_x L^2_v (\nu)} + (\rho^+)_{\T^3} + (\rho^-)_{\T^3} + |(u)_{\T^3}| + (\theta)_{\T^3} \right)  \\
	    \leq & C \eps \| G \|_{H^s_x L^2_v} \left( \| \nabla_x \mathbb{P} G \|^2_{H^{s-1}_x L^2_v} + \| \mathbb{P}^\perp G \|^2_{H^s_x L^2_v (\nu)} + (\rho^+)^2_{\T^3} + (\rho^-)^2_{\T^3} + (u)^2_{\T^3} + (\theta)^2_{\T^3} \right) \,,
	  \end{align}
	where the last inequality is implied by the Young's inequality and the part (1) of Lemma \ref{Lmm-CF-nu} or Lemma \ref{Lmm-nu-norm}.
	
	Collecting the estimates on the quantities $\textbf{F}_{13i}$ $(i = 1 , 2, \cdots , 5)$ in the previous, we obtain the bounds of $\textbf{F}_{13}$
	\begin{equation}
	  \begin{aligned}
	    \textbf{F}_{13} \leq & - \tfrac{\eps^2}{2} \tfrac{\d}{\d t} \big\| \nabla_x \times \l \partial^m_x G \cdot \mathsf{q}_1 , \widetilde{\Phi} \r_{L^2_v} \big\|^2_{L^2_x} + C \eps \big( \| \nabla_x \mathbb{P} G \|^2_{H^{s-1}_x L^2_v} + \| \mathbb{P}^\perp G \|^2_{H^s_x L^2_v (\nu)} \big) \\
	    + & C ( \| G \|_{H^s_x L^2_v} + \| B \|_{H^s_x} ) \\
	    & \qquad \times \big( \| E \|^2_{H^{s-1}} + \| \nabla_x B \|^2_{H^{s-2}_x} + \| \nabla_x \mathbb{P} G \|^2_{H^{s-1}_x L^2_v} + \| \mathbb{P}^\perp G \|^2_{H^s_x L^2_v (\nu)} \big) \\
	    + & C \| G \|_{H^s_x L^2_v} \big( (\rho^+)^2_{\T^3} + (\rho^-)^2_{\T^3} + (u)^2_{\T^3} + (\theta)^2_{\T^3} \big)
	  \end{aligned}
	\end{equation}
	holds for all $0 < \eps \leq 1$, where Lemma \ref{Lmm-nu-norm} is utilized. Then we derive from the bounds of $\textbf{F}_{11}$, $\textbf{F}_{12}$, $\textbf{F}_{13}$ and Lemma \ref{Lmm-nu-norm} that the term $\textbf{F}_1$ is controlled by
	\begin{equation}
	  \begin{aligned}
	    \textbf{F}_1 \leq & - \tfrac{\eps^2}{2} \tfrac{\d}{\d t} \big\| \nabla_x \times \l \partial^m_x G \cdot \mathsf{q}_1 , \widetilde{\Phi} \r_{L^2_v} \big\|^2_{L^2_x} - \eps \tfrac{\d}{\d t} \l \nabla_x \times \l \partial^m_x G \cdot \mathsf{q}_1 , \widetilde{\Phi} \r_{L^2_v} , \partial_t \partial^m_x B \r_{L^2_x} \\
	    + & C \eps \big( \| \nabla_x \mathbb{P} G \|^2_{H^{s-1}_x L^2_v} + \| \mathbb{P}^\perp G \|^2_{H^s_x L^2_v (\nu)} \big) \\
	    + & C \eps \big( \| \nabla_x \mathbb{P} G \|_{H^{s-1}_x L^2_v} + \| \mathbb{P}^\perp G \|_{H^s_x L^2_v (\nu)} \big) ( \| \nabla_x \partial^m_x B \|_{L^2_x} + \| \partial_t \partial^m_x B \|_{L^2_x} ) \\
	    + & C ( \| G \|_{H^s_x L^2_v} + \| B \|_{H^s_x} ) \\
	    & \qquad \times \big( \| E \|^2_{H^{s-1}_x} + \| \nabla_x B \|^2_{H^{s-2}_x} + \| \nabla_x \mathbb{P} G \|^2_{H^{s-1}_x L^2_v} + \| \mathbb{P}^\perp G \|^2_{H^s_x L^2_v (\nu)} \big) \\
	    + & C \| G \|_{H^s_x L^2_v} \big( (\rho^+)^2_{\T^3} + (\rho^-)^2_{\T^3} + (u)^2_{\T^3} + (\theta)^2_{\T^3} \big)
	  \end{aligned}
	\end{equation}
	for any $0 < \eps \leq 1$, where we employ Lemma \ref{Lmm-nu-norm}.
	
	For the term $\textbf{F}_2$, we deduce from the H\"older inequality, the the decomposition $G = \mathbb{P} G + \mathbb{P}^\perp G$ and the part (1) of Lemma \ref{Lmm-CF-nu} that
	\begin{equation}
	  \begin{aligned}
	    \textbf{F}_2 \leq & \| \nabla_x ( \nabla_x \partial^m_x G ) \|_{L^2_{x,v}} \| v \mathsf{q}_1 \widetilde{\Phi} \|_{L^2_v} \| \partial_t \partial^m_x B \|_{L^2_x} \\
	    \leq & C \| \nabla_x G \|_{H^{s-1}_x L^2_v} \| \partial_t \partial^m_x B \|_{L^2_x} \\
	    \leq & C ( \| \nabla_x \mathbb{P} G \|_{H^{s-1}_x L^2_v} + \| \mathbb{P}^\perp G \|_{H^s_x L^2_v (\nu)} ) \| \partial_t \partial^m_x B \|_{L^2_x} \,.
	  \end{aligned}
	\end{equation} 
	
	For the terms $\textbf{F}_3$ and $\textbf{F}_4$, it is yielded by the similar arguments in estimating the term $\textbf{F}_{132}$ that
	\begin{equation}
	  \begin{aligned}
	    \textbf{F}_3 + \textbf{F}_4 \leq C \eps \| G \|_{H^s_x L^2_v} \| E \|_{H^{s-1}_x} \| \partial_t \partial^m_x B \|_{L^2_x} \,.
	  \end{aligned}
	\end{equation}
	We also deduce the bound of the term $\textbf{F}_5$
	\begin{equation}
	  \begin{aligned}
	    \textbf{F}_5 \leq & C \eps ( \| B \|_{H^s_x} + \| G \|_{H^s_x L^2_v} ) \big( \| \nabla_x B \|^2_{H^{s-2}_x} + \| \nabla_x \mathbb{P} G \|^2_{H^{s-1}_x L^2_v} + \| \mathbb{P}^\perp G \|^2_{H^s_x L^2_v} \big)
	  \end{aligned}
	\end{equation}
	by employing the analogous arguments in estimating the term $\textbf{F}_{134}$.
	
	For the term $\textbf{F}_6$, by utilizing Lemma \ref{Lmm-Gamma-Torus} and the decomposition $G = \mathbb{P} G + \mathbb{P}^\perp G$, we get
	\begin{equation}
	  \begin{aligned}
	    \textbf{F}_6 \leq & C \| G \|_{H^s_x L^2_v} \| G \|_{H^s_x L^2_v (\nu)} \| \partial_t \partial^m_x B \|_{L^2_x} \| \mathsf{q}_1 \widetilde{\Phi} \|_{L^2_v (\nu)} \\
	    \leq & C \| G \|_{H^s_x L^2_v} \| \partial_t \partial^m_x B \|_{L^2_x} \\
	    & \times \big( \| \nabla_x \mathbb{P} G \|_{H^{s-1}_x L^2_v} + \| \mathbb{P}^\perp G \|_{H^s_x L^2_v} + |(\rho^+)_{\T^3}| + |(\rho^-)_{\T^3}| + |(u)_{\T^3}| + |(\theta)_{\T^3}| \big) \,.
	  \end{aligned}
	\end{equation}
	
	We summarize the estimates on $F_i$ $(i = 1,2, \cdots , 6)$, and obtain
	\begin{equation}\label{K-GEB-1}
	  \begin{aligned}
	    & \l \nabla_x \times \partial^m_x \mathcal{K} ( G,E,B ) , \partial_t \partial^m_x B \r_{L^2_x} \\
	    \leq & - \tfrac{\eps^2}{2} \tfrac{\d}{\d t} \big\| \nabla_x \times \l \partial^m_x G \cdot \mathsf{q}_1 , \widetilde{\Phi} \r_{L^2_v} \big\|^2_{L^2_x} - \eps \tfrac{\d}{\d t} \l \nabla_x \times \l \partial^m_x G \cdot \mathsf{q}_1 , \widetilde{\Phi} \r_{L^2_v} , \partial_t \partial^m_x B \r_{L^2_x} \\
	    + & C \eps \big( \| \nabla_x \mathbb{P} G \|^2_{H^{s-1}_x L^2_v} + \| \mathbb{P}^\perp G \|^2_{H^s_x L^2_v (\nu)} \big) \\
	    + & C \eps \big( \| \nabla_x \mathbb{P} G \|_{H^{s-1}_x L^2_v} + \| \mathbb{P}^\perp G \|_{H^s_x L^2_v (\nu)} \big) ( \| \nabla_x \partial^m_x B \|_{L^2_x} + \| \partial_t \partial^m_x B \|_{L^2_x} ) \\
	    + & C ( \| G \|_{H^s_x L^2_v} + \| B \|_{H^s_x} ) \Big[ \| E \|^2_{H^{s-1}_x} + \| \nabla_x B \|^2_{H^{s-2}_x} \\
	    & \qquad + \| \nabla_x \mathbb{P} G \|^2_{H^{s-1}_x L^2_v} + \| \mathbb{P}^\perp G \|^2_{H^s_x L^2_v (\nu)} \Big] \\
	    + & C \| G \|_{H^s_x L^2_v} \big( (\rho^+)^2_{\T^3} + (\rho^-)^2_{\T^3} + (u)^2_{\T^3} + (\theta)^2_{\T^3} \big) \\
	    + & C \| G \|_{H^s_x L^2_v} \| \partial_t \partial^m_x B \|_{L^2_x} \Big[ \| E \|_{H^{s-1}_x} + \| \nabla_x \mathbb{P} G \|_{H^{s-1}_x L^2_v} \\
	    & \qquad + \| \mathbb{P}^\perp G \|_{H^s_x L^2_v} + |(\rho^+)_{\T^3}| + |(\rho^-)_{\T^3}| + |(u)_{\T^3}| + |(\theta)_{\T^3}| \Big]
	  \end{aligned}
	\end{equation}
	for all $0 < \eps \leq 1$ and any multi-indexes $m \in \mathbb{N}^3$ with $|m| \leq s-2$ $(s \geq 3)$.
	
	Next, we estimate the quantity $ \l  \partial^m_x \mathcal{K} ( G, E, B ) ,  \delta_1 \nabla_x \times \partial^m_x B \r_{L^2_x} $. By the definition of $\mathcal{K} ( G, E, B )$ in \eqref{Ohm-Law-fluid}, we decompose this term as
	\begin{equation}
	  \begin{aligned}
	    & \l  \partial^m_x \mathcal{K} ( G, E, B ) ,  \delta_1 \nabla_x \times \partial^m_x B \r_{L^2_x} \\
	    = & \underset{\textbf{G}_1}{ \underbrace{ - \eps \l \partial_t \l \partial^m_x G \cdot \mathsf{q}_1 , \widetilde{\Phi} \r_{L^2_v} , \delta_1 \nabla_x \times \partial^m_x B \r_{L^2_x} }} \\
	    & \underset{\textbf{G}_2}{ \underbrace{ - \l \l v \cdot \nabla_x ( \partial^m_x G \cdot \mathsf{q}_1 ) , \widetilde{\Phi} \r_{L^2_v} , \delta_1 \nabla_x \times \partial^m_x B \r_{L^2_x} }} \\
	    & \underset{\textbf{G}_3}{ \underbrace{ - \eps \l \partial^m_x \l E \cdot \nabla_v ( G \cdot \mathsf{q}_2 ), \widetilde{\Phi} \r_{L^2_v} , \delta_1 \nabla_x \times \partial^m_x B \r_{L^2_x} }} \\
	    & + \underset{\textbf{G}_4}{ \underbrace{ \tfrac{1}{2} \eps \l \partial^m_x \l (E \cdot v ) ( G \cdot \mathsf{q}_2 ), \widetilde{\Phi} \r_{L^2_v} , \delta_1 \nabla_x \times \partial^m_x B \r_{L^2_x} }} \\
	    & \underset{\textbf{G}_5}{ \underbrace{ - \l \partial^m_x \l (v \times B) \cdot \nabla_v ( G \cdot \mathsf{q}_2 ), \widetilde{\Phi} \r_{L^2_v} , \delta_1 \nabla_x \times \partial^m_x B \r_{L^2_x} }} \\
	    & + \underset{\textbf{G}_6}{ \underbrace{ \l \partial^m_x \l \Gamma ( G , G ) \cdot \mathsf{q}_2 , \widetilde{\Phi} \r_{L^2_v} , \delta_1 \nabla_x \times \partial^m_x B \r_{L^2_x} }} \,. 
	  \end{aligned}
	\end{equation}
	For the term $\textbf{G}_1$, we have
	\begin{equation}
	  \begin{aligned}
	    \textbf{G}_1 = & - \eps \tfrac{\d}{\d t} \l \nabla_x \times \l \partial^m_x G \cdot \mathsf{q}_1 , \widetilde{\Phi} \r_{L^2_v} , \delta_1 \partial^m_x B \r_{L^2_x} \\
	    & +  \eps \l \nabla_x \times \l \partial^m_x G \cdot \mathsf{q}_1 , \widetilde{\Phi} \r_{L^2_v} , \delta_1 \partial_t \partial^m_x B \r_{L^2_x} \,,
	  \end{aligned}
	\end{equation}
	where the last term can be controlled by
	\begin{equation}\label{Bnd-G1-1}
	  \begin{aligned}
	    & \eps \l \nabla_x \times \l \partial^m_x G \cdot \mathsf{q}_1 , \widetilde{\Phi} \r_{L^2_v} , \delta_1 \partial_t \partial^m_x B \r_{L^2_x} \\
	    \leq & \eps \delta_1 \| \nabla_x \partial^m_x G \|_{L^2_{x,v}} \| \mathsf{q}_1 \widetilde{\Phi} \|_{L^2_v} \| \partial_t \partial^m_x B \|_{L^2_x} \\
	    \leq & C \eps \delta_1 \| \partial_t \partial^m_x B \|_{L^2_x} \big( \| \nabla_x \mathbb{P} G \|_{H^{s-1}_x L^2_v} + \| \mathbb{P} G \|_{H^s_x L^2_v} \big) \,.
	  \end{aligned}
	\end{equation}
	Here we make use of the H\"older inequality and the decomposition $G = \mathbb{P} G + \mathbb{P}^\perp G$. Then, we obtain the bound of $\textbf{G}_1$
	\begin{equation}
	  \begin{aligned}
	    \textbf{G}_1 \leq & - \eps \tfrac{\d}{\d t} \l \nabla_x \times \l \partial^m_x G \cdot \mathsf{q}_1 , \widetilde{\Phi} \r_{L^2_v} , \delta_1 \partial^m_x B \r_{L^2_x} \\
	    + & C \eps \delta_1 \| \partial_t \partial^m_x B \|_{L^2_x} \big( \| \nabla_x \mathbb{P} G \|_{H^{s-1}_x L^2_v} + \| \mathbb{P} G \|_{H^s_x L^2_v} \big)
	  \end{aligned}
	\end{equation}
	for all $|m| \leq s -2$ $(s \geq 3)$ and for small $\delta_1 \in ( 0 , 1 ]$ to be determined. Similar calculations on the inequality \eqref{Bnd-G1-1} reduce to
	\begin{equation}
	  \begin{aligned}
	    \textbf{G}_2 \leq C \eps \delta_1 \| \nabla_x \partial^m_x B \|_{L^2_x} \big( \| \nabla_x \mathbb{P} G \|_{H^{s-1}_x L^2_v} + \| \mathbb{P} G \|_{H^s_x L^2_v} \big) \,.
	  \end{aligned}
	\end{equation}
	Via the analogous arguments of the estimates on the term $\textbf{F}_{132}$, $\textbf{F}_{133}$ and $\textbf{F}_{134}$, one can get the bounds
	\begin{equation}
	  \begin{aligned}
	    \textbf{G}_3 + \textbf{G}_4 \leq C \eps \delta_1 \| G \|_{H^s_x L^2_v} \| E \|_{H^{s-1}_x} \| \nabla_x \partial^m_x B \|_{L^2_x}
	  \end{aligned}
	\end{equation}
	and
	\begin{equation}
	  \begin{aligned}
	    \textbf{G}_5 \leq C \eps \delta_1 ( \| B \|_{H^s_x} + \| G \|_{H^s_x L^2_v} ) \big( \| \nabla_x B \|^2_{H^{s-2}_x} + \| \nabla_x \mathbb{P} G \|^2_{H^{s-1}_x L^2_v} + \| \mathbb{P}^\perp G \|^2_{H^s_x L^2_v} \big) \,.
	  \end{aligned}
	\end{equation}
	We also deduce from the same operations as estimating the term $\textbf{F}_6$ that
	\begin{equation}
	  \begin{aligned}
	     \textbf{G}_6 \leq & C \delta_1 \| G \|_{H^s_x L^2_v} \| \nabla_x \partial^m_x B \|_{L^2_x} \\
	     & \times \big( \| \nabla_x \mathbb{P} G \|_{H^{s-1}_x L^2_v} + \| \mathbb{P}^\perp G \|_{H^s_x L^2_v} + |(\rho^+)_{\T^3}| + |(\rho^-)_{\T^3}| + |(u)_{\T^3}| + |(\theta)_{\T^3}| \big) \,.
	  \end{aligned}
	\end{equation}
	Via the summarization of the bounds $\textbf{G}_i$ $(i = 1, 2, \cdots, 6 )$, we derive
	\begin{equation}\label{K-GEB-2}
	  \begin{aligned}
	    & \l  \partial^m_x \mathcal{K} ( G, E, B ) ,  \delta_1 \nabla_x \times \partial^m_x B \r_{L^2_x} \leq - \eps \tfrac{\d}{\d t} \l \nabla_x \times \l \partial^m_x G \cdot \mathsf{q}_1 , \widetilde{\Phi} \r_{L^2_v} , \delta_1 \partial^m_x B \r_{L^2_x} \\
	    + & C \eps \delta_1 ( \| \partial_t \partial^m_x B \|_{L^2_x} + \| \nabla_x \partial^m_x B \|_{L^2_x} ) ( \| \nabla_x \mathbb{P} G \|_{H^{s-1}_x L^2_v} + \| \mathbb{P}^\perp G \|_{H^s_x L^2_v} ) \\
	    + & C \delta_1 ( \| G \|_{H^s_x L^2_v} + \| B \|_{H^s_x} ) ( \| \nabla_x B \|^2_{H^{s-2}_x} + \| E \|^2_{H^{s-1}_x} + \| \nabla_x \mathbb{P} G \|^2_{H^{s-1}_x L^2_v} + \| \mathbb{P}^\perp G \|^2_{H^s_x L^2_v} ) \\
	    + & C \delta_1 \| G \|_{H^s_x L^2_v} \big( (\rho^+)^2_{\T^3} + ( \rho^- )^2_{\T^3} + ( u )^2_{\T^3} + ( \theta )^2_{\T^3} \big)
	  \end{aligned}
	\end{equation}
	for all $0 < \eps \leq 1$ and for small $\delta_1 \in (0 , 1 ]$ to be determined.
	
	Consequently, via substituting the bounds \eqref{K-GEB-1} and \eqref{K-GEB-2} into the equality \eqref{B-energy-equ}, summing up for $|m| \leq s - 2$ $(s \geq 3)$ and utilizing the Young's inequality, we yield that
	\begin{equation}\label{B-decay}
	  \begin{aligned}
	    & \tfrac{1}{2} \tfrac{\d}{\d t} \big( \delta_1 \| \partial_t B + B \|^2_{H^{s-2}_x} + ( 1 - \delta_1 ) \| \partial_t B \|^2_{H^{s-2}_x} + \| \nabla_x B \|^2_{H^{s-2}_x} + \delta_1 ( \sigma - 1 ) \| B \|^2_{H^{s-2}_x} \big) \\
	    & + ( \sigma - \tfrac{3 \delta_1}{2} ) \| \partial_t B \|^2_{H^{s-2}_x} + \tfrac{\delta_1}{2} \| \nabla_x B \|^2_{H^{s-2}_x} \\
	    \leq & - \eps \frac{\d}{\d t} \sum_{|m| \leq s - 2} \Big[ \tfrac{\eps}{2} \big\| \nabla_x \times \l \partial^m_x G \cdot \mathsf{q}_1 , \widetilde{\Phi} \r_{L^2_v} \big\|^2_{L^2_x} \\
	    & \qquad \qquad \qquad \qquad + \Big\langle \nabla_x \times \l \partial^m_x G \cdot \mathsf{q}_1 , \widetilde{\Phi} \r_{L^2_v} , \partial_t \partial^m_x B + \delta_1 \partial^m_x B \Big\rangle_{L^2_x} \Big] \\
	    + & C ( \| G \|_{H^s_x L^2_v} + \| B \|_{H^s_x} ) \big[ \| \partial_t B \|^2_{H^{s-2}_x} + \| \nabla_x B \|^2_{H^{s-2}_x} \\
	    & \quad + \| E \|^2_{H^{s-1}_x} + \| \nabla_x \mathbb{P} G \|^2_{H^{s-1}_x L^2_v} + \| \mathbb{P}^\perp G \|^2_{H^s_x L^2_v (\nu)} \big] \\
	    + & \tfrac{C \eps}{\delta_1} ( \| \nabla_x \mathbb{P} G \|^2_{H^{s-1}_x L^2_v} + \| \mathbb{P}^\perp G \|^2_{H^s_x L^2_v (\nu)} ) + C \| G \|_{H^s_x L^2_v} \big( (\rho^+)^2_{\T^3} + ( \rho^- )^2_{\T^3} + ( u )^2_{\T^3} + ( \theta )^2_{\T^3} \big)
	  \end{aligned}
	\end{equation}
	for all $0 < \eps \leq 1$ and for small $\delta_1 \in (0 , 1 ]$ to be determined. \\
	
	{\em Step 2. Derivations of the dissipative structure $\| E \|^2_{H^{s-1}_x}$.} We take the derivative operator $\partial^m_x$ on the first two equation of \eqref{Mxw-E-Decay}, multiply by $\partial^m_x E$ and $\partial^m_x B$ respectively, and integrate by parts over $x \in \T^3$ for all $|m| \leq s - 1$ $(s \geq 3)$. We then deduce that
	\begin{equation}\label{E-decay-higer-derivative-equ}
	  \begin{aligned}
	    & \tfrac{1}{2} \tfrac{\d}{\d t} \big( \| \partial^m_x E \|^2_{L^2_x} + \| \partial^m_x B \|^2_{L^2_x} \big) + \sigma \| \partial^m_x E \|^2_{L^2_x} \\
	    = & \l \nabla_x \times \partial^m_x B , \partial^m_x E \r_{L^2_x} - \l \nabla_x \times \partial^m_x E , \partial^m_x B \r_{L^2_x} - \l \partial^m_x \mathcal{K} ( G, E, B ) , \partial^m_x E \r_{L^2_x} \\
	    = & - \l \partial^m_x \mathcal{K} ( G, E, B ) , \partial^m_x E \r_{L^2_x} \,,
	  \end{aligned}
	\end{equation}
	where we make use of the cancellation $ \l \nabla_x \times \partial^m_x B , \partial^m_x E \r_{L^2_x} - \l \nabla_x \times \partial^m_x E , \partial^m_x B \r_{L^2_x} = 0 $. So, we only need to control the term $ - \l \partial^m_x \mathcal{K} ( G, E, B ) , \partial^m_x E \r_{L^2_x} $ for all multi-indexes $m \in \mathbb{N}^3$ with $|m| \leq s - 1$ $(s \geq 3)$.  
	
	Recalling the definition of $\mathcal{K} ( G, E, B)$ in \eqref{Ohm-Law-fluid}, we have
	\begin{equation}\label{K-E}
	  \begin{aligned}
	    & - \l \partial^m_x \mathcal{K} ( G, E, B ) , \partial^m_x E \r_{L^2_x} \\
	    = & \underset{ \textbf{H}_1}{ \underbrace{ \eps \l \partial_t \l \partial^m_x G \cdot \mathsf{q}_1 , \widetilde{\Phi} \r_{L^2_v} , \partial^m_x E \r_{L^2_x} }} + \underset{ \textbf{H}_2}{ \underbrace{ \l \l v \cdot \nabla_x ( \partial^m_x G \cdot \mathsf{q}_1 ) , \widetilde{\Phi} \r_{L^2_v} , \partial^m_x E \r_{L^2_x} }}\\
	    & + \underset{ \textbf{H}_3}{ \underbrace{ \eps \l \partial^m_x \l E \cdot \nabla_v ( G \cdot \mathsf{q}_2 ) , \widetilde{\Phi} \r_{L^2_v} , \partial^m_x E \r_{L^2_x} }} \ \underset{ \textbf{H}_4}{ \underbrace{ - \tfrac{1}{2} \eps  \l \partial^m_x \l ( E \cdot v )  ( G \cdot \mathsf{q}_2 ) , \widetilde{\Phi} \r_{L^2_v} , \partial^m_x E \r_{L^2_x} }} \\
	    & + \underset{ \textbf{H}_5}{ \underbrace{ \eps \l \partial^m_x \l ( v \times B ) \cdot \nabla_v ( G \cdot \mathsf{q}_2 ) , \widetilde{\Phi} \r_{L^2_v} , \partial^m_x E \r_{L^2_x} }} \ \underset{ \textbf{H}_6}{ \underbrace{ - \l \partial^m_x \l \Gamma (G,G) \cdot \mathsf{q}_1 , \widetilde{\Phi} \r_{L^2_v} , \partial^m_x E \r_{L^2_x} }} \,.
	  \end{aligned}
	\end{equation}
	For the term $\textbf{H}_1$, we deduce from the first equation of \eqref{Mxw-E-Decay} and the definition of $\mathcal{K} (G, E, B)$ in \eqref{Ohm-Law-fluid} that
	\begin{equation}
	  \begin{aligned}
	    \textbf{H}_1 = & \eps \tfrac{\d}{\d t} \l \partial^m_x G \cdot \mathsf{q}_1 , \widetilde{\Phi} \cdot \partial^m_x E \r_{L^2_{x,v}} - \eps \l \l \partial^m_x G \cdot \mathsf{q}_1 , \widetilde{\Phi} \r_{L^2_v} , \partial_t \partial^m_x E \r_{L^2_x} \\
	    = & \eps \tfrac{\d}{\d t} \l \partial^m_x G \cdot \mathsf{q}_1 , \widetilde{\Phi} \cdot \partial^m_x E \r_{L^2_{x,v}} \underset{\textbf{H}_{11}}{ \underbrace{ - \eps \l \l \partial^m_x G \cdot \mathsf{q}_1 , \widetilde{\Phi} \r_{L^2_v} , \nabla_x \times \partial^m_x B \r_{L^2_x} }} \\
	    + & \underset{\textbf{H}_{12}}{ \underbrace{  \eps \l \l \partial^m_x G \cdot \mathsf{q}_1 , \widetilde{\Phi} \r_{L^2_v} , \sigma \partial^m_x E \r_{L^2_x} }} + \underset{\textbf{H}_{13}}{ \underbrace{  \eps \l \l \partial^m_x G \cdot \mathsf{q}_1 , \widetilde{\Phi} \r_{L^2_v} , \partial^m_x \mathcal{K} ( G, E, B ) \r_{L^2_x} }} \,.
	  \end{aligned}
	\end{equation}
	If $|m| \leq s - 2$, the H\"older inequality, the Poincar\'e inequality and the decomposition $G = \mathbb{P} G + \mathbb{P}^\perp G$ reduce to
	\begin{equation}
	  \begin{aligned}
	    \textbf{H}_{11} \leq & \eps \| \partial^m_x G \|_{L^2_{x,v}} \| \mathsf{q}_1 \widetilde{\Phi} \|_{L^2_v} \| \nabla_x \times \partial^m_x B \|_{L^2_x} \\
	    \leq & C \eps \| G \|_{H^s_x L^2_v} \| \nabla_x B \|_{H^{s-1}_x} \\
	    \leq & C \eps \| \nabla_x B \|_{H^{s-2}_x} \big( \| \nabla_x \mathbb{P} G \|_{H^{s-1}_x L^2_v} + \| \mathbb{P}^\perp G \|_{H^s_x L^2_v} \big) \\ 
	    & +  C \eps \| \nabla_x B \|_{H^{s-2}_x} \big(  |(\rho^+)_{\T^2}| + |(\rho^-)_{\T^2}| + |(u)_{\T^2}| + |(\theta)_{\T^2}| \big) \,.
	  \end{aligned}
	\end{equation}
	If $|m| = s -1$, by, additionally, integrating by parts over $x \in \T^3$, we have
	\begin{equation}
	  \begin{aligned}
	    \textbf{H}_{11} \leq & C \eps \| \nabla_x B \|_{H^{s-2}_x} \big( \| \nabla_x \mathbb{P} G \|_{H^{s-1}_x L^2_v} + \| \mathbb{P}^\perp G \|_{H^s_x L^2_v} \big)  \,.
	  \end{aligned}
	\end{equation} 
	In summary, for all multi-indexes $m \in \mathbb{N}^3$ with $|m| \leq s - 1$, the following bound holds:
	\begin{equation}
	  \begin{aligned}
	    \textbf{H}_{11} \leq &  C \eps \| \nabla_x B \|_{H^{s-2}_x} \big( \| \nabla_x \mathbb{P} G \|_{H^{s-1}_x L^2_v} + \| \mathbb{P}^\perp G \|_{H^s_x L^2_v} \big) \\ 
	    & +  C \eps \| \nabla_x B \|_{H^{s-2}_x} \big(  |(\rho^+)_{\T^2}| + |(\rho^-)_{\T^2}| + |(u)_{\T^2}| + |(\theta)_{\T^2}| \big) \,.
	  \end{aligned}
	\end{equation}
	Via the similar calculations of $\textbf{H}_{11}$ in the case $|m| \leq s - 2$, we have
	\begin{equation}
	  \begin{aligned}
	    \textbf{H}_{12} \leq & C \sigma \eps \| \partial^m_x E \|_{L^2_x} \big( \| \nabla_x \mathbb{P} G \|_{H^{s-1}_x L^2_v} + \| \mathbb{P}^\perp G \|_{H^s_x L^2_v} \big) \\ 
	    & +  C \sigma \eps \| \partial^m_x E \|_{L^2_x} \big(  |(\rho^+)_{\T^2}| + |(\rho^-)_{\T^2}| + |(u)_{\T^2}| + |(\theta)_{\T^2}| \big) 
	  \end{aligned}
	\end{equation}
	holds for any $|m| \leq s - 1$. 
	
	Again from utilizing the definition of $\mathcal{K} (G, E, B)$ in \eqref{Ohm-Law-fluid}, we derive
	\begin{equation}
	  \begin{aligned}
	    \textbf{H}_{13} = & - \tfrac{\eps^2}{2} \tfrac{\d}{\d t} \big\| \l \partial^m_x G \cdot \mathsf{q}_1 , \widetilde{\Phi} \r_{L^2_v} \big\|^2_{L^2_x} \\
	    & \underset{\textbf{H}_{131}}{ \underbrace{ - \eps \l \l \partial^m_x G \cdot \mathsf{q}_1 , \widetilde{\Phi} \r_{L^2_v} , \l v \cdot \nabla_x ( \partial^m_x G \cdot \mathsf{q}_1 ) , \widetilde{\Phi} \r_{L^2_v} \r_{L^2_x} }} \\
	    & \underset{\textbf{H}_{132}}{ \underbrace{  - \eps^2 \l \l \partial^m_x G \cdot \mathsf{q}_1 , \widetilde{\Phi} \r_{L^2_v} , \partial^m_x \l E \cdot \nabla_v ( G \cdot \mathsf{q}_2 ) , \widetilde{\Phi} \r_{L^2_v} \r_{L^2_x} }} \\
	    & + \underset{\textbf{H}_{133}}{ \underbrace{  \tfrac{1}{2} \eps^2 \l \l \partial^m_x G \cdot \mathsf{q}_1 , \widetilde{\Phi} \r_{L^2_v} , \partial^m_x \l (E \cdot v) ( G \cdot \mathsf{q}_2 ) , \widetilde{\Phi} \r_{L^2_v} \r_{L^2_x} }} \\
	    & \underset{\textbf{H}_{134}}{ \underbrace{ - \eps \l \l \partial^m_x G \cdot \mathsf{q}_1 , \widetilde{\Phi} \r_{L^2_v} , \partial^m_x \l ( v \times B ) \cdot \nabla_v ( G \cdot \mathsf{q}_2 ) , \widetilde{\Phi} \r_{L^2_v} \r_{L^2_x} }} \\
	    & + \underset{\textbf{H}_{135}}{ \underbrace{  \eps \l \l \partial^m_x G \cdot \mathsf{q}_1 , \widetilde{\Phi} \r_{L^2_v} , \partial^m_x \l \Gamma (G,G) \cdot \mathsf{q}_1 , \widetilde{\Phi} \r_{L^2_v} \r_{L^2_x} }} \,.
	  \end{aligned}
	\end{equation}
	The H\"older inequality reduces to
	\begin{equation}
	  \begin{aligned}
	    \textbf{H}_{131} \leq & C \eps \| \nabla_x G \|_{H^{s-1}_x L^2_v} \| G \|_{H^s_x L^2_v} \\
	    \leq & C \eps \big( \| \nabla_x \mathbb{P} G \|^2_{H^{s-1}_x L^2_v} + \| \mathbb{P}^\perp G \|^2_{H^s_x L^2_v} + (\rho^+)^2_{\T^3} + (\rho^-)^2_{\T^3} + (u)^2_{\T^3} + (\theta)^2_{\T^3} \big) \,,
	  \end{aligned}
	\end{equation}
	where the last inequality is implied by making use of the decomposition $G = \mathbb{P} G + \mathbb{P}^\perp G$, the Poincar\'e inequality and the Young's inequality. 	Via the analogous arguments of the estimates on the term $\textbf{F}_{132}$ and $\textbf{F}_{133}$, we can control the terms $\textbf{H}_{132}$ and $\textbf{H}_{133}$ in the case $|m| \leq s - 1$ with $m \neq 0$
	\begin{equation}
	  \begin{aligned}
	    \textbf{H}_{132} + \textbf{H}_{133} \leq C \eps^2 \| G \|_{H^s_x L^2_v} \| E \|_{H^{s-1}_x} \big( \| \nabla_x \mathbb{P} G \|_{H^{s-1}_x L^2_v} + \| \mathbb{P}^\perp G \|_{H^s_x L^2_v} \big) \,.
	  \end{aligned}
	\end{equation}
	If $m = 0$, the quantity $ \textbf{H}_{132} + \textbf{H}_{133} $ can be estimated as
	\begin{equation}
	  \begin{aligned}
	    \textbf{H}_{132} + \textbf{H}_{133} = & \eps^2 \l \l G \cdot \mathsf{q}_1 , \widetilde{\Phi} \r_{L^2_v} , \l E \cdot \nabla_v ( \tfrac{  \widetilde{\Phi} }{ \sqrt{M} } ) \sqrt{M} , G \cdot \mathsf{q}_2 \r_{L^2_v} \r_{L^2_x} \\
	    \leq & \eps^2 \| G \|^2_{L^4_x L^2_v} \| E \|_{L^2_x} \| \mathsf{q}_1 \widetilde{\Phi} \|_{L^2_v} \big\| \mathsf{q}_2 \nabla_v ( \tfrac{  \widetilde{\Phi} }{ \sqrt{M} } ) \sqrt{M} \big\|_{L^2_v} \\
	    \leq & C \eps^2 \| G \|^\frac{1}{2}_{L^2_x L^2_v} \| \nabla_x G \|^\frac{3}{4}_{L^2_x L^2_v} \| E \|_{L^2_x} \\
	    \leq & C \eps^2 \| G \|_{H^1_x L^2_v} \| E \|_{L^2_x} \big( \| \nabla_x \mathbb{P} G \|_{L^2_{x,v}} + \| \nabla_x \mathbb{P}^\perp G \|_{L^2_{x,v}} \big) \\
	    \leq & C \eps^2 \| G \|_{H^s_x L^2_v} \| E \|_{H^{s-1}_x} \big( \| \nabla_x \mathbb{P} G \|_{H^{s-1}_x L^2_v} + \| \mathbb{P}^\perp G \|_{H^s_x L^2_v} \big) \,,
	  \end{aligned}
	\end{equation}
	where we utilize the Sobolev interpolation inequality $\| f \|_{L^4_x (\T^3)} \leq C \| f \|^\frac{1}{4}_{L^2_x (\T^3)} \| \nabla_x f \|^\frac{3}{4}_{L^2_x (\T^3)}$. We thereby have
	\begin{equation}
	  \begin{aligned}
	    \textbf{H}_{132} + \textbf{H}_{133} \leq C \eps^2 \| G \|_{H^s_x L^2_v} \| E \|_{H^{s-1}_x} \big( \| \nabla_x \mathbb{P} G \|_{H^{s-1}_x L^2_v} + \| \mathbb{P}^\perp G \|_{H^s_x L^2_v} \big)
	  \end{aligned}
	\end{equation}
	for all $|m| \leq  s - 1$. Similarly, one can easily estimate 
	\begin{equation}
	  \begin{aligned}
	    \textbf{H}_{134} \leq C \eps \| B \|_{H^s_x} \big( \| \nabla_x \mathbb{P} G \|^2_{H^{s-1}_x L^2_v} + \| \mathbb{P}^\perp G \|^2_{H^s_x L^2_v} + (\rho^+)^2_{\T^3} + (\rho^-)^2_{\T^3} + (u)^2_{\T^3} + (\theta)^2_{\T^3} \big)
	  \end{aligned}
	\end{equation}
	for all $|m| \leq s - 1$. For the term $\textbf{H}_{135}$, we derive from Lemma \ref{Lmm-Gamma-Torus}, Lemma \ref{Lmm-nu-norm}, the decomposition of $G = \mathbb{P} G + \mathbb{P}^\perp G$ and the Poincar\'e inequality that
	\begin{equation}
	  \begin{aligned}
	    \textbf{H}_{135} \leq & C \eps \| G \|_{H^s_x L^2_v} \big( \| \nabla_x \mathbb{P} G \|^2_{H^{s-1}_x L^2_v} + \| \mathbb{P}^\perp G \|^2_{H^s_x L^2_v} + (\rho^+)^2_{\T^3} + (\rho^-)^2_{\T^3} + (u)^2_{\T^3} + (\theta)^2_{\T^3} \big)
	  \end{aligned}
	\end{equation}
	for any $|m| \leq s - 1$. Collecting the all estimates in the previous, we obtain
	\begin{equation}
	  \begin{aligned}
	    \textbf{H}_{13} \leq & - \tfrac{\eps^2}{2} \tfrac{\d}{\d t} \big\| \l \partial^m_x G \cdot \mathsf{q}_1 , \widetilde{\Phi} \r_{L^2_v} \big\|^2_{L^2_x} \\
	    + & C \eps \big( \| \nabla_x \mathbb{P} G \|^2_{H^{s-1}_x L^2_v} + \| \mathbb{P}^\perp G \|^2_{H^s_x L^2_v} + (\rho^+)^2_{\T^3} + (\rho^-)^2_{\T^3} + (u)^2_{\T^3} + (\theta)^2_{\T^3} \big) \\
	    + & C \eps ( \| G \|_{H^s_x L^2_v} + \| B \|_{H^s_x} ) \big( \| \nabla_x \mathbb{P} G \|^2_{H^{s-1}_x L^2_v} + \| \mathbb{P}^\perp G \|^2_{H^s_x L^2_v (\nu)} + \| E \|^2_{H^{s-1}_x} \big) \\
	    + & C \eps ( \| G \|_{H^s_x L^2_v} + \| B \|_{H^s_x} ) \big( (\rho^+)^2_{\T^3} + (\rho^-)^2_{\T^3} + (u)^2_{\T^3} + (\theta)^2_{\T^3} \big) \,,
	  \end{aligned}
	\end{equation}
	where we use the Young's inequality and Lemma \ref{Lmm-nu-norm}.
	
	We plug the estimates on the bounds of quantities $\textbf{H}_{11}$, $\textbf{H}_{12}$ and $\textbf{H}_{13}$ into the expression of $\textbf{H}_1$, which then gives us
	\begin{equation}\label{Bnd-H1}
	  \begin{aligned}
	    \textbf{H}_1 \leq & - \eps \tfrac{\d}{\d t} \Big[ \tfrac{\eps}{2} \big\| \l \partial^m_x G \cdot \mathsf{q}_1 , \widetilde{\Phi} \r_{L^2_v} \big\|^2_{L^2_x} - \l \partial^m_x G \cdot \mathsf{q}_1 , \widetilde{\Phi} \cdot \partial^m_x E \r_{L^2_{x,v}} \Big] \\
	    + & C \eps \big( \| \nabla_x B \|^2_{H^{s-2}_x} + \| \nabla_x \mathbb{P} G \|^2_{H^{s-1}_x L^2_v} + \| \mathbb{P}^\perp G \|^2_{H^s_x L^2_v (\nu)} \big) \\
	    + & C \eps \big( (\rho^+)^2_{\T^3} + (\rho^-)^2_{\T^3} + (u)^2_{\T^3} + (\theta)^2_{\T^3} \big) + \tfrac{\sigma}{4} \| \partial^m_x E \|^2_{L^2_x} \\
	    + & C \eps ( \| G \|_{H^s_x L^2_v} + \| B \|_{H^s_x} ) \big( \| \nabla_x \mathbb{P} G \|^2_{H^{s-1}_x L^2_v} + \| \mathbb{P}^\perp G \|^2_{H^s_x L^2_v (\nu)} + \| E \|^2_{H^{s-1}_x} \big) \\
	    + & C \eps ( \| G \|_{H^s_x L^2_v} + \| B \|_{H^s_x} ) \big( (\rho^+)^2_{\T^3} + (\rho^-)^2_{\T^3} + (u)^2_{\T^3} + (\theta)^2_{\T^3} \big) 
	  \end{aligned}
	\end{equation}
	for all $0 < \eps \leq 1$ and $|m| \leq s - 1$. Here we also make use of the Young's inequality and Lemma \ref{Lmm-nu-norm}.
	
	From the decomposition $G = \mathbb{P} G + \mathbb{P}^\perp G$ and the H\"older inequality, we deduce that
	\begin{equation}\label{Bnd-H2}
	  \begin{aligned}
	    \textbf{H}_2 \leq & \| \partial^m_x E \|_{L^2_x} \| \nabla_x \partial^m_x G \|^2_{L^2_{x,v}} \| v \mathsf{q}_1 \widetilde{\Phi} \|_{L^2_v} \\
	    \leq & C \| \partial^m_x E \|_{L^2_x} \big( \| \nabla_x \mathbb{P} G \|_{H^{s-1}_x L^2_v} + \| \mathbb{P}^\perp G \|_{H^s_x L^2_v (\nu)} \big) \\
	    \leq & \tfrac{\sigma}{4} \| \partial^m_x E \|^2_{L^2_x} + C \big( \| \nabla_x \mathbb{P} G \|^2_{H^{s-1}_x L^2_v} + \| \mathbb{P}^\perp G \|^2_{H^s_x L^2_v (\nu)} \big)
	  \end{aligned}
	\end{equation}
	for all $|m| \leq s - 1$. Similarly, we have
	\begin{equation}\label{Bnd-H3-H4}
	  \begin{aligned}
	    \textbf{H}_3 + \textbf{H}_4 = & - \eps \l \partial^m_x \l E \cdot \nabla_v ( \tfrac{\widetilde{\Phi}}{\sqrt{M}} ) \sqrt{M} , G \cdot \mathsf{q}_2 \r_{L^2_v} , \partial^m_x E \r_{L^2_x} \\
	    \leq & C \eps \| \partial^m_x E \|_{L^2_x} \| E \|_{H^{s-1}_x} \| G \|_{H^s_x L^2_v} \\
	    \leq & C \eps \| G \|_{H^s_x L^2_v} \| E \|^2_{H^{s-1}_x} \,.
	  \end{aligned}
	\end{equation}
	For the term $\textbf{H}_5$, we have
	\begin{equation}\label{Bnd-H5}
	  \begin{aligned}
	    \textbf{H}_5 \leq & C \eps \sum_{0 \neq m' \leq m} \| \partial^{m'}_x B \|_{L^2_x} \| \partial^{m - m'}_x G \|_{L^\infty_x L^2_v} \| \partial^m_x E \|_{L^2_v} \\
	    & + C \eps \| B \|_{L^4_x} \| \partial^m_x G \|_{L^4_x L^2_v} \| \partial^m_x E \|_{L^2_x} \\
	    \leq & C \eps \| G \|_{H^s_x L^2_v} \| \nabla_x B \|_{H^{s-2}_x} \| E \|_{H^{s-1}_x} \\
	    & + C \eps \| \nabla_x B \|^\frac{3}{4}_{L^2_x} \| B \|^\frac{1}{4}_{L^2_x} \| \nabla_x \partial^m_x G \|^\frac{3}{4}_{L^2_{x,v}} \| G \|^\frac{1}{4}_{L^2_{x,v}} \| \partial^m_x E \|_{L^2_x} \\
	    \leq & C \eps ( \| G \|_{H^s_x L^2_v} + \| B \|_{H^s_x} ) \| E \|_{H^{s-1}_x} \big( \| \nabla_x B \|_{H^{s-2}_x} + \| \nabla_x \mathbb{P} G \|_{H^{s-1}_x L^2_v} + \| \mathbb{P}^\perp G \|_{H^s_x L^2_v} \big)
	  \end{aligned}
	\end{equation}
	for all $|m| \leq s - 1$, where we utilize the H\"older inequality, the Young's inequality, the Sobolev embedding $H^2_x (\T^3) \hookrightarrow L^\infty_x (\T^3)$ and the Sobolev interpolation inequality $\| f \|_{L^4_x (\T^3)} \leq C \| f \|^\frac{1}{4}_{L^2_x (\T^3)} \| \nabla_x f \|^\frac{3}{4}_{L^2_x (\T^3)}$. For the term $\textbf{H}_6$, it is easily derived from Lemma \ref{Lmm-Gamma-Torus}, the decomposition $G = \mathbb{P} G + \mathbb{P}^\perp G$ and the Young's inequality that
	\begin{equation}\label{Bnd-H6}
	  \begin{aligned}
	    \textbf{H}_6 \leq & C \| G \|_{H^s_x L^2_v} \| G \|_{H^s_x L^2_v (\nu)} \| \partial^m_x E \|_{L^2_x} \| \mathsf{q}_1 \widetilde{\Phi} \|_{L^2_v(\nu)} \\
	    \leq & C \| G \|_{H^s_x L^2_v} \| E \|_{H^{s-1}_x} \Big( \| \nabla_x \mathbb{P} G \|_{H^{s-1}_x L^2_v} + \| \mathbb{P}^\perp G \|_{H^s_x L^2_v} + \sum_{ \bm{\gamma} \in \{ \rho^\pm , u, \theta \} } | ( \bm{\gamma} )_{\T^3} | \Big) \\
	    \leq & C \| G \|_{H^s_x L^2_v} \big( \| E \|^2_{H^{s-1}} + \| \nabla_x \mathbb{P} G \|^2_{H^{s-1}_x L^2_v} + \| \mathbb{P}^\perp G \|^2_{H^s_x L^2_v (\nu)} \big) \\
	    & + C \| G \|_{H^s_x L^2_v} \big( (\rho^+)^2_{\T^3} + (\rho^-)^2_{\T^3} + (u)^2_{\T^3} + (\theta)^2_{\T^3} \big) 
	  \end{aligned}
	\end{equation}
	for all $|m| \leq s - 1$.
	
	Consequently, we substitute the bounds \eqref{Bnd-H1}, \eqref{Bnd-H2}, \eqref{Bnd-H3-H4}, \eqref{Bnd-H5} and \eqref{Bnd-H6} into the equality \eqref{K-E}, which leads to
	\begin{equation}\label{Bnd-K-E}
	  \begin{aligned}
	    & - \l \partial^m_x \mathcal{K} ( G, E, B ) , \partial^m_x E \r_{L^2_x} \\
	    \leq & - \eps \tfrac{\d}{\d t} \Big[ \tfrac{\eps}{2} \big\| \l \partial^m_x G \cdot \mathsf{q}_1 , \widetilde{\Phi} \r_{L^2_v} \big\|^2_{L^2_x} - \l \partial^m_x G \cdot \mathsf{q}_1 , \widetilde{\Phi} \cdot \partial^m_x E \r_{L^2_{x,v}} \Big] \\
	    + & \tfrac{\sigma}{2} \| \partial^m_x E \|^2_{L^2_x} + C \big( \| \nabla_x B \|^2_{H^{s-2}_x} + \| \nabla_x \mathbb{P} G \|^2_{H^{s-1}_x L^2_v} + \| \mathbb{P}^\perp G \|^2_{H^s_x L^2_v} \big) \\
	    + & C \eps ( \| G \|_{H^s_x L^2_v} + \| B \|_{H^s_x} ) \big( \| \nabla_x \mathbb{P} G \|^2_{H^{s-1}_x L^2_v} + \| \mathbb{P}^\perp G \|^2_{H^s_x L^2_v (\nu)} + \| E \|^2_{H^{s-1}_x} + \| \nabla_x B \|^2_{H^{s-2}_x} \big) \\
	    + & C \eps ( 1 + \| G \|_{H^s_x L^2_v} + \| B \|_{H^s_x} ) \big( (\rho^+)^2_{\T^3} + (\rho^-)^2_{\T^3} + (u)^2_{\T^3} + (\theta)^2_{\T^3} \big) 
	  \end{aligned}
	\end{equation}
	for all $0 < \eps \leq 1$ and for any $|m| \leq s - 1$. Finally, plugging the inequality \eqref{Bnd-K-E} into the relation \eqref{E-decay-higer-derivative-equ} and summing up for all $|m| \leq s - 1$ reduce to
	\begin{equation}\label{E-decay}
	  \begin{aligned}
	    & \tfrac{1}{2} \tfrac{\d}{\d t} \big( \| E \|^2_{H^{s-1}_x} + \| B \|^2_{H^{s-1}} \big) + \tfrac{\sigma}{2} \| E \|^2_{H^{s - 1}_x} \\
	    \leq & - \eps \tfrac{\d}{\d t} \sum_{|m| \leq s - 1} \Big[ \tfrac{\eps}{2} \big\| \l \partial^m_x G \cdot \mathsf{q}_1 , \widetilde{\Phi} \r_{L^2_v} \big\|^2_{L^2_x} - \l \partial^m_x G \cdot \mathsf{q}_1 , \widetilde{\Phi} \cdot \partial^m_x E \r_{L^2_{x,v}} \Big] \\
	    + &  C \big( \| \nabla_x B \|^2_{H^{s-2}_x} + \| \nabla_x \mathbb{P} G \|^2_{H^{s-1}_x L^2_v} + \| \mathbb{P}^\perp G \|^2_{H^s_x L^2_v} \big) \\
	    + & C \eps ( \| G \|_{H^s_x L^2_v} + \| B \|_{H^s_x} ) \big( \| \nabla_x \mathbb{P} G \|^2_{H^{s-1}_x L^2_v} + \| \mathbb{P}^\perp G \|^2_{H^s_x L^2_v (\nu)} + \| E \|^2_{H^{s-1}_x} + \| \nabla_x B \|^2_{H^{s-2}_x} \big) \\
	    + & C \eps ( 1 + \| G \|_{H^s_x L^2_v} + \| B \|_{H^s_x} ) \big( (\rho^+)^2_{\T^3} + (\rho^-)^2_{\T^3} + (u)^2_{\T^3} + (\theta)^2_{\T^3} \big) 
	  \end{aligned}
	\end{equation}
	for any $0 < \eps \leq 1$.
	
	Let $\delta_2 \in (0, 1]$ be a small constant to be determined. Then, adding the $\delta_2$ times  of \eqref{E-decay} to \eqref{B-decay} tells us
	\begin{equation}\label{E-B-decay-1}
	  \begin{aligned}
	    & \tfrac{1}{2} \tfrac{\d}{\d t} \Big( \delta_2 \| E \|^2_{H^{s-1}} + \delta_2 \| B \|^2_{H^{s-1}_x} + \delta_1 ( \sigma - 1 ) \| B \|^2_{H^{s-2}_x} \\
	    & \qquad \quad + \delta_1 \| \partial_t B + B \|^2_{H^{s-2}_x} + ( 1 - \delta_1 ) \| \partial_t B \|^2_{H^{s-2}_x} + \| \nabla_x B \|^2_{H^{s-2}_x} \Big) \\
	    & + \tfrac{\sigma}{2} \delta_2 \| E \|^2_{H^{s-1}_x} + ( \sigma - \tfrac{3}{2} \delta_1 ) \| \partial_t B \|^2_{H^{s-2}_x} + ( \tfrac{\delta_1}{2} - C \delta_2 ) \| \nabla_x B \|^2_{H^{s-2}_x} \\
	    \leq & - \eps \tfrac{\d}{\d t} \sum_{|m| \leq s - 2} \Big[ \tfrac{\eps}{2} \big\| \nabla_x \times \big\langle \partial^m_x G \cdot \mathsf{q}_1 , \widetilde{\Phi} \big\rangle_{L^2_v} \big\|^2_{L^2_x} \\
	    & \qquad \quad + \l \nabla_x \times \big\langle \partial^m_x G \cdot \mathsf{q}_1 , \widetilde{\Phi} \big\rangle_{L^2_v} , \partial_t \partial^m_x B + \delta_1 \partial^m_x B \r_{L^2_x} \Big] \\
	    & - \eps \tfrac{\d}{\d t} \sum_{|m| \leq s - 1} \Big[ \tfrac{\eps}{2} \delta_2 \big\| \big\langle \partial^m_x G \cdot \mathsf{q}_1 , \widetilde{\Phi} \big\rangle_{L^2_v} \big\|^2_{L^2_v} - \delta_2 \big\langle \partial^m_x G \cdot \mathsf{q}_1 , \widetilde{\Phi} \cdot \partial^m_x E \big\rangle_{L^2_{x,v}}  \Big] \\
	    & + C \eps ( 1 + \tfrac{1}{\delta_1} ) \big( \| \nabla_x \mathbb{P} G \|^2_{H^{s-1}_x L^2_v} + \| \mathbb{P}^\perp G \|^2_{H^s_x L^2_v (\nu)} \big) \\
	    & + C ( \| G \|_{H^s_x L^2_v} + \| B \|_{H^s_x} ) \big( \| E \|^2_{H^{s-1}_x} + \| \partial_t B \|^2_{H^{s-2}_x} + \| \nabla_x B \|^2_{H^{s-2}_x} \big) \\
	    & + C ( \| G \|_{H^s_x L^2_v} + \| B \|_{H^s_x} ) \big( | \nabla_x \mathbb{P} G \|^2_{H^{s-1}_x L^2_v} + \| \mathbb{P}^\perp G \|^2_{H^s_x L^2_v (\nu)} \big) \\
	    & + C ( \eps + \| G \|_{H^s_x L^2_v} + \| B \|_{H^s_x} ) \big( (\rho^+)^2_{\T^3} + (\rho^-)^2_{\T^3} + (u)^2_{\T^3} + (\theta)^2_{\T^3} \big)
	  \end{aligned}
	\end{equation}
	for any $0 < \eps \leq 1$ and integer $s \geq 3$, where the small constants $\delta_1, \delta_2 \in (0, 1]$ are to be determined. We first take $\delta_1 \in ( 0 , \min \{ \tfrac{1}{2} , \tfrac{1}{3} \sigma  \}$ such that
	\begin{equation}
	  1 - \delta_1 \geq \tfrac{1}{2} > 0 \ \ \textrm{and } \ \sigma - \tfrac{3}{2} \delta_1 \geq \tfrac{1}{2} \sigma > 0 \,,
	\end{equation}
	and then choose $ \delta_2 \in ( 0 , \min \{ 1, \tfrac{1}{4 C} \delta_1 \} ] $ such that
	\begin{equation}
	  \tfrac{\delta_1}{2} - C \delta_2 \geq \tfrac{\delta_1}{4} > 0 \,,
	\end{equation}
	where the constant $C > 0$ is mentioned as in the inequality \eqref{E-B-decay-1} and independent of $\eps \in (0,1]$. Then the inequality \eqref{E-B-decay-1} implies \eqref{E-B-decay} and the proof of Proposition \ref{Prop-Decay-E-B} is finished.
\end{proof}

\subsection{Summarizations} In this subsection, we will summarize the all energy estimates derived from the previous four subsections. We first choose a constant $\eta_0 \in (0, 1]（$, independent of $\eps > 0$, such that for all $\eta \in ( 0, \eta_0 ]$
\begin{equation}
  \begin{aligned}
    \lambda - C \eta \geq \tfrac{\lambda}{2} > 0 \,,
  \end{aligned}
\end{equation}
where the constant $C > 0$ is mentioned as in Proposition \ref{Prop-MM-Est}. Combining the Young's inequality, the bound \eqref{Fluid-norm-2} and the inequality \eqref{Average-Bnd} in Proposition \ref{Prop-rho-u-theta-average}, we add the $\eta$ times of the relation \eqref{Fluid-Dissip} to the bound \eqref{Spatial-Bnd} and then deduce that
\begin{equation}\label{G-macro-micro-decay}
  \begin{aligned}
    & \tfrac{1}{2} \tfrac{\d}{\d t} \big( \| G \|^2_{H^s_x L^2_v} + \| E \|^2_{H^s_x} + \| B \|^2_{H^s_x} \big) + \eps \eta \tfrac{\d}{\d t} \mathscr{A}_s (G) (t) \\
    & + \tfrac{\lambda}{2 \eps^2} \| \mathbb{P}^\perp G \|^2_{H^s_x L^2_v (\nu)} + \eta \| \nabla_x \mathbb{P} G \|^2_{H^{s-1}_x L^2_v} + \eta \| \div_x E \|^2_{H^{s-1}_x} \\
    \leq & C \big( \| G \|_{H^s_x L^2_v} + \| G \|^4_{H^s_x L^2_v} + \| E \|_{H^s_x} + \| E \|^4_{H^s_x} + \| B \|_{H^s_x} + \| B \|^4_{H^s_x} \big) \\
    & \ \ \times \big( \| \nabla_x \mathbb{P} G \|^2_{H^{s-1}_x L^2_v} + \tfrac{1}{\eps^2} \| \mathbb{P}^\perp G \|^2_{H^s_x L^2_v (\nu) } + \| E \|^2_{H^{s-1}_x} + \| \nabla_x B \|^2_{H^{s-1}_x} \big) \\
    & + C \big( \| G \|_{H^s_x L^2_v} + \| E \|_{H^s_x} + \| B \|_{H^s_x} \big)  \sum_{|m| \leq s -1} \| \nabla_v \partial^m_x \mathbb{P}^\perp G \|^2_{L^2_{x,v} (\nu)}
  \end{aligned}
\end{equation}
holds for all $0 < \eps \leq 1$, $\eta \in (0, \eta_0 ]$ and integer $s \geq 3$.

We next take a constant $\eta_1 \in (0, 1]$, independent of $\eps > 0$, such that for all $\eta \in ( 0, \eta_0 ]$
\begin{equation}\label{Coeffic-etas}
  \left\{
    \begin{array}{l}
      1 - \eta \eta_1 \delta_1 \geq 1 - \eta_0 \eta_2 \delta_1 \geq \tfrac{1}{2} \,, \\[2mm]
      \tfrac{\lambda}{2} - C \eta \eta_1 \geq \tfrac{\lambda}{2} - C \eta_0 \eta_1 \geq \tfrac{\lambda}{4} \,, \\[2mm]
      \eta - C \eta \eta_1 \geq \tfrac{\eta}{2} \,,
    \end{array}
  \right.
\end{equation}
where the positive constant $C$ is given in Proposition \ref{Prop-Decay-E-B}. Then $\eta_1$ can be assumed as
\begin{equation*}
  \begin{aligned}
    \eta_1 = \min \{ 1, \tfrac{1}{2 C} , \tfrac{1}{2 \eta_0 \delta_1}, \tfrac{\lambda}{4 C \eta_0} \} \in ( 0, 1 ] \,.
  \end{aligned}
\end{equation*}
We now multiply the inequality \eqref{E-B-decay} by $\eta_1 \eta $ and then add it to the bound \eqref{G-macro-micro-decay}. We finally derive from the Young's inequality and the estimate \eqref{Average-Bnd} in Proposition \ref{Prop-rho-u-theta-average} that for all $0 < \eps \leq 1$
\begin{equation}\label{GEB-decay-1}
  \begin{aligned}
    & \tfrac{1}{2} \tfrac{\d}{\d t} \big( \| G \|^2_{H^s_x L^2_v} + \| E \|^2_{H^s_x} + \| B \|^2_{H^s_x} + \eta_1 \eta \mathscr{E}_1 (E,B) \big) + \eps \eta \tfrac{\d}{\d t} \big( \mathscr{A}_s (G) (t) + \eta_1 \mathscr{A}_s (E,B) (t) \big) \\
    & \big( \tfrac{\lambda}{2} - C \eta_1 \eta \big) \tfrac{1}{\eps^2} \| \mathbb{P}^\perp G \|^2_{H^s_x L^2_v (\nu)} + ( \eta - C \eta_1 \eta ) \| \nabla_x \mathbb{P} G \|^2_{H^{s-1}_x L^2_v} + \eta \| \div_x E \|^2_{H^{s-1}_x} + \eta_1 \eta \mathscr{D}_1 (E,B) \\
    \leq & C \big( \| G \|_{H^s_x L^2_v} + \| G \|^4_{H^s_x L^2_v} + \| E \|_{H^s_x} + \| E \|^4_{H^s_x} + \| B \|_{H^s_x} + \| B \|^4_{H^s_x} \big) \\
    & \ \ \times \big( \| \nabla_x \mathbb{P} G \|^2_{H^{s-1}_x L^2_v} + \tfrac{1}{\eps^2} \| \mathbb{P}^\perp G \|^2_{H^s_x L^2_v (\nu) } + \| E \|^2_{H^{s-1}_x} + \| \nabla_x B \|^2_{H^{s-1}_x} + \| \partial_t B \|^2_{H^{s-2}_x} \big) \\
    & + C \big( \| G \|_{H^s_x L^2_v} + \| E \|_{H^s_x} + \| B \|_{H^s_x} \big)  \sum_{|m| \leq s -1} \| \nabla_v \partial^m_x \mathbb{P}^\perp G \|^2_{L^2_{x,v} (\nu)} 
  \end{aligned}
\end{equation}
holds for any $\eta \in (0, \eta_0 ]$ and integer $s \geq 3$. 

For notational simplicity, we define the following functionals:
\begin{equation}\label{E-D-eta-Energies}
  \begin{aligned}
    & \mathscr{E}_{\eta} ( G, E, B ) = \| G \|^2_{H^s_x L^2_v} + \| E \|^2_{H^s_x} + \| B \|^2_{H^s_x} + \eta_1 \eta \mathscr{E}_1 (E, B) \,, \\
    & \mathscr{A}_s (G,E,B) (t) = \mathscr{A}_s (G) (t) + \eta_1 \mathscr{A}_s (E,B) (t) \,, \\
    & \mathscr{D}_{\eta} ( G, E, B ) = \tfrac{\lambda}{4 \eps^2} \| \mathbb{P}^\perp G \|^2_{H^s_x L^2_v (\nu)} + \tfrac{\eta}{2} \| \nabla_x \mathbb{P} G \|^2_{H^{s-1}_x L^2_v} + \eta_1 \eta  \mathscr{D}_1 (E, B) \,, \\
    & \mathscr{D}_w (G) = \sum_{|m| \leq s -1} \| \nabla_v \partial^m_x \mathbb{P}^\perp G \|^2_{L^2_{x,v} (\nu)} \,.
  \end{aligned}
\end{equation}
We emphasize that the coefficients relations \eqref{Coeffic-etas} ensure that 
\begin{equation}
  \begin{aligned}
    & \| B \|^2_{H^s_x} - \eta_1 \eta \delta_1 \| B \|^2_{H^{s-2}_x} \\
    = & \tfrac{1}{2} \| B \|^2_{H^s_x} + ( \tfrac{1}{2} - \eta_1 \eta \delta_1 ) \| B \|^2_{H^{s-2}_x} + \tfrac{1}{2} \sum_{s-1 \leq |m| \leq s} \| \partial^m_x B \|^2_{L^2_x} \\
    \geq & \tfrac{1}{2} \| B \|^2_{H^s_x} \geq 0 
  \end{aligned}
\end{equation}
for all $\eta \in ( 0, \eta_0 ]$, which means the functional $ \mathscr{E}_{\eta, \eta_1} ( G, E, B ) $ is nonnegative for all $\eta \in ( 0, \eta_0 ]$. Furthermore, we have
\begin{equation}
  \begin{aligned}
    \| G \|^2_{H^s_x L^2_v} + \| E \|^2_{H^s_x} + \| B \|^2_{H^s} \leq 2 \mathscr{E}_{\eta} ( G, E, B ) \,, 
  \end{aligned}
\end{equation}
and 
\begin{equation}
  \begin{aligned}
    \| \nabla_x \mathbb{P} G \|^2_{H^{s-1}_x L^2_v} + \tfrac{1}{\eps^2} \| \mathbb{P}^\perp G \|^2_{H^s_x L^2_v (\nu) } + \| E \|^2_{H^{s-1}_x} + \| \nabla_x B \|^2_{H^{s-1}_x} + \| \partial_t B \|^2_{H^{s-2}_x} \\
    \leq C ( 1 + \tfrac{1}{\eta} ) \mathscr{D}_{\eta} ( G, E, B )
  \end{aligned}
\end{equation}
for all $\eta \in ( 0, \eta_0 ]$. 

Consequently, we have derived the following proposition from the inequality \eqref{GEB-decay-1}:
\begin{proposition}\label{Prop-Spatial-Summary}
	Assume that $(G, E, B)$ is the solution to the perturbed VMB system \eqref{VMB-G} constructed in Proposition \ref{Prop-Local-Solutn}. Then there are constants $\eta_0 , \eta_1 \in (0, 1]$ and $C > 0$, independent of $\eps > 0$, such that
	\begin{equation}\label{GEB-decay}
	  \begin{aligned}
	    & \tfrac{1}{2} \tfrac{\d}{\d t} \mathscr{E}_{\eta} (G, E, B) + \eps \eta \tfrac{\d}{\d t} \mathscr{A}_s (G, E, B) (t) + \mathscr{D}_{\eta} (G, E, B) \\
	    \leq & C ( 1 + \tfrac{1}{\eta} ) \big[ \mathscr{E}^\frac{1}{2}_{\eta} (G, E, B) + \mathscr{E}^2_{\eta} (G, E, B) \big] \mathscr{D}_{\eta} (G, E, B) \\
	    & + C \mathscr{E}^\frac{1}{2}_{\eta} (G, E, B) \mathscr{D}_w (G)
	  \end{aligned}
	\end{equation}
	holds for all $0 < \eta \leq \eta_0$ and $0 < \eps \leq 1$.
\end{proposition}

\begin{remark}
	The parameter $\eta \in ( 0, \eta_0 ]$ given in Proposition \ref{Prop-Spatial-Summary} is small to be determined such that the unsigned functional $\eps \eta \mathscr{A}_s (G,E,B)$ defined in \eqref{E-D-eta-Energies} will be dominated by $\mathscr{E}_\eta (G,E,B)$ for all $0 < \eps \leq 1$.
\end{remark}

One notices that the energy inequality \eqref{GEB-decay} is not closed, because so far the quantity $\mathscr{D}_w (G)$ is uncontrolled.

\section{Energy estimates for the $(x,v)$-mixed derivatives and global solutions}
\label{Sec:Unif-Mix-Bnd-Global}

In this section, based on the bound \eqref{GEB-decay} in Proposition \ref{Prop-Spatial-Summary}, we will derive the energy estimates on the $(x,v)$-mixed derivatives of the kinetic part $\mathbb{P}^\perp G$ to control the the energy functional $\mathscr{D}_\eta (G)$ defined in \eqref{E-D-eta-Energies}. Then we can obtain a closed energy inequality of the perturbed VMB system \eqref{VMB-G-drop-eps}, which is uniform in $\eps \in (0, 1]$. One notices that for the hydrodynamic part $\mathbb{P} G$,
\begin{equation}\label{G-Macro}
  \begin{aligned}
    \| p(v) \partial^m_\alpha \mathbb{P} G \|_{L^2_{x,v}} \leq C \| \partial^m_x \mathbb{P} G \|_{L^2_{x,v}}
  \end{aligned}
\end{equation}
holds for any polynomial $p(v)$, which has been estimated in Proposition \ref{Prop-Spatial-Summary}. Furthermore, when we compute the $L^2_{x,v}$-norm of $\partial^m_\alpha \mathbb{P}^\perp G$ for $|m| + |\alpha| \leq s$ with $\alpha \neq 0$ in the later, there is an uncontrolled term $ \tfrac{1}{\eps^2} \sum_{\substack{ |m'| + |\alpha'| \leq s \\ \alpha' < \alpha }} \| \partial^{m'}_{\alpha'} \mathbb{P}^\perp G \|^2_{L^2_{x,v}(\nu)} $ in the right-hand side of \eqref{Mix-5}. However, we observe that the orders of $v$-derivatives in this term is strictly less that $|\alpha|$, so that we can employ an induction over $|\alpha|$, which ranges from $0$ to $s$. To be more precise, one can first inductively derive the following lemma.

\begin{lemma}\label{Lmm-Mix-Energy}
	Assume that $(G, E, B)$ is the solution to the perturbed VMB system \eqref{VMB-G} constructed in Proposition \ref{Prop-Local-Solutn}. Let $s \geq 3$ be any fixed integer. For any given $0 \leq k \leq s$, $|\alpha| \leq k$, there are positive constants $C_{|\alpha|}$, $C_k^*$, $\delta_k$, $\delta^*_k$, $\varrho_k$ and $\varrho^*_k$, independent of $0 < \eps \leq 1$ and $0 < \eta \leq \eta_0$, such that
	\begin{equation}\label{Mix-Derivatives}
	  \begin{aligned}
	     & \tfrac{\d}{\d t} \bigg\{  \varrho_k \mathscr{E}_\eta (G,E,B) + \varrho^*_k \eps \eta \mathscr{A}_s (G,E,B) + \sum_{|m| + |\alpha| \leq s \,, \, |\alpha| \leq k} \tfrac{C_{|\alpha|}\eta}{1 + \eta} \| \partial^m_\alpha \mathbb{P}^\perp G \|^2_{L^2_{x,v}} \bigg\} \\
	    & \qquad + \tfrac{\delta_k \eta}{1 + \eta} \tfrac{1}{\eps^2} \sum_{|m| + |\alpha| \leq s \,, \, |\alpha| \leq k} \| \partial^m_\alpha \mathbb{P}^\perp G \|^2_{L^2_{x,v}(\nu)} + \delta^*_k  \mathscr{D}_\eta (G,E,B)  \\
	    \leq & C_k^* \Big\{ \mathscr{E}_\eta^\frac{1}{2} (G,E,B) \| \mathbb{P}^\perp G \|^2_{\widetilde{H}^s_{x,v}(\nu)} + ( 1 + \tfrac{1}{\eta} ) \big[ \mathscr{E}_\eta^\frac{1}{2} (G,E,B) + \mathscr{E}_\eta^2 (G,E,B) \big] \\
	    & \qquad + ( 1 + \mathscr{E}_\eta (G,E,B) ) \big( \mathscr{E}_\eta (G,E,B) + \| \mathbb{P}^\perp G \|^2_{\widetilde{H}^s_{x,v}(\nu)} \big) \Big\} \mathscr{D}_\eta (G,E,B)
	  \end{aligned}
	\end{equation}
	holds for all $0 < \eps \leq 1$ and $0 < \eta \leq \eta_0$, where the constant $\eta_0 > 0$ is mentioned in Proposition \ref{Prop-Spatial-Summary}. Here the functionals $\mathscr{E}_\eta (G,E,B)$, $\mathscr{A}_s (G,E,B)$ and $\mathscr{D}_\eta (G,E,B)$ are defined in  \eqref{E-D-eta-Energies}.
\end{lemma}

\begin{proof}[Proof of Lemma \ref{Lmm-Mix-Energy}]
	We now rewrite the first $G$-equation of \eqref{VMB-G-drop-eps} as
	\begin{equation}\label{VMB-G-Mixed}
	  \begin{aligned}
	    & \partial_t \mathbb{P}^\perp G + \tfrac{1}{\eps} \big[ v \cdot \nabla_x + \mathsf{q} (\eps E + v \times B) \cdot \nabla_v \big] \mathbb{P}^\perp G + \tfrac{1}{\eps^2} \mathscr{L} \mathbb{P}^\perp G \\
	    = & \tfrac{1}{\eps} (E \cdot v) \sqrt{M} \mathsf{q}_1 + \tfrac{1}{2} \mathsf{q} ( E \cdot v ) G + \tfrac{1}{\eps} \Gamma (G,G) \\
	    & - \partial_t \mathbb{P} G - \tfrac{1}{\eps} \big[ v \cdot \nabla_x + \mathsf{q} (\eps E + v \times B) \cdot \nabla_v \big] \mathbb{P} G \,.
	  \end{aligned}
	\end{equation}
	For all $|m| + |\alpha| \leq s$ and $\alpha \neq 0$, we take the derivative operator $\partial^m_\alpha$ in the equation \eqref{VMB-G-Mixed} and then we obtain
	\begin{equation}\label{VMB-G-Mixed-Derivative}
	  \begin{aligned}
	    & \partial_t \partial^m_\alpha \mathbb{P}^\perp G + \tfrac{1}{\eps} \big[ v \cdot \nabla_x + \mathsf{q} ( \eps E + v \times B ) \cdot \nabla_v \big] \partial^m_\alpha \mathbb{P}^\perp G + \tfrac{1}{\eps^2} \partial^m_\alpha \mathscr{L} \mathbb{P}^\perp G \\
	    = & \tfrac{1}{\eps} \partial^m_\alpha \big[ (E \cdot v) \sqrt{M} \mathsf{q}_1 \big] + \tfrac{1}{2} \partial^m_\alpha \big[ \mathsf{q} (E \cdot v) G \big] + \tfrac{1}{\eps} \partial^m_\alpha \Gamma (G, G) \\
	    & - \tfrac{1}{\eps} \sum_{|\alpha'|=1} C_\alpha^{\alpha'} \partial^{\alpha'}_v v \cdot \nabla_x \partial^m_{\alpha - \alpha'} \mathbb{P}^\perp G - \tfrac{1}{\eps} \sum_{|\alpha'|=1} C_\alpha^{\alpha'} \mathsf{q} ( \partial^{\alpha'}_v v \times B ) \cdot \nabla_v \mathbb{P}^\perp G \\
	    & - \tfrac{1}{\eps} \sum_{m' < m} \sum_{|\alpha'| \leq 1} C_m^{m'} C_\alpha^{\alpha'} \mathsf{q} ( \partial^{\alpha'}_v v \times \partial^{m-m'}_x B ) \cdot \nabla_v \partial^{m'}_{\alpha - \alpha'} \mathbb{P}^\perp G \\
	    & - \partial_t \partial^m_\alpha \mathbb{P} G - \tfrac{1}{\eps} \partial^m_\alpha \big[ v \cdot \nabla_x + \mathsf{q} ( \eps E + v \times B ) \cdot \nabla_v \big] \mathbb{P} G \,.
	  \end{aligned}
	\end{equation}
	We take the inner product of \eqref{VMB-G-Mixed-Derivative} over $\T^3 \times \R^3$ with $\partial^m_\alpha \mathbb{P}^\perp G$. More precisely, we obtain
	\begin{equation}\label{M}
	  \begin{aligned}
	    & \tfrac{1}{2} \tfrac{\d}{\d t} \| \partial^m_\alpha \mathbb{P}^\perp G \|^2_{L^2_{x,v}} + \underset{M_1}{ \underbrace{ \tfrac{1}{\eps^2} \l \partial^m_\alpha \mathscr{L} \mathbb{P}^\perp G , \partial^m_\alpha \mathbb{P}^\perp G \r_{L^2_{x,v}} }} \\
	    = & \underset{M_2}{ \underbrace{ \tfrac{1}{\eps} \l \partial^m_\alpha \big[ (E \cdot v) \sqrt{M} \mathsf{q}_1 \big] , \partial^m_\alpha \mathbb{P}^\perp G \r_{L^2_{x,v}} }} + \underset{M_3}{ \underbrace{ \tfrac{1}{2} \l \partial^m_\alpha \big[ \mathsf{q} ( E \cdot v ) G \big] , \partial^m_\alpha \mathbb{P}^\perp G \r_{L^2_{x,v}} }} \\
	    & + \underset{M_4}{ \underbrace{ \tfrac{1}{\eps} \l \partial^m_\alpha \Gamma (G, G) , \partial^m_\alpha \mathbb{P}^\perp G \r_{L^2_{x,v}} }} \ \underset{M_5}{ \underbrace{ - \tfrac{1}{\eps} \sum_{|\alpha'|=1} C_\alpha^{\alpha'} \l \partial^{\alpha'}_v v \cdot \nabla_x \partial^m_{\alpha - \alpha'} \mathbb{P}^\perp G , \partial^m_\alpha \mathbb{P}^\perp G \r_{L^2_{x,v}} }} \\
	    & \underset{M_6}{ \underbrace{ - \tfrac{1}{\eps} \sum_{|\alpha'|=1} C_\alpha^{\alpha'} \l \mathsf{q} ( \partial^{\alpha'}_v v \times B ) \cdot \nabla_v \partial^m_{\alpha - \alpha'} \mathbb{P}^\perp G , \partial^m_\alpha \mathbb{P}^\perp G \r_{L^2_{x,v}} }} \\
	    & \underset{M_7}{ \underbrace{ - \tfrac{1}{\eps} \sum_{m' < m} \sum_{|\alpha'| \leq 1} C_m^{m'} C_\alpha^{\alpha'} \l \mathsf{q} ( \partial^{\alpha'}_v v \times \partial^{m-m'}_x B ) \cdot \nabla_v \partial^{m'}_{\alpha - \alpha'} \mathbb{P}^\perp G , \partial^m_\alpha \mathbb{P}^\perp G \r_{L^2_{x,v}} }} \\
	    & \underset{M_8}{ \underbrace{ - \l \partial_t \partial^m_\alpha \mathbb{P} G , \partial^m_\alpha \mathbb{P}^\perp G \r_{L^2_{x,v}} }} \ \underset{M_9}{ \underbrace{ - \tfrac{1}{\eps} \l \partial^m_\alpha \big[ v \cdot \nabla_x + \mathsf{q} ( \eps E + v \times B ) \cdot \nabla_v \big] \mathbb{P} G , \partial^m_\alpha \mathbb{P}^\perp G \r_{L^2_{x,v}} }} \,.
	  \end{aligned}
	\end{equation}
	
	Recalling the decomposition of $\mathscr{L}$ in Lemma \ref{Lmm-L} (1), we have
	\begin{equation}
	  \begin{aligned}
	    \l \partial^m_\alpha \mathscr{L} \mathbb{P}^\perp G , \partial^m_\alpha \mathbb{P}^\perp G \r_{L^2_{x,v}} = 2 \l \partial^m_\alpha (\nu (v) \mathbb{P}^\perp G ) , \partial^m_\alpha \mathbb{P}^\perp G \r_{L^2_{x,v}} - \l \partial^m_\alpha \mathscr{K} \mathbb{P}^\perp G , \partial^m_\alpha \mathbb{P}^\perp G \r_{L^2_{x,v} } \,.
	  \end{aligned}
	\end{equation}
	From Lemma \ref{Lmm-nu-norm} (2), we derive
	\begin{equation}
	  \begin{aligned}
	    \l \partial^m_\alpha (\nu (v) \mathbb{P}^\perp G ) , \partial^m_\alpha \mathbb{P}^\perp G \r_{L^2_{x,v}} \geq C_5 \| \partial^m_\alpha \mathbb{P}^\perp G \|^2_{L^2_{x,v}(\nu)} - C_6 \sum_{\alpha' < \alpha} \| \partial^m_{\alpha'} \mathbb{P}^\perp G \|^2_{L^2_{x,v}} \,.
	  \end{aligned}
	\end{equation}
	Moreover, Lemma \ref{Lmm-L} (2) tells us that for any $\delta > 0$, there is a $C(\delta) > 0$ such that
	\begin{equation}
	  \begin{aligned}
	    \l \partial^m_\alpha \mathscr{K} \mathbb{P}^\perp G , \partial^m_\alpha \mathbb{P}^\perp G \r_{L^2_{x,v}} \leq \delta \| \partial^m_\alpha \mathbb{P}^\perp G \|^2_{L^2_{x,v}(\nu)} + C(\delta) \| \partial^m_x \mathbb{P}^\perp G \|^2_{L^2_{x,v}} \,.
	  \end{aligned}
	\end{equation}
	Thus taking $\delta = C_5 > 0$, $\lambda_0 = C_5 > 0$ and $\lambda_1 = C(C_5) + 2 C_6 > 0$ implies that the quantity $M_1$ has the lower bound
	\begin{equation}\label{M1}
	  \begin{aligned}
	    M_1 = \tfrac{1}{\eps^2} \l \partial^m_\alpha \mathscr{L} \mathbb{P}^\perp G , \partial^m_\alpha \mathbb{P}^\perp G \r_{L^2_{x,v}} \geq \tfrac{\lambda_0}{\eps^2} \| \partial^m_\alpha \mathbb{P}^\perp G \|^2_{L^2_{x,v} (\nu)} - \tfrac{\lambda_1}{\eps^2} \sum_{\alpha' < \alpha} \| \partial^m_{\alpha'} \mathbb{P}^\perp G \|^2_{L^2_{x,v}} \,.
	  \end{aligned}
	\end{equation}
	
	Since $|m| + |\alpha| \leq s$ and $\alpha \neq 0$, $0 \leq |m| \leq s - 1$. Then the term $M_2$ can be estimated as
	\begin{equation}\label{M2}
	  \begin{aligned}
	    M_2 = & \tfrac{1}{\eps} \l \partial^m_x E \cdot \partial^{\alpha}_v ( v \sqrt{M} \mathsf{q}_1 ) , \partial^m_\alpha \mathbb{P}^\perp_\alpha G \r_{L^2_{x,v} } \\
	    \leq & \tfrac{1}{\eps} \| \partial^m_x E \|_{L^2_x} \| \partial^\alpha_v ( v \sqrt{M} \mathsf{q}_1 ) \|_{L^2_{x,v}} \| \partial^m_\alpha \mathbb{P}^\perp G \|_{L^2_{x,v}} \\
	    \leq & \tfrac{C}{\eps} \| E \|_{H^{s-1}_x} \| \partial^m_\alpha \mathbb{P}^\perp G \|_{L^2_{x,v}} \leq \tfrac{C}{\eps} \| E \|_{H^{s-1}_x} \| \partial^m_\alpha \mathbb{P}^\perp G \|_{L^2_{x,v}(\nu)} \\
	    \leq & \tfrac{\gamma}{\eps^2} \| \partial^m_\alpha \mathbb{P}^\perp G \|^2_{L^2_{x,v}(\nu)} + C_\gamma \| E \|^2_{H^{s-1}_x} \ 
	  \end{aligned}
	\end{equation}
	for small $\gamma > 0$ to be determined, where the H\"older inequality, Young's inequality and Part (1) of Lemma \ref{Lmm-nu-norm} are utilized here.
	
	Next we estimate the term $M_3$. It can be directly calculated by using the decomposition $G = \mathbb{P} G + \mathbb{P}^\perp G$ that
	\begin{equation}\label{M3-decomp}
	  \begin{aligned}
	    M_3 = & \underset{M_{31}}{\underbrace{ \tfrac{1}{2} \sum_{|\alpha'| \leq 1} C_\alpha^{\alpha'} \l \partial^m_x E \cdot ( \partial^{\alpha'}_v v \otimes \mathsf{q} \partial^{\alpha - \alpha'}_v \mathbb{P} G ), \partial^m_\alpha \mathbb{P}^\perp G \r_{L^2_{x,v} } }} \\
	    + & \underset{M_{32}}{\underbrace{ \tfrac{1}{2} \sum_{m' < m} \sum_{|\alpha'| \leq 1} C_m^{m'} C_\alpha^{\alpha'}  \l \partial^{m'}_x E \cdot ( \partial^{\alpha'}_v v \otimes \mathsf{q} \partial^{m-m'}_{\alpha - \alpha'} \mathbb{P} G ) , \partial^m_\alpha \mathbb{P}^\perp G \r_{L^2_{x,v}} }} \\
	    + & \underset{M_{33}}{\underbrace{ \tfrac{1}{2} \sum_{|\alpha'| \leq 1} C_\alpha^{\alpha'} \l \partial^m_x E \cdot ( \partial^{\alpha'}_v v \otimes \mathsf{q} \partial^{\alpha - \alpha'}_v \mathbb{P}^\perp G ), \partial^m_\alpha \mathbb{P}^\perp G \r_{L^2_{x,v} } }} \\
	    + & \underset{M_{34}}{\underbrace{ \tfrac{1}{2} \sum_{m' < m} \sum_{|\alpha'| \leq 1} C_m^{m'} C_\alpha^{\alpha'}  \l \partial^{m'}_x E \cdot ( \partial^{\alpha'}_v v \otimes \mathsf{q} \partial^{m-m'}_{\alpha - \alpha'} \mathbb{P}^\perp G ) , \partial^m_\alpha \mathbb{P}^\perp G \r_{L^2_{x,v}} }} \,.
	  \end{aligned}
	\end{equation}
	For the term $M_{31}$, we drive from the H\"older inequality, the Sobolev embedding $H^2_x (\T^3) \hookrightarrow L^\infty_x (\T^3)$, the inequality \eqref{G-Macro} and the part (1) of Lemma \ref{Lmm-nu-norm} that
	\begin{equation}\label{M31}
	  \begin{aligned}
	    M_{31} \leq & C \sum_{|\alpha'|\leq 1} \| \partial^m_x E \|_{L^2_x} \| \partial^{\alpha'}_v v \otimes  \mathsf{q} \partial^{\alpha-\alpha'}_v \mathbb{P} G \|_{L^\infty_x L^2_v} \| \partial^m_\alpha \mathbb{P}^\perp G \|_{L^2_{x,v}} \\
	    \leq & C \| E \|_{H^{s-1}_x} \| \mathbb{P} G \|_{L^\infty_x L^2_v} \| \partial^m_\alpha \mathbb{P}^\perp G \|_{L^2_{x,v}} \leq C \| E \|_{H^{s-1}_x} \| \mathbb{P} G \|_{H^2_x L^2_v} \| \partial^m_\alpha \mathbb{P}^\perp G \|_{L^2_{x,v}} \\
	    \leq & C \| E \|_{H^{s-1}_x} \| \mathbb{P} G \|_{H^s_x L^2_v} \| \partial^m_\alpha \mathbb{P}^\perp G \|_{L^2_{x,v}(\nu)} \,.
	  \end{aligned}
	\end{equation}
	Here we require the integer $s \geq 2$. For the $M_{32}$, we additionally derive from the Sobolev embedding $H^1_x (\T^3) \hookrightarrow L^4_x (\T^3)$ that
	\begin{equation}\label{M32}
	  \begin{aligned}
	    M_{32} \leq & C \sum_{m' < m} \sum_{|\alpha'| \leq 1 } \| \partial^{m'}_x E \|_{L^4_x} \| \partial^{\alpha'}_v v \otimes \mathsf{q} \partial^{m-m'}_{\alpha-\alpha'} \mathbb{P} G \|_{L^4_x L^2_v} \| \partial^m_\alpha \mathbb{P}^\perp G \|_{L^2_{x,v}} \\
	    \leq & C  \sum_{m' < m} \sum_{|\alpha'| \leq 1 } \| \partial^{m'}_x E \|_{H^1_x} \| \partial^{m-m'}_x \mathbb{P} G \|_{H^1_x L^2_v} \| \partial^m_\alpha \mathbb{P}^\perp G \|_{L^2_{x,v}(\nu)} \\
	    \leq & C \| E \|_{H^{s-1}_x} \| \mathbb{P} G \|_{H^s_x L^2_v} \| \partial^m_\alpha \mathbb{P}^\perp G \|_{L^2_{x,v}(\nu)} \,. 
	  \end{aligned}
	\end{equation}
	In the term $M_{33}$, if $m=0$ and $1 \leq |\alpha| \leq s$, which satisfy $|m|+|\alpha| \leq s$ and $\alpha \neq 0$, one easily derive from the H\"older inequality, the part (1) of Lemma \ref{Lmm-CF-nu} and the Sobolev embedding $H^2_x (\T^3) \hookrightarrow L^\infty_x (\T^3)$ and $H^1_x (\T^3) \hookrightarrow L^4_x (\T^3)$ that
	\begin{equation}
	  \begin{aligned}
	    & \tfrac{1}{2} \sum_{|\alpha'| \leq 1} C_\alpha^{\alpha'} \l E \cdot ( \partial^{\alpha'}_v v \otimes \mathsf{q} \partial^{\alpha - \alpha'}_v \mathbb{P}^\perp G ), \partial^\alpha_v \mathbb{P}^\perp G \r_{L^2_{x,v} } \\
	    \leq & C \sum_{|\alpha'| = 1} \| E \|_{L^4_x} \left\| \tfrac{\partial^{\alpha'}_v v}{\nu (v)} \right\|_{L^\infty_v} \| \partial^{\alpha-\alpha'}_v \mathbb{P}^\perp G \|_{L^4_x L^2_v(\nu)} \| \partial^\alpha_v \mathbb{P}^\perp G \|_{L^2_{x,v} (\nu)} \\ 
	    & + \| E \|_{L^\infty_x} \left\| \tfrac{ v}{\nu (v)} \right\|_{L^\infty_v} \| \partial^\alpha_v \mathbb{P}^\perp G \|^2_{L^2_{x,v} (\nu)} \\
	    \leq & C \sum_{|\alpha'| = 1} \| E \|_{H^1_x} \| \partial^{\alpha-\alpha'}_v \mathbb{P}^\perp G \|_{H^1_x L^2_v(\nu)} \| \partial^\alpha_v \mathbb{P}^\perp G \|_{L^2_{x,v} (\nu)} + C \| E \|_{H^2_x} \| \partial^\alpha_v \mathbb{P}^\perp G \|^2_{L^2_{x,v} (\nu)} \\
	    \leq & C \| E \|_{H^{s}_x} \| \partial^\alpha_v \mathbb{P}^\perp G \|_{L^2_{x,v} (\nu)} \sum_{\alpha' \leq \alpha} \| \partial^{\alpha'}_v \mathbb{P}^\perp G \|_{L^2_{x,v}(\nu)} \,.
	  \end{aligned}
	\end{equation}
	Here $s \geq 2$ is required. In the term $M_{33}$, if $m, \alpha \neq 0$ and $|m| + |\alpha| \leq s$, we have $1 \leq |m|, |\alpha| \leq s - 1$. Then we estimate that
	\begin{equation}
	  \begin{aligned}
	    & \tfrac{1}{2} \sum_{|\alpha'| \leq 1} C_\alpha^{\alpha'} \l \partial^m_x E \cdot ( \partial^{\alpha'}_v v \otimes \mathsf{q} \partial^{\alpha-\alpha'}_v \mathbb{P}^\perp G ) , \partial^m_\alpha \mathbb{P}^\perp G \r_{L^2_{x,v}} \\
	    \leq & C \sum_{|\alpha'| \leq 1} \| \partial^m_x E \|_{L^4_x} \big\| \tfrac{\partial^{\alpha'}_v v}{\nu (v)} \big\|_{L^\infty_v} \| \partial^{\alpha-\alpha'}_v \mathbb{P}^\perp G \|_{L^4_x L^2_v (\nu)} \| \partial^m_\alpha \mathbb{P}^\perp G \|_{L^2_{x,v}(\nu)} \\
	    \leq & C \sum_{|\alpha'| \leq 1} \| \partial^m_x E \|_{H^1_x} \| \partial^{\alpha-\alpha'}_v \mathbb{P}^\perp G \|_{H^1_x L^2_v (\nu)} \| \partial^m_\alpha \mathbb{P}^\perp G \|_{L^2_{x,v}(\nu)} \\
	    \leq & C \| E \|_{H^s_x} \| \partial^m_\alpha \mathbb{P}^\perp G \|_{L^2_{x,v}(\nu)} \sum_{\alpha' \leq \alpha} \| \partial^m_{\alpha'} \mathbb{P}^\perp G \|_{L^2_{x,v}(\nu)} \,,
	  \end{aligned}
	\end{equation}
	where we make use of the Sobolev embedding $H^1_x (\T^3) \hookrightarrow L^4_x (\T^3)$. We thereby obtain the bound of $M_{33}$
	\begin{equation}\label{M33}
	  \begin{aligned}
	    M_{33} \leq & C \| E \|_{H^s_x} \| \partial^m_\alpha \mathbb{P}^\perp G \|_{L^2_{x,v}(\nu)} \sum_{\alpha' \leq \alpha} \| \partial^m_{\alpha'} \mathbb{P}^\perp G \|_{L^2_{x,v}(\nu)}
	  \end{aligned}
	\end{equation}
	holds for all $|m| + |\alpha| \leq s$ and $\alpha \neq 0$. In the term $M_{34}$, since $m' < m$ and $0 \leq |m| \leq s - 1$, we have $|m'| \leq |m| - 1 \leq s - 2$. We then deduce that
	\begin{equation}\label{M34}
	  \begin{aligned}
	    M_{34} \leq & C \sum_{m' < m} \sum_{|\alpha'| \leq 1} \| \partial^{m'}_x E \|_{L^\infty_x} \big\| \tfrac{\partial^{\alpha'}_v v}{\nu (v)} \big\|_{L^\infty_v} \| \partial^{m-m'}_{\alpha - \alpha'} \mathbb{P}^\perp G \|_{L^2_{x,v} (\nu)} \| \partial^m_\alpha \mathbb{P}^\perp G \|_{L^2_{x,v}(\nu)} \\
	    \leq & C \sum_{m' < m} \sum_{|\alpha'| \leq 1} \| \partial^{m'}_x E \|_{H^2_x} \| \partial^{m-m'}_{\alpha - \alpha'} \mathbb{P}^\perp G \|_{L^2_{x,v} (\nu)} \| \partial^m_\alpha \mathbb{P}^\perp G \|_{L^2_{x,v}(\nu)} \\
	    \leq & C \| E \|_{H^s_x} \| \partial^m_\alpha \mathbb{P}^\perp G \|_{L^2_{x,v}(\nu)} \sum_{m' \leq m , \, \alpha' \leq \alpha} \| \partial^{m'}_{\alpha'} \mathbb{P}^\perp G \|_{L^2_{x,v}(\nu)} \,,
	  \end{aligned}
	\end{equation}
	where the Sobolev embedding $H^2_x (\T^3) \hookrightarrow L^\infty_x (\T^3)$ is used. We then derive the bound of $M_3$ from plugging the inequalities \eqref{M31}, \eqref{M32}, \eqref{M33} and \eqref{M34} into the equality \eqref{M3-decomp} that
	\begin{equation}\label{M3}
	  \begin{aligned}
	    M_3 \leq & C \| E \|_{H^{s-1}_x} \| \mathbb{P} G \|_{H^s_x L^2_v} \| \partial^m_\alpha \mathbb{P}^\perp G \|_{L^2_{x,v}(\nu)} \\
	    + & C \| E \|_{H^s_x} \| \partial^m_\alpha \mathbb{P}^\perp G \|_{L^2_{x,v}(\nu)} \sum_{m' \leq m , \, \alpha' \leq \alpha} \| \partial^{m'}_{\alpha'} \mathbb{P}^\perp G \|_{L^2_{x,v}(\nu)} \\
	    \leq & C \| E \|_{H^{s-1}_x} \| \mathbb{P} G \|_{H^s_x L^2_v} \| \partial^m_\alpha \mathbb{P}^\perp G \|_{L^2_{x,v}(\nu)} + C \| E \|_{H^s_x} \| \mathbb{P}^\perp G \|_{\widetilde{H}^s_{x,v} (\nu)} \| \partial^m_\alpha \mathbb{P}^\perp G \|_{L^2_{x,v}(\nu)} 
	  \end{aligned}
	\end{equation}
	for all $|m| + |\alpha| \leq s$ and $\alpha \neq 0$.
	
	Next we estimate the term $M_4$ in \eqref{M}. Via Lemma \ref{Lmm-Gamma-Torus} and the decomposition $G = \mathbb{P} G + \mathbb{P}^\perp G$, we easily estimate that
	\begin{equation}
	  \begin{aligned}
	    M_4 \leq & \tfrac{C_\Gamma}{\eps} \| G \|_{H^s_{x,v}} \| G \|_{H^s_{x,v} (\nu)} \| \partial^m_\alpha \mathbb{P}^\perp G \|_{L^2_{x,v} (\nu)} \\
	    \leq & \tfrac{C}{\eps} \big( \| \mathbb{P} G \|_{H^s_{x,v}} + \| \mathbb{P}^\perp G \|_{H^s_{x,v}} \big) \big( \| \mathbb{P} G \|_{H^s_{x,v}(\nu)} + \| \mathbb{P}^\perp G \|_{H^s_{x,v}(\nu)} \big) \| \partial^m_\alpha \mathbb{P}^\perp G \|_{L^2_{x,v} (\nu)} \,.
	  \end{aligned}
	\end{equation}
	Furthermore, by the Poincar\'e inequality, the relation \eqref{G-Macro} and Proposition \ref{Prop-rho-u-theta-average}, we have
	\begin{equation}
	  \begin{aligned}
	    & \| \mathbb{P} G \|_{H^s_{x,v}} + \| \mathbb{P} G \|_{H^s_{x,v}(\nu)} \leq C \| \mathbb{P} G \|_{H^s_x L^2_v} \\
	    \leq & C \| \nabla_x \mathbb{P} G \|_{H^{s-1}_x L^2_v} + C \big( |(\rho^+)_{\T^3}| + |(\rho^-)_{\T^3}| + |(u)_{\T^3}| + |(\theta)_{\T^3}| \big) \\
	    + & C \big( \| \rho^+ - (\rho^+)_{\T^3} \|_{L^2_x} + \| \rho^- - (\rho^-)_{\T^3} \|_{L^2_x} + \| u - (u)_{\T^3} \|_{L^2_x} + \| \theta - (\theta)_{\T^3} \|_{L^2_x} \big) \\
	    \leq & C \| \nabla_x \mathbb{P} G \|_{H^{s-1}_x L^2_v} + C ( \| E \|_{L^2_x} + \| B \|_{H^1_x} ) ( \| E \|_{L^2_x} + \| \nabla_x B \|_{L^2_x} ) \\
	    \leq & C \| \nabla_x \mathbb{P} G \|_{H^{s-1}_x L^2_v} + C ( \| E \|_{H^s_x} + \| B \|_{H^s_x} ) ( \| E \|_{H^{s-1}_x} + \| \nabla_x B \|_{H^{s-2}_x} ) \,.
	  \end{aligned}
	\end{equation}
	Here $s \geq 2$ is required. We thereby obtain
	\begin{equation}\label{M4}
	  \begin{aligned}
	    M_4 \leq & \tfrac{C}{\eps} \big( \| \mathbb{P} G \|_{H^s_x L^2_v} + \| \mathbb{P}^\perp G \|_{H^s_{x,v}} \big) \big( \| \nabla_x \mathbb{P} G \|_{H^{s-1}_x L^2_v} + \| \mathbb{P}^\perp G \|_{H^s_{x,v}(\nu)} \big) \| \partial^m_\alpha \mathbb{P}^\perp G \|_{L^2_{x,v}(\nu)} \\
	    + & \tfrac{C}{\eps} ( \| E \|_{H^s_x} + \| B \|_{H^s_x} ) \big( \| \mathbb{P} G \|_{H^s_x L^2_v} + \| \mathbb{P}^\perp G \|_{H^s_{x,v}} \big) \\
	    & \quad \times ( \| E \|_{H^{s-1}_x} + \| \nabla_x B \|_{H^{s-2}_x} ) \| \partial^m_\alpha \mathbb{P}^\perp G \|_{L^2_{x,v}(\nu)} \,.
	  \end{aligned}
	\end{equation}
	
	Next we consider the quantity $M_5$ in \eqref{M}. Via the H\"older inequality and the part (1) of Lemma \ref{Lmm-nu-norm}, one easily yields that
	\begin{equation}\label{M5}
	  \begin{aligned}
	    M_5 \leq & \tfrac{C}{\eps} \sum_{|\alpha'| = 1} \| \nabla_x \partial^m_{\alpha - \alpha'} \mathbb{P}^\perp G \|_{L^2_{x,v}} \| \partial^m_\alpha \mathbb{P}^\perp G \|_{L^2_{x,v}} \\
	    \leq & \tfrac{C}{\eps} \sum_{\alpha' < \alpha} \| \nabla_x \partial^m_{\alpha'} \mathbb{P}^\perp G \|_{L^2_{x,v}(\nu)} \| \partial^m_\alpha \mathbb{P}^\perp G \|_{L^2_{x,v}(\nu)} \\
	    \leq & \tfrac{\gamma}{\eps^2} \| \partial^m_\alpha \mathbb{P}^\perp G \|^2_{L^2_{x,v}(\nu)} + C_\gamma \sum_{\substack{|m'|+|\alpha'| \leq s \\ \alpha' < \alpha}} \| \partial^{m'}_{\alpha'} \mathbb{P}^\perp G \|^2_{L^2_{x,v}(\nu)} 
	  \end{aligned}
	\end{equation}
	for small $\gamma > 0$ to be determined.
	
	The term $M_6$ in \eqref{M} will be estimated as follows:
	\begin{equation}\label{M6}
	  \begin{aligned}
	    M_6 \leq & \tfrac{C}{\eps} \sum_{|\alpha'|=1} \| B \|_{L^\infty_x} \| \nabla_v \partial^m_{\alpha-\alpha'} \mathbb{P}^\perp G \|_{L^2_{x,v}} \| \partial^m_\alpha \mathbb{P}^\perp G \|_{L^2_{x,v}} \\
	    \leq & \tfrac{C}{\eps} \| B \|_{H^2_x} \| \mathbb{P}^\perp G \|_{\widetilde{H}^s_{x,v}} \| \partial^m_\alpha \mathbb{P}^\perp G \|_{L^2_{x,v}} \\
	    \leq & \tfrac{C}{\eps} \| B \|_{H^s_x} \| \mathbb{P}^\perp G \|_{\widetilde{H}^s_{x,v}(\nu)} \| \partial^m_\alpha \mathbb{P}^\perp G \|_{L^2_{x,v}(\nu)} \,,
	  \end{aligned}
	\end{equation}
	where the H\"older inequality, the Sobolev embedding $H^2_x (\T^3) \hookrightarrow L^\infty_x (\T^3)$ and Lemma \ref{Lmm-nu-norm} (1) are utilized. Here $s \geq 2$ is required.
	
	Next we estimate the term $M_7$ in \eqref{M}. We first divide $M_7$ into two parts:
	\begin{equation}\label{M7-Decomp}
	  \begin{aligned}
	    M_7 = & \underset{M_{71}}{\underbrace{ - \tfrac{1}{\eps} \sum_{m' < m} \sum_{|\alpha'|=1} C_m^{m'} C_\alpha^{\alpha'} \l \mathsf{q} ( \partial^{\alpha'}_v v \times \partial^{m-m'}_x B ) \cdot \nabla_v \partial^{m'}_{\alpha-\alpha'} \mathbb{P}^\perp G , \partial^m_\alpha \mathbb{P}^\perp G \r_{L^2_{x,v}} }} \\
	    & \underset{M_{72}}{\underbrace{ - \tfrac{1}{\eps} \sum_{m' < m} C_m^{m'} \l \mathsf{q} ( v \times \partial^{m-m'}_x B ) \cdot \nabla_v \partial^{m'}_\alpha \mathbb{P}^\perp G , \partial^m_\alpha \mathbb{P}^\perp G \r_{L^2_{x,v}} }}\,.
	  \end{aligned}
	\end{equation}
	Since $|m|+|\alpha| \leq s$, $|\alpha| \geq 1$, $m' < m$ and $|\alpha'|=1$ in the term $M_{71}$, we easily have
	\begin{equation*}
	  \begin{aligned}
	    & |m'| + |\alpha-\alpha'| + 1 = |m'| + |\alpha| \leq |m| + |\alpha| - 1 \leq s - 1 \,, \\
	    & |m-m'| \leq |m| \leq s - |\alpha| \leq s - 1 \,.
	  \end{aligned}
	\end{equation*}
	Then the term $M_{71}$ can be estimated by
	\begin{equation}\label{M71}
	  \begin{aligned}
	    M_{71} \leq & \tfrac{C}{\eps} \sum_{m' < m} \sum_{|\alpha'|=1} \| \partial^{m-m'}_x B \|_{L^4_x} \| \nabla_v \partial^{m'}_{\alpha-\alpha'} \mathbb{P}^\perp G \|_{L^4_x L^2_v} \| \partial^m_\alpha \mathbb{P}^\perp G \|_{L^2_{x,v}} \\
	    \leq & \tfrac{C}{\eps} \sum_{m' < m} \sum_{|\alpha'|=1} \| \partial^{m-m'}_x B \|_{H^1_x} \| \nabla_v \partial^{m'}_{\alpha-\alpha'} \mathbb{P}^\perp G \|_{H^1_x L^2_v} \| \partial^m_\alpha \mathbb{P}^\perp G \|_{L^2_{x,v}} \\
	    \leq & \tfrac{C}{\eps} \| B \|_{H^s_x} \| \mathbb{P}^\perp G \|_{\widetilde{H}^s_{x,v}} \| \partial^m_\alpha \mathbb{P}^\perp G \|_{L^2_{x,v}} \\
	    \leq & \tfrac{C}{\eps} \| B \|_{H^s_x} \| \mathbb{P}^\perp G \|_{\widetilde{H}^s_{x,v}(\nu)} \| \partial^m_\alpha \mathbb{P}^\perp G \|_{L^2_{x,v}(\nu)} \,,
	  \end{aligned}
	\end{equation}
	where the H\"older inequality, the Sobolev embedding $H^1_x (\T^3) \hookrightarrow L^4_x (\T^3)$ and the part (1) of Lemma \ref{Lmm-nu-norm}. In the term $M_{72}$, if $m' = 0$, we have $|m| \geq |m'|+1 = 1$ and $|\alpha| \leq s - |m| \leq s - 1$. So we need to control the quantity $ - \frac{1}{\eps} \big\langle \mathsf{q} ( v \times \partial^m_x B ) \cdot \nabla_v \partial^\alpha_v \mathbb{P}^\perp G , \partial^m_\alpha \mathbb{P}^\perp G \big\rangle_{L^2_{x,v}} $. For the case $|m| = s - 1$, $|\alpha| \leq s - |m| = 1$. Then we have
	\begin{equation}\label{M72-1}
	  \begin{aligned}
	    & - \frac{1}{\eps} \l \mathsf{q} ( v \times \partial^m_x B ) \cdot \nabla_v \partial^\alpha_v \mathbb{P}^\perp G , \partial^m_\alpha \mathbb{P}^\perp G \r_{L^2_{x,v}} \\
	    \leq & \tfrac{C}{\eps} \big\| \tfrac{v}{\nu} \big\|_{L^\infty_v} \| \partial^m_x B \|_{L^4_x} \| \nabla_v \partial^\alpha_v \mathbb{P}^\perp G \|_{L^4_x L^2_v (\nu)} \| \partial^m_\alpha \mathbb{P}^\perp G \|_{L^2_{x,v}(\nu)} \\
	    \leq & \tfrac{C}{\eps} \| \partial^m_x B \|_{H^1_x} \| \nabla_v \partial^\alpha_v \mathbb{P}^\perp G \|_{H^1_x L^2_v (\nu)} \| \partial^m_\alpha \mathbb{P}^\perp G \|_{L^2_{x,v}(\nu)} \\
	    \leq & \tfrac{C}{\eps} \| B \|_{H^s_x} \| \mathbb{P}^\perp G \|_{\widetilde{H}^3_{x,v}(\nu)} \| \partial^m_\alpha \mathbb{P}^\perp G \|_{L^2_{x,v}(\nu)} \\
	    \leq & \tfrac{C}{\eps} \| B \|_{H^s_x} \| \mathbb{P}^\perp G \|_{\widetilde{H}^s_{x,v}(\nu)} \| \partial^m_\alpha \mathbb{P}^\perp G \|_{L^2_{x,v}(\nu)} \,,
	  \end{aligned}
	\end{equation}
	where we make use of the H\"older inequality, the Sobolev embedding $H^1_x (\T^3) \hookrightarrow L^4_x (\T^3)$ and the part (1) of Lemma \ref{Lmm-CF-nu}. Here $s \geq 3$ is required. For the case $1 \leq |m| \leq s - 2$, we know $|\alpha| + 1 \leq s + 1 - |m| \leq s$. Then we derive from the Sobolev embedding $H^2_x (\T^3) \hookrightarrow L^\infty_x (\T^3)$ and Lemma \ref{Lmm-CF-nu} (1) that
	\begin{equation}\label{M72-2}
	  \begin{aligned}
	    & - \frac{1}{\eps} \l \mathsf{q} ( v \times \partial^m_x B ) \cdot \nabla_v \partial^\alpha_v \mathbb{P}^\perp G , \partial^m_\alpha \mathbb{P}^\perp G \r_{L^2_{x,v}} \\
	    \leq & \tfrac{C}{\eps} \big\| \tfrac{v}{\nu (v)} \big\|_{L^\infty_v} \| \partial^m_x B \|_{L^\infty_x} \| \nabla_v \partial^\alpha_v \mathbb{P}^\perp G \|_{L^2_{x,v}(\nu)} \| \partial^m_\alpha \mathbb{P}^\perp G \|_{L^2_{x,v}(\nu)} \\
	    \leq & \tfrac{C}{\eps} \| \partial^m_x B \|_{H^2_x} \| \nabla_v \partial^\alpha_v \mathbb{P}^\perp G \|_{L^2_{x,v}(\nu)} \| \partial^m_\alpha \mathbb{P}^\perp G \|_{L^2_{x,v}(\nu)} \\
	    \leq & \tfrac{C}{\eps} \| B \|_{H^s_x} \| \mathbb{P}^\perp G \|_{\widetilde{H}^s_{x,v}(\nu)} \| \partial^m_\alpha \mathbb{P}^\perp G \|_{L^2_{x,v}(\nu)} \,.
	  \end{aligned}
	\end{equation}
	In the term $M_{72}$, if $m' \neq 0$, $m > m' \neq 0$ and $|m| + |\alpha| \leq s$ with $\alpha \neq 0$ imply that
	\begin{equation*}
	  \begin{aligned}
	    & |m-m'| = |m| - |m'| \leq |m| - 1 \leq s - 2 \,, \\
	    & |m'| + |\alpha| + 1 \leq |m| - 1 + |\alpha| + 1 \leq s \,.
	  \end{aligned}
	\end{equation*}
	Then we estimate that
	\begin{equation}\label{M72-3}
	  \begin{aligned}
	    & - \tfrac{1}{\eps} \sum_{0 \neq m' < m} C_m^{m'} \l \mathsf{q} ( v \times \partial^{m-m'}_x B ) \cdot \nabla_v \partial^{m'}_\alpha \mathbb{P}^\perp G , \partial^m_\alpha \mathbb{P}^\perp G \r_{L^2_{x,v}} \\
	    \leq & \tfrac{C}{\eps} \sum_{0 \neq m' < m} \big\| \tfrac{v}{\nu(v)} \big\|_{L^\infty_v} \| \partial^{m-m'}_x B \|_{L^\infty_x} \| \nabla_v \partial^{m'}_\alpha \mathbb{P}^\perp G \|_{L^2_{x,v}(\nu)} \| \partial^m_\alpha \mathbb{P}^\perp G \|_{L^2_{x,v}(\nu)} \\
	    \leq & \tfrac{C}{\eps} \sum_{0 \neq m' < m} \| \partial^{m-m'}_x B \|_{H^2_x} \| \nabla_v \partial^{m'}_\alpha \mathbb{P}^\perp G \|_{L^2_{x,v}(\nu)} \| \partial^m_\alpha \mathbb{P}^\perp G \|_{L^2_{x,v}(\nu)} \\
	    \leq & \tfrac{C}{\eps} \| B \|_{H^s_x} \| \mathbb{P}^\perp G \|_{\widetilde{H}^s_{x,v}(\nu)} \| \partial^m_\alpha \mathbb{P}^\perp G \|_{L^2_{x,v}(\nu)} \,,
	  \end{aligned}
	\end{equation}
	where the Sobolev embedding $H^2_x (\T^3) \hookrightarrow L^\infty_x (\T^3)$ and Lemma \ref{Lmm-CF-nu} (1) are utilized. Therefore, the bounds \eqref{M72-1}, \eqref{M72-2} and \eqref{M72-3} reduce to
	\begin{equation}\label{M72}
	  \begin{aligned}
	    M_{72} \leq \tfrac{C}{\eps} \| B \|_{H^s_x} \| \mathbb{P}^\perp G \|_{\widetilde{H}^s_{x,v}(\nu)} \| \partial^m_\alpha \mathbb{P}^\perp G \|_{L^2_{x,v}(\nu)} 
	  \end{aligned}
	\end{equation} 
	for all $|m|+|\alpha| \leq s$ and $\alpha \neq 0$. As a consequence, substituting the bounds \eqref{M71} and \eqref{M72} into the equality \eqref{M7-Decomp} tells us
	\begin{equation}\label{M7}
	  \begin{aligned}
	    M_7 \leq \tfrac{C}{\eps} \| B \|_{H^s_x} \| \mathbb{P}^\perp G \|_{\widetilde{H}^s_{x,v}(\nu)} \| \partial^m_\alpha \mathbb{P}^\perp G \|_{L^2_{x,v}(\nu)} \,.
	  \end{aligned}
	\end{equation}
	
	Next we estimate the term $M_8$ in the relation \eqref{M}. Recalling the expression of $\mathbb{P} G$ in \eqref{VMB-Proj-Smp}, i.e.,
	\begin{equation*}
	  \begin{aligned}
	    \mathbb{P} G = \rho^+ \phi_1 (v) + \rho^- \phi_2 (v) + \sum_{i=1}^3 u_i \phi_{i+2} (v) + \theta \phi_6 (v) \,,
	  \end{aligned}
	\end{equation*}
	we estimate the term $M_8$ as
	\begin{equation}\label{M8-1}
	  \begin{aligned}
	    M_8 = & - \l \partial_t ( \partial^m_x \rho^+ \phi_1 + \partial^m_x \rho^- \phi_2 + \sum_{i=1}^3 \partial^m_x u_i \phi_{i+2} + \partial^m_x \theta \phi_6 ) , \partial^m_\alpha \mathbb{P}^\perp G \r_{L^2_{x,v}} \\
	    \leq & C \big( \| \partial_t \partial^m_x \rho^+ \|_{L^2_x} + \| \partial_t \partial^m_x \rho^- \|_{L^2_x} + \| \partial_t \partial^m_x u \|_{L^2_x} + \| \partial_t \partial^m_x \theta \|_{L^2_x}  \big) \| \partial^m_\alpha \mathbb{P}^\perp G \|_{L^2_{x,v}} \,,
	  \end{aligned}
	\end{equation}
	where we make use of the fact $\phi_i (v) \in L^2_v$ for $1 \leq i \leq 6$.	Recall that the hydrodynamic coefficients $\rho^\pm$, $u$ and $\theta$ obey the relations \eqref{Deco-6-moments}, namely,
	\begin{equation}\label{M8-2}
	  \left\{
	    \begin{array}{l}
	      \partial_t \rho^+ = \tfrac{1}{\eps} \big[ - \div_x u + \l \Theta ( \mathbb{P}^\perp G ) + \Psi (\mathbb{P} G) , \phi_1 (v) \r_{L^2_v} \big] \,, \\ [2mm]
	      \partial_t \rho^- = \tfrac{1}{\eps} \big[ - \div_x u + \l \Theta ( \mathbb{P}^\perp G ) + \Psi (\mathbb{P} G) , \phi_2 (v) \r_{L^2_v} \big] \,, \\ [2mm]
	      \partial_t u_i = \tfrac{1}{\eps} \big[ - \partial_i ( \tfrac{\rho^+ + \rho^-}{2} + \theta ) + \l \Theta ( \mathbb{P}^\perp G ) + \Psi (\mathbb{P} G) , \phi_{i+2} (v) \r_{L^2_v} \big] \ \ \textrm{for} \ \ 1 \leq i \leq 3 \,, \\ [2mm]
	      \partial_t \theta = \tfrac{1}{\eps} \big[ - \tfrac{2}{3} \div_x u + \tfrac{1}{3} \l \Theta ( \mathbb{P}^\perp G ) + \Psi (\mathbb{P} G) , \phi_6 (v) \r_{L^2_v} \big] \,.
	    \end{array}
	  \right.
	\end{equation}
	For the quantity $\| \partial_t \partial^m_x \rho^+ \|_{L^2_x}$, we derive from the first equation of \eqref{M8-2} and the expression of $\mathbb{P} G$ in \eqref{VMB-Proj-Smp} that for all $|m| \leq s - 1$
	\begin{equation}\label{M8-3}
	  \begin{aligned}
	    & \| \partial_t \partial^m_x \rho^+ \|_{L^2_x} \leq \tfrac{1}{\eps} \| \div_x \partial^m_x u \|_{L^2_x} + \tfrac{1}{\eps} \left\| \partial^m_x \l \Theta (\mathbb{P}^\perp G ) + \Psi ( \mathbb{P} G ) , \phi_1 (v) \r_{L^2_v} \right\|_{L^2_x} \\
	    \leq & \tfrac{1}{\eps} \| \nabla_x \mathbb{P} G \|_{H^{s-1}_x L^2_v} + \underset{M_{81}}{\underbrace{ \tfrac{1}{\eps} \big\| \partial^m_x \langle \Theta (\mathbb{P}^\perp G) , \phi_1 (v) \rangle_{L^2_v} \big\|_{L^2_x} }} + \underset{M_{82}}{\underbrace{ \tfrac{1}{\eps} \big\| \partial^m_x \langle \Psi (\mathbb{P} G) , \phi_1 (v) \rangle_{L^2_v} \big\|_{L^2_x} }} \,.
	  \end{aligned}
	\end{equation}
	Recalling that $\phi_1 (v) \in \textrm{Ker} (\mathscr{L})$, $( \partial_t + \tfrac{1}{\eps} \mathscr{L} ) \mathbb{P}^\perp G \in \textrm{Ker}^\perp (\mathscr{L})$ and $ \Theta (\mathbb{P}^\perp G) $ defined in \eqref{Theta-Pperp-G}, i.e.,
	\begin{equation*}
	  \begin{aligned}
	    \Theta(\mathbb{P}^\perp G) = - \left( \eps \partial_t + v \cdot \nabla_x + \tfrac{1}{\eps} \mathscr{L} + \mathsf{q} ( \eps E + v \times B ) \cdot \nabla_v - \tfrac{1}{2} \eps \mathsf{q} ( E \cdot v ) \right) \mathbb{P}^\perp G \,,
	  \end{aligned}
	\end{equation*}
	one easily has
	\begin{equation}\label{M81-0}
	  \begin{aligned}
	    M_{81} = & \tfrac{1}{\eps} \big\| \partial^m_x \langle ( v \cdot \nabla_x + \mathsf{q} (\eps E + v \times B) \cdot \nabla_v - \tfrac{1}{2} \eps \mathsf{q} (E \cdot v) ) \mathbb{P}^\perp G , \phi_1 (v) \rangle_{L^2_v} \big\|_{L^2_x} \\
	    \leq & \underset{M_{811}}{\underbrace{ \tfrac{1}{\eps} \big\| \langle v \cdot \nabla_x \partial^m_x \mathbb{P}^\perp G , \phi_1(v) \rangle_{L^2_v} \big\|_{L^2_x} }} + \underset{M_{812}}{\underbrace{ \tfrac{1}{\eps} \big\| \partial^m_x \langle \eps \mathsf{q} E \cdot \nabla_v \mathbb{P}^\perp G , \phi_1 (v) \rangle_{L^2_v} \big\|_{L^2_x} }} \\
	    + & \underset{M_{813}}{\underbrace{ \tfrac{1}{\eps} \big\| \partial^m_x \langle \mathsf{q} (v \times B) \cdot \nabla_v \mathbb{P}^\perp G , \phi_1 (v) \rangle_{L^2_v} \big\|_{L^2_x} }} + \underset{M_{814}}{\underbrace{ \tfrac{1}{\eps} \big\| \partial^m_x \langle \tfrac{1}{2} \eps \mathsf{q} (E \cdot v) \mathbb{P}^\perp G , \phi_1 (v) \rangle_{L^2_v} \big\|_{L^2_x} }} 
	  \end{aligned}
	\end{equation}
	for $|m| \leq s - 1$. Since $\phi_1 (v)$ contains the exponential decay factor $\exp \big( - \tfrac{|v|^2}{4} \big)$, we know that $p(v) \phi_1 (v) \in L^2_v$ for any polynomial $p(v)$. As a result, we derive from the H\"older inequality and Lemma \ref{Lmm-nu-norm} (1) that the term $M_{811}$ is bounded by
	\begin{equation}\label{M811}
	  \begin{aligned}
	    M_{811} \leq \tfrac{1}{\eps} \| v \otimes \phi_1 (v) \|_{L^2_v} \| \nabla_x \partial^m_x \mathbb{P}^\perp G \|_{L^2_{x,v}} \leq \tfrac{C}{\eps} \| \mathbb{P}^\perp G \|_{H^s_x L^2_v} \leq \tfrac{C}{\eps} \| \mathbb{P}^\perp G \|_{H^s_x L^2_v (\nu)} \,.
	  \end{aligned}
	\end{equation}
	In the term $M_{812}$, since $|m| \leq s - 1$, the relation $|m-m'|+1 = |m| - |m'| + 1 \leq |m| \leq s - 1 $ holds if $m' \neq 0$. Then we compute that
	\begin{equation}\label{M812}
	  \begin{aligned}
	    M_{812} \leq & \sum_{m' \leq m} C_m^{m'} \big\| \langle \partial^{m'}_x E \cdot \nabla_v \partial^{m-m'}_x \mathbb{P}^\perp G , \mathsf{q} \phi_1 (v) \rangle_{L^2_v} \big\|_{L^2_x} \\
	    \leq & C \| E \|_{L^\infty_x} \| \nabla_v \partial^m_x \mathbb{P}^\perp G \|_{L^2_{x,v}} \| \mathsf{q} \phi_1 (v) \|_{L^2_v} \\
	    + & C \sum_{0 \neq m' \leq m} \| \partial^{m'}_x E \|_{L^4_x} \| \nabla_v \partial^{m-m'}_x \mathbb{P}^\perp G \|_{L^4_x L^2_v} \| \mathsf{q} \phi_1 (v) \|_{L^2_v} \\
	    \leq & C \| E \|_{H^2_x} \| \nabla_v \partial^m_x \mathbb{P}^\perp G \|_{L^2_{x,v}} + C \sum_{0 \neq m' \leq m} \| \partial^{m'}_x E \|_{H^1_x} \| \nabla_v \partial^{m-m'}_x \mathbb{P}^\perp G \|_{H^1_x L^2_v} \\
	    \leq & C \| E \|_{H^s_x} \sum_{m' \leq m} \| \nabla_v \partial^{m'}_x \mathbb{P}^\perp G \|_{L^2_{x,v}(\nu)} \\
	    \leq & C \| E \|_{H^s_x} \| \mathbb{P}^\perp G \|_{\widetilde{H}^s_{x,v}(\nu)} \,,
	  \end{aligned}
	\end{equation}
	where we make use of the H\"older inequality, the Sobolev embedding $H^2_x (\T^3) \hookrightarrow L^\infty_x (\T^3)$, $H^1_x (\T^3) \hookrightarrow L^4_x (\T^3)$ and the part (1) of Lemma \ref{Lmm-nu-norm}. Here $s \geq 2$ is required. In the term $M_{813}$, since $|m'| , |m-m'| \leq |m| \leq s - 1$, we derive from the integration by parts over $v \in \T^3$, the Sobolev embedding $H^1_x (\T^3) \hookrightarrow L^4_x (\T^3)$, the H\"older inequality and the part (1) of Lemma \ref{Lmm-nu-norm} that
	\begin{equation}\label{M813}
	  \begin{aligned}
	    M_{813} = & \tfrac{1}{\eps} \big\| - \partial^m_x \langle \mathsf{q} ( v \times B ) \otimes \mathbb{P}^\perp G , \nabla_v \phi_1 (v) \rangle_{L^2_v}  \big\|_{L^2_x} \\
	    \leq & \tfrac{1}{\eps} \sum_{m' \leq m} C_m^{m'} \big\| \langle \partial^{m-m'}_x B \otimes \partial^{m'}_x \mathbb{P}^\perp G , \mathsf{q} (v \times \nabla_v) \phi_1 (v) \rangle_{L^2_v} \big\|_{L^2_x} \\
	    \leq & \tfrac{C}{\eps} \sum_{m' \leq m} \| \partial^{m-m'}_x B \|_{L^4_x} \| \partial^{m'}_x \mathbb{P}^\perp G \|_{L^4_x L^2_v} \| \mathsf{q} ( v \times \nabla_v ) \phi_1 (v) \|_{L^2_v} \\
	    \leq & \tfrac{C}{\eps} \sum_{m' \leq m} \| \partial^{m-m'}_x B \|_{H^1_x} \| \partial^{m'}_x \mathbb{P}^\perp G \|_{H^1_x L^2_v} \\
	    \leq & \tfrac{C}{\eps} \| B \|_{H^s_x} \| \mathbb{P}^\perp G \|_{H^s_x L^2_v} \leq \tfrac{C}{\eps} \| B \|_{H^s_x} \| \mathbb{P}^\perp G \|_{H^s_x L^2_v(\nu)} \,.
	  \end{aligned}
	\end{equation}
	The term $M_{814}$ in \eqref{M81-0} can be estimated by
	\begin{equation}\label{M814}
	  \begin{aligned}
	    M_{814} \leq &  \tfrac{1}{2} \sum_{m' \leq m} C_m^{m'} \big\| \langle \mathsf{q} ( \partial^{m'}_x E \cdot v ) \partial^{m-m'}_x \mathbb{P}^\perp G , \phi_1 (v) \rangle_{L^2_v} \big\|_{L^2_x} \\
	    \leq & \tfrac{1}{2} \sum_{m' \leq m} C_m^{m'} \| \partial^{m'}_x E \|_{L^4_x} \| \partial^{m-m'}_x \mathbb{P}^\perp G \|_{L^4_x L^2_v} \| v \otimes \phi_1 (v) \|_{L^2_v} \\
	    \leq & C \sum_{m' \leq m} \| \partial^{m'}_x E \|_{H^1_x} \| \partial^{m-m'}_x E \|_{H^1_x} \| \partial^{m-m'}_x \mathbb{P}^\perp G \|_{H^1_x L^2_v} \\
	    \leq & C \| E \|_{H^s_x} \| \mathbb{P}^\perp G \|_{H^s_x L^2_v} \leq C \| E \|_{H^s_x} \| \mathbb{P}^\perp G \|_{H^s_x L^2_v(\nu)} \,,
	  \end{aligned}
	\end{equation}
	where the H\"older inequality, the Sobolev embedding $H^2_x (\T^3) \hookrightarrow L^\infty_x (\T^3)$ and Lemma \ref{Lmm-nu-norm} (1) are utilized. Plugging the bounds \eqref{M811}, \eqref{M812}, \eqref{M813} and \eqref{M814} into \eqref{M81-0} reduces to
	\begin{equation}\label{M81}
	  \begin{aligned}
	    M_{81} \leq \tfrac{C}{\eps} \| \mathbb{P}^\perp G \|_{H^s_x L^2_v(\nu)} + \tfrac{C}{\eps} \big( \eps \| E \|_{H^s_x} + \| B \|_{H^s_x} \big) \big( \| \mathbb{P}^\perp G \|_{H^s_x L^2_v(\nu)} + \| \mathbb{P}^\perp G \|_{\widetilde{H}^s_{x,v}(\nu)} \big) \,.
	  \end{aligned}
	\end{equation}
	Next we estimate the term $M_{82}$ in \eqref{M8-3}. Recalling the expression of $\Psi (\mathbb{P} G)$ defined in \eqref{Psi-P-G}, i.e.,
	\begin{equation*}
	  \begin{aligned}
	    \Psi (\mathbb{P} G) = - \mathsf{q} ( \eps E + v \times B ) \cdot \nabla_v \mathbb{P} G + \tfrac{1}{2} \eps \mathsf{q} ( E \cdot v ) \mathbb{P} G \,,
	  \end{aligned}
	\end{equation*}
	one can derive bound of the $M_{82}$ that
	\begin{equation}\label{M82-0}
	  \begin{aligned}
	    M_{82} \leq & \underset{M_{821}}{\underbrace{ \tfrac{1}{\eps} \big\| \partial^m_x \langle \eps E \cdot \nabla_v \mathbb{P} G , \mathsf{q} \phi_1 (v) \rangle_{L^2_v} \big\|_{L^2_x} }} + \underset{M_{822}}{\underbrace{ \tfrac{1}{\eps} \big\| \partial^m_x \langle (v \times B) \cdot \nabla_v \mathbb{P} G , \mathsf{q} \phi_1 (v) \rangle_{L^2_v} \big\|_{L^2_x} }} \\
	    & + \underset{M_{823}}{\underbrace{ \tfrac{1}{\eps} \big\| \partial^m_x \langle \tfrac{1}{2} \eps ( E \cdot v ) \mathbb{P} G , \mathsf{q} \phi_1 (v) \rangle_{L^2_v} \big\|_{L^2_x} }} \,.
	  \end{aligned}
	\end{equation}
	For the term $M_{821}$ in \eqref{M82-0}, we have
	\begin{equation}\label{M821}
	  \begin{aligned}
	    M_{821} \leq & \sum_{m' \leq m} C_m^{m'} \big\| \langle \partial^{m'}_x E \cdot \nabla_v \mathsf{q} \phi_1 (v) , \partial^{m-m'}_x \mathbb{P} G \rangle_{L^2_v} \big\|_{L^2_x} \\
	    \leq & C \sum_{m' \leq m} \| \partial^{m'}_x E \|_{L^3_x} \| \partial^{m-m'}_x \mathbb{P} G \|_{L^6_x L^2_v} \| \nabla_v \mathsf{q} \phi_1 (v) \|_{L^2_v} \\
	    \leq & C \sum_{m' \leq m} \| \partial^{m'}_x E \|_{H^1_x} \| \nabla_x \partial^{m-m'}_x \mathbb{P} G \|_{L^2_{x,v}} \leq C \| E \|_{H^s_x} \| \nabla_x \mathbb{P} G \|_{H^{s-1}_x L^2_v} \,,
	  \end{aligned}
	\end{equation}
	where we make use of the integration by parts over $v \in \R^3$, the H\"older inequality, the Sobolev embedding $H^1_x (\T^3) \hookrightarrow L^3_x (\T^3)$ and the Sobolev interpolation inequality $\| f \|_{L^6_x} \leq C \| \nabla_x f \|_{L^2_x}$. For the terms $M_{822}$ and $M_{823}$ in \eqref{M82-0}, we similarly estimate that
	\begin{equation}\label{M822}
	  \begin{aligned}
	    M_{822} \leq & \tfrac{1}{\eps} \sum_{m' \leq m} C_m^{m'} \big\| \langle \partial^{m'}_x B \otimes \partial^{m-m'}_x \mathbb{P} G , v \times \nabla_v (\mathsf{q} \phi_1 (v)) \rangle_{L^2_v} \big\|_{L^2_x} \\
	    \leq & \tfrac{C}{\eps} \sum_{m' \leq m} \| \partial^{m'}_x B \|_{L^3_x} \| \partial^{m-m'}_x \mathbb{P} G \|_{L^6_x L^2_v} \| v \times \nabla_v (\mathsf{q} \phi_1 (v)) \|_{L^2_v} \\
	    \leq & \tfrac{C}{\eps}  \sum_{m' \leq m} \| \partial^{m'}_x B \|_{H^1_x} \| \nabla_x \partial^{m-m'}_x \mathbb{P} G \|_{L^2_{x,v}} \\
	    \leq & \tfrac{C}{\eps} \| B \|_{H^s_x} \| \nabla_x \mathbb{P} G \|_{H^{s-1}_x L^2_v}
	  \end{aligned}
	\end{equation}
	and
	\begin{equation}\label{M823}
	  \begin{aligned}
	    M_{823} \leq & \tfrac{1}{2} \sum_{m' \leq m} C_m^{m'} \big\| \langle \partial^{m-m'}_x E \otimes \partial^{m'}_x \mathbb{P} G , v \otimes \mathsf{q} \phi_1 (v) \rangle_{L^2_v} \big\|_{L^2_x} \\
	    \leq & C \sum_{m' \leq m} \| \partial^{m-m'}_x E \|_{L^3_x} \| \partial^{m'}_x \mathbb{P} G \|_{L^6_x L^2_v} \| v \otimes \mathsf{q} \phi_1 (v) \|_{L^6_v} \\
	    \leq & C \sum_{m' \leq m} \| \partial^{m-m'}_x E \|_{H^1_x} \| \nabla_x \partial^{m'}_x \mathbb{P} G \|_{L^2_{x,v}} \leq C \| E \|_{H^s_x} \| \nabla_x \mathbb{P} G \|_{H^{s-1}_x L^2_v} \,.
	  \end{aligned}
	\end{equation}
	Substituting the inequality \eqref{M821}, \eqref{M822} and \eqref{M823} into the relation \eqref{M82-0}, we obtain
	\begin{equation}\label{M82}
	  \begin{aligned}
	    M_{82} \leq \tfrac{C}{\eps} ( \eps \| E \|_{H^s_x} + \| B \|_{H^s_x} ) \| \nabla_x \mathbb{P} G \|_{H^{s-1}_x L^2_v} \,.
	  \end{aligned}
	\end{equation}
	We thereby deduce from plugging the bounds \eqref{M81} and \eqref{M82} into the relation \eqref{M8-3} that
	\begin{equation}\label{M8-4}
	  \begin{aligned}
	    \| & \partial_t \partial^m_x \rho^+ \|_{L^2_x} \leq  \tfrac{C}{\eps} \big( \| \nabla_x \mathbb{P} G \|_{H^{s-1}_x L^2_v} + \| \mathbb{P}^\perp G \|_{H^s_x L^2_v (\nu)} \big) \\
	    & + \tfrac{C}{\eps} ( \eps \| E \|_{H^s_x} + \| B \|_{H^s_x} ) \big( \| \nabla_x \mathbb{P} G \|_{H^{s-1}_x L^2_v} + \| \mathbb{P}^\perp G \|_{H^s_x L^2_v (\nu)} + \| \mathbb{P}^\perp G \|_{\widetilde{H}^s_{x,v}(\nu)} \big) \,.
	  \end{aligned}
	\end{equation}
	Furthermore, via the analogous argument in estimating the norm $\| \partial_t \partial^m_x \rho^+ \|_{L^2_x}$ in \eqref{M8-4}, one can easily yield that
	\begin{equation}\label{M8-5}
	  \begin{aligned}
	    & \| \partial_t \partial^m_x \rho^- \|_{L^2_x} + \| \partial_t \partial^m_x u \|_{L^2_x} + \| \partial_t \partial^m_x \theta \|_{L^2_x} \\
	    \leq & \tfrac{C}{\eps} \big( \| \nabla_x \mathbb{P} G \|_{H^{s-1}_x L^2_v} + \| \mathbb{P}^\perp G \|_{H^s_x L^2_v (\nu)} \big) \\
	    & + \tfrac{C}{\eps} ( \eps \| E \|_{H^s_x} + \| B \|_{H^s_x} ) \big( \| \nabla_x \mathbb{P} G \|_{H^{s-1}_x L^2_v} + \| \mathbb{P}^\perp G \|_{H^s_x L^2_v (\nu)} + \| \mathbb{P}^\perp G \|_{\widetilde{H}^s_{x,v}(\nu)} \big) \,.
	  \end{aligned}
	\end{equation}
	We then plug the bounds \eqref{M8-4} and \eqref{M8-5} into the relation \eqref{M8-2} and obtain
	\begin{equation}\label{M8}
	  \begin{aligned}
	    M_8 \leq & \tfrac{C}{\eps} \big( \| \nabla_x \mathbb{P} G \|_{H^{s-1}_x L^2_v} + \| \mathbb{P}^\perp G \|_{H^s_x L^2_v (\nu)} \big) \| \partial^m_x \mathbb{P}^\perp G \|_{L^2_{x,v}} \\
	    & + \tfrac{C}{\eps} ( \eps \| E \|_{H^s_x} + \| B \|_{H^s_x} ) \| \partial^m_x \mathbb{P}^\perp G \|_{L^2_{x,v}} \\
	    & \qquad \times \big( \| \nabla_x \mathbb{P} G \|_{H^{s-1}_x L^2_v} + \| \mathbb{P}^\perp G \|_{H^s_x L^2_v (\nu)} + \| \mathbb{P}^\perp G \|_{\widetilde{H}^s_{x,v}(\nu)} \big) \\
	    \leq & C_\eta  \big( \| \nabla_x \mathbb{P} G \|^2_{H^{s-1}_x L^2_v} + \| \mathbb{P}^\perp G \|^2_{H^s_x L^2_v (\nu)} \big) + \tfrac{\eta}{\eps^2} \| \partial^m_x \mathbb{P}^\perp G \|^2_{L^2_{x,v}(\nu)} \\
	    & + \tfrac{C}{\eps} ( \eps \| E \|_{H^s_x} + \| B \|_{H^s_x} ) \| \partial^m_x \mathbb{P}^\perp G \|_{L^2_{x,v}(\nu)} \\
	    & \qquad \times \big( \| \nabla_x \mathbb{P} G \|_{H^{s-1}_x L^2_v} + \| \mathbb{P}^\perp G \|_{H^s_x L^2_v (\nu)} + \| \mathbb{P}^\perp G \|_{\widetilde{H}^s_{x,v}(\nu)} \big) 
	  \end{aligned}
	\end{equation}
	for small $\eta > 0$ to be determined, where the Young's inequality and Lemma \ref{Lmm-nu-norm} are utilized.
	
	Next we estimate the term $M_9$ in \eqref{M}. We first decompose it as three parts and then estimate them term by term. More precisely,
	\begin{equation}\label{M9-1}
	  \begin{aligned}
	    M_9 = & \underset{M_{91}}{\underbrace{ - \tfrac{1}{\eps} \l v \cdot \nabla_x \partial^m_\alpha \mathbb{P} G , \partial^m_\alpha \mathbb{P}^\perp G \r_{L^2_{x,v}} }} \ \underset{M_{92}}{\underbrace{ - \l \mathsf{q} \partial^m_\alpha ( E \cdot \nabla_v \mathbb{P} G ) , \partial^m_\alpha \mathbb{P}^\perp G \r_{L^2_{x,v}} }} \\
	    & \underset{M_{93}}{\underbrace{ - \tfrac{1}{\eps} \l \mathsf{q} \partial^m_\alpha \big( (v\times B) \cdot \nabla_v \mathbb{P} G \big) , \partial^m_\alpha \mathbb{P}^\perp G \r_{L^2_{x,v}} }} \,.
	  \end{aligned}
	\end{equation}
	Since $|m| \leq s - |\alpha| \leq s - 1$, it is easily derived from the H\"older inequality, the bound \eqref{G-Macro}, the Young's inequality and Lemma \ref{Lmm-nu-norm} (1) that
	\begin{equation}\label{M91}
	  \begin{aligned}
	    M_{91} \leq & \tfrac{1}{\eps} \| v \cdot \nabla_x \partial^m_\alpha \mathbb{P} G \|_{L^2_{x,v}} \| \partial^m_\alpha \mathbb{P}^\perp G \|_{L^2_{x,v}} \\
	    \leq & \tfrac{C}{\eps} \| \nabla_x \partial^m_x \mathbb{P} G \|_{L^2_{x,v}} \| \partial^m_\alpha \mathbb{P}^\perp G \|_{L^2_{x,v}} \\
	    \leq & \tfrac{C}{\eps} \| \nabla_x \mathbb{P} G \|_{H^{s-1}_x L^2_v} \| \partial^m_\alpha \mathbb{P}^\perp G \|_{L^2_{x,v}(\nu)} \\
	    \leq & C_\gamma \| \nabla_x \mathbb{P} G \|^2_{H^{s-1}_x L^2_v} + \tfrac{\gamma}{\eps^2} \| \partial^m_\alpha \mathbb{P}^\perp G \|^2_{L^2_{x,v}(\nu)}
	  \end{aligned}
	\end{equation}
	for small $\gamma > 0$ to be determined. For the term $M_{92}$ in \eqref{M9-1}, we employ the H\"older inequality, the Sobolev embedding $H^1_x (\T^3) \hookrightarrow L^3_x (\T^3)$, the Sobolev interpolation inequality $\| f \|_{L^6_x} \leq C \| \nabla_x f \|_{L^2_x}$, the bound \eqref{G-Macro} and the part (1) of Lemma \ref{Lmm-nu-norm} to estimate
	\begin{equation}\label{M92}
	  \begin{aligned}
	    M_{92} = & - \sum_{m' \leq m} C_m^{m'} \l \mathsf{q} ( \partial^{m'}_x E \cdot \nabla_v \partial^{m-m'}_\alpha \mathbb{P} G ) , \partial^m_\alpha \mathbb{P}^\perp G \r_{L^2_{x,v}} \\
	    \leq & C \sum_{m' \leq m} \| \partial^{m'}_x E \|_{L^3_x} \| \nabla_v \partial^{m-m'}_\alpha \mathbb{P} G \|_{L^6_x L^2_v} \| \partial^m_\alpha \mathbb{P}^\perp G \|_{L^2_{x,v}} \\
	    \leq & C \sum_{m' \leq m} \| \partial^{m'}_x E \|_{H^1_x} \| \nabla_x \nabla_v \partial^{m-m'}_\alpha \mathbb{P} G \|_{L^2_{x,v}} \| \partial^m_\alpha \mathbb{P}^\perp G \|_{L^2_{x,v}} \\
	    \leq & C \| E \|_{H^s_x} \sum_{m' \leq m} \| \nabla_x \partial^{m-m'}_x \mathbb{P} G \|_{L^2_{x,v}} \| \partial^m_\alpha \mathbb{P}^\perp G \|_{L^2_{x,v}(\nu)} \\
	    \leq & C \| E \|_{H^s_x} \| \nabla_x \mathbb{P} G \|_{H^{s-1}_x L^2_v} \| \partial^m_\alpha \mathbb{P}^\perp G \|_{L^2_{x,v}(\nu)} \,.
	  \end{aligned}
	\end{equation}
	Similarly in \eqref{M92}, one can calculate that
	\begin{equation}\label{M93}
	  \begin{aligned}
	    M_{93} = & \tfrac{1}{\eps} \sum_{m' \leq m} C_m^{m'} \l \mathsf{q} \partial^{m-m'}_x B \cdot \partial^{m'}_\alpha ( v \times \nabla_v \mathbb{P} G ) , \partial^m_\alpha \mathbb{P}^\perp G \r_{L^2_{x,v}} \\
	    \leq & \tfrac{C}{\eps} \sum_{m' \leq m} \| \partial^{m-m'}_x B \|_{L^3_x} \| \partial^{m'}_\alpha ( v \times \nabla_v \mathbb{P} G ) \|_{L^6_x L^2_v} \| \partial^m_\alpha \mathbb{P}^\perp G \|_{L^2_{x,v}} \\
	    \leq & \tfrac{C}{\eps} \sum_{m' \leq m} \| \partial^{m-m'}_x B \|_{H^1_x} \| \nabla_x \partial^{m'}_\alpha ( v \times \nabla_v \mathbb{P} G ) \|_{L^2_{x,v}} \| \partial^m_\alpha \mathbb{P}^\perp G \|_{L^2_{x,v}} \\
	    \leq & \tfrac{C}{\eps} \sum_{m' \leq m} \| \partial^{m-m'}_x B \|_{H^1_x} \| \nabla_x \partial^{m'}_x \mathbb{P} G \|_{L^2_{x,v}} \| \partial^m_\alpha \mathbb{P}^\perp G \|_{L^2_{x,v}(\nu)} \\
	    \leq & \tfrac{C}{\eps} \| B \|_{H^s_x} \| \nabla_x \mathbb{P} G \|_{H^{s-1}_x L^2_v} \| \partial^m_\alpha \mathbb{P}^\perp G \|_{L^2_{x,v}(\nu)} \,.
	  \end{aligned}
	\end{equation}
	Substituting the bounds \eqref{M91}, \eqref{M92} and \eqref{M93} into the equality \eqref{M9-1}, we obtain that
	\begin{equation}\label{M9}
	  \begin{aligned}
	    M_9 \leq & \tfrac{\gamma}{\eps^2} \| \partial^m_\alpha \mathbb{P}^\perp G \|^2_{L^2_{x,v}(\nu)} + C_\gamma \| \nabla_x \mathbb{P} G \|^2_{H^{s-1}_x L^2_v} \\
	    + & \tfrac{C}{\eps} \big( \eps \| E \|_{H^s_x} + \| B \|_{H^s_x} \big) \| \nabla_x \mathbb{P} G \|_{H^{s-1}_x L^2_v} \| \partial^m_\alpha \mathbb{P}^\perp G \|_{L^2_{x,v}(\nu)}
	  \end{aligned}
	\end{equation}
	holds for all $|m| + |\alpha| \leq s$ and $\alpha \neq 0$, where $\gamma > 0$ is small to be determined.
	
	Finally, we plug the bounds \eqref{M1}, \eqref{M2}, \eqref{M3}, \eqref{M4}, \eqref{M5}, \eqref{M6}, \eqref{M7}, \eqref{M8} and \eqref{M9} into the equality \eqref{M} and then obtain
	\begin{equation}\label{Mix-1}
	  \begin{aligned}
	    & \tfrac{1}{2} \tfrac{\d}{\d t} \| \partial^m_\alpha \mathbb{P}^\perp G \|^2_{L^2_{x,v}} + \tfrac{\lambda_0 - 4 \gamma}{\eps^2} \| \partial^m_\alpha \mathbb{P}^\perp G \|^2_{L^2_{x,v}(\nu)} \\
	    \leq & \tfrac{\lambda_1}{\eps^2} \sum_{\alpha' < \alpha} \| \partial^m_{\alpha'} \mathbb{P}^\perp G \|^2_{L^2_{x,v}(\nu)} + C_\gamma \sum_{\substack{ |m'| + |\alpha'| \leq s \\ \alpha' < \alpha }} \| \partial^{m'}_{\alpha'} \mathbb{P}^\perp G \|^2_{L^2_{x,v}(\nu)} \\
	    + & C_\gamma \big( \| E \|^2_{H^{s-1}_x} + \| \mathbb{P}^\perp G \|^2_{H^s_x L^2_v (\nu)} + 2 \| \nabla_x \mathbb{P} G \|^2_{H^{s-1}_x L^2_v} \big) \\
	    + & \tfrac{C}{\eps} ( 1 + \| E \|_{H^s_x} + \| B \|_{H^s_x} ) ( \| \mathbb{P} G \|_{H^s_x L^2_v} + \| \mathbb{P}^\perp G \|_{H^s_{x,v}} ) \\
	    & \qquad \times ( \| E \|_{H^{s-1}_x} + \| \nabla_x B \|_{H^{s-2}_x} ) \| \partial^m_\alpha \mathbb{P}^\perp G \|_{L^2_{x,v}(\nu)} \\
	    + & \tfrac{C}{\eps} \big( \eps \| E \|_{H^s_x} + \| B \|_{H^s_x} + \| \mathbb{P} G \|_{H^s_x L^2_v} + \| \mathbb{P}^\perp G \|_{H^s_{x,v}} \big) \\
	    &  \times \big( \| \nabla_x \mathbb{P} G \|_{H^{s-1}_x L^2_v} + \| \mathbb{P}^\perp G \|_{H^s_x L^2_v (\nu)} + \| \mathbb{P}^\perp G \|_{\widetilde{H}^s_{x,v}(\nu)} \big) \| \partial^m_\alpha \mathbb{P}^\perp G \|_{L^2_{x,v}(\nu)}
	  \end{aligned}
	\end{equation}
	for all $0 < \eps \leq 1$, $|m|+|\alpha| \leq s$ with $\alpha \neq 0$ and for small $\gamma > 0 $ to be determined. We take $\gamma = \tfrac{1}{16} \lambda_0 > 0$ and then we derive from the previous bound \eqref{Mix-1} and the Young's inequality that
	\begin{equation}\label{Mix-2}
	  \begin{aligned}
	    & \tfrac{1}{2} \tfrac{\d}{\d t} \| \partial^m_\alpha \mathbb{P}^\perp G \|^2_{L^2_{x,v}} + \tfrac{\lambda_0}{2 \eps^2} \| \partial^m_\alpha \mathbb{P}^\perp G \|^2_{L^2_{x,v}(\nu)} \\
	    \leq & \tfrac{C}{\eps^2} \sum_{\substack{ |m'| + |\alpha'| \leq s \\ \alpha' < \alpha }} \| \partial^{m'}_{\alpha'} \mathbb{P}^\perp G \|^2_{L^2_{x,v}(\nu)} + C \big( \| E \|^2_{H^{s-1}_x} + \| \mathbb{P}^\perp G \|^2_{H^s_x L^2_v(\nu)} + \| \nabla_x \mathbb{P} G \|^2_{H^{s-1}_x L^2_v} \big) \\
	    + & C ( 1 + \| E \|^2_{H^s_x} + \| B \|^2_{H^s_x} ) ( \| \mathbb{P} G \|^2_{H^s_x L^2_v} + \| \mathbb{P}^\perp G \|^2_{H^s_{x,v}} ) ( \| E \|^2_{H^{s-1}_x} + \| \nabla_x B \|^2_{H^{s-2}_x} )  \\
	    + & C \big( \eps^2 \| E \|^2_{H^s_x} + \| B \|^2_{H^s_x} + \| \mathbb{P} G \|^2_{H^s_x L^2_v} + \| \mathbb{P}^\perp G \|^2_{H^s_{x,v}} \big) \\
	    &  \times \big( \| \nabla_x \mathbb{P} G \|^2_{H^{s-1}_x L^2_v} + \| \mathbb{P}^\perp G \|^2_{H^s_x L^2_v (\nu)} + \| \mathbb{P}^\perp G \|^2_{\widetilde{H}^s_{x,v}(\nu)} \big) 
	  \end{aligned}
	\end{equation}
	for all $0 < \eps \leq 1$ and $|m|+|\alpha| \leq s$ with $\alpha \neq 0$. 
	
	Recalling the definitions of the energy functionals $\mathscr{E}_\eta (G,E,B)$ and $\mathscr{D}_\eta (G, E , B)$ in \eqref{E-D-eta-Energies}, one can obtain that the following inequalities
	\begin{equation}\label{Mix-3}
	  \begin{aligned}
	    & ( 1 + \| E \|^2_{H^s_x} + \| B \|^2_{H^s_x} ) ( \| \mathbb{P} G \|^2_{H^s_x L^2_v} + \| \mathbb{P}^\perp G \|^2_{H^s_{x,v}} ) \\
	    & \qquad \qquad \qquad \qquad \qquad \qquad  \leq C ( 1 + \mathscr{E}_\eta (G,E,B) ) ( \mathscr{E}_\eta (G,E,B) + \| \mathbb{P}^\perp G \|^2_{\widetilde{H}^s_{x,v}} ) \,, \\
	    & \eps^2 \| E \|^2_{H^s_x} + \| B \|^2_{H^s_x} + \| \mathbb{P} G \|^2_{H^s_x L^2_v} + \| \mathbb{P}^\perp G \|^2_{H^s_{x,v}} \leq C ( \mathscr{E}_\eta (G,E,B) + \| \mathbb{P}^\perp G \|^2_{\widetilde{H}^s_{x,v}} ) \,, \\
	    & \| E \|^2_{H^{s-1}_x} + \| \nabla_x B \|^2_{H^{s-2}_x} + \| \nabla_x \mathbb{P} G \|^2_{H^{s-2}_x L^2_v} + \| \mathbb{P}^\perp G \|^2_{H^s_x L^2_v (\nu)} \leq C ( 1 + \tfrac{1}{\eta} ) \mathscr{D}_\eta (G,E,B)
	  \end{aligned}
	\end{equation}
	hold for all $0 < \eta \leq \eta_0$ and $0 < \eps \leq 1$. Then the relations \eqref{Mix-2} and \eqref{Mix-3} reduce to
	\begin{equation}\label{Mix-4}
	  \begin{aligned}
	    & \tfrac{1}{2} \tfrac{\d}{\d t} \| \partial^m_\alpha \mathbb{P}^\perp G \|^2_{L^2_{x,v}} + \tfrac{\lambda_0}{2 \eps^2} \| \partial^m_\alpha \mathbb{P}^\perp G \|^2_{L^2_{x,v}(\nu)} \\
	    \leq & \tfrac{C}{\eps^2} \sum_{\substack{ |m'| + |\alpha'| \leq s \\ \alpha' < \alpha }} \| \partial^{m'}_{\alpha'} \mathbb{P}^\perp G \|^2_{L^2_{x,v}(\nu)} + C ( 1 + \tfrac{1}{\eta} ) \mathscr{D}_\eta (G,E,B) \\
	    + & C ( 1 + \tfrac{1}{\eta} ) ( 1 + \mathscr{E}_\eta (G,E,B) ) ( \mathscr{E}_\eta (G,E,B) + \| \mathbb{P}^\perp G \|^2_{\widetilde{H}^s_{x,v}} ) \mathscr{D}_\eta (G,E,B)
	  \end{aligned}
	\end{equation}
	for all $|m|+|\alpha| \leq s$ with $\alpha \neq 0$ and for any $0 < \eta \leq \eta_0$, $0 < \eps \leq 1$.
	
	Noticing that the energy functional $\mathscr{D}_w (G)$ defined in \eqref{E-D-eta-Energies} can be dominated by
	\begin{equation}\label{Bnd-Dw-G}
	  \begin{aligned}
	    \mathscr{D}_w (G) \leq \| \mathbb{P}^\perp G \|^2_{\widetilde{H}^s_{x,v}(\nu)} \,,
	  \end{aligned}
	\end{equation}
	we derive from multiplying the inequality \eqref{Mix-4} by $\tfrac{\eta}{2 C ( 1 + \eta
		 )}$ and adding it to the inequality \eqref{GEB-decay} that
	\begin{equation}\label{Mix-5}
	  \begin{aligned}
	    & \tfrac{\d}{\d t} \Big\{ \tfrac{1}{2} \mathscr{E}_\eta (G,E,B) + \eps \eta \mathscr{A}_s (G,E,B) + \tfrac{\eta}{4 C ( 1 + \eta )} \| \partial^m_\alpha \mathbb{P}^\perp G \|^2_{L^2_{x,v}} \Big\} \\
	    & + \tfrac{\eta}{4 C ( 1 + \eta )} \tfrac{1}{\eps^2} \| \partial^m_\alpha \mathbb{P}^\perp G \|^2_{L^2_{x,v}(\nu)} + \tfrac{1}{2} \mathscr{D}_\eta (G,E,B) \\
	    \leq & \tfrac{\eta}{2 ( 1 + \eta )} \tfrac{1}{\eps^2} \sum_{\substack{ |m'| + |\alpha'| \leq s \\ \alpha' < \alpha }} \| \partial^{m'}_{\alpha'} \mathbb{P}^\perp G \|^2_{L^2_{x,v}(\nu)} + C \mathscr{E}_\eta^\frac{1}{2} (G,E,B) \| \mathbb{P}^\perp G \|^2_{\widetilde{H}^s_{x,v}(\nu)} \\
	    & + \tfrac{1}{2} \big( 1 + \mathscr{E}_\eta (G,E,B) \big) \big( \mathscr{E}_\eta (G,E,B) + \| \mathbb{P}^\perp G \|^2_{\widetilde{H}^s_{x,v}(\nu)} \big) \mathscr{D}_\eta (G,E,B)  \\
	    & + C ( 1 + \tfrac{1}{\eta} ) \big[ \mathscr{E}_\eta^\frac{1}{2} (G,E,B) + \mathscr{E}_\eta^2 (G,E,B) \big] \mathscr{D}_\eta (G,E,B)
	  \end{aligned}
	\end{equation}
	holds for all $|m|+|\alpha| \leq s$ with $\alpha \neq 0$ and for any $0 < \eta \leq \eta_0$, $0 < \eps \leq 1$. 
	
	We observe that the previous energy inequality \eqref{Mix-5} is not closed, since the first term $ \tfrac{\eta}{2 ( 1 + \eta )} \tfrac{1}{\eps^2} \sum_{\substack{ |m'| + |\alpha'| \leq s \\ \alpha' < \alpha }} \| \partial^{m'}_{\alpha'} \mathbb{P}^\perp G \|^2_{L^2_{x,v}(\nu)} $ in the right-hand side of \eqref{Mix-5} is uncontrolled. One notices that the highest oder of $v$-derivatives in that term is less that $|\alpha|=k$, which inspires us that we can employ an induction over the $|\alpha|=k$, the order of $v$-derivative, to prove the energy bound \eqref{Mix-Derivatives}.
	
	For $k=0$, the energy bound \eqref{GEB-decay} in Proposition \ref{Prop-Spatial-Summary} and the inequality \eqref{Bnd-Dw-G} imply that \eqref{Mix-Derivatives} holds. Now we assume the lemma is valid for $k$. For $|\alpha| = k + 1$, summing up for $|m| + |\alpha| \leq s$ with $|\alpha| = k + 1$ in the inequality \eqref{Mix-5}, we obtain
	\begin{equation}\label{Mix-6}
	  \begin{aligned}
	    & \tfrac{\d}{\d t} \bigg\{ \tfrac{1}{2} N_{k+1} \mathscr{E}_\eta (G,E,B) + N_{k+1} \eps \eta \mathscr{A}_s (G,E,B) \\
	    & \qquad + \sum_{|m|+|\alpha| \leq s , |\alpha|=k+1} \tfrac{\eta}{4 C ( 1 + \eta )} \| \partial^m_\alpha \mathbb{P}^\perp G \|^2_{L^2_{x,v}} \bigg\} \\
	    & + \tfrac{\eta}{2 (1 + \eta)} \tfrac{1}{\eps^2} \sum_{|m|+|\alpha| \leq s , |\alpha|=k+1} \| \partial^m_\alpha \mathbb{P}^\perp G \|^2_{L^2_{x,v}(\nu)} + \tfrac{1}{2} N_{k+1} \mathscr{D}_\eta (G,E,B) \\
	    \leq & \tfrac{N_{k+1} \eta}{2 (1 + \eta)} \tfrac{1}{\eps^2} \sum_{|m|+|\alpha| \leq s , |\alpha| \leq k} \| \partial^m_\alpha \mathbb{P}^\perp G \|^2_{L^2_{x,v}(\nu)} + C N_{k+1} \Big\{ \mathscr{E}_\eta^\frac{1}{2} (G,E,B) \| \mathbb{P}^\perp G \|^2_{\widetilde{H}^s_{x,v}(\nu)} \\
	    & + \big( 1 + \mathscr{E}_\eta (G,E,B) \big) \big( \mathscr{E}_\eta (G,E,B) + \| \mathbb{P}^\perp G \|^2_{\widetilde{H}^s_{x,v}(\nu)} \big) \\
	    & + ( 1 + \tfrac{1}{\eta} ) \big[ \mathscr{E}_\eta^\frac{1}{2} (G,E,B) + \mathscr{E}_\eta^2 (G,E,B) \big] \Big\} \mathscr{D}_\eta (G,E,B) \,.
	  \end{aligned}
	\end{equation}
	Here $N_{k+1} > 0$ denotes the number of all possible $(m,\alpha)$ such that $|m| + |\alpha| \leq s$, $|\alpha| = k + 1$. By the assumption of the induction, \eqref{Mix-Derivatives} is valid for the case $k$. In order to absorb the first term on the right-hand side in \eqref{Mix-6} by the last second term on the left-hand side of \eqref{Mix-Derivatives}, we multiply \eqref{Mix-6} by $\tfrac{\delta_k}{N_{k+1}}$ and add it to \eqref{Mix-Derivatives}. We then get
	\begin{equation}\label{Mix-7}
	  \begin{aligned}
	    \tfrac{\d}{\d t} & \bigg\{ (\varrho_k + \tfrac{\delta_k}{2}) \mathscr{E}_\eta (G,E,B) + \sum_{|m|+|\alpha| \leq s , |\alpha| \leq k} \tfrac{C_{|\alpha|} \eta}{1 + \eta} \| \partial^m_\alpha \mathbb{P}^\perp G \|^2_{L^2_{x,v}} \\
	    & + (\varrho_k^* + \delta_k) \eps \eta \mathscr{A}_s (G,E,B) + \sum_{|m| + |\alpha| \leq s , |\alpha| = k + 1} \tfrac{\delta_k}{4 C N_{k+1}} \tfrac{\eta}{1 + \eta} \| \partial^m_\alpha \mathbb{P}^\perp G \|^2_{L^2_{x,v}} \bigg\} \\
	    & + ( \delta_k^* + \tfrac{\delta_k}{2} ) \mathscr{D}_\eta (G,E,B) + \tfrac{\delta_k \eta}{2 (1 + \eta)} \tfrac{1}{\eps^2} \sum_{|m|+|\alpha| \leq s, |\alpha| \leq k} \| \partial^m_\alpha \mathbb{P}^\perp G \|^2_{L^2_{x,v}(\nu)} \\
	    & + \tfrac{\delta_k}{N_{k+1}} \tfrac{\eta}{2 (1 + \eta)} \tfrac{1}{\eps^2} \sum_{|m| + |\alpha| \leq s , |\alpha| = k + 1} \| \partial^m_\alpha \mathbb{P}^\perp G \|^2_{L^2_{x,v}(\nu)} \\
	    \leq & ( C_k^* + C \delta_k ) \Big\{ \mathscr{E}_\eta^\frac{1}{2} (G,E,B) \| \mathbb{P}^\perp G \|^2_{\widetilde{H}^s_{x,v}(\nu)} + ( 1 + \tfrac{1}{\eta} ) \big[ \mathscr{E}_\eta^\frac{1}{2} (G,E,B) + \mathscr{E}_\eta^2 (G,E,B) \big] \\
	    & + ( 1 + \mathscr{E}_\eta (G,E,B) ) \big( \mathscr{E}_\eta (G,E,B) + \| \mathbb{P}^\perp G \|^2_{\widetilde{H}^s_{x,v}(\nu)} \big) \Big\} \mathscr{D}_\eta (G,E,B) \,.
	  \end{aligned}
	\end{equation}
	We thereby conclude our lemma from the previous inequality \eqref{Mix-7} by letting
	\begin{equation*}
	  \begin{aligned}
	    & C_{k+1} = \tfrac{\delta_k}{4 C N_{k+1}} \,, \qquad \qquad \! C^*_{k+1} = C_k^* + C \delta_k \,, \\
	    & \varrho_{k+1} = \varrho_k + \tfrac{\delta_k}{2} \,, \qquad \qquad \varrho_{k+1}^* = \varrho_k^* + \delta_k \,, \\ 
	    & \delta_{k+1} = \min \big\{ \tfrac{\delta_k}{2} , \tfrac{\delta_k}{2 N_{k+1}} \big\} \,, \ \delta_{k+1}^* = \delta_k^* + \tfrac{\delta_k}{2} \,.
	  \end{aligned}
	\end{equation*}
\end{proof}

Next we derive a closed energy estimate (uniform in $\eps$) of the perturbed VMB system \eqref{VMB-G-drop-eps} from the inequality \eqref{Mix-Derivatives} in Lemma \ref{Lmm-Mix-Energy}. We first analyze the unsigned functional $\mathscr{A}_s (G,E,B)$ defined in \eqref{E-D-eta-Energies}, i.e.,
\begin{equation*}
  \begin{aligned}
    \mathscr{A}_s (G,E,B) = \mathscr{A}_s (G) + \eta_1 \mathscr{A}_s (E,B) \,,
  \end{aligned}
\end{equation*}
where $\mathscr{A}_s (G)$ is defined in \eqref{Quantity-A} and $\mathscr{A}_s (E,B)$ is given in \eqref{Quantity-A-EB}. Via the H\"older inequality, one easily deduces that for $s \geq 3$
\begin{equation}\label{Quantity-A-G-bnd}
  \begin{aligned}
    | \mathscr{A}_s & (G) | \leq C \sum_{|m| \leq s - 1} \sum_{i=1}^3 \bigg\{ \| \partial^m_x u_i \|_{L^2_x} \big( \| \partial_i \partial^m_x \rho^+ \|_{L^2_x} + \| \partial_i \partial^m_x \rho^- \|_{L^2_x} \big) \\
    & + \sum_{j=1}^3 \| \partial^m_x \mathbb{P}^\perp G \|_{L^2_{x,v}} \| \partial_j \partial^m_x u_i \|_{L^2_x} \| \zeta_{ij} (v) \|_{L^2_v} + \| \partial^m_x \mathbb{P}^\perp G \|_{L^2_{x,v}} \| \partial_i \partial^m_x \theta \|_{L^2_x} \| \widetilde{\zeta}_i (v) \|_{L^2_v} \\
    & + \| \partial^m_x \mathbb{P}^\perp G \|_{L^2_{x,v}} ( \| \partial_i \partial^m_x \rho^+ \|_{L^2_x} + \| \partial_i \partial^m_x \rho^- \|_{L^2_x} ) ( \| \zeta_i^+ (v) \|_{L^2_v} + \| \zeta_i^- (v) \|_{L^2_v} ) \bigg\} \\
    & \leq C \sum_{|m| \leq s - 1} \bigg\{ \| \partial^m_x u \|_{L^2_x} ( \| \nabla_x \partial^m_x \rho^+ \|_{L^2_x} + \| \nabla_x \partial^m_x \rho^- \|_{L^2_x} ) \\
    & + \| \partial^m_x \mathbb{P}^\perp G \|_{L^2_{x,v}} \big( \| \nabla_x \partial^m_x \rho^+ \|_{L^2_x} + \| \nabla_x \partial^m_x \rho^- \|_{L^2_x} + \| \nabla_x \partial^m_x u \|_{L^2_x} + \| \nabla_x \partial^m_x \theta \|_{L^2_x}  \big) \bigg\} \\
    & \leq C \sum_{|m| \leq s - 1} \big( \| \partial^m_x \mathbb{P} G \|_{L^2_{x,v}} + \| \partial^m_x \mathbb{P}^\perp G \|_{L^2_{x,v}} \big) \| \nabla_x \partial^m_x \mathbb{P} G \|_{L^2_{x,v}} \\
    & \leq C\!_{_\#}^{(1)} \| G \|^2_{H^s_x L^2_v} 
  \end{aligned}
\end{equation}
for some $C\!_{_\#}^{(1)} > 0$, independent of $\eps \in ( 0, 1 ]$, where we make use of the facts
\begin{equation}
  \begin{aligned}
    & \|\partial^m_x \rho^+ \|_{L^2_x} + \| \partial^m_x \rho^- \|_{L^2_x} + \| \partial^m_x u \|_{L^2_x} + \| \partial^m_x \theta \|_{L^2_x} \leq C \| \partial^m_x \mathbb{P} G \|_{L^2_{x,v}} \,, \\
    & \| \partial^m_x \mathbb{P}^\perp G \|_{L^2_{x,v}} + \| \partial^m_x \mathbb{P} G \|_{L^2_{x,v}} \leq C \| \partial^m_x G \|_{L^2_{x,v}}
  \end{aligned}
\end{equation}
for all $m \in \mathbb{N}^3$ and $\zeta_{ij} (v)$, $\widetilde{\zeta}_i (v)$, $\zeta_i^+ (v)$, $\zeta_i^- (v) \in L^2_v$. Considering the functional $\mathscr{A}_s (E,B)$, defined in \eqref{Quantity-A-EB}, we yield that
\begin{equation}\label{Quantity-A-EB-bnd}
  \begin{aligned}
    \eta_1 \big| \mathscr{A}_s & (E,B) \big| \leq C \sum_{|m| \leq s - 1} \Big( \eps \| \partial^m_x G \|^2_{L^2_{x,v}} \| \mathsf{q}_1 \widetilde{\Phi} \|^2_{L^2_v} + \| \partial^m_x G \|_{L^2_{x,v}} \| \partial^m_x E \|_{L^2_x} \|_{L^2_v} \Big) \\
    & + C \sum_{|m| \leq s - 2} \Big[ \eps \| \nabla_x \partial^m_x G \|^2_{L^2_{x,v}} \| \mathsf{q}_1 \widetilde{\Phi} \|^2_{L^2_v} \\
    & + \| \nabla_x \partial^m_x G \|_{L^2_{x,v}} \| \mathsf{q}_1 \widetilde{\Phi} \|_{L^2_v} \big( \| \partial_t \partial^m_x B \|_{L^2_x} + \| \partial^m_x B \|_{L^2_x}  \big) \Big] \\
    \leq & C \sum_{|m| \leq s - 1} \big( \| \partial^m_x G \|^2_{L^2_{x,v}} + \| \partial^m_x E \|^2_{L^2_x} \big) \\
    & + C \sum_{|m| \leq s - 2} \Big[ \| \nabla_x \partial^m_x G \|^2_{L^2_{x,v}} + \| \nabla_x \partial^m_x G \|_{L^2_{x,v}} \big( \| \nabla_x \times \partial^m_x E \|_{L^2_x} + \| \partial^m_x B \|_{L^2_x} \big) \Big] \\
    \leq & C\!_{_\#}^{(2)} \big( \| G \|^2_{H^s_x L^2_v} + \| E \|^2_{H^s_x} + \| B \|^2_{H^s_x} \big)
  \end{aligned}
\end{equation}
holds for any $0 < \eps \leq 1$ and for some constant $ C\!_{_\#}^{(2)} > 0 $, independent of $\eps \in (0,1]$. Here the H\"older inequality, the Young's inequality, the bound $\| \mathsf{q}_1 \widetilde{\Phi} \|_{L^2_v} \leq C$ and the third Faraday equation of \eqref{VMB-G-drop-eps} are utilized. Combining the bounds \eqref{Quantity-A-G-bnd} and \eqref{Quantity-A-EB-bnd}, one immediately yields that
\begin{equation}\label{Quantity-A-GEB-bnd}
  \begin{aligned}
    \big| \mathscr{A}_s (G,E,B) \big| \leq & |\mathscr{A}_s (G)| + \eta_1 |\mathscr{A}_s (E,B)| \\
    \leq & C\!_{_\#} \big( \| G \|^2_{H^s_x L^2_v} + \| E \|^2_{H^s_x} + \| B \|^2_{H^s_x} \big)
  \end{aligned}
\end{equation}
for all $0 < \eps \leq 1$, where $ C\!_{_\#} = \max \big\{ C\!_{_\#}^{(1)} , C\!_{_\#}^{(2)} \big\} > 0 $.

Now we take $\eta_2 = \min \big\{ \eta_0 , \tfrac{\varrho_s}{2 \varrho_s^* C\!_{_\#}} \big\} > 0$, where $\eta_0 > 0$ is mentioned in Proposition \ref{Prop-Spatial-Summary}, the constants $\varrho_s > 0$ and $\varrho_s^* > 0$ are given in Lemma \ref{Lmm-Mix-Energy}. Then we introduce an energy functional
\begin{equation}
  \begin{aligned}
    \mathcal{E}_s (G,E,B) = \varrho_s \mathscr{E}_{\eta_2} (G,E,B) + \varrho_s^* \eps \eta_2 \mathscr{A}_s (G,E,B) + \sum_{|m| + |\alpha| \leq s , |\alpha| \leq s} \tfrac{C_{|\alpha|} \eta_2}{1 + \eta_2} \| \partial^m_\alpha \mathbb{P}^\perp G \|^2_{L^2_{x,v}}
  \end{aligned}
\end{equation}
and an energy dissipative rate
\begin{equation}
  \begin{aligned}
    \mathcal{D}_s (G,E,B) = \tfrac{\delta_s \eta_2}{1 + \eta_2} \tfrac{1}{\eps^2} \sum_{|m| + |\alpha| \leq s , |\alpha| \leq s} \| \partial^m_\alpha \mathbb{P}^\perp G \|^2_{L^2_{x,v}(\nu)} + \delta_s^* \mathscr{D}_{\eta_2} (G,E,B) \,,
  \end{aligned}
\end{equation}
where the positive constants $\varrho_s$, $\varrho_s^*$, $C_{|\alpha|}$ $\delta_s$ and $\delta_s^*$ are mentioned in Lemma \ref{Lmm-Mix-Energy}. We remark that the chosen number $\eta_2 = \min \big\{ \eta_0 , \tfrac{\varrho_s}{2 \varrho_s^* C\!_{_\#}} \big\} > 0$ is such that the energy $\mathcal{E}_s (G,E,B)$ is nonnegative for all $0 < \eps \leq 1$. Indeed, by \eqref{Quantity-A-GEB-bnd} we know that for any $0 < \eta \leq \eta_0$ and for all $0 < \eps \leq 1$
\begin{equation}
  \begin{aligned}
    & \varrho_s \mathscr{E}_\eta (G,E,B) + \varrho_s^* \eps \eta \mathscr{A}_s (G,E,B) \\
    \geq & \varrho_s \big( \| G \|^2_{H^s_x L^2_v} + \| E \|^2_{H^s_x} + \| B \|^2_{H^s_x} \big) - \varrho_s^* \eta \big| \mathscr{A}_s (G,E,B) \big| \\
    \geq & ( \varrho_s - \varrho_s^* C\!_{_\#} \eta ) \big( \| G \|^2_{H^s_x L^2_v} + \| E \|^2_{H^s_x} + \| B \|^2_{H^s_x} \big) \,.
  \end{aligned}
\end{equation}
So, if we require the right-hand side of the previous inequality is nonnegative, we noly need to choose an $\eta > 0$ such that
\begin{equation*}
  \begin{aligned}
    \varrho_s - \varrho_s^* C\!_{_\#} \eta > 0 \qquad \textrm{and } \qquad 0 < \eta \leq \eta_0 \,.
  \end{aligned}
\end{equation*}
Without loss of generality, we take 
\begin{equation}
  \begin{aligned}
    \eta = \eta_2 = \min \big\{ \eta_0 , \tfrac{\varrho_s}{2 \varrho_s^* C\!_{_\#}} \big\} > 0 \,.
  \end{aligned}
\end{equation}
We thereby obtain the positivity of the energy $\mathcal{E}_s ( G,E,B )$.

One notices that
\begin{equation}
  \begin{aligned}
    \| \mathbb{P}^\perp G \|^2_{\widetilde{H}^s_{x,v}(\nu)} \leq C \mathcal{D}_s (G,E,B)
  \end{aligned}
\end{equation}
holds for all $0 < \eps \leq 1$ and for some $C > 0$, independent of $\eps \in (0, 1 ]$. Therefore, from Lemma \ref{Lmm-Mix-Energy}, we immediately derive the following proposition.

\begin{proposition}\label{Prop-Uniform-Bnd-VMB}
	Let $s \geq 3$ be an integer and $0 < \eps \leq 1$. Assume that $(G_\eps, E_\eps, B_\eps)$ is the solution to the perturbed VMB system \eqref{VMB-G} constructed in Proposition \ref{Prop-Local-Solutn}. Then there is a constant $C > 0$, independent of $\eps \in (0 , 1 ]$, such that
	\begin{equation}\label{VMB-Uniform-Inq}
	  \begin{aligned}
	    \tfrac{\d}{\d t} \mathcal{E}_s (G_\eps, E_\eps, B_\eps) + \mathcal{D}_s (G_\eps, E_\eps, B_\eps) \leq C \big[ \mathcal{E}_s^\frac{1}{2} (G_\eps, E_\eps, B_\eps) + \mathcal{E}_s^2 (G_\eps, E_\eps, B_\eps) \big] \mathcal{D}_s (G_\eps, E_\eps, B_\eps)
	  \end{aligned}
	\end{equation}
	holds for all $0 < \eps \leq 1$.
\end{proposition}

Based on the uniform inequality \eqref{VMB-Uniform-Inq} in Proposition \ref{Prop-Uniform-Bnd-VMB}, we will give the proof of Theorem \ref{Main-Thm-1}.

\begin{proof}[\bf Proof of Theorem \ref{Main-Thm-1}: Global solutions]
	First, it is easy to know that there are constants $c_0$, $C_0 > 0$, independent of $\eps$, such that
	\begin{equation}\label{Energy-Equiv}
	  \begin{aligned}
	    & c_0 \mathbb{E}_s ( G_\eps, E_\eps, B_\eps ) \leq  \mathcal{E}_s (G_\eps , E_\eps, B_\eps) \leq C_0 \mathbb{E}_s ( G_\eps, E_\eps, B_\eps ) \,, \\
	    & c_0 \mathbb{D}_s ( G_\eps, E_\eps, B_\eps ) \leq \mathcal{D}_s (G_\eps , E_\eps, B_\eps) \leq C_0 \mathbb{D}_s ( G_\eps, E_\eps, B_\eps ) \,,
	  \end{aligned}
	\end{equation}
	where the energy functional $\mathbb{E}_s ( G_\eps, E_\eps, B_\eps )$ and energy dissipative rate functional $ \mathbb{D}_s ( G_\eps, E_\eps, B_\eps ) $ are defined in \eqref{Energy-E-D}. From \eqref{Energy-Equiv}, the differential inequality \eqref{VMB-Uniform-Inq} in Proposition \ref{Prop-Uniform-Bnd-VMB} and the energy bound \eqref{Local-Energy-Bnd} in Proposition \ref{Prop-Local-Solutn}, we deduce that for any $[t_1 , t_2] \subset [0, T^*]$ and $0 < \eps \leq 1$
	\begin{equation}
	  \begin{aligned}
	    & \big| \mathcal{E}_s (G_\eps, E_\eps, B_\eps) (t_2) - \mathcal{E}_s (G_\eps, E_\eps, B_\eps) (t_1) \big| \\
	    \leq & C \int_{t_1}^{t_2} \big[ \mathcal{E}_s^\frac{1}{2} (G_\eps, E_\eps, B_\eps) + \mathcal{E}_s^2 (G_\eps, E_\eps, B_\eps) \big] \mathcal{D}_s (G_\eps, E_\eps, B_\eps) \d t \\
	    \leq & C \sup_{0 \leq t \leq T^*} \Big[ \mathbb{E}_s^\frac{1}{2} (G_\eps , E_\eps , B_\eps) + \mathbb{E}_s^2 (G_\eps , E_\eps, B_\eps) \Big] \int_{t_1}^{t_2} \mathbb{D}_s (G_\eps, E_\eps, B_\eps) \d t \\
	    \leq & C ( \sqrt{\ell} + \ell^2 ) \left[ \int_{t_1}^{t_2} \tfrac{1}{\eps^2} \| \mathbb{P}^\perp G_\eps \|^2_{H^s_{x,v}(\nu)} \d t + \ell (t_2 - t_1) \right] \\
	    \rightarrow & 0 \quad \textrm{as} \quad t_2 \rightarrow t_1 \,,
	  \end{aligned}
	\end{equation}
	where $\ell > 0$ is mentioned in Proposition \ref{Prop-Local-Solutn}. Thus the local solution $(G_\eps, E_\eps, B_\eps)$ constructed in Proposition \ref{Prop-Local-Solutn} is such that the energy functional $\mathcal{E}_s (G_\eps, E_\eps, B_\eps) (t)$ is continuous in $t \in [0, T^*]$.
	
	We now define
	\begin{equation}
	  \begin{aligned}
	    T = \sup \Big\{ \tau \geq 0 ; C \sup_{t \in [0,\tau]} \big[ \mathcal{E}_s^\frac{1}{2} (G_\eps, E_\eps, B_\eps) + \mathcal{E}_s^2 (G_\eps, E_\eps, B_\eps) \big] (t) \leq \tfrac{1}{2}  \Big\} \geq 0 \,.
	  \end{aligned}
	\end{equation}
	From the relations \eqref{Energy-Equiv} and the initial condition in Theorem \ref{Main-Thm-1}, we have
	\begin{equation}
	  \begin{aligned}
	    c_0 \mathbb{E} ( G_\eps^{in}, E_\eps^{in}, B_\eps^{in} ) \leq \mathcal{E}_s (G_\eps, E_\eps, B_\eps) (0) \leq C_0 \mathbb{E} ( G_\eps^{in}, E_\eps^{in}, B_\eps^{in} ) \leq C_0 \ell_0 \,,
	  \end{aligned}
	\end{equation}
	where $\ell_0 \in ( 0, 1]$ is small to be determined. If we take $0 < \ell_0 \leq \min \big\{ 1 , \ell, \tfrac{1}{16 C^2 ( \sqrt{C_0} + C_0^2 )^2} \big\}$ ($\ell > 0$ is mentioned in Proposition \ref{Prop-Local-Solutn}), we have
	\begin{equation}\label{Initial-Bnd}
	  \begin{aligned}
	    & C \big[ \mathcal{E}_s^\frac{1}{2} (G_\eps, E_\eps, B_\eps) (0) + \mathcal{E}_s^2 (G_\eps, E_\eps, B_\eps) (0) \big] \\
	    \leq & C \big( \sqrt{C_0 \ell_0} + ( C_0 \ell_0 )^2 \big) \leq C ( \sqrt{C_0} + C_0^2 ) \sqrt{\ell_0} \leq \tfrac{1}{4} < \tfrac{1}{2} \,.
	  \end{aligned}
	\end{equation}
	Then, the continuity of $\mathcal{E} (G_\eps, E_\eps, B_\eps) (t)$ implies that $T > 0$. Consequently, we derive from the definition of $T$ and the inequality \eqref{VMB-Uniform-Inq} that for all $t \in [0,T]$ and $0 < \eps \leq 1$ 
	\begin{equation}
	  \begin{aligned}
	    \tfrac{\d}{\d t} \mathcal{E}_s (G_\eps, E_\eps , B_\eps) + \tfrac{1}{2} \mathcal{D}_s (G_\eps, E_\eps, B_\eps) \leq 0 \,.
	  \end{aligned}
	\end{equation}
	Then we yield that for all $t \in [0,T]$ and $0 < \eps \leq 1$
	\begin{equation}\label{Energy-Extend}
	  \begin{aligned}
	    \mathcal{E}_s (G_\eps, E_\eps , B_\eps) (t) + \int_0^t \tfrac{1}{2} \mathcal{D}_s (G_\eps, E_\eps, B_\eps) (\tau) \d \tau \leq \mathcal{E}_s (G_\eps, E_\eps , B_\eps) (0)  \leq C_0 \ell_0 \,,
	  \end{aligned}
	\end{equation}
	which immediately implies by the initial bound \eqref{Initial-Bnd} that
	\begin{equation}
	  \begin{aligned}
	    C \sup_{t \in [0,T]} \big[ \mathcal{E}_s^\frac{1}{2} (G_\eps, E_\eps, B_\eps) + \mathcal{E}_s^2 (G_\eps, E_\eps, B_\eps) \big] (t) \leq \tfrac{1}{4} < \tfrac{1}{2} \,. 
	  \end{aligned}
	\end{equation}
	Thus, the continuity of $ \mathcal{E} (G_\eps, E_\eps , B_\eps) (t) $ implies that $T = + \infty$. In other words, the local solution constructed in Proposition \ref{Prop-Local-Solutn} can be extended globally. Moreover, the uniform energy bound \eqref{Uniform-Bnd} can be derived from \eqref{Energy-Equiv} and \eqref{Energy-Extend}. Then the proof of Theorem \ref{Main-Thm-1} is completed.	
\end{proof}

\section{Limit to two fluid incompressible Navier-Stokes-Fourier-Maxwell equations with Ohm's law}
\label{Sec:Limits}

In this section, we will derive the two fluid incompressible Navier-Stokes-Fourier-Maxwell equations \eqref{INSFM-Ohm} with Ohm's law from the perturbed two-species Vlasov-Maxwell-Boltzmann \eqref{VMB-G} as $\eps \rightarrow 0$.

\subsection{Local conservation laws} We first introduce the following fluid variables
\begin{equation}\label{Fluid-Quanities}
  \begin{aligned}
    \rho_\eps = \tfrac{1}{2} \langle G_\eps , \mathsf{q}_2 \sqrt{M} \rangle_{L^2_v} \,, \ u_\eps = \tfrac{1}{2} \langle G_\eps , \mathsf{q}_2 v \sqrt{M} \rangle_{L^2_v} \,, \ \theta_\eps = \tfrac{1}{2} \langle G_\eps , \mathsf{q}_2 ( \tfrac{|v|^2}{3} - 1 ) \sqrt{M} \rangle_{L^2_v} \,, \\
    n_\eps = \langle G_\eps , \mathsf{q}_1 \sqrt{M} \rangle_{L^2_v} \,,\ j_\eps = \tfrac{1}{\eps} \langle G_\eps , \mathsf{q}_1 v \sqrt{M} \rangle_{L^2_v} \,,\ w_\eps = \tfrac{1}{\eps} \langle G_\eps , \mathsf{q}_1 ( \tfrac{|v|^2}{3} - 1 ) \sqrt{M} \rangle_{L^2_v} \,.
  \end{aligned}
\end{equation}
Then we can derive the following local conservation laws from the solutions $(G_\eps , E_\eps, B_\eps)$ constructed in Theorem \ref{Main-Thm-1}. These conservation laws exactly can be referred to \cite{Arsenio-SRM-2016}. However, for convenience for readers, we justify them here.

\begin{lemma}\label{Lmm-Local-Consvtn-Law}
	Assume that  $(G_\eps , E_\eps, B_\eps)$ is the solutions to the perturbed two-species VMB equations \eqref{VMB-G} constructed in Theorem \ref{Main-Thm-1}. Then the following local conservation laws holds:
	\begin{equation}\label{Local-Consvtn-Law}
	  \left\{
	    \begin{array}{l}
	      \partial_t \rho_\eps + \tfrac{1}{\eps} \div_x \, u_\eps = 0 \,, \\
	      \partial_t u_\eps + \tfrac{1}{\eps} \nabla_x ( \rho_\eps + \theta_\eps ) + \div_x \, \big\langle \widehat{A} (v) \sqrt{M} , \tfrac{1}{\eps} \mathcal{L} ( \tfrac{G_\eps \cdot \mathsf{q}_2}{2} ) \big\rangle_{L^2_v} = \tfrac{1}{2} ( n_\eps E_\eps + j_\eps \times B_\eps ) \,, \\
	      \partial_t \theta_\eps + \tfrac{2}{3} \tfrac{1}{\eps} \div_x \, u_\eps + \tfrac{2}{3} \div_x \, \big\langle \widehat{B} (v) \sqrt{M} , \tfrac{1}{\eps} \mathcal{L} ( \tfrac{G_\eps \cdot \mathsf{q}_2}{2} ) \big\rangle_{L^2_v} = \tfrac{\eps}{3} j_\eps \cdot E_\eps \,, \\
	      \partial_t n_\eps + \div_x \, j_\eps = 0 \,, \\
	      \partial_t E_\eps - \nabla_x \times B_\eps = - j_\eps \,, \\
	      \partial_t B_\eps + \nabla_x \times E_\eps = 0 \,, \\
	      \div_x \, E_\eps = n_\eps \,, \quad \div_x B_\eps = 0 \,.
	    \end{array}
	  \right.
	\end{equation}
\end{lemma}

\begin{proof}[Proof of Lemma \ref{Lmm-Local-Consvtn-Law}]
	From the definition of $j_\eps$ in \eqref{Fluid-Quanities}, we easily derive last four relations of \eqref{Local-Consvtn-Law}. We thereby need only verify the first four relations in \eqref{Local-Consvtn-Law}. 
	
	{\em Step 1. Conservation law of $\rho_\eps$.} We multiply the first $G_\eps$-equation of \eqref{VMB-G} by $ \tfrac{\mathsf{q}_2 \sqrt{M}}{2} \in \textrm{Ker} (\mathscr{L}) $ and integrate over $v \in \R^3$. Then we obtain
	\begin{equation}
	  \begin{aligned}
	    \partial_t \rho_\eps + \underset{I_1}{\underbrace{ \tfrac{1}{\eps} \big\langle v \cdot \nabla_x G_\eps , \tfrac{\mathsf{q}_2 \sqrt{M}}{2} \big\rangle_{L^2_v} }} + \underset{I_2}{\underbrace{ \tfrac{1}{\eps} \big\langle ( \eps E_\eps + v \times B_\eps ) \cdot \nabla_v G_\eps , \tfrac{\mathsf{q}_2 \sqrt{M}}{2} \big\rangle_{L^2_v} }} \\
	    + \underset{I_3 \ = \, 0}{\underbrace{ \tfrac{1}{\eps^2} \big\langle \mathscr{L} G_\eps , \tfrac{\mathsf{q}_2 \sqrt{M}}{2} \big\rangle_{L^2_v} }} \ \  \underset{I_4}{\underbrace{ - \tfrac{1}{\eps} \big\langle (E_\eps \cdot v) \sqrt{M} \mathsf{q}_1 , \tfrac{\mathsf{q}_2 \sqrt{M}}{2} \big\rangle_{L^2_v} }} \\
	    = \underset{I_5}{\underbrace{ \tfrac{1}{2} \big\langle \mathsf{q} ( E_\eps \cdot v ) G_\eps , \tfrac{\mathsf{q}_2 \sqrt{M}}{2} \big\rangle_{L^2_v} }} + \underset{I_6 \ = \, 0}{\underbrace{ \tfrac{1}{\eps} \big\langle \Gamma (G_\eps , G_\eps) , \tfrac{\mathsf{q}_2 \sqrt{M}}{2} \big\rangle_{L^2_v} }} \,.
	  \end{aligned}
	\end{equation}
	For the term $I_1$, we have
	\begin{equation}
	  \begin{aligned}
	    I_1 = \tfrac{1}{\eps} \div_x \, \big\langle v \sqrt{M} , \tfrac{G_\eps \cdot \mathsf{q}_2}{2} \big\rangle_{L^2_v} = \tfrac{1}{\eps} \div_x \, u_\eps \,.
	  \end{aligned}
	\end{equation}
	For the term $I_2$, we compute that
	\begin{equation}
	  \begin{aligned}
	    I_2 = - \tfrac{1}{\eps} \big\langle \eps E_\eps + v \times B_\eps , \tfrac{G_\eps \cdot \mathsf{q}_1}{2} \nabla_v \sqrt{M} \big\rangle_{L^2_v} = \tfrac{1}{2} E_\eps \cdot \big\langle \tfrac{G_\eps \cdot \mathsf{q}_1}{2} , v \sqrt{M} \big\rangle_{L^2_v} \,.
	  \end{aligned}
	\end{equation}
	Since $\mathsf{q}_1 \cdot \mathsf{q}_2 = 0$, we know that $I_4 = 0$. For the term $I_5$, we have
	\begin{equation}
	  \begin{aligned}
	    I_5 = \tfrac{1}{2} E_\eps \cdot \big\langle \tfrac{G_\eps \cdot \mathsf{q}_1}{2} , v \sqrt{M} \big\rangle_{L^2_v} \,.
	  \end{aligned}
	\end{equation}
	Collecting the above relations, we deduce that
	\begin{equation}
	  \begin{aligned}
	    \partial_t \rho_\eps + \tfrac{1}{\eps} \div_x \, u_\eps = 0 \,,
	  \end{aligned}
	\end{equation}
	hence the first equation of \eqref{Local-Consvtn-Law} holds.
	
	{\em Step 2. Conservation law of $u_\eps$.} Multiplying the first $G_\eps$-equation of \eqref{VMB-G} by $ \tfrac{\mathsf{q}_2 v \sqrt{M}}{2} \in \textrm{Ker} (\mathscr{L}) $ and integrating over $v \in \R^3$ , we have
	\begin{equation}
	  \begin{aligned}
	    \partial_t u_\eps + \underset{I\!I_1}{\underbrace{ \tfrac{1}{\eps} \big\langle v \cdot \nabla_x G_\eps , \tfrac{\mathsf{q}_2 v \sqrt{M}}{2} \big\rangle_{L^2_v} }} + \underset{I\!I_2}{\underbrace{ \tfrac{1}{\eps} \big\langle \mathsf{q} ( \eps E_\eps + v \times B_\eps ) \cdot \nabla_v G_\eps , \tfrac{\mathsf{q}_2 v \sqrt{M}}{2} \big\rangle_{L^2_v} }}\\
	    + \underset{I\!I_3 \ = \, 0}{\underbrace{ \tfrac{1}{\eps^2} \big\langle \mathscr{L} G_\eps , \tfrac{\mathsf{q}_2 v \sqrt{M}}{2} \big\rangle_{L^2_v} }} \ \ \underset{I\!I_4 \ = \, 0}{\underbrace{ - \tfrac{1}{\eps} \big\langle (E_\eps \cdot v) \sqrt{M} \mathsf{q}_1 , \tfrac{\mathsf{q}_2 v \sqrt{M}}{2} \big\rangle_{L^2_v} }} \\
	    = \underset{I\!I_5}{\underbrace{ \tfrac{1}{2} \big\langle \mathsf{q} (E_\eps \cdot v) G_\eps , \tfrac{\mathsf{q}_2 v \sqrt{M}}{2} \big\rangle_{L^2_v} }} + \underset{I\!I_6 \ = \, 0}{\underbrace{ \tfrac{1}{\eps} \big\langle \Gamma (G_\eps , G_\eps) , \tfrac{\mathsf{q}_2 v \sqrt{M}}{2} \big\rangle_{L^2_v} }}\,,
	  \end{aligned}
	\end{equation}
	where $I\!I_4 = 0$ is derived from $\mathsf{q}_1 \cdot \mathsf{q}_2 = 0$. For the term $I\!I_1$, we deduce that
	\begin{equation}
	  \begin{aligned}
	    I\!I_1  = & \tfrac{1}{\eps} \div_x \, \big\langle v \otimes v \sqrt{M} , \tfrac{G_\eps \cdot \mathsf{q}_2}{2} \big\rangle_{L^2_v} \\
	    = & \tfrac{1}{\eps} \div_x \, \big\langle A(v) \sqrt{M} , \tfrac{G_\eps \cdot \mathsf{q}_2}{2} \big\rangle_{L^2_v} + \tfrac{1}{\eps} \div_x \, \big\langle \tfrac{|v|^2}{3} \mathbb{I}_3 \sqrt{M} , \tfrac{G_\eps \cdot \mathsf{q}_2}{2} \big\rangle_{L^2_v} \\
	    = & \div_x \, \big\langle \widehat{A} (v) \sqrt{M} , \tfrac{1}{\eps} \mathcal{L} ( \tfrac{G_\eps \cdot \mathsf{q}_2}{2} ) \big\rangle_{L^2_v} + \tfrac{1}{\eps} \nabla_x \Big( \big\langle G_\eps , \tfrac{\mathsf{q}_2 \sqrt{M}}{2} \big\rangle_{L^2_v} + \big\langle G_\eps , \tfrac{\mathsf{q}_2 \sqrt{M}}{2} ( \tfrac{|v|^2}{3} - 1 ) \big\rangle_{L^2_v} \Big) \\
	    = & \div_x \, \big\langle \widehat{A} (v) \sqrt{M} , \tfrac{1}{\eps} \mathcal{L} ( \tfrac{G_\eps \cdot \mathsf{q}_2}{2} ) \big\rangle_{L^2_v} + \tfrac{1}{\eps} \nabla_x (\rho_\eps + \theta_\eps) \,,
	  \end{aligned}
	\end{equation}
	where $\mathcal{L}$ is the standard linearized Boltzmann collision operator, $A(v) = v \otimes v - \tfrac{|v|^2}{3} \mathbb{I}_3 $ with $A \sqrt{M} \in \textrm{Ker}^\perp (\mathcal{L})$, $\widehat{A} (v) $ is such that $\mathcal{L} ( \widehat{A} \sqrt{M} ) = A \sqrt{M}$ with $\widehat{A} \sqrt{M} \in \textrm{Ker}^\perp (\mathcal{L}) $ and $\mathbb{I}_3$ is the $3 \times 3$ unitary matrix. For the term $I\!I_2$, we have
	\begin{equation}
	  \begin{aligned}
	    I\!I_2 = & - \tfrac{1}{\eps} \big\langle G_\eps , ( \eps E_\eps + v \times B_\eps ) \cdot \nabla_v ( \tfrac{\mathsf{q}_1 v \sqrt{M}}{2} ) \big\rangle_{L^2_v} \\
	    = & - \tfrac{1}{\eps} \big\langle G_\eps , ( \eps E_\eps + v \times B_\eps ) \cdot ( \mathbb{I}_3 - \tfrac{1}{2} v \otimes v ) \tfrac{\mathsf{q}_1 \sqrt{M}}{2} \big\rangle_{L^2_v} \\
	    = & - \tfrac{1}{2} \big\langle G_\eps , \mathsf{q}_1 \sqrt{M} \big\rangle_{L^2_v} E_\eps - \tfrac{1}{2 \eps} \big\langle G_\eps, \mathsf{q}_1 v \sqrt{M} \big\rangle_{L^2_v} \times B_\eps + \tfrac{1}{4} \big\langle E_\eps \cdot v , G_\eps \cdot \mathsf{q}_1 \sqrt{M} \big\rangle_{L^2_v} \\
	    = & - \tfrac{1}{2} n_\eps E_\eps - \tfrac{1}{2} j_\eps \times B_\eps + \tfrac{1}{4} \big\langle E_\eps \cdot v , G_\eps \cdot \mathsf{q}_1 \sqrt{M} \big\rangle_{L^2_v} \,,
	  \end{aligned}
	\end{equation}
	where the cancellation $(v \times B_\eps) \cdot v = 0$ is utilized. For the term $I\!I_5$, we derive that
	\begin{equation}
	  \begin{aligned}
	    I\!I_5 = \tfrac{1}{4} \big\langle E_\eps \cdot v , G_\eps \cdot \mathsf{q}_1 \sqrt{M} \big\rangle_{L^2_v} \,.
	  \end{aligned}
	\end{equation}
	Collecting the previous calculations of the terms $I\!I_1$, $I\!I_2$ and $I\!I_5$, we obtain
	\begin{equation}
	  \begin{aligned}
	    \partial_t u_\eps + \tfrac{1}{\eps} \nabla_x (\rho_\eps + \theta_\eps) + \div_x \, \big\langle \widehat{A} (v) \sqrt{M} , \tfrac{1}{\eps} \mathcal{L} ( \tfrac{G_\eps \cdot \mathsf{q}_2}{2} ) \big\rangle_{L^2_v} = \tfrac{1}{2} ( n_\eps E_\eps + j_\eps \times B_\eps ) \,,
	  \end{aligned}
	\end{equation}
	then the second equations of \eqref{Local-Consvtn-Law} holds.
	
	{\em Step 3. Conservation law of $\theta_\eps$.} We take $L^2_v$-inner product via multiplying the first $G_\eps$-equation of \eqref{VMB-G} by $ \tfrac{1}{2} \mathsf{q}_2 ( \tfrac{|v|^2}{3} - 1 ) \sqrt{M} \in \textrm{Ker} (\mathscr{L}) $. More precisely, we have
	\begin{equation}
	  \begin{aligned}
	    \partial_t \theta_\eps + \underset{I\!I\!I_1}{\underbrace{ \tfrac{1}{\eps} \big\langle v \cdot \nabla_x G_\eps , \tfrac{1}{2} \mathsf{q}_2 ( \tfrac{|v|^2}{3} - 1 ) \sqrt{M} \big\rangle_{L^2_v} }} \\
	    + \underset{I\!I\!I_2}{\underbrace{ \tfrac{1}{\eps} \big\langle \mathsf{q} ( \eps E_\eps + v \times B_\eps ) \cdot \nabla_v G_\eps , \tfrac{1}{2} \mathsf{q}_2 ( \tfrac{|v|^2}{3} - 1 ) \sqrt{M} \big\rangle_{L^2_v} }} \\
	    + \underset{I\!I\!I_3 \ = \, 0}{\underbrace{ \tfrac{1}{\eps^2} \big\langle \mathscr{L} G_\eps , \tfrac{1}{2} \mathsf{q}_2 ( \tfrac{|v|^2}{3} - 1 ) \sqrt{M} \big\rangle_{L^2_v} }} \ \ \underset{I\!I\!I_4 \ = \, 0}{\underbrace{ - \tfrac{1}{\eps} \big\langle (E_\eps \cdot v) \mathsf{q}_1 \sqrt{M} , \tfrac{1}{2} \mathsf{q}_2 ( \tfrac{|v|^2}{3} - 1 ) \sqrt{M} \big\rangle_{L^2_v} }} \\
	    = \underset{I\!I\!I_5}{\underbrace{ \tfrac{1}{2} \big\langle \mathsf{q} (E_\eps \cdot v) G_\eps , \tfrac{1}{2} \mathsf{q}_2 ( \tfrac{|v|^2}{3} - 1 ) \sqrt{M} \big\rangle_{L^2_v} }} + \underset{I\!I\!I_6 \ = \, 0}{\underbrace{ \tfrac{1}{\eps} \big\langle \Gamma (G_\eps , G_\eps) , \tfrac{1}{2} \mathsf{q}_2 ( \tfrac{|v|^2}{3} - 1 ) \sqrt{M} \big\rangle_{L^2_v} }} \,,
	  \end{aligned}
	\end{equation}
	where $I\!I\!I_3 = 0$ is implied by $\mathsf{q}_1 \cdot \mathsf{q}_2 = 0$. For the term $I\!I\!I_1$, we have
	\begin{equation}
	  \begin{aligned}
	    I\!I\!I_1 = & \tfrac{1}{2 \eps} \div_x \, \big\langle G_\eps \cdot \mathsf{q}_2 \sqrt{M} , v ( \tfrac{|v|^2}{3} - 1 ) \big\rangle_{L^2_v} \\
	    = & \tfrac{1}{3 \eps} \div_x \, \big\langle G_\eps \cdot \mathsf{q}_2 \sqrt{M} , v ( \tfrac{|v|^2}{2} - \tfrac{5}{2} ) + v \big\rangle_{L^2_v} \\
	    = & \tfrac{2}{3} \tfrac{1}{\eps} \div_x \big\langle \tfrac{G_\eps \cdot \mathsf{q}_2}{2} , B(v) \sqrt{M} \big\rangle_{L^2_v} + \tfrac{2}{3} \tfrac{1}{\eps} \div_x \, u_\eps \\
	    = & \tfrac{2}{3} \div_x \big\langle \widehat{B} (v) \sqrt{M} , \tfrac{1}{\eps} \mathcal{L} ( \tfrac{G_\eps \cdot \mathsf{q}_2}{2} ) + \tfrac{2}{3} \tfrac{1}{\eps} \div_x \, u_\eps \,,
	  \end{aligned}
	\end{equation}
	where $B(v) = v ( \tfrac{|v|^2}{2} - \tfrac{5}{2} )$ with $B \sqrt{M} \in \textrm{Ker}^\perp (\mathcal{L})$ and $\widehat{B} (v)$ is such that $\mathcal{L} ( \widehat{B} \sqrt{M} ) = B \sqrt{M}$ with $\widehat{B} (v) \sqrt{M} \in \textrm{Ker}^\perp (\mathcal{L})$. For the term $I\!I\!I_2$, we have
	\begin{equation}
	  \begin{aligned}
	    I\!I\!I_2 = & - \tfrac{1}{2 \eps} \big\langle G_\eps \cdot \mathsf{q}_1 , ( \eps E_\eps + v \times B_\eps ) \cdot \nabla_v [ ( \tfrac{|v|^2}{3} - 1 ) \sqrt{M} ] \big\rangle_{L^2_v} \\
	    = & - \tfrac{1}{2} \big\langle G_\eps \cdot \mathsf{q}_1 , E_\eps \cdot [ \tfrac{2}{3} v - \tfrac{1}{2} v ( \tfrac{|v|^2}{3} - 1 ) ] \sqrt{M} \big\rangle_{L^2_v} \\
	    = & - \tfrac{\eps}{3} E_\eps \cdot \big\langle G_\eps , \mathsf{q}_1 v \sqrt{M} \big\rangle_{L^2_v} + \tfrac{1}{4} \big\langle E_\eps \cdot v , G_\eps \cdot \mathsf{q}_1 ( \tfrac{|v|^2}{3} - 1 ) \sqrt{M} \big\rangle_{L^2_v} \\
	    = & - \tfrac{\eps}{3} j_\eps \cdot E_\eps + I\!I\!I_5 \,,
	  \end{aligned}
	\end{equation}
	where we make use of the cancellation $(v \times B_\eps) \cdot v = 0$. Thus, collecting the all relations derived in the previous yields that
	\begin{equation}
	  \begin{aligned}
	    \partial_t \theta_\eps + \tfrac{2}{3} \tfrac{1}{\eps} \div_x \, u_\eps + \tfrac{2}{3} \tfrac{1}{\eps} \div_x \big\langle \widehat{B} (v) \sqrt{M} , \tfrac{1}{\eps} \mathcal{L} ( \tfrac{G_\eps \cdot \mathsf{q}_1}{2} ) \big\rangle_{L^2_v} = \tfrac{\eps}{3} j_\eps \cdot E_\eps \,,
	  \end{aligned}
	\end{equation}
	and the third equation of \eqref{Local-Consvtn-Law} holds.
	
	{\em Conservation law of $n_\eps$.} We take $L^2_v$-inner product via multiplying the first equation of \eqref{VMB-G} by $ \mathsf{q}_1 \sqrt{M} \in \textrm{Ker} (\mathscr{L}) $. We then obtain
	\begin{equation}
	  \begin{aligned}
	    \partial_t n_\eps + \underset{I\!V_1}{\underbrace{ \tfrac{1}{\eps} \big\langle v \cdot \nabla_x G_\eps , \mathsf{q}_1 \sqrt{M} \big\rangle_{L^2_v} }} + \underset{I\!V_2}{\underbrace{ \tfrac{1}{\eps} \big\langle \mathsf{q} ( \eps E_\eps + v \times B_\eps ) \cdot \nabla_v G_\eps , \mathsf{q}_1 \sqrt{M} \big\rangle_{L^2_v} }} \\
	    + \underset{I\!V_3 \ = \, 0}{\underbrace{ \tfrac{1}{\eps^2} \big\langle \mathscr{L} G_\eps , \mathsf{q}_1 \sqrt{M} \big\rangle_{L^2_v} }} \ \ \underset{I\!V_4 \ = \, 0}{\underbrace{ - \tfrac{1}{\eps} \big\langle ( E_\eps \cdot v ) \sqrt{M} \mathsf{q}_1 , \mathsf{q}_1 \sqrt{M} \big\rangle_{L^2_v} }} \\
	    = \underset{I\!V_5}{\underbrace{ \tfrac{1}{2} \big\langle \mathsf{q} ( E_\eps \cdot v ) G_\eps , \mathsf{q}_1 \sqrt{M} \big\rangle_{L^2_v} }} + \underset{I\!V_6 \ = \, 0}{\underbrace{ \tfrac{1}{\eps} \big\langle \Gamma (G_\eps , G_\eps) , \mathsf{q}_1 \sqrt{M} \big\rangle_{L^2_v} }} \,,
	  \end{aligned}
	\end{equation}
	where $I\!V_4 = 0$ is derived from $ \langle v \sqrt{M} , \sqrt{M} \rangle_{L^2_v} = 0 $. For the term $I\!V_1$, we have
	\begin{equation}
	  \begin{aligned}
	    I\!V_1 = & \tfrac{1}{\eps} \div_x \, \big\langle G_\eps , \mathsf{q}_1 \sqrt{M} \big\rangle_{L^2_v} = \div_x \, j_\eps \,.
	  \end{aligned}
	\end{equation}
	For the term $I\!V_2$, we deduce from the integration by parts over $v \in \R^3$ and the cancellation $(v \times B_\eps) \cdot v = 0$ that
	\begin{equation}
	  \begin{aligned}
	    I\!V_2 = & \tfrac{1}{\eps} \big\langle \mathsf{q} ( \eps E_\eps + v \times B_\eps ) , \mathsf{q}_1 \tfrac{1}{2} v \sqrt{M} \big\rangle_{L^2_v} = \tfrac{1}{2} \big\langle E_\eps \cdot v , G_\eps \cdot \mathsf{q}_2 \sqrt{M} \big\rangle_{L^2_v} = I\!V_5  \,.
	  \end{aligned}
	\end{equation}
	Then, we derive the fourth equation of \eqref{Local-Consvtn-Law} 
	\begin{equation}
	  \begin{aligned}
	    \partial_t n_\eps + \div_x \, j_\eps = 0
	  \end{aligned}
	\end{equation}
	from the microscopic kinetic $G_\eps$-equation of \eqref{VMB-G}. We emphasize that the conservation law of $n_\eps$ can also be derived from the Amper\'e equation $\partial_t E_\eps - \nabla_x \times B_\eps = - j_\eps$ with the constraint $\div_x \, E_\eps = n_\eps$, which means that the conservation law of $n_\eps$ demonstrates the unity of micro and macro. As a result, the proof of Lemma \ref{Lmm-Local-Consvtn-Law} is finished.
\end{proof}

\subsection{Limits from the global energy estimate} Based on Theorem \ref{Main-Thm-1}, the Cauchy problem \eqref{VMB-G}-\eqref{IC-VMB-G} admits a global solution $( G_\eps , E_\eps , B_\eps )$ belonging to $L^\infty ( \R^+ ; H^s_{x,v} ) $, $ L^\infty (\R^+ ; H^s_x) $ and $ L^\infty (\R^+ ; H^s_x)$, which subjects to the global energy estimate \eqref{Uniform-Bnd}, namely, there is a positive constant, independent of $\eps$, such that
\begin{equation}\label{Uniform-Bnd-1}
  \begin{aligned}
    \sup_{t \geq 0} \big( \| G_\eps (t) \|^2_{H^s_{x,v}} + \| E_\eps (t) \|^2_{H^s_x} + \| B_\eps (t) \|^2_{H^s_x} \big)  \leq C \,,
  \end{aligned}
\end{equation}
and
\begin{equation}\label{Uniform-Bnd-2}
  \begin{aligned}
    \int_0^\infty \| \mathbb{P}^\perp G_\eps (t) \|^2_{H^s_{x,v}(\nu)} \d t \leq C \eps^2 \,.
  \end{aligned}
\end{equation}
From the energy bound \eqref{Uniform-Bnd-1}, there are $G \in L^\infty (\R^+; H^s_{x,v})$, $E,B \in L^\infty (\R^+ ; H^s_x)$, such that
\begin{equation}\label{Convg-GEB}
  \begin{aligned}
    G_\eps \rightarrow G \quad \textrm{weakly-}\star \ \textrm{for} \ t \geq 0 \,, \ \textrm{weakly in } H^s_{x,v}  \,, \\
    E_\eps \rightarrow E \quad \textrm{weakly-}\star \ \textrm{for} \ t \geq 0 \,, \ \textrm{weakly in } H^s_x  \,, \\
    B_\eps \rightarrow B \quad \textrm{weakly-}\star \ \textrm{for} \ t \geq 0  \,, \ \textrm{weakly in } H^s_x \,,
  \end{aligned}
\end{equation}
as $\eps \rightarrow 0$. The limits may hold for some subsequences. But, for convenience, we still employ the original notations of the sequences to denote by the subsequences throughout this paper. From the energy dissipation bound \eqref{Uniform-Bnd-2} and the inequality $\| \mathbb{P}^\perp G_\eps \|^2_{H^s_{x,v}(\nu)} \leq C \| \mathbb{P}^\perp G_\eps \|^2_{H^s_{x,v}} $ derived from the part (1) of Lemma \ref{Lmm-nu-norm}, we have
\begin{equation}\label{Convg-G-perp}
  \begin{aligned}
    \mathbb{P}^\perp G_\eps \rightarrow 0 \quad \textrm{strongly in } \ L^2 (\R^+; H^s_{x,v})
  \end{aligned}
\end{equation}
as $\eps \rightarrow 0$. We thereby deduce from combining the first convergence in \eqref{Convg-GEB} and \eqref{Convg-G-perp} that
\begin{equation}
  \begin{aligned}
    \mathbb{P}^\perp G = 0 \,,
  \end{aligned}
\end{equation}
which immediately means that there are $(\rho^+ , \rho^- , u , \theta) \in L^\infty_t ( \R^+ ; H^s_x )$ such that
\begin{equation}\label{Limit-G}
  \begin{aligned}
    G(t,x,v) = & \rho^+ (t,x) \tfrac{\mathsf{q}_1 + \mathsf{q}_2}{2} \sqrt{M (v)} + \rho^- (t,x) \tfrac{\mathsf{q}_2 - \mathsf{q}_1}{2} \sqrt{M (v)} \\
    + & u (t,x) \cdot v \mathsf{q}_2 \sqrt{M(v)} + \theta (t,x) ( \tfrac{|v|^2}{2} - \tfrac{3}{2} ) \mathsf{q}_2 \sqrt{M(v)} \,,
  \end{aligned}
\end{equation}
where $\mathsf{q}_1 = [1,-1] \in \R^2$ and $\mathsf{q}_2 = [1,1] \in \R^2$.

Via the definitions of $\rho_\eps$, $u_\eps$, $\theta_\eps$ and $n_\eps$ in \eqref{Fluid-Quanities} and the uniform energy bound \eqref{Uniform-Bnd-1}, we obtain
\begin{equation}\label{Uniform-Bnd-3}
  \begin{aligned}
    \sup_{t \geq 0} \big( \| \rho_\eps \|_{H^s_x} + \| u_\eps \|_{H^s_x} + \| \theta_\eps \|_{H^s_x} + \| n_\eps \|_{H^s_x} \big) \leq C\,.
  \end{aligned}
\end{equation}
We thereby deduce the following convergences from the convergence of \eqref{Convg-GEB} and the limit function $G(t,x,v)$ given in \eqref{Limit-G} that
\begin{equation}\label{Convg-rho-u-theta-n}
  \begin{aligned}
    & \rho_\eps = \tfrac{1}{2} \langle G_\eps , \mathsf{q}_2 \sqrt{M} \rangle_{L^2_v} \rightarrow \tfrac{1}{2} \langle G , \mathsf{q}_2 \sqrt{M} \rangle_{L^2_v} = \tfrac{1}{2} (\rho^+ + \rho^-) \overset{\triangle}{=} \rho \,, \\
    & u_\eps = \tfrac{1}{2} \langle G_\eps , \mathsf{q}_2 v \sqrt{M} \rangle_{L^2_v} \rightarrow \tfrac{1}{2} \langle G , \mathsf{q}_2 v \sqrt{M} \rangle_{L^2_v} = u \,, \\
    & \theta_\eps = \tfrac{1}{2} \langle G_\eps , \mathsf{q}_2 ( \tfrac{|v|^2}{3} - 1 ) \sqrt{M} \rangle_{L^2_v} \rightarrow \tfrac{1}{2} \langle G , \mathsf{q}_2 ( \tfrac{|v|^2}{3} - 1 ) \sqrt{M} \rangle_{L^2_v} = \theta \,, \\
    & n_\eps = \langle G_\eps , \mathsf{q}_1 \sqrt{M} \rangle_{L^2_v} \rightarrow \langle G , \mathsf{q}_1 \sqrt{M} \rangle_{L^2_v} = \rho^+ - \rho^- \overset{\triangle}{=} n \,,
  \end{aligned}
\end{equation}
weakly-$\star$ for $t \geq 0$ and weakly in $H^s_x $ as $\eps \rightarrow 0$. It remains to find the limits of $j_\eps$ and $w_\eps$ defined in \eqref{Fluid-Quanities}. We first analyze the relations between $\Pi^\perp_{\mathcal{L}} G_\eps^\pm$ and $\mathbb{P}^\perp G_\eps$. By the definition of the operator $\Pi_{\mathcal{L}}^\perp$, we have
\begin{equation}\label{Relt-perp-B-VMB-1}
  \begin{aligned}
    \Pi_{\mathcal{L}}^\perp G_\eps^+ = & G_\eps^+ - \sum_{1 \leq i \leq 5} \langle G_\eps^+ , \chi_i \rangle_{L^2_v} \chi_i \\
    = & G_\eps^+ - \tfrac{1}{2} \langle G_\eps , (\mathsf{q}_1 + \mathsf{q}_2) \sqrt{M} \rangle_{L^2_v} \sqrt{M} \\
    & - \Big[ \tfrac{1}{2} \langle G_\eps , \mathsf{q}_2 v \sqrt{M} \rangle_{L^2_v} + \tfrac{1}{2} \langle G_\eps , \mathsf{q}_1 v \sqrt{M} \rangle_{L^2_v} \Big] \cdot v \sqrt{M} \\
    & - \Big[ \tfrac{1}{2} \langle G_\eps , \mathsf{q}_2 ( \tfrac{|v|^2}{3} - 1 ) \sqrt{M} \rangle_{L^2_v} + \tfrac{1}{2} \langle G_\eps , \mathsf{q}_1 ( \tfrac{|v|^2}{3} - 1 ) \sqrt{M} \rangle_{L^2_v} \Big] ( \tfrac{|v|^2}{2} - \tfrac{3}{2} ) \sqrt{M} \\
    = & G_\eps^+ - ( \mathbb{P} G_\eps )^+ - \tfrac{\eps}{2} j_\eps \cdot v \sqrt{M} - \tfrac{\eps}{2} w_\eps ( \tfrac{|v|^2}{2} - \tfrac{3}{2} ) \sqrt{M} \\
    = & ( \mathbb{P}^\perp G_\eps )^+ - \tfrac{\eps}{2} j_\eps \cdot v \sqrt{M} - \tfrac{\eps}{2} w_\eps ( \tfrac{|v|^2}{2} - \tfrac{3}{2} ) \sqrt{M} \,,
  \end{aligned}
\end{equation}
where the relations \eqref{Relation-VMN-Boltzmann-Fluid} and \eqref{VMB-Proj-Smp} are utilized. Moreover, similar calculations in the previous process \eqref{Relt-perp-B-VMB-1} yield that
\begin{equation}\label{Relt-perp-B-VMB-2}
  \begin{aligned}
    \Pi_{\mathcal{L}}^\perp G_\eps^- = ( \mathbb{P}^\perp G_\eps )^+ + \tfrac{\eps}{2} j_\eps \cdot v \sqrt{M} + \tfrac{\eps}{2} w_\eps ( \tfrac{|v|^2}{2} - \tfrac{3}{2} ) \sqrt{M} \,.
  \end{aligned}
\end{equation}
We thereby derive from \eqref{Relt-perp-B-VMB-1} and \eqref{Relt-perp-B-VMB-2} that 
\begin{equation}\label{Relt-perp-B-VMB-3}
  \begin{aligned}
    \| \partial^m_x ( \mathbb{P}^\perp G_\eps )^\pm \|^2_{L^2_{x,v}} = & \| \partial^m_x \Pi_{\mathcal{L}}^\perp G_\eps^\pm \|^2_{L^2_{x,v}} + \tfrac{\eps^2}{4} \| \partial^m_x j_\eps \|^2_{L^2_x} \langle v \sqrt{M} , v \sqrt{M} \rangle_{L^2_v} \\
    & + \tfrac{\eps^2}{4} \| \partial^m_x w_\eps \|^2_{L^2_x} \langle (\tfrac{|v|^2}{2} - \tfrac{3}{2}) \sqrt{M} , (\tfrac{|v|^2}{2} - \tfrac{3}{2}) \sqrt{M} \rangle_{L^2_v} \\
    = & \| \partial^m_x \Pi_{\mathcal{L}}^\perp G_\eps^\pm \|^2_{L^2_{x,v}} + \tfrac{\eps^2}{4} \| \partial^m_x j_\eps \|^2_{L^2_x} + \tfrac{3 \eps^2}{8} \| \partial^m_x w_\eps \|^2_{L^2_x} 
  \end{aligned}
\end{equation}
for all $|m| \leq s$, where $ \langle v \sqrt{M} , v \sqrt{M} \rangle_{L^2_v} = 1 $ and $ \langle (\tfrac{|v|^2}{2} - \tfrac{3}{2}) \sqrt{M} , (\tfrac{|v|^2}{2} - \tfrac{3}{2}) \sqrt{M} \rangle_{L^2_v} = \tfrac{3}{2} $ are used. Then, from the part (1) of Lemma \ref{Lmm-nu-norm}, the energy dissipation bound \eqref{Uniform-Bnd-2} and the relation \eqref{Relt-perp-B-VMB-3}, we deduce that
\begin{equation}\label{Uniform-Bnd-4}
  \begin{aligned}
    & \| j_\eps \|^2_{L^2 (\R^+; H^s_x)} + \| w_\eps \|^2_{L^2 (\R^+; H^s_x)} \\
    \leq & \tfrac{C}{\eps^2} \int_0^\infty \| \mathbb{P}^\perp G_\eps (t)  \|^2_{H^s_x L^2_v} \d t \leq \tfrac{C}{\eps^2} \int_0^\infty \| \mathbb{P}^\perp G_\eps (t) \|^2_{H^s_{x,v}(\nu)} \d t \leq C^* \,.
  \end{aligned}
\end{equation}
Consequently, there are functions $j$, $w \in L^2 (\R^+ ; H^s_x)$ such that
\begin{equation}\label{Convg-j-w}
  \begin{aligned}
    j_\eps \rightarrow j \quad \textrm{and }\ w_\eps \rightarrow w
  \end{aligned}
\end{equation}
weakly in $L^2 (\R^+ ; H^s_x)$ as $\eps \rightarrow 0$.

\subsection{Convergences to limiting equations} In this subsection, we will derive the two fluid incompressible Navier-Stokes-Fourier-Maxwell equations \eqref{INSFM-Ohm} with Ohm's law from the conservation laws \eqref{Local-Consvtn-Law} in Lemma \ref{Lmm-Local-Consvtn-Law} and the convergences obtained in the previous subsection.

\subsubsection{Incompressibility and Boussinesq relation} 

From the first equation of \eqref{Local-Consvtn-Law} in Lemma \ref{Lmm-Local-Consvtn-Law} and the energy uniform bound \eqref{Uniform-Bnd-1}, it is easy to deduce
\begin{equation}
  \begin{aligned}
    \div_x \, u_\eps = - \eps \partial_t \rho_\eps \rightarrow 0
  \end{aligned}
\end{equation}
in the sense of distributions as $\eps \rightarrow 0$, which imply that by combining with the convergence \eqref{Convg-rho-u-theta-n}
\begin{equation}\label{Incompressibility}
  \begin{aligned}
    \div_x \, u = 0 \,.
  \end{aligned}
\end{equation}
Via the second equation of \eqref{Local-Consvtn-Law}, we have
\begin{equation}
  \begin{aligned}
    \nabla_x ( \rho_\eps + \theta_\eps ) = - \eps \partial_t u_\eps - \div_x \, \big\langle \widehat{A} (v) \sqrt{M} , \mathcal{L} ( \tfrac{G_\eps \cdot \mathsf{q}_2}{2} ) \big\rangle_{L^2_v} + \tfrac{\eps}{2} ( n_\eps E_\eps + j_\eps \times B_\eps ) \,.
  \end{aligned}
\end{equation}
Noticing that
\begin{equation}
  \begin{aligned}
    \tfrac{1}{2} G_\eps \cdot \mathsf{q}_2 = \tfrac{1}{2} \mathbb{P}^\perp G_\eps \cdot \mathsf{q}_2 + \rho_\eps \sqrt{M} + u_\eps \cdot v \sqrt{M} + \theta_\eps ( \tfrac{|v|^2}{2} - \tfrac{3}{2} ) \sqrt{M} \,,
  \end{aligned}
\end{equation}
one has
\begin{equation}\label{Bolz-L-G-q}
  \begin{aligned}
    \mathcal{L} ( \tfrac{G_\eps \cdot \mathsf{q}_2}{2} ) = \mathcal{L} ( \tfrac{\mathbb{P}^\perp G_\eps \cdot \mathsf{q}_2}{2} ) \,.
  \end{aligned}
\end{equation}
Thus, we have
\begin{equation}
  \begin{aligned}
    \div_x \, \big\langle \widehat{A} (v) \sqrt{M} , \mathcal{L} ( \tfrac{G_\eps \cdot \mathsf{q}_2}{2} ) \big\rangle_{L^2_v} = \div_x \, \big\langle \widehat{A} (v) \sqrt{M} , \mathcal{L} ( \tfrac{\mathbb{P}^\perp G_\eps \cdot \mathsf{q}_2}{2} ) \big\rangle_{L^2_v} \\
    = \tfrac{1}{2} \div_x \, \big\langle A (v) \sqrt{M} , \mathbb{P}^\perp G_\eps \cdot \mathsf{q}_2 ) \big\rangle_{L^2_v} \,,
  \end{aligned}
\end{equation}
where the self-adjointness of $\mathcal{L}$ is utilized. Then we derive from the H\"older inequality, the part (1) of Lemma \ref{Lmm-nu-norm} and the uniform energy dissipation bound \eqref{Uniform-Bnd-2} that
\begin{equation}
  \begin{aligned}
    \int_0^\infty \big\| \div_x \, \big\langle \widehat{A} (v) \sqrt{M} , \mathcal{L} ( \tfrac{G_\eps \cdot \mathsf{q}_2}{2} ) \big\rangle_{L^2_v} \big\|^2_{H^{s-1}_x} \d t \leq C \int_0^\infty \| \mathbb{P}^\perp G_\eps \|^2_{H^s_{x,v}(\nu)} \d t \leq C \eps^2 \,.
  \end{aligned}
\end{equation}
Furthermore, from the uniform bounds \eqref{Uniform-Bnd-1}, \eqref{Uniform-Bnd-3} and \eqref{Uniform-Bnd-4}, we easily derive that
\begin{equation}
  \begin{aligned}
    \sup_{t \geq 0} \big\| \tfrac{\eps}{2} ( n_\eps E_\eps + j_\eps \times B_\eps ) \big\|_{H^s_x} \leq C \eps \,.
  \end{aligned}
\end{equation}
Consequently, it is easy to deduce that
\begin{equation}
  \begin{aligned}
    \nabla_x (\rho_\eps + \theta_\eps) \rightarrow 0
  \end{aligned}
\end{equation}
in the sense of distributions as $\eps \rightarrow 0$, which, combining with the convergence \eqref{Convg-rho-u-theta-n}, gives the Boussinesq relation
\begin{equation}\label{Boussinesq}
  \begin{aligned}
    \rho + \theta  = 0 \,.
  \end{aligned}
\end{equation}

\subsubsection{Convergences of $\tfrac{3}{5} \theta_\eps - \tfrac{2}{5} \rho_\eps$, $\mathcal{P} u_\eps$, $n_\eps$, $E_\eps$ and $B_\eps$}

Before doing this, we introduce the following Aubin-Lions-Simon Theorem, a fundamental result of compactness in the study of nonlinear evolution problems, which can be referred to Theorem II.5.16 of \cite{Boyer-Fabrie-2013BOOK} or \cite{Simon-1987-AMPA}, for instance.
\begin{lemma}[Aubin-Lions-Simon Theorem]\label{Lmm-Aubin-Lions-Simon}
	Let $B_0 \subset B_1 \subset B_2$ be three Banach spaces. We assume that the embedding of $B_1$ in $B_2$ is continuous ans that the embedding of $B_0$ in $B_1$ is compact. Let $p$, $r$ be such that $1 \leq p, r \leq + \infty$. For $T > 0$, we define
	\begin{equation}
	  \begin{aligned}
	    E_{p,r} = \Big\{ u \in L^p ( 0, T; B_0 ) , \partial_t u \in L^r (0, T; B_2) \Big\} \,.
	  \end{aligned}
	\end{equation}
	\begin{enumerate}
		\item If $p < + \infty$, the embedding of $E_{p,r}$ in $L^p (0,T; B_1)$ is compact.
		
		\item If $p = + \infty$ and if $r > 1$, the embedding of $E_{p,r}$ in $C( 0,T; B_1 )$.
	\end{enumerate}
\end{lemma}
We emphasize that the reflexivity of the spaces considered in Lemma \ref{Lmm-Aubin-Lions-Simon} is not assumed.

We now consider the convergence of $\tfrac{3}{5} \theta_\eps - \tfrac{2}{5} \rho_\eps$. The third equation of \eqref{Local-Consvtn-Law} multiplied by $\tfrac{3}{5}$ minus $\tfrac{2}{5}$ times of the first equation of \eqref{Local-Consvtn-Law} gives
\begin{equation}
  \begin{aligned}
    \partial_t ( \tfrac{3}{5} \theta_\eps - \tfrac{2}{5} \rho_\eps ) + \tfrac{2}{5} \div_x \, \big\langle \widehat{B} (v) \sqrt{M} , \tfrac{1}{\eps} \mathcal{L} ( \tfrac{G_\eps \cdot \mathsf{q}_2}{2} ) \big\rangle_{L^2_v} = \tfrac{\eps}{5} j_\eps \cdot E_\eps \,.
  \end{aligned}
\end{equation}
Noticing the relation \eqref{Bolz-L-G-q}, we yield that
\begin{equation}
  \begin{aligned}
    \| \partial_t ( \tfrac{3}{5} \theta_\eps - \tfrac{2}{5} \rho_\eps ) \|_{H^{s-1}_x} = & \Big\| \tfrac{\eps}{5} j_\eps \cdot E_\eps - \tfrac{2}{5} \div_x \, \big\langle \widehat{B} (v) \sqrt{M} , \tfrac{1}{\eps} \mathcal{L} ( \tfrac{G_\eps \cdot \mathsf{q}_2}{2} ) \big\rangle_{L^2_v} \Big\|_{H^{s-1}_x} \\
    = & \Big\| \tfrac{\eps}{5} j_\eps \cdot E_\eps - \tfrac{2}{5} \div_x \, \big\langle \widehat{B} (v) \sqrt{M} , \tfrac{1}{\eps} \mathcal{L} ( \tfrac{ \mathbb{P}^\perp G_\eps \cdot \mathsf{q}_2}{2} ) \big\rangle_{L^2_v} \Big\|_{H^{s-1}_x} \\
    = & \Big\| \tfrac{\eps}{5} j_\eps \cdot E_\eps - \tfrac{1}{5 \eps} \div_x \, \big\langle B (v) \sqrt{M} , \mathbb{P}^\perp G_\eps \cdot \mathsf{q}_2 \big\rangle_{L^2_v} \Big\|_{H^{s-1}_x} \\
    \leq & C \| j_\eps \cdot E_\eps \|_{H^{s-1}_x} + \tfrac{C}{\eps} \| B(v) \sqrt{M} \|_{L^2_v} \| \nabla_x \mathbb{P}^\perp G_\eps \|_{H^{s-1}_x L^2_v} \\
    \leq & C \| j_\eps \|_{H^s_x} \| E_\eps \|_{H^s_x} + \tfrac{C}{\eps} \| \mathbb{P}^\perp G_\eps \|_{H^s_{x,v}(\nu)} \,,
  \end{aligned}
\end{equation}
which immediately derives from the uniform energy bounds \eqref{Uniform-Bnd-1}, \eqref{Uniform-Bnd-3} and \eqref{Uniform-Bnd-4} that
\begin{equation}\label{Local-t-theta-rho}
  \begin{aligned}
    & \| \partial_t ( \tfrac{3}{5} \theta_\eps - \tfrac{2}{5} \rho_\eps ) \|_{L^\infty(0,T ; H^{s-1}_x )} \\
    \leq & C \| j_\eps \|_{L^2(\R^+ ; H^s_x)} \| E_\eps \|_{L^\infty (\R^+ ; H^s_x)} + \tfrac{C}{\eps} \| \mathbb{P}^\perp G_\eps \|_{L^2(\R^+ ; H^s_{x,v}(\nu))} \leq C 
  \end{aligned}
\end{equation}
for any $T > 0$ and $0 < \eps \leq 1$. It is easily derived from the definition of $\rho_\eps$, $\theta_\eps$ in \eqref{Fluid-Quanities} and the uniform energy bound \eqref{Uniform-Bnd-1} that
\begin{equation}\label{Local-theta-rho}
  \begin{aligned}
    \big\| ( \tfrac{3}{5} \theta_\eps - \tfrac{2}{5} \rho_\eps ) \big\|_{L^\infty(0,T; H^s_x)} \leq C 
  \end{aligned}
\end{equation}
for all $T>0$ and $0 < \eps \leq 1$. One notices that
\begin{equation}\label{Embeddings-Hs}
  \begin{aligned}
    H^s_x \hookrightarrow H^{s-1}_x \hookrightarrow H^{s-1}_x \,,
  \end{aligned}
\end{equation}
where the embedding of $H^s_x $ in $ H^{s-1}_x$ is compact and the embedding of  $H^{s-1}_x $ in $ H^{s-1}_x$ is naturally continuous. Then, from Aubin-Lions-Simon Theorem in Lemma \ref{Lmm-Aubin-Lions-Simon}, the bounds \eqref{Local-t-theta-rho}, \eqref{Local-theta-rho} and the embeddings \eqref{Embeddings-Hs}, we deduce that there is a $\widetilde{\theta} \in L^\infty (\R^+ ; H^s_x) \cap C(\R^+ ; H^{s-1}_x)$ such that 
\begin{equation}
  \begin{aligned}
    \tfrac{3}{5} \theta_\eps - \tfrac{2}{5} \rho_\eps \rightarrow \widetilde{\theta}
  \end{aligned}
\end{equation}
strongly in $C(0, T ; H^{s-1}_x)$ for any $T > 0$ as $\eps \rightarrow 0$. Combining with the convergences \eqref{Convg-rho-u-theta-n}, we know that $ \widetilde{\theta} = \tfrac{3}{5} \theta - \tfrac{2}{5} \rho $. Then the relation \eqref{Boussinesq} and $\theta = ( \tfrac{3}{5} \theta - \tfrac{2}{5} \rho ) + \tfrac{2}{5} ( \rho + \theta )$ give us $\widetilde{\theta} = \theta$. As a result, 
\begin{equation}
  \begin{aligned}
    \tfrac{3}{5} \theta_\eps - \tfrac{2}{5} \rho_\eps \rightarrow \theta
  \end{aligned}
\end{equation}
strongly in $C(\R^+; H^{s-1}_x)$ as $\eps \rightarrow 0$, where $\theta \in L^\infty (\R^+; H^s_x) \cap C(\R^+; H^{s-1}_x)$. Noticing that $\theta_\eps = ( \tfrac{3}{5} \theta_\eps - \tfrac{2}{5} \rho_\eps ) + \tfrac{2}{5} (\rho_\eps + \theta_\eps)$, we thereby derive from the convergences \eqref{Convg-rho-u-theta-n} that
\begin{equation}
  \begin{aligned}
    \rho_\eps + \theta_\eps \rightarrow 0
  \end{aligned}
\end{equation}
weakly-$\star$ in $t \geq 0$, weakly in $H^s_x$ and strongly in $H^{s-1}_x$ as $\eps \rightarrow 0$.

Next we consider the convergence of $\mathcal{P} u_\eps$, where $\mathcal{P}$ is the Leray projection operator. Taking $\mathcal{P}$ on the second equation of \eqref{Local-Consvtn-Law} gives 
\begin{equation}\label{Local-Pu}
   \begin{aligned}
     \partial_t \mathcal{P} u_\eps + \mathcal{P} \div_x \, \big\langle \widehat{A} (v) \sqrt{M} , \tfrac{1}{\eps} \mathcal{L} ( \tfrac{G_\eps \cdot \mathsf{q}_2}{2} ) \big\rangle_{L^2_v} = \tfrac{1}{2} \mathcal{P} ( n_\eps E_\eps + j_\eps \times B_\eps ) \,. 
   \end{aligned}
\end{equation}
It is easy to derive from the relation \eqref{Bolz-L-G-q}, the H\"older inequality, the bound $\| A(v) \sqrt{M} \|_{L^2_v} \leq C$, the calculus inequality and the part (1) of Lemma \ref{Lmm-nu-norm} that
\begin{equation}
  \begin{aligned}
    \| \partial_t \mathcal{P} u_\eps \|_{H^{s-1}_x} = & \Big\| \tfrac{1}{2} \mathcal{P} ( n_\eps E_\eps + j_\eps \times B_\eps ) - \mathcal{P} \div_x \, \big\langle \widehat{A} (v) \sqrt{M} , \tfrac{1}{\eps} \mathcal{L} ( \tfrac{G_\eps \cdot \mathsf{q}_2}{2} ) \big\rangle_{L^2_v} \Big\|_{H^{s-1}_x} \\
    = & \Big\| \tfrac{1}{2} \mathcal{P} ( n_\eps E_\eps + j_\eps \times B_\eps ) - \mathcal{P} \div_x \, \big\langle \widehat{A} (v) \sqrt{M} , \tfrac{1}{\eps} \mathcal{L} ( \tfrac{\mathbb{P}^\perp G_\eps \cdot \mathsf{q}_2}{2} ) \big\rangle_{L^2_v} \Big\|_{H^{s-1}_x} \\
    = & \Big\| \tfrac{1}{2} \mathcal{P} ( n_\eps E_\eps + j_\eps \times B_\eps ) - \mathcal{P} \div_x \, \big\langle A (v) \sqrt{M} , \tfrac{1}{2 \eps}  \mathbb{P}^\perp G_\eps \cdot \mathsf{q}_2 \big\rangle_{L^2_v} \Big\|_{H^{s-1}_x} \\
    \leq & C \| n_\eps E_\eps \|_{H^{s-1}_x} + C \| j_\eps \times B_\eps \|_{H^{s-1}_x} \\
    & + \tfrac{C}{\eps} \| A(v) \sqrt{M} \|_{L^2_v} \| \nabla_x \mathbb{P}^\perp G_\eps \|_{H^{s-1}_x L^2_v} \\
    \leq & C \| n_\eps \|_{H^s_x} \| E_\eps \|_{H^s_x} + C \| j_\eps \|_{H^s_x} \| B_\eps \|_{H^s_x} + \tfrac{C}{\eps} \| \mathbb{P}^\perp G_\eps \|_{H^s_{x,v}(\nu)} \,,
  \end{aligned}
\end{equation}
which, by the uniform energy bounds \eqref{Uniform-Bnd-1}, \eqref{Uniform-Bnd-3} and \eqref{Uniform-Bnd-4}, implies that
\begin{equation}\label{Local-t-Pu}
  \begin{aligned}
    & \| \partial_t \mathcal{P} u_\eps \|_{L^2(0,T; H^{s-1}_x)}  \leq C \| n_\eps \|_{L^\infty(\R^+; H^s_x)} \| E_\eps \|_{L^2(\R^+ ; H^{s-1}_x)} \sqrt{T} \\
    & + C \| j_\eps \|_{L^2(\R^+ ; H^s_x)} \| B_\eps \|_{L^\infty(\R^+; H^s_x)} + \tfrac{C}{\eps} \| \mathbb{P}^\perp G_\eps \|_{L^2(\R^+; H^s_{x,v}(\nu))} \\ 
    \leq & C \sqrt{T} + C
  \end{aligned}
\end{equation}
for any $T > 0$ and $0 < \eps \leq 1$. Furthermore, from the definition of $u_\eps$ in \eqref{Fluid-Quanities} and the uniform energy bound \eqref{Uniform-Bnd-1}, we derive that for all $T > 0$ and $0 < \eps \leq 1$
\begin{equation}\label{Local-Pu-1}
  \begin{aligned}
    \| \mathcal{P} u_\eps \|_{L^\infty(0,T ; H^s_x)} \leq C \,.
  \end{aligned}
\end{equation}
Then, from Aubin-Lions-Simon Theorem in Lemma \ref{Lmm-Aubin-Lions-Simon}, the bounds \eqref{Local-t-Pu}, \eqref{Local-Pu-1} and the embeddings \eqref{Embeddings-Hs}, we derive that there is a $\widetilde{u} \in L^\infty (\R^+ ; H^s_x) \cap C(\R^+; H^{s-1}_x)$ such that 
\begin{equation}
  \begin{aligned}
    \mathcal{P} u_\eps \rightarrow \widetilde{u} 
  \end{aligned}
\end{equation}
strongly in $C(0,T; H^{s-1}_x)$ for all $T > 0$ as $\eps \rightarrow 0$. Furthermore, from the convergences \eqref{Convg-rho-u-theta-n} and the incompressibility \eqref{Incompressibility}, we deduce
\begin{equation}
  \begin{aligned}
    \widetilde{u} = \mathcal{P} u = u \,.
  \end{aligned}
\end{equation}
Consequently, 
\begin{equation}
  \begin{aligned}
    \mathcal{P} u_\eps \rightarrow u
  \end{aligned}
\end{equation}
strongly in $C(\R^+; H^{s-1}_x)$ as $\eps \rightarrow 0$, where $u \in L^\infty (\R^+; H^s_x) \cap C(\R^+; H^{s-1}_x)$. Furthermore, we know that
\begin{equation}
  \begin{aligned}
    \mathcal{P}^\perp u_\eps \rightarrow 0
  \end{aligned}
\end{equation}
weakly-$\star$ in $t \geq 0$, weakly in $H^s_x$ and strongly in $H^{s-1}_x$ as $\eps \rightarrow 0$.

Next, we consider the convergence of $n_\eps$. From the local conservation laws \eqref{Local-Consvtn-Law} in Lemma \ref{Lmm-Local-Consvtn-Law}, we know that $n_\eps$ satisfies
\begin{equation}
  \begin{aligned}
    \partial_t n_\eps  + \div_x \, j_\eps = 0 \,.
  \end{aligned}
\end{equation}
Then, we have
\begin{equation}
  \begin{aligned}
    \| \partial_t n_\eps \|_{H^{s-1}_x} = \| \div_x \, j_\eps \|_{H^{s-1}_x} \leq C \| j_\eps \|_{H^s_x} \,,
  \end{aligned}
\end{equation}
which yields by using the bound \eqref{Uniform-Bnd-4} that
\begin{equation}\label{Local-t-n}
  \begin{aligned}
    \| \partial_t n_\eps \|_{L^2 (0,T; H^{s-1}_x)} \leq C \| j_\eps \|_{L^2(\R^+ ; H^s_x)} \leq C 
  \end{aligned}
\end{equation}
for any $T > 0$ and $0 < \eps \leq 1$. Moreover, from the bound \eqref{Uniform-Bnd-3}, we know that for any $T > 0$ and $0 < \eps \leq 1$
\begin{equation}\label{Local-n}
  \begin{aligned}
    \| n_\eps \|_{L^\infty (0,T ; H^s_x)} \leq C \,.
  \end{aligned}
\end{equation}
Then, it is derived from Aubin-Lions-Simon Theorem in Lemma \ref{Lmm-Aubin-Lions-Simon}, the bounds \eqref{Local-t-n}, \eqref{Local-n} and the embeddings \eqref{Embeddings-Hs} that $n_\eps \rightarrow n$ strongly in $C(0,T;H^{s-1}_x)$ for any $T > 0$ as $\eps \rightarrow 0$. Hence, we have
\begin{equation}
  \begin{aligned}
    n_\eps \rightarrow n
  \end{aligned}
\end{equation}
strongly in $C(\R^+; H^{s-1}_x)$ as $\eps \rightarrow 0$, where $n \in L^\infty (\R^+ ; H^s_x) \cap C(\R^+; H^{s-1}_x)$.

We next consider the convergences of $E_\eps$ and $B_\eps$. Noticing that $E_\eps$ and $B_\eps$ subject to 
\begin{equation}
  \begin{aligned}
    \partial_t E_\eps - \nabla_x \times B_\eps = - j_\eps \,, \ \partial_t B_\eps + \nabla_x \times E_\eps = 0 \,,
  \end{aligned}
\end{equation}
we deduce that
\begin{equation}
  \begin{aligned}
    & \| \partial_t E_\eps \|_{H^{s-1}_x} + \| \partial_t B_\eps \|_{H^{s-1}_x} \\
    = & \| \nabla_x \times B_\eps - j_\eps \|_{H^{s-1}_x} + \| \nabla_x \times E_\eps \|_{H^{s-1}_x} \\
    \leq & C ( \| B_\eps \|_{H^s_x}  + \| j_\eps \|_{H^s_x} + \| E_\eps \|_{H^s_x} ) \,,
  \end{aligned}
\end{equation}
which reduces to
\begin{equation}\label{Local-t-E-B}
  \begin{aligned}
    & \| \partial_t E_\eps \|_{L^2(0,T;H^{s-1}_x)} + \| \partial_t B_\eps \|_{L^2(0,T; H^{s-1}_x) } \\
    \leq & C ( \| B_\eps \|_{L^\infty (\R^+; H^s_x)} + \| E_\eps \|_{L^\infty (\R^+; H^s_x)} ) \sqrt{T} + C \| j_\eps \|_{L^2(\R^+; H^s_x)} \\
    \leq & C( \sqrt{T} + 1 )
  \end{aligned}
\end{equation}
for all $T > 0$ and $0 < \eps \leq 1$. Here the uniform energy bounds \eqref{Uniform-Bnd-1} and \eqref{Uniform-Bnd-4} are utilized. Moreover, from the bound \eqref{Uniform-Bnd-1}, we have
\begin{equation}\label{Local-E-B}
  \begin{aligned}
    \| E_\eps \|_{L^\infty (0,T;H^s_x)} + \| B_\eps \|_{L^\infty (0,T;H^s_x)} \leq C
  \end{aligned}
\end{equation}
for any $T > 0$ and $0 < \eps \leq 1$. Then, from Aubin-Lions-Simon Theorem in Lemma \ref{Lmm-Aubin-Lions-Simon}, the uniform bounds \eqref{Local-t-E-B}, \eqref{Local-E-B} and the embeddings  \eqref{Embeddings-Hs}, we deduce that $E_\eps \rightarrow E$ and $B_\eps \rightarrow B$ strongly in $C(0,T;H^{s-1}_x)$ for any $T > 0$ as $\eps \rightarrow 0$. Namely, we have
\begin{equation}
\begin{aligned}
E_\eps \rightarrow E \quad \textrm{and} \quad B_\eps \rightarrow B
\end{aligned}
\end{equation}
strongly in $C(\R^+;H^{s-1}_x)$ as $\eps \rightarrow 0$, where $E, B \in L^\infty(\R^+; H^s_x) \cap C(\R^+; H^{s-1}_x)$.

In summary, we have deduced the following convergences:
\begin{equation}\label{Convg-strong}
  \begin{aligned}
    \big( \mathcal{P} u_\eps , \tfrac{3}{5} \theta_\eps - \tfrac{2}{5} \rho_\eps , n_\eps , E_\eps , B_\eps \big) \rightarrow (u, \theta, n, E, B)
  \end{aligned}
\end{equation}
strongly in $C(\R^+; H^{s-1}_x)$ as $\eps \rightarrow 0$, where $(u, \theta, n, E, B) \in L^\infty(\R^+; H^s_x) \cap C(\R^+ ; H^{s-1}_x)$, and
\begin{equation}\label{Convg-weak-to-0}
  \begin{aligned}
    \big( \mathcal{P}^\perp u_\eps , \rho_\eps + \theta_\eps \big) \rightarrow ( 0, 0 )
  \end{aligned}
\end{equation}
weakly-$\star$ in $t \geq 0$, weakly in $H^s_x$ and strongly in $H^{s-1}_x$ as $\eps \rightarrow 0$, and
\begin{equation}\label{Convg-weak-jw}
  \begin{aligned}
    (j_\eps, w_\eps) \rightarrow (j, w)
  \end{aligned}
\end{equation}
weakly in $L^2(\R^+; H^s_x)$ as $\eps \rightarrow 0$, where $(j,w) \in L^2 (\R^+; H^s_x)$.

\subsubsection{Ohm's law and energy equivalence relation}

We now derive the Ohm's law in the last second equality of \eqref{INSFM-Ohm} and energy equivalence relation in the last first equality of \eqref{INSFM-Ohm}. As in \eqref{Ohm-Law-Kinetic}, multiplying the first $G_\eps$-equation of \eqref{VMB-G} by $\mathsf{q}_1$ gives
\begin{equation}
  \begin{aligned}
    \tfrac{1}{\eps} ( \mathcal{L} + \mathfrak{L} ) ( G_\eps \cdot \mathsf{q}_1 ) = - \eps \partial_t ( G_\eps \cdot \mathsf{q}_1 ) - v \cdot \nabla_x ( G_\eps \cdot \mathsf{q}_1 ) - ( \eps E_\eps + v \times B_\eps ) \cdot \nabla_v (G_\eps \cdot \mathsf{q}_2) \\
    + 2 E_\eps \cdot \Phi (v) + \tfrac{1}{2} \eps ( E_\eps \cdot v ) ( G_\eps \cdot \mathsf{q}_2 ) + \Gamma (G_\eps , G_\eps) \cdot \mathsf{q}_1 \,,
  \end{aligned}
\end{equation}
where $\Phi (v) = v \sqrt{M}$ and the operator $\mathfrak{L}$ is defined in \eqref{Mathfrak-L}. By the definition of $\mathbb{P} G_\eps$ in \eqref{VMB-Proj-Smp}, $\mathsf{q}_1 \cdot \mathsf{q}_1 = 2$ and $\mathsf{q}_1 \cdot \mathsf{q}_2 = 0$, we have
\begin{equation}\label{PG-q1}
  \begin{aligned}
    \mathbb{P} G_\eps \cdot \mathsf{q}_1 = n_\eps \sqrt{M} 
  \end{aligned}
\end{equation}
and
\begin{equation}\label{PG-q2}
   \begin{aligned}
     \mathbb{P} G_\eps \cdot \mathsf{q}_2 = \rho_\eps \sqrt{M} + \underset{2 h_\eps}{\underbrace{ 2 u_\eps \cdot v \sqrt{M} + 2 \theta_\eps ( \tfrac{|v|^2}{2} - \tfrac{3}{2} ) \sqrt{M} }} \,.
   \end{aligned}
\end{equation}
Then, by the decomposition $G_\eps = \mathbb{P} G_\eps + \mathbb{P}^\perp G_\eps$, the relations \eqref{PG-q1} and \eqref{PG-q2}, we can calculate 
\begin{equation}
  \begin{aligned}
    v \cdot \nabla_x ( G_\eps \cdot \mathsf{q}_1 ) = & v \cdot \nabla_x ( \mathbb{P} G_\eps \cdot \mathsf{q}_1 ) + v \cdot \nabla_x ( \mathbb{P}^\perp G_\eps \cdot \mathsf{q}_1 ) \\
    = & \nabla_x n_\eps \cdot \Phi (v) + v \cdot \nabla_x ( \mathbb{P}^\perp G_\eps \cdot \mathsf{q}_1 ) \,,
  \end{aligned}
\end{equation}
and
\begin{equation}
  \begin{aligned}
    ( v \times B_\eps ) \cdot \nabla_v ( G_\eps \cdot \mathsf{q}_2) = & ( v \times B_\eps ) \cdot \nabla_v ( \mathbb{P} G_\eps \cdot \mathsf{q}_2) + ( v \times B_\eps ) \cdot \nabla_v ( \mathbb{P}^\perp G_\eps \cdot \mathsf{q}_2) \\
    = & - (v \times B_\eps) \cdot v \big( \rho_\eps + 2 u_\eps \cdot v + 2 \theta_\eps ( \tfrac{|v|^2}{2} - \tfrac{3}{2} ) \big) \sqrt{M} \\
    & + ( v \times B_\eps ) \cdot u_\eps \sqrt{M} + ( v \times B_\eps ) \cdot \nabla_v ( \mathbb{P}^\perp G_\eps \cdot \mathsf{q}_2) \\
    = & - 2 ( u_\eps \times B_\eps ) \cdot \Phi (v) + ( v \times B_\eps ) \cdot \nabla_v ( \mathbb{P}^\perp G_\eps \cdot \mathsf{q}_2) \,,
  \end{aligned}
\end{equation}
where the cancellation $(v \times B_\eps) \cdot v = 0$ and the relation $( v \times B_\eps ) \cdot u_\eps = - ( u_\eps \cdot B_\eps ) \cdot v$ are utilized. Furthermore, we have
\begin{equation}
  \begin{aligned}
    \Gamma ( G_\eps , G_\eps ) \cdot \mathsf{q}_1 = \Gamma ( \mathbb{P} G_\eps , \mathbb{P} G_\eps ) \cdot \mathsf{q}_1 + \Gamma ( \mathbb{P}^\perp G_\eps , \mathbb{P}^\perp G_\eps ) \cdot \mathsf{q}_1 \\
    + \Gamma ( \mathbb{P}^\perp G_\eps , \mathbb{P} G_\eps ) \cdot \mathsf{q}_1 + \Gamma ( \mathbb{P} G_\eps , \mathbb{P}^\perp G_\eps ) \cdot \mathsf{q}_1 \,.
  \end{aligned}
\end{equation}
For the term $\Gamma (\mathbb{P} G_\eps , \mathbb{P} G_\eps) \cdot \mathsf{q}_1$, we derive from the fact $\mathcal{Q} (\sqrt{M}, \sqrt{M}) = 0$, the definition of $h_\eps \in \textrm{Ker} (\mathcal{L})$ in \eqref{PG-q2}, the cancellation $\mathcal{L} h_\eps = 0$, the relations \eqref{PG-q1}-\eqref{PG-q2}, the definitions of operators $\mathcal{L}$ in \eqref{Linear-Oprt-B} and $\mathfrak{L}$ in \eqref{Mathfrak-L} that
\begin{equation}
  \begin{aligned}
    \Gamma (\mathbb{P} G_\eps , \mathbb{P} G_\eps) \cdot \mathsf{q}_1 = & \mathcal{Q} ( \mathbb{P} G_\eps^+ , \mathbb{P} G_\eps^+ + \mathbb{P} G_\eps^- ) - \mathcal{Q} ( \mathbb{P} G_\eps^- , \mathbb{P} G_\eps^+ + \mathbb{P} G_\eps^- ) \\
    = & \mathcal{Q} ( \mathbb{P} G_\eps^+ - \mathbb{P} G_\eps^- , \mathbb{P} G_\eps^+ + \mathbb{P} G_\eps^- ) \\
    = & \mathcal{Q} ( \mathbb{P} G_\eps \cdot \mathsf{q}_1 , \mathbb{P} G_\eps \cdot \mathsf{q}_2 ) \\
    = & \mathcal{Q} ( n_\eps \sqrt{M} , \rho_\eps \sqrt{M} + 2 h_\eps ) = 2 n_\eps \mathcal{Q} (\sqrt{M} , h_\eps) \\
    = & n_\eps \big[ \mathcal{Q} (h_\eps , \sqrt{M}) + \mathcal{Q} (\sqrt{M}, h_\eps) - \mathcal{Q} (h_\eps , \sqrt{M}) + \mathcal{Q} ( \sqrt{M} , h_\eps ) \big] \\
    = & - n_\eps \mathcal{L} h_\eps + n_\eps \mathfrak{L} h_\eps = n_\eps ( \mathcal{L} + \mathfrak{L} ) h_\eps - 2 n_\eps \mathcal{L} h_\eps \\
    = &  n_\eps ( \mathcal{L} + \mathfrak{L} ) h_\eps \,.
  \end{aligned}
\end{equation}
Consequently, from the previous equalities and the relations \eqref{PG-q1}-\eqref{PG-q2}, we have
\begin{equation}\label{Ohm-L+L-eps}
  \begin{aligned}
    \tfrac{1}{\eps} (\mathcal{L} + \mathfrak{L}) ( G_\eps \cdot \mathsf{q}_1 ) = 2 \big( - \tfrac{1}{2} \nabla_x n_\eps + E_\eps + u_\eps \times B_\eps ) \cdot \Phi (v) + n_\eps ( \mathcal{L} + \mathfrak{L} ) h_\eps + \mathcal{X}_\eps \,,
  \end{aligned}
\end{equation}
where $\mathcal{X}_\eps$ has the form of
\begin{equation}\label{X-eps}
  \begin{aligned}
    \mathcal{X}_\eps = & - \eps \partial_t ( G_\eps \cdot \mathsf{q}_1 ) - v \cdot \nabla_x ( \mathbb{P}^\perp G_\eps \cdot \mathsf{q}_1 ) - \eps E_\eps \cdot \nabla_v ( G_\eps \cdot \mathsf{q}_2 ) \\
    & - ( v \times B_\eps ) \cdot \nabla_v ( \mathbb{P}^\perp G_\eps \cdot \mathsf{q}_2 ) + \tfrac{1}{2} \eps ( E_\eps \cdot v ) ( G_\eps \cdot \mathsf{q}_2 ) \\
    & + \Gamma ( \mathbb{P}^\perp G_\eps , \mathbb{P}^\perp G_\eps ) \cdot \mathsf{q}_1 + \Gamma ( \mathbb{P}^\perp G_\eps , \mathbb{P} G_\eps ) \cdot \mathsf{q}_1 + \Gamma ( \mathbb{P} G_\eps , \mathbb{P}^\perp G_\eps ) \cdot \mathsf{q}_1 \,.
  \end{aligned}
\end{equation}

As shown in the equality \eqref{Ohm-Law-fluid}, via multiplying the relation \eqref{Ohm-L+L-eps} by $\widetilde{\Phi} (v)$ defined in \eqref{Tilde-Phi-Phi}, using Lemma \ref{Lmm-mathfrak-L} and the definition of $j_\eps$ in \eqref{Fluid-Quanities}, we obtain
\begin{equation}
  \begin{aligned}
    j_\eps = & 2 ( - \tfrac{1}{2} \nabla_x n_\eps + E_\eps + u_\eps \times B_\eps ) \cdot \langle \Phi (v) \otimes \widehat{\Phi} (v) \rangle_{L^2_v} + \langle n_\eps h_\eps , \Phi (v) \rangle_{L^2_v} + \langle \mathcal{X}_\eps , \widetilde{\Phi} (v) \rangle_{L^2_v} \\
    = & 2 ( - \tfrac{1}{2} \nabla_x n_\eps + E_\eps + u_\eps \times B_\eps ) \cdot \tfrac{1}{2} \sigma \mathbb{I}_3 + n_\eps u_\eps + \langle \mathcal{X}_\eps , \widetilde{\Phi} (v) \rangle_{L^2_v} \\
    = & n_\eps u_\eps + \sigma ( - \tfrac{1}{2} \nabla_x n_\eps + E_\eps + u_\eps \times B_\eps ) + \langle \mathcal{X}_\eps , \widetilde{\Phi} (v) \rangle_{L^2_v} \,.
  \end{aligned}
\end{equation}
Then, we have
\begin{equation}\label{Ohm-law-eps}
  \begin{aligned}
    j_\eps = & n_\eps \mathcal{P} u_\eps + \sigma ( - \tfrac{1}{2} \nabla_x n_\eps + E_\eps + \mathcal{P} u_\eps \times B_\eps ) \\
    & + n_\eps \mathcal{P}^\perp u_\eps + \sigma \mathcal{P}^\perp u_\eps \times B_\eps + \langle \mathcal{X}_\eps , \widetilde{\Phi} (v) \rangle_{L^2_v}  
  \end{aligned}
\end{equation}
by utilizing the decomposition $u_\eps = \mathcal{P} u_\eps + \mathcal{P}^\perp u_\eps$, where $\mathcal{P}$ is the Leray projection. From the convergences \eqref{Convg-strong}, we easily know that for $s \geq 3$
\begin{equation}\label{Convg-strong-Ohm}
  \begin{aligned}
    n_\eps \mathcal{P} u_\eps + \sigma ( - \tfrac{1}{2} \nabla_x n_\eps + E_\eps + \mathcal{P} u_\eps \times B_\eps ) \longrightarrow n u + \sigma ( - \tfrac{1}{2} \nabla_x n + E + u \times B )
  \end{aligned}
\end{equation}
strongly in $C(\R^+; H^{s-2}_x)$ as $\eps \rightarrow 0$. Moreover, it is easily deduced from the convergences \eqref{Convg-strong} and \eqref{Convg-weak-to-0} that
\begin{equation}\label{Convg-weak-uB-0}
  \begin{aligned}
     n_\eps \mathcal{P}^\perp u_\eps + \sigma \mathcal{P}^\perp u_\eps \times B_\eps \rightarrow 0
  \end{aligned}
\end{equation}
weakly-$\star$ in $t \geq 0$ and weakly in $H^{s-1}_x$ as $\eps \rightarrow 0$.

Next we prove that
\begin{equation}\label{X-eps-Limit}
  \begin{aligned}
    \langle \mathcal{X}_\eps , \widetilde{\Phi} (v) \rangle_{L^2_v} \rightarrow 0
  \end{aligned}
\end{equation}
in the sense of distribution as $\eps \rightarrow 0$. Indeed, for any $T > 0$, let a vector-valued text function $\psi (t,x) \in C^1 ( 0,T; C_c^\infty (\T^3) )$, $\psi(0,x) = \psi_0 (x) \in C_c^\infty (\T^3)$ and $\psi (t,x) = 0$ for $t \geq T'$, where $T' < T$. Then, from the uniform bound \eqref{Uniform-Bnd-1} and the initial energy bounds given in Theorem \ref{Main-Thm-1} that
\begin{equation}
  \begin{aligned}
    & \Big| - \eps \int_0^T \int_{\T^3} \langle \partial_t ( G_\eps \cdot \mathsf{q}_1 ) , \widetilde{\Phi} (v) \rangle_{L^2_v} \cdot \psi (t,x) \d x \d t \Big| \\
    = & \Big| \eps \int_{\T^3} \langle G_\eps^{in} \cdot \mathsf{q}_1 , \widetilde{\Phi} (v) \rangle_{L^2_v} \cdot \psi_0 (x) \d x  + \eps \int_0^T \int_{\T^3} \langle G_\eps \cdot \mathsf{q}_1 , \widetilde{\Phi} (v) \rangle_{L^2_v} \cdot \partial_t \psi (t, x) \d x \d t \Big| \\
    \leq & C \eps \| \widetilde{\Phi} (v) \|_{L^2_v} \big(  \| G_\eps^{in} \|_{L^2_{x,v}} \| \psi_0 \|_{L^2_x} + \| G_\eps \|_{L^\infty(\R^+; L^2_{x,v})} \| \partial_t \psi \|_{L^\infty (0,T;L^2_x)} T \big) \\
    \leq & C (\psi, T) \eps \big( \| G_\eps^{in} \|_{H^s_{x,v}} + \| G_\eps \|_{L^\infty(\R^+; H^s_{x,v})} \big) \\
    \leq & C^* ( \psi , T ) \eps \rightarrow 0
  \end{aligned}
\end{equation}
as $\eps \rightarrow 0$, which means that
\begin{equation}\label{X-eps-Limit-1}
  \begin{aligned}
    - \eps \langle \partial_t ( G_\eps \cdot \mathsf{q}_1 ) , \widetilde{\Phi} (v) \rangle_{L^2_v} \rightarrow 0
  \end{aligned}
\end{equation}
in the sense of distribution as $\eps \rightarrow 0$. It is yielded that by the H\"older inequality and the part (1) of Lemma \ref{Lmm-nu-norm}
\begin{equation}
  \begin{aligned}
    & \big\| - \langle v \cdot \nabla_x ( \mathbb{P}^\perp G_\eps \cdot \mathsf{q}_1 ) , \widetilde{\Phi} (v) \rangle_{L^2_v} \big\|_{H^{s-1}_x} \\
    \leq & C \| v \otimes \widetilde{\Phi} (v) \|_{L^2_v} \| \nabla_x \mathbb{P}^\perp G_\eps \|_{H^{s-1}_x L^2_v} \leq C \| \mathbb{P}^\perp G_\eps \|_{H^s_{x,v}(\nu)} \,,
  \end{aligned}
\end{equation}
which implies by the uniform energy dissipation bound \eqref{Uniform-Bnd-2} that
\begin{equation}
  \begin{aligned}
    \big\| - \langle v \cdot \nabla_x ( \mathbb{P}^\perp G_\eps \cdot \mathsf{q}_1 ) , \widetilde{\Phi} (v) \rangle_{L^2_v} \big\|_{L^2(\R^+;H^{s-1}_x)} \leq C \| \mathbb{P}^\perp G_\eps \|_{L^2(\R^+;H^s_{x,v}(\nu))} \leq C \eps \,.
  \end{aligned}
\end{equation}
Then we have
\begin{equation}\label{X-eps-Limit-2}
  \begin{aligned}
    - \langle v \cdot \nabla_x ( \mathbb{P}^\perp G_\eps \cdot \mathsf{q}_1 ) , \widetilde{\Phi} (v) \rangle_{L^2_v} \rightarrow 0
  \end{aligned}
\end{equation}
strongly in $L^2(\R^+;H^{s-1}_x)$ as $\eps \rightarrow 0$. Similarly, we can estimate by using the bounds \eqref{Uniform-Bnd-1} and \eqref{Uniform-Bnd-2} that
\begin{equation}
  \begin{aligned}
    & \| - \langle (v \times B_\eps) \cdot \nabla_v ( \mathbb{P}^\perp G_\eps \cdot \mathsf{q}_2 ) , \widetilde{\Phi} (v) \rangle_{L^2_v} \|_{L^2(\R^+; H^{s-1}_x)} \\
    & \leq C \| B_\eps \|_{L^\infty (\R^+; H^s_x)} \| \mathbb{P}^\perp G_\eps \|_{L^2(\R^+;H^s_{x,v}(\nu))} \leq C \eps \,.
  \end{aligned}
\end{equation}
Therefore, we have
\begin{equation}\label{X-eps-Limit-3}
  \begin{aligned}
    - \langle (v \times B_\eps) \cdot \nabla_v ( \mathbb{P}^\perp G_\eps \cdot \mathsf{q}_2 ) , \widetilde{\Phi} (v) \rangle_{L^2_v} \rightarrow 0
  \end{aligned}
\end{equation}
strongly in $L^2(\R^+;H^{s-1}_x)$ as $\eps \rightarrow 0$. By the H\"older inequality and the energy bound \eqref{Uniform-Bnd-1}, we have
\begin{equation}
  \begin{aligned}
    & \big\| \langle \tfrac{1}{2} \eps (E_\eps \cdot v) ( G_\eps \cdot \mathsf{q}_2 ) - \eps E_\eps \cdot \nabla_v (G_\eps \cdot  \mathsf{q}_2) , \widetilde{\Phi} (v) \rangle_{L^2_v} \big\|_{L^\infty (\R^+; H^{s-1}_x)} \\
    \leq & C \eps \| v \otimes \widetilde{\Phi} (v) \|_{L^2_v} \| E_\eps \otimes G_\eps \|_{L^\infty (\R^+; H^{s-1}_x L^2_v)} \\
    & + C \eps \| \widetilde{\Phi} (v) \|_{L^2_v} \| E_\eps \cdot \nabla_v G_\eps \|_{L^\infty (\R^+;  H^{s-1}_x L^2_v)} \\
    \leq & C \eps \| E_\eps \|_{L^\infty (\R^+ ; H^s_x)} \big( \| G_\eps \|_{L^\infty (\R^+; H^s_x L^2_v)} + \| \nabla_v G_\eps \|_{L^\infty (\R^+ ; H^{s-1}_x L^2_v)} \big) \\
    \leq & C \eps \| E_\eps \|_{L^\infty (\R^+ ; H^s_x)} \| G_\eps \|_{L^\infty (\R^+; H^s_{x,v})} \leq C \eps \,.
  \end{aligned}
\end{equation}
We thereby obtain
\begin{equation}\label{X-eps-Limit-4}
  \begin{aligned}
    \langle \tfrac{1}{2} \eps (E_\eps \cdot v) ( G_\eps \cdot \mathsf{q}_2 ) - \eps E_\eps \cdot \nabla_v (G_\eps \cdot  \mathsf{q}_2) , \widetilde{\Phi} (v) \rangle_{L^2_v} \rightarrow 0
  \end{aligned}
\end{equation}
strongly in $L^2(\R^+;H^{s-1}_x)$ as $\eps \rightarrow 0$. For any $T > 0$, we take any vector-valued test function $\varphi (t,x) \in C_c^\infty ( [0,T] \times \T^3 )$. Then, by employing the uniform bound \eqref{Uniform-Bnd-2}, Lemma \ref{Lmm-nu-norm} and \ref{Lmm-Gamma-Torus}, we yield that
\begin{equation}
  \begin{aligned}
    & \Big|  \int_0^T \int_{\T^3} \langle \Gamma ( \mathbb{P}^\perp G_\eps , \mathbb{P}^\perp G_\eps ) \cdot \mathsf{q}_1 , \widetilde{\Phi} (v) \rangle_{L^2_v} \cdot \varphi (t,x) \d x \d t \Big| \\
    = & \Big|  \int_0^T \langle \Gamma ( \mathbb{P}^\perp G_\eps , \mathbb{P}^\perp G_\eps ) , \mathsf{q}_1  \widetilde{\Phi} (v) \varphi (t,x) \rangle_{L^2_{x,v}} \d t \Big| \\
    \leq & \int_0^T \| \mathbb{P}^\perp G_\eps \|_{H^s_{x,v}} \| \mathbb{P}^\perp G_\eps \|_{H^s_{x,v}(\nu)} \| \mathsf{q}_1 \widetilde{\Phi} (v) \varphi (t,x) \|_{L^2_{x,v}(\nu)} \d t \\
    \leq & C (\varphi) \| \mathbb{P}^\perp G_\eps \|^2_{L^2 (\R^+; H^s_{x,v}(\nu))} \leq C (\varphi) \eps^2 \,.
  \end{aligned}
\end{equation}
Thus, we know that
\begin{equation}\label{X-eps-Limit-5}
  \begin{aligned}
    \langle \Gamma ( \mathbb{P}^\perp G_\eps , \mathbb{P}^\perp G_\eps ) \cdot \mathsf{q}_1 , \widetilde{\Phi} (v) \rangle_{L^2_v} \rightarrow 0
  \end{aligned}
\end{equation}
in the sense of distribution as $\eps \rightarrow 0$. Analogously, one easily derives that
\begin{equation}
  \begin{aligned}
    & \Big| \int_0^T \int_{\T^3} \langle \Gamma ( \mathbb{P} G_\eps , \mathbb{P}^\perp G_\eps ) \cdot \mathsf{q}_1 + \Gamma ( \mathbb{P}^\perp G_\eps , \mathbb{P} G_\eps ) \cdot \mathsf{q}_1 , \widetilde{\Phi} (v) \rangle_{L^2_v} \cdot \varphi (t,x) \d x \d t \Big| \\
    & \leq C (\varphi) \sqrt{T} \| G_\eps \|_{L^\infty (\R^+; H^s_{x,v})} \| \mathbb{P}^\perp G_\eps \|_{L^2 (\R^+; H^s_{x,v}(\nu))} \leq C (\varphi , T) \eps \,, 
  \end{aligned}
\end{equation}
which immediately implies that
\begin{equation}\label{X-eps-Limit-6}
  \begin{aligned}
    \langle \Gamma ( \mathbb{P} G_\eps , \mathbb{P}^\perp G_\eps ) \cdot \mathsf{q}_1 + \Gamma ( \mathbb{P}^\perp G_\eps , \mathbb{P} G_\eps ) \cdot \mathsf{q}_1 , \widetilde{\Phi} (v) \rangle_{L^2_v} \rightarrow 0
  \end{aligned}
\end{equation}
in the sense of distribution as $\eps \rightarrow 0$. Combining the convergences \eqref{X-eps-Limit-1}, \eqref{X-eps-Limit-2}, \eqref{X-eps-Limit-3}, \eqref{X-eps-Limit-4}, \eqref{X-eps-Limit-5}, \eqref{X-eps-Limit-6} with the definition of $\mathcal{X}_\eps$ in \eqref{X-eps}, we deduce the validity of the convergence \eqref{X-eps-Limit}. Consequently, based on the convergences \eqref{Convg-weak-jw}, \eqref{Convg-strong-Ohm}, \eqref{Convg-weak-uB-0} and \eqref{X-eps-Limit}, the equality \eqref{Ohm-law-eps} implies that
\begin{equation}
  \begin{aligned}
    j = nu + \sigma \big( - \tfrac{1}{2} \nabla_x n + E + u \times B \big) \,.
  \end{aligned}
\end{equation}

Next we verify rigorously the energy equivalence relation, hence the last equation in \eqref{INSFM-Ohm}. Recall that $\Psi (v) = ( \tfrac{|v|^2}{2} - \tfrac{3}{2} ) \sqrt{M}$ and $\widetilde{\Psi} (v)$ is such that $(\mathcal{L} + \mathfrak{L}) \widetilde{\Psi} (v) = \Psi (v)$, which can be seen in Lemma \ref{Lmm-mathfrak-L}. Then, from the definition of $w_\eps$ in \eqref{Fluid-Quanities} and the equation \eqref{Ohm-L+L-eps}, we deduce that
\begin{equation}
  \begin{aligned}
    w_\eps = & \tfrac{1}{\eps} \langle G_\eps \cdot \mathsf{q}_1 , \Psi (v) \rangle_{L^2_v} = \big\langle \tfrac{1}{\eps} ( \mathcal{L} + \mathfrak{L} ) (G_\eps \cdot \mathsf{q}_1) , \widetilde{\Psi} (v) \big\rangle_{L^2_v} \\
    = & 2 ( - \tfrac{1}{2} \nabla_x n_\eps + E_\eps + u_\eps \times B_\eps ) \cdot \langle \Phi (v) , \widetilde{\Psi} (v) \rangle_{L^2_v} \\
    & + \langle ( \mathcal{L} + \mathfrak{L} ) (n_\eps h_\eps) , \widetilde{\Psi} (v) \rangle_{L^2_v} + \langle \mathcal{X}_\eps , \widetilde{\Psi} (v) \rangle_{L^2_v} \\
    = & \langle n_\eps h_\eps , \Psi (v) \rangle_{L^2_v} + \langle \mathcal{X}_\eps , \widetilde{\Psi} (v) \rangle_{L^2_v} \,, 
  \end{aligned}
\end{equation}
where $ \langle \Phi (v) , \widetilde{\Psi} (v) \rangle_{L^2_v} = 0$ and the self-adjointness of $\mathcal{L} + \mathfrak{L}$ are also utilized. Moreover, direct calculation reduces to
\begin{equation}
  \begin{aligned}
    \langle n_\eps h_\eps , \Psi (v) \rangle_{L^2_v} = & n_\eps u_\eps \cdot \langle v \sqrt{M} , ( \tfrac{|v|^2}{2} - \tfrac{3}{2} ) \sqrt{M} \rangle_{L^2_v} \\
    & + n_\eps \theta_\eps \langle ( \tfrac{|v|^2}{2} - \tfrac{3}{2} ) \sqrt{M} , ( \tfrac{|v|^2}{2} - \tfrac{3}{2} ) \sqrt{M} \rangle_{L^2_v} \\
    = & \tfrac{3}{2} n_\eps \theta_\eps = \tfrac{3}{2} n_\eps ( \tfrac{3}{5} \theta_\eps - \tfrac{2}{5} \rho_\eps ) + \tfrac{3}{5} n_\eps ( \rho_\eps + \theta_\eps ) \,.
  \end{aligned}
\end{equation}
We thereby have
\begin{equation}\label{EER-eps}
  \begin{aligned}
    w_\eps = \tfrac{3}{2} n_\eps ( \tfrac{3}{5} \theta_\eps - \tfrac{2}{5} \rho_\eps ) + \tfrac{3}{5} n_\eps ( \rho_\eps + \theta_\eps ) + \langle \mathcal{X}_\eps , \widetilde{\Psi} (v) \rangle_{L^2_v} \,.
  \end{aligned}
\end{equation}
From the strong convergences \eqref{Convg-strong}, one easily deduces that
\begin{equation}\label{Limit-EER-1}
  \begin{aligned}
    \tfrac{3}{2} n_\eps ( \tfrac{3}{5} \theta_\eps - \tfrac{2}{5} \rho_\eps ) \rightarrow \tfrac{3}{2} n \theta 
  \end{aligned}
\end{equation}
strongly in $C(\R^+; H^{s-1}_x)$ as $\eps \rightarrow 0$. It is also derived from the convergences \eqref{Convg-strong} and \eqref{Convg-weak-to-0} that
\begin{equation}\label{Limit-EER-2}
  \begin{aligned}
    \tfrac{3}{5} n_\eps ( \rho_\eps + \theta_\eps ) \rightarrow 0
  \end{aligned}
\end{equation}
weakly-$\star$ in $t \geq 0$ and weakly in $H^{s-1}_x$ as $\eps \rightarrow 0$. Furthermore, similar arguments in analyzing the convergence \eqref{X-eps-Limit} give
\begin{equation}\label{Limit-EER-3}
  \begin{aligned}
    \langle \mathcal{X}_\eps , \widetilde{\Psi} (v) \rangle_{L^2_v} \rightarrow 0
  \end{aligned}
\end{equation}
in the sense of distribution as $\eps \rightarrow 0$. Finally, by plugging the convergences \eqref{Convg-weak-jw}, \eqref{Limit-EER-1}, \eqref{Limit-EER-2} and \eqref{Limit-EER-3} into the equation \eqref{EER-eps}, we obtain the energy equivalence relation
\begin{equation}
  \begin{aligned}
    w =  \tfrac{3}{2} n \theta \,.
  \end{aligned}
\end{equation}

\subsubsection{Equations of $u$ and $\theta$} We first calculate the term 
$$ \big\langle \widehat{\Xi} (v) \sqrt{M} , \tfrac{1}{\eps} \mathcal{L} ( \tfrac{G_\eps \cdot \mathsf{q}_2}{2} ) \big\rangle_{L^2_v} \,,$$
where $\widehat{\Xi} = A$ or $B$. Via multiplying the first $G_\eps$-equation in \eqref{VMB-G} by $\mathsf{q}_2$ and direct calculations, we obtain
\begin{equation}
\begin{aligned}
\partial_t ( G_\eps \cdot \mathsf{q}_2 ) + \tfrac{1}{\eps} v \cdot \nabla_x ( G_\eps \cdot \mathsf{q}_2 ) + \tfrac{1}{\eps} ( \eps E_\eps + v \times B_\eps ) \cdot \nabla_v ( G_\eps \cdot \mathsf{q}_1 ) + \tfrac{1}{\eps^2} \mathcal{L} ( G_\eps \cdot \mathsf{q}_2 )  \\
= \tfrac{1}{2} ( E_\eps \cdot v ) ( G_\eps \cdot \mathsf{q}_1 ) + \tfrac{1}{\eps} \mathcal{Q} ( G_\eps \cdot \mathsf{q}_2 ,  G_\eps \cdot \mathsf{q}_2 ) \,.
\end{aligned}
\end{equation}
Following the standard formal derivations of fluid dynamic limits of Boltzmann equation (see \cite{BGL-1991-JSP}, for instance), we obtain
\begin{equation}\label{Diffusion-A}
  \begin{aligned}
    \big\langle \widehat{A} (v) \sqrt{M} , \tfrac{1}{\eps} \mathcal{L} ( \tfrac{G_\eps \cdot \mathsf{q}_2}{2} ) \big\rangle_{L^2_v} =  u_\eps \otimes u_\eps - \tfrac{|u_\eps|^2}{3} \mathbb{I}_3   - \mu \Sigma (u_\eps) - R_{\eps, A}
  \end{aligned}
\end{equation}
and
\begin{equation}\label{Diffusion-B}
  \begin{aligned}
    \big\langle \widehat{B} (v) \sqrt{M} , \tfrac{1}{\eps} \mathcal{L} ( \tfrac{G_\eps \cdot \mathsf{q}_2}{2} ) \big\rangle_{L^2_v} = \tfrac{5}{2} u_\eps \theta_\eps - \tfrac{5}{2} \kappa \nabla_x \theta_\eps - R_{\eps , B} \,,
  \end{aligned}
\end{equation}
where $\Sigma (u_\eps) = \nabla_x u_\eps + \nabla_x u_\eps^\top - \tfrac{2}{3} \div_x \, u_\eps \mathbb{I}_3$, and for $\Xi = A$ or $B$, $R_{\eps, \Xi}$ are of the form
\begin{equation}\label{R-eps-Xi}
  \begin{aligned}
    R&_{\eps, \Xi} = \eps \big\langle \widehat{\Xi} (v) \sqrt{M} , \partial_t ( \tfrac{G_\eps \cdot \mathsf{q}_2}{2} ) \big\rangle_{L^2_v} + \big\langle \widehat{\Xi} (v) \sqrt{M} v \cdot \nabla_x  ( \tfrac{\mathbb{P}^\perp G_\eps \cdot \mathsf{q}_2}{2} ) \big\rangle_{L^2_v} \\
    & + \big\langle \widehat{\Xi} (v) \sqrt{M} , ( \eps E_\eps + v \times B_\eps ) \cdot \nabla_v ( \tfrac{\mathbb{P}^\perp G_\eps \cdot \mathsf{q}_1}{2} ) \big\rangle_{L^2_v} \\
    & - \eps \big\langle \widehat{\Xi} (v) \sqrt{M} , \tfrac{1}{2} ( E_\eps \cdot v ) ( \tfrac{G_\eps \cdot \mathsf{q}_1}{2} ) \big\rangle_{L^2_v} - 2 \big\langle \widehat{\Xi} (v) \sqrt{M} , \mathcal{Q} \big( \tfrac{\mathbb{P}^\perp G_\eps \cdot \mathsf{q}_2}{2} , \tfrac{\mathbb{P}^\perp G_\eps \cdot \mathsf{q}_2}{2} \big) \big\rangle_{L^2_v} \\
    & - 2 \big\langle \widehat{\Xi} (v) \sqrt{M} , \mathcal{Q} \big( \Pi ( \tfrac{ G_\eps \cdot \mathsf{q}_2}{2} ) , \tfrac{\mathbb{P}^\perp G_\eps \cdot \mathsf{q}_2}{2} \big) \big\rangle_{L^2_v} - 2 \big\langle \widehat{\Xi} (v) \sqrt{M} , \mathcal{Q} \big( \tfrac{\mathbb{P}^\perp G_\eps \cdot \mathsf{q}_2}{2} , \Pi ( \tfrac{ G_\eps \cdot \mathsf{q}_2}{2} ) \big) \big\rangle_{L^2_v} \,.
  \end{aligned}
\end{equation}
Here we also make use of the relation $\Pi^\perp ( G_\eps \cdot \mathsf{q}_2 ) = \mathbb{P}^\perp G_\eps \cdot \mathsf{q}_2 $, derived from the equalities \eqref{Relt-perp-B-VMB-1} and \eqref{Relt-perp-B-VMB-2}.

For the vector field $u_\eps$, we decompose $u_\eps = \mathcal{P} u_\eps + \mathcal{P}^\perp u_\eps$, where $\mathcal{P}^\perp = \mathcal{I} - \mathcal{P} = \nabla_x \Delta_x^{-1} \div_x$ is a gradient operator, where $\mathcal{I}$ is the identical mapping. Then, via plugging the relation \eqref{Diffusion-A} into the equation \eqref{Local-Pu}, we have
\begin{equation}\label{Pu-eps}
  \begin{aligned}
    \partial_t \mathcal{P} u_\eps + \mathcal{P} \div_x \, ( \mathcal{P} u_\eps \otimes \mathcal{P} u_\eps ) - \mu \Delta_x \mathcal{P} u_\eps = \tfrac{1}{2} \mathcal{P} ( n_\eps E_\eps + j_\eps \times B_\eps ) + R_{\eps, u} \,,
  \end{aligned}
\end{equation}
where
\begin{equation}\label{R-eps-u}
  \begin{aligned}
    R_{\eps, u} = \mathcal{P} \div_x \, R_{\eps, A} - \mathcal{P} \div_x \, ( \mathcal{P} u_\eps \otimes \mathcal{P}^\perp u_\eps + \mathcal{P}^\perp u_\eps \otimes \mathcal{P} u_\eps + \mathcal{P}^\perp u_\eps \otimes \mathcal{P}^\perp u_\eps ) \,.
  \end{aligned}
\end{equation}
Noticing that $\theta_\eps = ( \tfrac{3}{5} \theta_\eps - \tfrac{2}{5} \rho_\eps ) + \tfrac{2}{5} (\rho_\eps + \theta_\eps)$, we substitute the relation \eqref{Diffusion-B} into the equation \eqref{Local-theta-rho} and then obtain
\begin{equation}\label{3/5theta-2/5rho-eps}
  \begin{aligned}
    \partial_t ( \tfrac{3}{5} \theta_\eps - \tfrac{2}{5} \rho_\eps ) + \div_x \, \big[ \mathcal{P} u_\eps ( \tfrac{3}{5} \theta_\eps - \tfrac{2}{5} \rho_\eps ) \big] - \kappa \Delta_x ( \tfrac{3}{5} \theta_\eps - \tfrac{2}{5} \rho_\eps ) = R_{\eps, \theta} \,,
  \end{aligned}
\end{equation}
where
\begin{equation}\label{R-eps-theta}
  \begin{aligned}
    R_{\eps, \theta} = & \tfrac{2}{5} \div_x \, R_{\eps, B} - \tfrac{2}{5} \div_x \, \big[ \mathcal{P} u_\eps ( \rho_\eps + \theta_\eps ) \big] - \div_x \, \big[ \mathcal{P}^\perp u_\eps ( \tfrac{3}{5} \theta_\eps - \tfrac{2}{5} \rho_\eps ) \big] \\
    & + \tfrac{\eps}{5} j_\eps \cdot E_\eps - \tfrac{2}{5} \div_x \, \big[ \mathcal{P}^\perp u_\eps (\rho_\eps + \theta_\eps) \big] + \tfrac{2}{5} \kappa \Delta_x (\rho_\eps + \theta_\eps) \,.
  \end{aligned}
\end{equation}

Now we take the limit from \eqref{Pu-eps} to obtain the $u$-equation of \eqref{INSFM-Ohm}. For any $T > 0$, let a vector-valued text function $\psi (t,x) \in C^1 ( 0,T; C_c^\infty (\T^3) )$ with $\div_x \psi = 0$, $\psi(0,x) = \psi_0 (x) \in C_c^\infty (\T^3)$ and $\psi (t,x) = 0$ for $t \geq T'$, where $T' < T$. We multiply \eqref{Pu-eps} by $\psi (t,x)$ and integrate by parts over $(t,x) \in [0,T] \times \T^3$. Then we obtain 
\begin{equation}
  \begin{aligned}
    & \int_0^T \int_{\T^3} \partial_t \mathcal{P} u_\eps \cdot \psi (t,x) \d x \d t \\
    = & - \int_{\T^3} \mathcal{P} u_\eps (0,x) \cdot \psi (0,x) \d x - \int_0^T \int_{\T^3} \mathcal{P} u_\eps \cdot \partial_t \psi (t,x) \d x \d t \\
    = & - \int_{\T^3} \mathcal{P} \langle G_\eps^{in} , \tfrac{1}{2} \mathsf{q}_2 v \sqrt{M} \rangle_{L^2_v} \cdot \psi_0 (x) \d x - \int_0^T \int_{\T^3} \mathcal{P} u_\eps \cdot \partial_t \psi (t,x) \d x \d t \,.
  \end{aligned}
\end{equation}
From the initial conditions (3) in Theorem \ref{Main-Thm-2} and the convergence \eqref{Convg-strong}, we deduce that
\begin{equation}
  \begin{aligned}
    \int_{\T^3} \mathcal{P} \langle G_\eps^{in} , \tfrac{1}{2} \mathsf{q}_2 v \sqrt{M} \rangle_{L^2_v} \cdot \psi_0 (x) \d x \rightarrow \int_{\T^3} \mathcal{P} \langle G^{in} , \tfrac{1}{2} \mathsf{q}_2 v \sqrt{M} \rangle_{L^2_v} \cdot \psi_0 (x) \d x \\
    = \int_{\T^3} \mathcal{P} u^{in} \cdot \psi_0 (x) \d x
  \end{aligned}
\end{equation}
and
\begin{equation}
  \begin{aligned}
    \int_0^T \int_{\T^3} \mathcal{P} u_\eps \cdot \partial_t \psi (t,x) \d x \d t \rightarrow \int_0^T \int_{\T^3} u \cdot \partial_t \psi (t,x) \d x \d t
  \end{aligned}
\end{equation}
as $\eps \rightarrow 0$. Namely, we have
\begin{equation}\label{Limit-u-1}
  \begin{aligned}
    \int_0^T \int_{\T^3} \partial_t \mathcal{P} u_\eps \cdot \psi (t,x) \d x \d t \rightarrow - \int_{\T^3} \mathcal{P} u^{in} \cdot \psi_0 (x) \d x - \int_0^T \int_{\T^3} u \cdot \partial_t \psi (t,x) \d x \d t
  \end{aligned}
\end{equation}
as $\eps \rightarrow 0$. It is implied by the strong convergences \eqref{Convg-strong} that
\begin{equation}\label{Limit-u-2}
  \begin{aligned}
    & \mathcal{P} \div_x \, ( \mathcal{P} u_\eps \otimes \mathcal{P} u_\eps ) \rightarrow \mathcal{P} \div_x \, ( u \otimes u ) \quad \textrm{strongly in } \quad C(\R^+; H^{s-2}_x) \,, \\
    & \mu \Delta_x \mathcal{P} u_\eps \rightarrow \mu \Delta_x u \qquad\qquad\qquad\qquad\quad \textrm{strongly in } \quad C(\R^+; H^{s-3}_x) \,, \\
    & \mathcal{P} (n_\eps E_\eps) \rightarrow \mathcal{P} (n E) \quad\qquad\qquad\qquad\quad\ \, \textrm{strongly in } \quad C(\R^+; H^{s-1}_x) \,,
  \end{aligned}
\end{equation}
as $\eps \rightarrow 0$, where $s \geq 3$. Furthermore, we deduce from the convergences \eqref{Convg-weak-jw} and \eqref{Convg-strong} that
\begin{equation}\label{Limit-u-3}
  \begin{aligned}
    \mathcal{P} ( j_\eps \times B_\eps ) \rightarrow \mathcal{P} (j \times B)
  \end{aligned}
\end{equation}
weakly in $L^2_{loc} (\R^+ ; H^{s-1}_x)$ as $\eps \rightarrow 0$.

It remains to prove
\begin{equation}\label{Limit-R-eps-u}
  \begin{aligned}
    R_{\eps, u} \rightarrow 0
  \end{aligned}
\end{equation}
in the sense of distribution as $\eps \rightarrow 0$, where $R_{\eps, u}$ is defined in \eqref{R-eps-u}. Indeed, by employing the convergences \eqref{Convg-strong} and \eqref{Convg-weak-to-0}, one can obtain
\begin{equation}\label{Limit-R-eps-u-1}
  \begin{aligned}
    \mathcal{P} \div_x ( \mathcal{P} u_\eps \otimes \mathcal{P}^\perp u_\eps + \mathcal{P}^\perp u_\eps \otimes \mathcal{P} u_\eps + \mathcal{P}^\perp u_\eps \otimes \mathcal{P}^\perp u_\eps ) \rightarrow 0
  \end{aligned}
\end{equation}
weakly-$\star$ in $t \geq 0$ and strongly in $H^{s-2}_x$ as $\eps \rightarrow 0$. Moreover, by employing the similar arguments in the convergence \eqref{X-eps-Limit}, we know that for $\Xi = A$ or $B$
\begin{equation}\label{Limit-R-eps-u-2}
  \begin{aligned}
    \div_x \, R_{\eps , \Xi} \rightarrow 0
  \end{aligned}
\end{equation}
in the sense of distribution as $\eps \rightarrow 0$, where $R_{\eps , \Xi}$ are defined in \eqref{R-eps-Xi}. Thus, the convergences \eqref{Limit-R-eps-u-1} and \eqref{Limit-R-eps-u-2} imply the convergence \eqref{Limit-R-eps-u}. Collecting the limits \eqref{Limit-u-1}, \eqref{Limit-u-2}, \eqref{Limit-u-3} and \eqref{Limit-R-eps-u} yields that $u \in L^\infty (\R^+ ; H^s_x) \cap C (\R^+ ; H^{s-1}_x)$ obeys
\begin{equation}
  \begin{aligned}
    \partial_t u + \mathcal{P} \div_x \, ( u \otimes u ) - \mu \Delta_x u = \tfrac{1}{2} \mathcal{P} ( n E + j \times B )
  \end{aligned}
\end{equation}
with the initial data
\begin{equation}
  \begin{aligned}
    u (0,x)  = \mathcal{P} u^{in} (x) \,.
  \end{aligned}
\end{equation}

Finally, we take the limit from \eqref{3/5theta-2/5rho-eps} to the third $\theta$-equation in \eqref{INSFM-Ohm} as $\eps \rightarrow 0$. For any $T > 0$, let $\xi (t,x)$ be a test function satisfying $\xi (t,x) \in C^1 (0,T;C_c^\infty (\T^3))$ with $\xi (0,x) = \xi_0 (x) \in C_c^\infty (\T^3)$ and $\xi (t,x) = 0$ for $t \geq T'$, where $T' < T$. From the initial conditions (3) in Theorem \ref{Main-Thm-2} and the strong convergence \eqref{Convg-strong}, we deduce that 
\begin{equation}\label{Limit-theta-1}
  \begin{aligned}
    & \int_0^T \int_{\T^3} \partial_t ( \tfrac{3}{5} \theta_\eps - \tfrac{2}{3} \rho_\eps ) (t,x) \xi (t,x) \d x \d t \\
    = & - \int_{\T^3} \langle G_\eps^{in} , \tfrac{1}{2} \mathsf{q}_2 \big( \tfrac{3}{5} ( \tfrac{|v|^2}{3} - 1 ) - \tfrac{2}{5} \big) \sqrt{M} \rangle_{L^2_v} \xi_0 (x) \d x \\
    & - \int_0^T \int_{\T^3} ( \tfrac{3}{5} \theta_\eps - \tfrac{2}{3} \rho_\eps ) (t,x) \partial_t  \xi (t,x) \d x \d t \\
    \rightarrow & - \int_{\T^3} \langle G^{in} , \tfrac{1}{2} \mathsf{q}_2 \big( \tfrac{3}{5} ( \tfrac{|v|^2}{3} - 1 ) - \tfrac{2}{5} \big) \sqrt{M} \rangle_{L^2_v} \xi_0 (x) \d x - \int_0^T \int_{\T^3} \theta (t,x) \partial_t  \xi (t,x) \d x \d t \\
    = & - \int_{\T^3} \big( \tfrac{3}{5} \theta^{in} - \tfrac{2}{5} \rho^{in} \big) \xi_0 (x) \d x - \int_0^T \int_{\T^3} \theta (t,x) \partial_t  \xi (t,x) \d x \d t
  \end{aligned}
\end{equation}
as $\eps \rightarrow 0$. It is derived from the strong convergence \eqref{Convg-strong} that
\begin{equation}\label{Limit-theta-2}
  \begin{aligned}
    & \div_x \, \big[ \mathcal{P} u_\eps ( \tfrac{3}{5} \theta_\eps - \tfrac{2}{5} \rho_\eps ) \big] \rightarrow \div_x \, ( u \theta ) \quad \textrm{strongly in } \quad C(\R^+ ; H^{s-2}_x) \,, \\
    & \kappa \Delta_x ( \tfrac{3}{5} \theta_\eps - \tfrac{2}{5} \rho_\eps ) \rightarrow \kappa \Delta_x \theta \qquad \qquad \ \  \ \, \textrm{strongly in } \qquad C(\R^+; H^{s-3}_x)
  \end{aligned}
\end{equation}
as $\eps \rightarrow 0$, where $s \geq 3$. 

It remains to prove
\begin{equation}\label{Limit-R-eps-theta}
  \begin{aligned}
    R_{\eps, \theta} \rightarrow 0
  \end{aligned}
\end{equation}
in the sense of distribution as $\eps \rightarrow 0$, where $R_{\eps , \theta}$ is defined in \eqref{R-eps-theta}. Indeed, from the convergences \eqref{Convg-strong} and \eqref{Convg-weak-to-0}, we derive that
\begin{equation}\label{Limit-R-eps-theta-1}
  \begin{aligned}
    \tfrac{2}{5} \div_x \, \big[ \mathcal{P} u_\eps ( \rho_\eps + \theta_\eps ) \big] + \tfrac{2}{5} \div_x \, \big[ \mathcal{P}^\perp u_\eps ( \rho_\eps + \theta_\eps ) \big] + \div_x \, \big[ \mathcal{P}^\perp u_\eps ( \tfrac{3}{5} \theta_\eps - \tfrac{2}{5} \rho_\eps ) \big] \rightarrow 0
  \end{aligned}
\end{equation}
weakly-$\star$ in $t \geq 0$ and strongly in $H^{s-2}_x$ as $\eps \rightarrow 0$. Moreover, the convergence \eqref{Convg-weak-to-0} tells us
\begin{equation}\label{Limit-R-eps-theta-2}
  \begin{aligned}
    \kappa \Delta_x (\rho_\eps + \theta_\eps) \rightarrow 0
  \end{aligned}
\end{equation}
weakly-$\star$ in $t \geq 0$ and strongly in $H^{s-3}_x$ as $\eps \rightarrow 0$, where $s \geq 3$. Because
\begin{equation}
  \begin{aligned}
    \| \tfrac{\eps}{5} j_\eps \cdot E_\eps \|_{L^2 (\R^+; H^s_x)} \leq C \eps \| j_\eps \|_{L^2 (\R^+; H^s_x)} \| E_\eps \|_{L^\infty (\R^+; H^s_x)} \leq C \eps
  \end{aligned}
\end{equation}
holds by utilizing the uniform bounds \eqref{Uniform-Bnd-1} and \eqref{Uniform-Bnd-4}, we know that
\begin{equation}\label{Limit-R-eps-theta-3}
  \begin{aligned}
    \tfrac{\eps }{3} j_\eps \cdot E_\eps \rightarrow 0
  \end{aligned}
\end{equation}
strongly in $L^2 (\R^+ ; H^s_x)$ as $\eps \rightarrow 0$. Consequently, the convergences \eqref{Limit-R-eps-u-2}, \eqref{Limit-R-eps-theta-1}, \eqref{Limit-R-eps-theta-2} and \eqref{Limit-R-eps-theta-3} imply the convergence \eqref{Limit-R-eps-theta}. It is yielded by collecting the convergences \eqref{Limit-theta-1}, \eqref{Limit-R-eps-u-2} and \eqref{Limit-R-eps-theta} that $\theta \in L^\infty (\R^+ ; H^s_x) \cap C(\R^+ ; H^{s-1}_x)$ subjects to
\begin{equation}
  \begin{aligned}
    \partial_t \theta + \div_x \, (u \theta) = \kappa \Delta_x \theta
  \end{aligned}
\end{equation}
with the initial data
\begin{equation}
  \begin{aligned}
    \theta (0,x) = \tfrac{3}{5} \theta^{in} (x) - \tfrac{2}{5} \rho^{in} (x) \,.
  \end{aligned}
\end{equation}

\subsubsection{Equations of the electromagnetic fields $E$ and $B$}

We will derive the equations of $E$ and $B$ from taking limit in the last five equations of \eqref{Local-Consvtn-Law} as $\eps \rightarrow 0$. For any $ T > 0 $, let $\omega (t,x)$ be a vector-valued test function such that $\omega (t,x) \in C^1 (0,T; C^\infty_c (\T^3))$ with $\omega (0,x) = \omega_0 (x) \in C_c^\infty (\T^3)$ and $\omega (t,x) = 0$ for $t \geq T'$, where $T' < T$. It is deduced from the initial condition (3) in Theorem \ref{Main-Thm-2} and the strong convergence \eqref{Convg-strong} that
\begin{equation}\label{Limit-EB-1}
  \begin{aligned}
    \int_0^T \int_{\T^3} \partial_t Y_\eps \cdot \omega (t,x) \d x \d t = &  - \int_{\T^3} Y_\eps^{in} \cdot \omega_0 (x) \d x -  \int_0^T \int_{\T^3} Y_\eps \cdot \partial_t \omega (t,x) \d x \d t \\
    \rightarrow & - \int_{\T^3} Y^{in} \cdot \omega_0 (x) \d x -  \int_0^T \int_{\T^3} Y \cdot \partial_t \omega (t,x) \d x \d t
  \end{aligned}
\end{equation}
as $\eps \rightarrow 0$, where the symbol $Y$ can be $E$ or $B$. Moreover, from the convergence \eqref{Convg-strong}, one derives that for $Y = E$ or $B$
\begin{equation}\label{Limit-EB-2}
  \begin{aligned}
    \nabla_x  \times Y_\eps \rightarrow \nabla_x \times Y \quad \textrm{and} \quad \div_x \, Y_\eps \rightarrow \div_x \, Y
  \end{aligned}
\end{equation}
strongly in $C (\R^+; H^{s-2}_x)$ as $\eps \rightarrow 0$. Thus, from the convergences \eqref{Convg-strong}, \eqref{Convg-weak-jw}, \eqref{Limit-EB-1} and \eqref{Limit-EB-2}, the last four equations of \eqref{Local-Consvtn-Law} reduce to
\begin{equation}
  \left\{
    \begin{array}{l}
      \partial_t E - \nabla_x \times B = - j \,, \\
      \partial_t B + \nabla_x \times E = 0 \,, \\
      \div_x \, E = n \,, \quad \div_x \, B = 0 \,,
    \end{array}
  \right.
\end{equation}
with the initial data
\begin{equation}
  \begin{aligned}
    E (0,x) = E^{in} (x) \,, \quad B(0,x) = B^{in} (x) \,.
  \end{aligned}
\end{equation}

We now take limit from the local conservation law $\partial_t n_\eps + \div_x \, j_\eps = 0$ as $\eps \rightarrow 0$. For any $T > 0$, we take any scalar test function $\varUpsilon (t,x) \in C^1 (0,T; C_c^\infty (\T^3))$ satisfying $\varUpsilon (0,x) = \varUpsilon_0 (x) \in C_c^\infty (\T^3)$ and $\varUpsilon (t,x) = 0$ for $t \geq T'$, where $T' < T$. Then, from the initial condition (3) in Theorem \ref{Main-Thm-2} and the strong convergence \eqref{Convg-strong}, we imply that
\begin{equation}
  \begin{aligned}
    & \int_0^T \int_{\T^3} \partial_t n_\eps \varUpsilon (t,x) \d x \d t \\
    = & - \int_{\T^3} \langle G_\eps^{in}, \mathsf{q}_1 \sqrt{M} \rangle_{L^2_v} \varUpsilon_0 (x) \d x - \int_0^T \int_{\T^3} n_\eps  \partial_t \varUpsilon (t,x) \d x \d t \\
    \rightarrow & - \int_{\T^3} \langle G^{in}, \mathsf{q}_1 \sqrt{M} \rangle_{L^2_v} \varUpsilon_0 (x) \d x - \int_0^T \int_{\T^3} n \,  \partial_t \varUpsilon (t,x) \d x \d t \\
    = & - \int_{\T^3} n^{in} \, \varUpsilon_0 (x) \d x - \int_0^T \int_{\T^3} n \,  \partial_t \varUpsilon (t,x) \d x \d t
  \end{aligned}
\end{equation}
as $\eps \rightarrow 0$. Moreover, the convergence \eqref{Convg-weak-jw} implies that
\begin{equation}
  \begin{aligned}
    \div_x \, j_\eps \rightarrow \div_x \, j
  \end{aligned}
\end{equation}
weakly in $L^2 (\R^+ ; H^{s-1}_x)$ as $\eps \rightarrow 0$. Then, we have derived from the microscopic local conservation law that
\begin{equation}\label{Equ-n}
  \begin{aligned}
    \partial_t n + \div_x \, j = 0
  \end{aligned}
\end{equation}
with the initial data
\begin{equation}\label{IC-Equ-n}
  \begin{aligned}
    n (0,x) = n^{in} (x) \,.
  \end{aligned}
\end{equation}

\subsubsection{Summarization}

Collecting all above convergence results, we have shown that 
$$(u, \theta, n, E, B) \in C(\R^+; H^{s-1}_x ) \cap L^\infty (\R^+; H^s_x) $$ 
satisfy the following two fluid incompressible Navier-Stokes-Fourier-Maxwell equations with Ohm's law
\begin{equation*}
  \left\{
    \begin{array}{l}
      \partial_t u + u \cdot \nabla_x u - \mu \Delta_x u + \nabla_x p = \tfrac{1}{2} ( n E + j \times B ) \,, \qquad \div_x \, u = 0 \,, \\ [2mm]
      \partial_t \theta + u \cdot \nabla_x \theta - \kappa \Delta_x \theta = 0 \,, \qquad\qquad\qquad\qquad\qquad\quad\ \, \rho + \theta = 0 \,, \\ [2mm]
      \partial_t E - \nabla_x \times B = - j \,, \qquad\qquad\qquad\qquad\qquad\qquad\ \ \ \, \div_x \, E = n \,, \\ [2mm]
      \partial_t B + \nabla_x \times E = 0 \,, \qquad\qquad\qquad\qquad\qquad\qquad\qquad \div_x \, B = 0 \,, \\ [2mm]
      \qquad \qquad j - nu = \sigma \big( - \tfrac{1}{2} \nabla_x n + E + u \times B \big) \,, \qquad\quad\,\ w = \tfrac{3}{2} n \theta \,,
    \end{array}
  \right.
\end{equation*}
with initial data
\begin{equation}
  \begin{aligned}
    u (0,x) = \mathcal{P} u^{in} (x) \,, \ \theta (0,x) = \tfrac{3}{5} \theta^{in} (x) - \tfrac{2}{5} \rho^{in} (x) \,, \ E(0,x) = E^{in} (x) \,, \ B(0,x) = B^{in} (x) \,.
  \end{aligned}
\end{equation}
Moreover, from the uniform bound \eqref{Uniform-Bnd} in Theorem \ref{Main-Thm-1} and the convergence \eqref{Convg-GEB}, we have
\begin{equation}
  \begin{aligned}
    \sup_{t \geq 0} \big( & \| G \|^2_{H^s_{x,v}} + \| E \|^2_{H^s_x} + \| B \|^2_{H^s_x} \big) (t) \leq \sup_{t \geq 0} \big( \| G_\eps \|^2_{H^s_{x,v}} + \| E_\eps \|^2_{H^s_x} + \| B_\eps \|^2_{H^s_x} \big) (t) \\
    = & \sup_{t \geq 0} \mathbb{E}_s ( G_\eps , E_\eps , B_\eps ) (t) \leq c_1 \mathbb{E}_s ( G_\eps^{in} , E_\eps^{in} , B_\eps^{in} ) \rightarrow c_1 \mathbb{E}_s ( G^{in} , E^{in} , B^{in} ) 
  \end{aligned}
\end{equation}
as $\eps \rightarrow 0$. Hence
\begin{equation}
  \begin{aligned}
    \sup_{t \geq 0} \big( & \| G \|^2_{H^s_{x,v}} + \| E \|^2_{H^s_x} + \| B \|^2_{H^s_x} \big) (t) \leq c_1 \mathbb{E}_s ( G^{in} , E^{in} , B^{in} ) \leq c_1 \ell_0 \,.
  \end{aligned}
\end{equation}
Since $G$ is of the form \eqref{Limit-G} and $\rho + \theta = 0$, there are positive generic constants $C_h$ and $C_l$ such that
\begin{equation}
  \begin{aligned}
    C_l \big( \| u \|^2_{H^s_x} + \| \theta \|^2_{H^s_x} + \| n \|^2_{H^s_x} \big) \leq \| G \|^2_{H^s_{x,v}} \leq C_h \big( \| u \|^2_{H^s_x} + \| \theta \|^2_{H^s_x} + \| n \|^2_{H^s_x} \big) \,.
  \end{aligned}
\end{equation}
Consequently, the solution $(u,\theta,n,E,B)$ to the two fluid incompressible Navier-Stokes-Fourier-Maxwell equations \eqref{INSFM-Ohm} with Ohm's law constructed above admits the energy bound
\begin{equation}
  \begin{aligned}
    \sup_{t \geq 0} \big( & \| u \|^2_{H^s_x} + \| \theta \|^2_{H^s_x} + \| n \|^2_{H^s_x} + \| E \|^2_{H^s_x} + \| B \|^2_{H^s_x} \big) (t) \leq c_1^*  \mathbb{E}_s ( G^{in} , E^{in} , B^{in} ) \leq c_1^* \ell_0 \,,
  \end{aligned}
\end{equation}
where $c_1^* = c_1 + \tfrac{c_1}{C_l}  > 0$.

\appendix

\section{Construction of local solutions}\label{Sec:Appendix-Local-Solutn}

In this section, we will construct a unique local-in-time solution to the perturbed VMB system \eqref{VMB-G} for all $0 < \eps \leq 1$, hence prove Proposition \ref{Prop-Local-Solutn}. Fixed $\eps \in (0,1]$, the construction is based on a uniform energy estimate for a sequence of iterating approximate solutions. We consider the following linear iterating approximate sequence $(n \geq 0)$ for solving the perturbed VMB system \eqref{VMB-G} with initial data \eqref{IC-VMB-G}:
\begin{equation}\label{Iter-Appro-Systm}
\left\{
\begin{array}{l}
\big\{ \partial_t + \tfrac{1}{\eps} v \cdot \nabla_x + \tfrac{1}{\eps} \mathsf{q} ( \eps E_\eps^n + v \times B_\eps^n ) \cdot \nabla_v + \tfrac{1}{\eps^2} \mathscr{L} - \tfrac{1}{2} \mathsf{q} ( E_\eps^n \cdot v ) \big\} G_\eps^{n+1} \\
\qquad \qquad \qquad \qquad \qquad \qquad \quad = \tfrac{1}{\eps} ( E_\eps^{n+1} \cdot v ) \mathsf{q}_1 \sqrt{M} + \tfrac{1}{\eps} \Gamma ( G_\eps^n , G_\eps^n ) \,, \\
\partial_t E_\eps^{n+1} - \nabla_x \times B_\eps^{n+1} = - \tfrac{1}{\eps} \int_{\R^3} G_\eps^{n+1} \cdot \mathsf{q} v \sqrt{M} \d v \,, \\
\partial_t B_\eps^{n+1} + \nabla_x \times E_\eps^{n+1} = 0 \,, \\
\div_x \, E_\eps^{n+1} = \int_{\R^3} G_\eps^{n+1} \cdot  \mathsf{q}_1 \sqrt{M} \d v \,, \quad \div_x \, B_\eps^{n+1} = 0 \,,
\end{array}
\right.
\end{equation}
with initial data
\begin{equation}
\begin{aligned}
G_\eps^{n+1} |_{t=0} = G_\eps^{in} (x,v) \,, \ E_\eps^{n+1} |_{t=0} = E_\eps^{in} (x) \,, \ B_\eps^{n+1} |_{t=0} = B_\eps^{in} (x) \,.
\end{aligned}
\end{equation}
We start with $G_\eps^0 (t,x,v) = G_\eps^{in} (x,v)$, $E_\eps^0 (t,x) = E_\eps^{in} (x)$ and $B_\eps^0 (t,x) = B_\eps^{in} (x)$. It is standard from the linear theory to verify that the sequence $( G_\eps^{n}, E_\eps^n , B_\eps^n )$ is well-defined for all $n \geq 0$ on a maximal interval $[ 0, T_\eps^n )$. Our goal is to get a uniform in $n$ estimate for the iterating energy 
\begin{equation}\label{Iter-Energy}
\begin{aligned}
\mathfrak{E}_n (t) = \mathbb{E}_s (G_\eps^n , E_\eps^n , B_\eps^n) (t) + \tfrac{1}{\eps^2} \int_0^{t} \| \mathbb{P}^\perp G_\eps^n \|^2_{H^s_{x,v}(\nu)} (\tau) \d \tau \,.
\end{aligned}
\end{equation}
We notice that $\mathfrak{E}_n (0) = \mathbb{E}_s (G_\eps^{in}, E_\eps^{in}, B_\eps^{in})$ is independent of $n \geq 0$. For simplicity, we denote by $\mathfrak{E} (0) = \mathfrak{E}_n (0)$. Furthermore, we call $ \widetilde{\mathfrak{E}}_n (t) $ is an {\em instant iterating energy} of $\mathfrak{E}_n (t)$, if there is energy functional $\widetilde{\mathfrak{E}}_n (t)$ satisfying 
\begin{equation}
\begin{aligned}
c_* \mathfrak{E}_n (t) \leq \widetilde{\mathfrak{E}}_n (t) \leq C_* \mathfrak{E}_n (t)
\end{aligned}
\end{equation}
for some positive constant $c_*$ and $C_*$, which are both independent of $n$ and $\eps$. We also denote by $\widetilde{\mathfrak{E}}_n (0) = \widetilde{\mathfrak{E}} (0)$ for all $n \geq 0$.

Now	we derive the following lemma.

\begin{lemma}\label{Lmm-Iter-Appro-Bnd}
	There exist an instant iterating energy $\widetilde{\mathfrak{E}}_n (t)$, small $\ell > 0$ and $T^* > 0$, independent of $\eps$, such that if $T^* \leq \sqrt{\ell}$, $\widetilde{\mathfrak{E}} (0) \leq \ell$ and $\sup_{0 \leq t \leq T^*} \widetilde{\mathfrak{E}}_n (t) \leq 2 \ell$, then
	\begin{equation}
	\begin{aligned}
	\sup_{0 \leq t \leq T^*} \widetilde{\mathfrak{E}}_{n+1} (t) \leq 2 \ell \,.
	\end{aligned}
	\end{equation}
\end{lemma}

If Lemma \ref{Lmm-Iter-Appro-Bnd} holds, we know that $ \sup_{0 \leq t \leq T^*} \widetilde{\mathfrak{E}}_n (t) \leq 2 \ell $ holds for all $n \geq 0$. Namely, $ sup_{0 \leq t \leq T^*} \mathfrak{E}_n (t) \leq 2 C_* \ell $ for all $n \geq 0$. In order to complete the proof of Proposition \ref{Prop-Local-Solutn}, We thereby employ the compactness arguments, take $n \rightarrow \infty$ and obtain a solution $(G_\eps, E_\eps, B_\eps)$ for any fixed $0 < \eps \leq 1$ from Lemma \ref{Lmm-Iter-Appro-Bnd}.

\begin{proof}[Proof of Lemma \ref{Lmm-Iter-Appro-Bnd}]
	For $|m| \leq s \, (s \geq 3)$, taking $\partial^m_x$ derivatives of the first $G_\eps^{n+1}$-equation in \eqref{Iter-Appro-Systm} and taking inner product with $\partial^m_x G_\eps^{n+1}$ over $\T^3 \times \R^3$, we obtain
	\begin{equation}\label{Iter-Inq-0}
	\begin{aligned}
	\tfrac{1}{2} & \tfrac{\d }{\d t} \| \partial^m_x G_\eps^{n+1} \|^2_{L^2_{x,v}} + \tfrac{1}{\eps^2} \langle \mathscr{L} \partial^m_x G_\eps^{n+1} , \partial^m_x G_\eps^{n+1} \rangle_{L^2_{x,v}} \\
	& - \tfrac{1}{\eps} \langle ( \partial^m_x E_\eps^{n+1} \cdot v ) \mathsf{q}_1 \sqrt{M} , \partial^m_x G_\eps^{n+1} \rangle_{L^2_{x,v}} \\
	= & \tfrac{1}{2} \langle  \mathsf{q} \partial^m_x [ (E_\eps^n \cdot v) G_\eps^{n+1} ] , \partial^m_x G_\eps^{n+1} \rangle_{L^2_{x,v}} + \tfrac{1}{\eps} \langle  \partial^m_x \Gamma (  G_\eps^n , G_\eps^n ) , \partial^m_x G_\eps^{n+1} \rangle_{L^2_{x,v}} \\
	& - \tfrac{1}{\eps} \sum_{0 \neq m' \leq m} \langle \mathsf{q} \partial^{m'}_x ( \eps E_\eps^n + v \times B_\eps^n ) \cdot \nabla_v \partial^{m-m'}_x G_\eps^{n+1} , \partial^m_x G_\eps^{n+1} \rangle_{L^2_{x,v}} \,.
	\end{aligned}
	\end{equation}
	We derive from the part (3) of Lemma \ref{Lmm-L} that there is a $\lambda > 0$ such that
	\begin{equation}\label{Iter-Inq-1}
	\begin{aligned}
	\tfrac{1}{\eps^2} \langle \mathscr{L} \partial^m_x G_\eps^{n+1} , \partial^m_x G_\eps^{n+1} \rangle_{L^2_{x,v}} \geq \tfrac{\lambda}{\eps^2} \| \partial^m_x \mathbb{P}^\perp G_\eps^{n+1} \|^2_{L^2_{x,v}(\nu)} \,.
	\end{aligned}
	\end{equation}
	Next, by the Maxwell system of $(E_\eps^{n+1}, B_\eps^{n+1})$ in \eqref{Iter-Appro-Systm}, we have
	\begin{equation}\label{Iter-Inq-2}
	\begin{aligned}
	- \tfrac{1}{\eps} \langle ( \partial^m_x E_\eps^{n+1} \cdot v ) \mathsf{q}_1 & \sqrt{M} , \partial^m_x G_\eps^{n+1} \rangle_{L^2_{x,v}} = \big\langle \partial^m_x E_\eps^{n+1} , - \tfrac{1}{\eps} \langle \partial^m_x G_\eps^{n+1} , \mathsf{q}_1 v \sqrt{M} \rangle_{L^2_v} \big\rangle_{L^2_x} \\
	= & \langle \partial^m_x E_\eps^{n+1} , \partial_t \partial^m_x E_\eps^{n+1} - \nabla_x \times \partial^m_x B_\eps^{n+1} \rangle_{L^2_x} \\
	= & \tfrac{1}{2} \tfrac{\d}{\d t} \| \partial^m_x E_\eps^{n+1} \|^2_{L^2_x} - \langle \partial^m_x B_\eps^{n+1} , \nabla_x \times \partial^m_x E_\eps^{n+1} \rangle_{L^2_x} \\
	= & \tfrac{1}{2} \tfrac{\d}{\d t} \| \partial^m_x E_\eps^{n+1} \|^2_{L^2_x} + \langle \partial^m_x B_\eps^{n+1} , \partial_t \partial^m_x B_\eps^{n+1} \rangle_{L^2_x} \\
	= & \tfrac{1}{2} \tfrac{\d}{\d t} \big( \| \partial^m_x E_\eps^{n+1} \|^2_{L^2_x} + \| \partial^m_x B_\eps^{n+1} \|^2_{L^2_x} \big) \,.
	\end{aligned}
	\end{equation}
	We now estimate the term $\tfrac{1}{2} \langle  \mathsf{q} \partial^m_x [ (E_\eps^n \cdot v) G_\eps^{n+1} ] , \partial^m_x G_\eps^{n+1} \rangle_{L^2_{x,v}}$ for all $|m| \leq s$. By employing the decomposition $G_\eps^{n+1} = \mathbb{P} G_\eps^{n+1} + \mathbb{P}^\perp G_\eps^{n+1}$, we obtain
	\begin{equation}
	\begin{aligned}
	& \tfrac{1}{2} \langle  \mathsf{q} \partial^m_x [ (E_\eps^n \cdot v) G_\eps^{n+1} ] , \partial^m_x G_\eps^{n+1} \rangle_{L^2_{x,v}} \\
	= & \underset{X_1}{ \underbrace{ \tfrac{1}{2} \sum_{m' \leq m} C_m^{m'} \big\langle ( \partial^{m'}_x E_\eps^n \cdot v ) \mathsf{q} \partial^{m-m'}_x \mathbb{P} G_\eps^{n+1} , \partial^m_x \mathbb{P} G_\eps^{n+1} \big\rangle_{L^2_{x,v}} } } \\
	& + \underset{X_2}{ \underbrace{ \tfrac{1}{2} \sum_{m' \leq m} C_m^{m'} \big\langle ( \partial^{m'}_x E_\eps^n \cdot v ) \mathsf{q} \partial^{m-m'}_x \mathbb{P} G_\eps^{n+1} , \partial^m_x \mathbb{P}^\perp G_\eps^{n+1} \big\rangle_{L^2_{x,v}} } } \\
	& + \underset{X_3}{ \underbrace{ \tfrac{1}{2} \sum_{m' \leq m} C_m^{m'} \big\langle ( \partial^{m'}_x E_\eps^n \cdot v ) \mathsf{q} \partial^{m-m'}_x \mathbb{P}^\perp G_\eps^{n+1} , \partial^m_x \mathbb{P} G_\eps^{n+1} \big\rangle_{L^2_{x,v}} } } \\
	& + \underset{X_4}{ \underbrace{ \tfrac{1}{2} \sum_{m' \leq m} C_m^{m'} \big\langle ( \partial^{m'}_x E_\eps^n \cdot v ) \mathsf{q} \partial^{m-m'}_x \mathbb{P}^\perp G_\eps^{n+1} , \partial^m_x \mathbb{P}^\perp G_\eps^{n+1} \big\rangle_{L^2_{x,v}} } } \,.
	\end{aligned}
	\end{equation}
	By the H\"older inequality and the Sobolev embedding theory, one easily yields that
	\begin{equation}
	\begin{aligned}
	X_1 \leq C \| E_\eps^n \|_{H^s_x} \| \mathbb{P} G_\eps^{n+1} \|^2_{H^s_x L^2_v} \leq C \mathbb{E}_s^\frac{1}{2} (G_\eps^n, E_\eps^n, B_\eps^n) \mathbb{E}_s ( G_\eps^{n+1}, E_\eps^{n+1}, B_\eps^{n+1} ) \,,
	\end{aligned}
	\end{equation}
	and
	\begin{equation}
	\begin{aligned}
	X_2 + X_3 \leq & C \| E_\eps^n \|_{H^s_x} \| \mathbb{P} G_\eps^{n+1} \|_{H^s_x L^2_v} \| \mathbb{P}^\perp G_\eps^{n+1} \|_{H^s_x L^2_v (\nu)} \\
	\leq & C \mathbb{E}_s^\frac{1}{2} (G_\eps^n, E_\eps^n, B_\eps^n) \mathbb{E}_s^\frac{1}{2} ( G_\eps^{n+1}, E_\eps^{n+1}, B_\eps^{n+1} ) \| \mathbb{P}^\perp G_\eps^{n+1} \|_{H^s_{x,v}(\nu)} \,,
	\end{aligned}
	\end{equation}
	and
	\begin{equation}
	\begin{aligned}
	X_4 \leq C \| E_\eps^n \|_{H^s_x} \| \mathbb{P}^\perp G_\eps^{n+1} \|^2_{H^s_x L^2_v (\nu)} \leq C \mathbb{E}_s^\frac{1}{2} (G_\eps^n, E_\eps^n, B_\eps^n) \| \mathbb{P}^\perp G_\eps^{n+1} \|^2_{H^s_x L^2_v (\nu)} \,.
	\end{aligned}
	\end{equation}
	In summary, we have
	\begin{equation}\label{Iter-Inq-3}
	\begin{aligned}
	& \tfrac{1}{2} \langle  \mathsf{q} \partial^m_x [ (E_\eps^n \cdot v) G_\eps^{n+1} ] , \partial^m_x G_\eps^{n+1} \rangle_{L^2_{x,v}} \\
	\leq & C \mathbb{E}_s^\frac{1}{2} (G_\eps^n, E_\eps^n, B_\eps^n) \Big[ \mathbb{E}_s ( G_\eps^{n+1}, E_\eps^{n+1}, B_\eps^{n+1} ) + \| \mathbb{P}^\perp G_\eps^{n+1} \|^2_{H^s_x L^2_v (\nu)} \Big] \,.
	\end{aligned}
	\end{equation}
	We now employ Lemma \ref{Lmm-Gamma-Torus} to estimate the term $\tfrac{1}{\eps} \langle  \partial^m_x \Gamma (  G_\eps^n , G_\eps^n ) , \partial^m_x G_\eps^{n+1} \rangle_{L^2_{x,v}}$ for all $|m| \leq s$. More precisely,
	\begin{equation}\label{Iter-Inq-4}
	\begin{aligned}
	& \tfrac{1}{\eps} \langle  \partial^m_x \Gamma (  G_\eps^n , G_\eps^n ) , \partial^m_x G_\eps^{n+1} \rangle_{L^2_{x,v}} = \tfrac{1}{\eps} \langle  \partial^m_x \Gamma (  G_\eps^n , G_\eps^n ) , \partial^m_x \mathbb{P}^\perp G_\eps^{n+1} \rangle_{L^2_{x,v}} \\
	\leq & \tfrac{C}{\eps} \| G_\eps^n \|_{H^s_x L^2_v} \big( \| \mathbb{P} G_\eps^n \|_{H^s_x L^2_v} + \| \mathbb{P}^\perp G_\eps^n \|_{H^s_x L^2_v (\nu)} \big) \| \mathbb{P}^\perp G_\eps^{n+1} \|_{H^s_x L^2_v (\nu)} \\
	\leq & \tfrac{C}{\eps} \mathbb{E}_s^\frac{1}{2} ( G_\eps^n , E_\eps^n , B_\eps^n ) \Big( \mathbb{E}_s^\frac{1}{2} ( G_\eps^n , E_\eps^n , B_\eps^n ) + \| \mathbb{P}^\perp G_\eps^n \|_{H^s_x L^2_v (\nu)} \Big) \| \mathbb{P}^\perp G_\eps^{n+1} \|_{H^s_x L^2_v (\nu)} \,.
	\end{aligned}
	\end{equation}
	Next we estimate the term $- \tfrac{1}{\eps} \sum_{0 \neq m' \leq m} \langle \mathsf{q} \partial^{m'}_x ( \eps E_\eps^n + v \times B_\eps^n ) \cdot \nabla_v \partial^{m-m'}_x G_\eps^{n+1} , \partial^m_x G_\eps^{n+1} \rangle_{L^2_{x,v}}$ for all $|m| \leq s$. By employing the relation $G_\eps^{n+1} = \mathbb{P} G_\eps^{n+1} + \mathbb{P}^\perp G_\eps^{n+1}$, it can be decomposed as four parts:
	\begin{equation}
	\begin{aligned}
	& - \tfrac{1}{\eps} \sum_{0 \neq m' \leq m} C_m^{m'} \big\langle \mathsf{q} \partial^{m'}_x ( \eps E_\eps^n + v \times B_\eps^n ) \cdot \nabla_v \partial^{m-m'}_x G_\eps^{n+1} , \partial^m_x G_\eps^{n+1} \big\rangle_{L^2_{x,v}} \\
	= & \underset{Y_1}{ \underbrace{ - \tfrac{1}{\eps} \sum_{0 \neq m' \leq m} C_m^{m'} \big\langle \mathsf{q} \partial^{m'}_x ( \eps E_\eps^n + v \times B_\eps^n ) \cdot \nabla_v \partial^{m-m'}_x \mathbb{P} G_\eps^{n+1} , \partial^m_x \mathbb{P} G_\eps^{n+1} \big\rangle_{L^2_{x,v}}  } } \\
	&  \underset{Y_2}{ \underbrace{ - \tfrac{1}{\eps} \sum_{0 \neq m' \leq m} C_m^{m'} \big\langle \mathsf{q} \partial^{m'}_x ( \eps E_\eps^n + v \times B_\eps^n ) \cdot \nabla_v \partial^{m-m'}_x \mathbb{P}^\perp G_\eps^{n+1} , \partial^m_x \mathbb{P} G_\eps^{n+1} \big\rangle_{L^2_{x,v}}  } } \\
	&  \underset{Y_3}{ \underbrace{ - \tfrac{1}{\eps} \sum_{0 \neq m' \leq m} C_m^{m'} \big\langle \mathsf{q} \partial^{m'}_x ( \eps E_\eps^n + v \times B_\eps^n ) \cdot \nabla_v \partial^{m-m'}_x \mathbb{P} G_\eps^{n+1} , \partial^m_x \mathbb{P}^\perp G_\eps^{n+1} \big\rangle_{L^2_{x,v}}  } } \\
	&  \underset{Y_4}{ \underbrace{ - \tfrac{1}{\eps} \sum_{0 \neq m' \leq m} C_m^{m'} \big\langle \mathsf{q} \partial^{m'}_x ( \eps E_\eps^n + v \times B_\eps^n ) \cdot \nabla_v \partial^{m-m'}_x \mathbb{P}^\perp G_\eps^{n+1} , \partial^m_x \mathbb{P}^\perp G_\eps^{n+1} \big\rangle_{L^2_{x,v}}  } } \,.
	\end{aligned}
	\end{equation}
	We notice that there is a singular expression 
	\begin{equation}
	\begin{aligned}
	\tfrac{1}{\eps} \sum_{0 \neq m' \leq m} C_m^{m'} \big\langle \mathsf{q} \partial^{m'}_x ( v \times B_\eps^n ) \cdot \nabla_v \partial^{m-m'}_x \mathbb{P} G_\eps^{n+1} , \partial^m_x \mathbb{P} G_\eps^{n+1} \big\rangle_{L^2_{x,v}} 
	\end{aligned}
	\end{equation}
	occurring in the term $Y_1$. Recalling the definition of $\mathbb{P} G$ in \eqref{VMB-Proj}, we have
	\begin{equation}
	\begin{aligned}
	\nabla_v \mathbb{P} G_\eps^{n+1} = (u_\eps^{n+1} + v \theta_\eps^{n+1} ) \mathsf{q}_2 \sqrt{M} - \tfrac{1}{2} v \mathbb{P} G_\eps^{n+1} \,,
	\end{aligned}
	\end{equation}
	where $\mathsf{q}_2 = [1,1] \in \mathbb{R}^2$, $\theta_\eps^{n+1} = \frac{1}{2} \l G_\eps^{n+1} , \frac{2}{3} \phi_6 \r_{L^2_v}$ and the vector field $u_\eps^{n+1} \in \mathbb{R}^3$ with the components $u^{n+1}_{\eps,i} = \frac{1}{2} \l G_\eps^{n+1} , \phi_{2+i} \r_{L^2_v}$ ($i = 1,2,3$). Then we have
	\begin{align}
	\no & \tfrac{1}{\eps} \sum\limits_{0 \neq m' \leq m} C_m^{m'} \big\langle \mathsf{q} \partial^{m'}_x (v \times B_\eps^n ) \cdot \nabla_v \partial^{m-m'}_x \mathbb{P} G_\eps^{n+1} , \partial^m_x \mathbb{P} G_\eps^{n+1} \big\rangle_{L^2_{x,v}} \\
	\no = & \tfrac{1}{\eps} \sum\limits_{0 \neq m' \leq m} C_m^{m'} \big\langle v \times \partial^{m'}_x B_\eps^n \cdot \big[ (\partial^{m-m'}_x u_\eps^{n+1} + v \partial^{m-m'}_x \theta_\eps^{n+1} ) \mathsf{q} \mathsf{q}_2 \sqrt{M} \\
	\no & \qquad \qquad - \tfrac{1}{2} v \mathsf{q} \partial^{m-m'}_x \mathbb{P} G_\eps^{n+1} \big] , \partial^m_x \mathbb{P} G_\eps^{n+1} \big\rangle_{L^2_{x,v}} \\
	\no = & \tfrac{1}{\eps} \sum\limits_{0 \neq m' \leq m} C_m^{m'} \big\langle ( v \times \partial^{m'}_x B_\eps^n ) \cdot \partial^{m-m'}_x u_\eps^{n+1} , \mathsf{q}_1 \sqrt{M} \cdot \partial^m_x \mathbb{P} G_\eps^{n+1} \big\rangle_{L^2_{x,v}} \\
	\no & + \tfrac{1}{\eps} \sum\limits_{0 \neq m' \leq m} C_m^{m'} \big\langle ( v \times \partial^{m'}_x B_\eps^n ) \cdot v , \big( \partial^{m-m'}_x \theta_\eps^{n+1} \mathsf{q}_1 \sqrt{M} \\
	\no & \qquad \qquad - \frac{1}{2} \mathsf{q} \partial^{m-m'}_x \mathbb{P} G_\eps^{n+1}  \big) \cdot \partial^m_x \mathbb{P} G_\eps^{n+1} \big\rangle_{L^2_{x,v}} \\
	\no = & \tfrac{1}{\eps} \sum\limits_{0 \neq m' \leq m} C_m^{m'} \big\langle ( v \times \partial^{m'}_x B_\eps^n ) \cdot \partial^{m-m'}_x u_\eps^{n+1} , \partial^m_x \l G_\eps^{n+1}, \phi_1 - \phi_2 \r_{L^2_v} M \big\rangle_{L^2_{x,v}} \\
	\no = & \tfrac{1}{\eps} \sum\limits_{0 \neq m' \leq m} C_m^{m'} \big\langle ( \l v , M \r_{L^2_v} \times \partial^{m'}_x B_\eps^n ) \cdot \partial^{m-m'}_x u_\eps^{n+1} , \partial^m_x \l G_\eps^{n+1} , \phi_1 - \phi_2 \r_{L^2_v} \big\rangle_{L^2_x} \\
	= & 0 \,.
	\end{align}
	We thereby know that the term $Y_1$ does not involve the singularity. Then, the term $Y_1$ can be bounded by
	\begin{equation}
	\begin{aligned}
	Y_1 \leq & C \| E_\eps^n \|_{H^s_x} \| \mathbb{P} G_\eps^{n+1} \|_{H^s_x L^2_v} \| \nabla_x \mathbb{P} G_\eps^{n+1} \|_{H^{s-1}_x L^2_v} \\
	\leq & C \mathbb{E}_s^\frac{1}{2} (G_\eps^n , E_\eps^n , B_\eps^n) \mathbb{E}_s (G_\eps^{n+1} , E_\eps^{n+1} , B_\eps^{n+1}) \,.
	\end{aligned}
	\end{equation}
	Moreover, the term $Y_2$, $Y_3$ and $Y_4$ can also be controlled as
	\begin{equation}
	\begin{aligned}
	Y_2 + Y_3 \leq & \tfrac{C}{\eps} ( \eps \| E_\eps^n \|_{H^s_x} + \| B_\eps^n \|_{H^s_x} ) \| \mathbb{P} G_\eps^{n+1} \|_{H^s_x L^2_v} \| \mathbb{P}^\perp G_\eps^{n+1} \|_{H^s_x L^2_v} \\
	\leq & \tfrac{C}{\eps}  \mathbb{E}_s^\frac{1}{2} (G_\eps^n , E_\eps^n , B_\eps^n) \mathbb{E}_s^\frac{1}{2} (G_\eps^{n+1} , E_\eps^{n+1} , B_\eps^{n+1}) \| \mathbb{P}^\perp G_\eps^{n+1} \|_{H^s_x L^2_v (\nu)} \,,
	\end{aligned}
	\end{equation}
	and
	\begin{equation}
	\begin{aligned}
	Y_4 \leq & \tfrac{C}{\eps} ( \eps \| E_\eps^n \|_{H^s_x} + \| B_\eps^n \|_{H^s_x} ) \| \mathbb{P}^\perp G_\eps^{n+1} \|_{H^s_x L^2_v (\nu)} \sum_{|m| \leq s - 1} \| \nabla_v \partial^m_x \mathbb{P}^\perp G_\eps^{n+1} \|_{L^2_{x,v}(\nu)} \\
	\leq & \tfrac{C}{\eps} \mathbb{E}_s^\frac{1}{2} (G_\eps^n , E_\eps^n , B_\eps^n) \| \mathbb{P}^\perp G_\eps^{n+1} \|_{\widetilde{H}^s_{x,v} (\nu)} \| \mathbb{P}^\perp G_\eps^{n+1} \|_{H^s_x L^2_v (\nu)} \,.
	\end{aligned}
	\end{equation}
	Consequently, we obtain
	\begin{equation}\label{Iter-Inq-5}
	\begin{aligned}
	- \tfrac{1}{\eps} & \sum_{0 \neq m' \leq m} C_m^{m'} \big\langle \mathsf{q} \partial^{m'}_x ( \eps E_\eps^n + v \times B_\eps^n ) \cdot \nabla_v \partial^{m-m'}_x G_\eps^{n+1} , \partial^m_x G_\eps^{n+1} \big\rangle_{L^2_{x,v}} \\
	& \leq C \mathbb{E}_s^\frac{1}{2} (G_\eps^n , E_\eps^n , B_\eps^n) \mathbb{E}_s (G_\eps^{n+1} , E_\eps^{n+1} , B_\eps^{n+1})  + \tfrac{C}{\eps} \mathbb{E}_s^\frac{1}{2} (G_\eps^n, E_\eps^n, B_\eps^n) \\
	& \times \Big[ \mathbb{E}_s^\frac{1}{2} ( G_\eps^{n+1} , E_\eps^{n+1} , B_\eps^{n+1} ) + \| \mathbb{P}^\perp G_\eps^{n+1} \|_{\widetilde{H}^s_{x,v}(\nu)} \Big] \| \mathbb{P}^\perp G_\eps^{n+1} \|_{H^s_x L^2_v (\nu)} \,.
	\end{aligned}
	\end{equation}
	Plugging the bounds \eqref{Iter-Inq-1}, \eqref{Iter-Inq-2}, \eqref{Iter-Inq-3}, \eqref{Iter-Inq-4} and \eqref{Iter-Inq-5} into the relation \eqref{Iter-Inq-0} reduces to
	\begin{equation}\label{Iter-Spatial-Bnd}
	\begin{aligned}
	& \tfrac{\d }{\d t} \big( \| G_\eps^{n+1} \|^2_{H^s_x L^2_v} + \| E_\eps^{n+1} \|^2_{H^s_x} + \| B_\eps^{n+1} \|^2_{H^s_x} \big) + \tfrac{\lambda}{\eps^2} \| \mathbb{P}^\perp G_\eps^{n+1} \|^2_{H^s_x L^2_v (\nu)} \\
	& \leq C \mathbb{E}_s^2 ( G_\eps^n , E_\eps^n , B_\eps^n ) + C \Big[ 1 + \mathbb{E}_s ( G_\eps^n , E_\eps^n , B_\eps^n ) \Big] \mathbb{E}_s ( G_\eps^{n+1} , E_\eps^{n+1} , B_\eps^{n+1} ) \\
	& + C \Big[ \mathbb{E}_s^\frac{1}{2} ( G_\eps^n , E_\eps^n , B_\eps^n ) + \mathbb{E}_s ( G_\eps^n , E_\eps^n , B_\eps^n ) \Big] \Big( \| \mathbb{P}^\perp G_\eps^n \|^2_{H^s_x L^2_v (\nu)} + \| \mathbb{P}^\perp G_\eps^{n+1} \|^2_{H^s_{x,v}(\nu)} \Big)
	\end{aligned}
	\end{equation}
	for all $0 < \eps \leq 1$.
	
	Next we estimate the mixed $(x,v)$-derivative energy bound. We first rewrite the $G_\eps^{n+1}$-equation of \eqref{Iter-Appro-Systm} as
	\begin{equation}
	\begin{aligned}
	& \big\{ \partial_t + \tfrac{1}{\eps} v \cdot \nabla_x + \tfrac{1}{\eps} \mathsf{q} ( \eps E_\eps^n + v \times B_\eps^n ) \cdot \nabla_v + \tfrac{1}{\eps^2} \mathscr{L} - \tfrac{1}{2} \mathsf{q} ( E_\eps^n \cdot v ) \big\} \mathbb{P}^\perp G_\eps^{n+1} \\
	& = - \big\{ \partial_t + \tfrac{1}{\eps} v \cdot \nabla_x + \tfrac{1}{\eps} \mathsf{q} ( \eps E_\eps^n + v \times B_\eps^n ) \cdot \nabla_v - \tfrac{1}{2} \mathsf{q} ( E_\eps^n \cdot v ) \big\} \mathbb{P} G_\eps^{n+1} \\
	& + \tfrac{1}{\eps} ( E_\eps^{n+1} \cdot v ) \mathsf{q}_1 \sqrt{M} + \tfrac{1}{\eps} \Gamma (G_\eps^n , G_\eps^n) \,.
	\end{aligned}
	\end{equation}
	For all $|m| + |\alpha| \leq s$ and $\alpha \neq 0$, we take the derivative operator $\partial^m_\alpha$ in \eqref{Iter-Appro-Decomp} and take inner product with $\partial^m_\alpha \mathbb{P}^\perp G_\eps^{n+1}$ over $\T^3 \times \R^3$. Then we obtain
	\begin{equation}\label{Z1-9}
	\begin{aligned}
	& \tfrac{1}{2} \tfrac{\d}{\d t} \| \partial^m_\alpha \mathbb{P}^\perp G_\eps^{n+1} \|^2_{L^2_{x,v}} + \underset{Z_1}{ \underbrace{ \tfrac{1}{\eps^2} \l \partial^m_\alpha \mathscr{L} \mathbb{P}^\perp G_\eps^{n+1} , \partial^m_\alpha \mathbb{P}^\perp G_\eps^{n+1} \r_{L^2_{x,v}} }} \\
	= & \underset{Z_2}{ \underbrace{ \tfrac{1}{\eps} \l \partial^m_\alpha \big[ (E_\eps^n \cdot v) \sqrt{M} \mathsf{q}_1 \big] , \partial^m_\alpha \mathbb{P}^\perp G_\eps^{n+1} \r_{L^2_{x,v}} }} + \underset{Z_3}{ \underbrace{ \tfrac{1}{2} \l \partial^m_\alpha \big[ \mathsf{q} ( E_\eps^n \cdot v ) G_\eps^{n+1} \big] , \partial^m_\alpha \mathbb{P}^\perp G_\eps^{n+1} \r_{L^2_{x,v}} }} \\
	& + \underset{Z_4}{ \underbrace{ \tfrac{1}{\eps} \l \partial^m_\alpha \Gamma (G_\eps^n, G_\eps^n) , \partial^m_\alpha \mathbb{P}^\perp G_\eps^{n+1} \r_{L^2_{x,v}} }} \ \underset{Z_5}{ \underbrace{ - \l \partial_t \partial^m_\alpha \mathbb{P} G_\eps^{n+1} , \partial^m_\alpha \mathbb{P}^\perp G_\eps^{n+1} \r_{L^2_{x,v}} }} \\
	& \underset{Z_6}{ \underbrace{ - \tfrac{1}{\eps} \sum_{|\alpha'|=1} C_\alpha^{\alpha'} \l \partial^{\alpha'}_v v \cdot \nabla_x \partial^m_{\alpha - \alpha'} \mathbb{P}^\perp G_\eps^{n+1} , \partial^m_\alpha \mathbb{P}^\perp G_\eps^{n+1} \r_{L^2_{x,v}} }} \\
	& \underset{Z_7}{ \underbrace{ - \tfrac{1}{\eps} \sum_{|\alpha'|=1} C_\alpha^{\alpha'} \l \mathsf{q} ( \partial^{\alpha'}_v v \times B_\eps^n ) \cdot \nabla_v \partial^m_{\alpha - \alpha'} \mathbb{P}^\perp G_\eps^{n+1} , \partial^m_\alpha \mathbb{P}^\perp G_\eps^{n+1} \r_{L^2_{x,v}} }} \\
	& \underset{Z_8}{ \underbrace{ - \tfrac{1}{\eps} \sum_{m' < m} \sum_{|\alpha'| \leq 1} C_m^{m'} C_\alpha^{\alpha'} \l \mathsf{q} ( \partial^{\alpha'}_v v \times \partial^{m-m'}_x B_\eps^n ) \cdot \nabla_v \partial^{m'}_{\alpha - \alpha'} \mathbb{P}^\perp G_\eps^{n+1} , \partial^m_\alpha \mathbb{P}^\perp G_\eps^{n+1} \r_{L^2_{x,v}} }} \\
	&  \underset{Z_9}{ \underbrace{ - \tfrac{1}{\eps} \l \partial^m_\alpha \big[ v \cdot \nabla_x + \mathsf{q} ( \eps E_\eps^n + v \times B_\eps^n ) \cdot \nabla_v \big] \mathbb{P} G_\eps^{n+1} , \partial^m_\alpha \mathbb{P}^\perp G_\eps^{n+1} \r_{L^2_{x,v}} }} \,.
	\end{aligned}
	\end{equation}
	
	Recalling the decomposition of $\mathscr{L}$ in Lemma \ref{Lmm-L} (1), we have
	\begin{equation}
	\begin{aligned}
	\big\langle \partial^m_\alpha \mathscr{L} \mathbb{P}^\perp G_\eps^{n+1} , \partial^m_\alpha \mathbb{P}^\perp G_\eps^{n+1} \big\rangle_{L^2_{x,v}} = & 2 \big\langle \partial^m_\alpha (\nu (v) \mathbb{P}^\perp G_\eps^{n+1} ) , \partial^m_\alpha \mathbb{P}^\perp G_\eps^{n+1} \big\rangle_{L^2_{x,v}} \\
	& - \big\langle \partial^m_\alpha \mathscr{K} \mathbb{P}^\perp G_\eps^{n+1} , \partial^m_\alpha \mathbb{P}^\perp G_\eps^{n+1} \big\rangle_{L^2_{x,v} } \,.
	\end{aligned}
	\end{equation}
	From Lemma \ref{Lmm-nu-norm} (2), we derive
	\begin{equation}
	\begin{aligned}
	\big\langle \partial^m_\alpha (\nu (v) \mathbb{P}^\perp G_\eps^{n+1} ) , \partial^m_\alpha \mathbb{P}^\perp G_\eps^{n+1} \big\rangle_{L^2_{x,v}} \geq C_5 \| \partial^m_\alpha \mathbb{P}^\perp G_\eps^{n+1} \|^2_{L^2_{x,v}(\nu)} - C_6 \sum_{\alpha' < \alpha} \| \partial^m_{\alpha'} \mathbb{P}^\perp G_\eps^{n+1} \|^2_{L^2_{x,v}} \,.
	\end{aligned}
	\end{equation}
	Moreover, Lemma \ref{Lmm-L} (2) tells us that for any $\delta > 0$, there is a $C(\delta) > 0$ such that
	\begin{equation}
	\begin{aligned}
	\big\langle \partial^m_\alpha \mathscr{K} \mathbb{P}^\perp G_\eps^{n+1} , \partial^m_\alpha \mathbb{P}^\perp G_\eps^{n+1} \big\rangle_{L^2_{x,v}} \leq \delta \| \partial^m_\alpha \mathbb{P}^\perp G_\eps^{n+1} \|^2_{L^2_{x,v}(\nu)} + C(\delta) \| \partial^m_x \mathbb{P}^\perp G_\eps^{n+1} \|^2_{L^2_{x,v}} \,.
	\end{aligned}
	\end{equation}
	Thus taking $\delta = C_5 > 0$, $\lambda_0 = C_5 > 0$ and $\lambda_1 = C(C_5) + 2 C_6 > 0$ implies that the quantity $M_1$ has the lower bound
	\begin{equation}\label{Z1}
	\begin{aligned}
	Z_1 = \tfrac{1}{\eps^2} \big\langle \partial^m_\alpha \mathscr{L} \mathbb{P}^\perp G_\eps^{n+1} , \partial^m_\alpha \mathbb{P}^\perp G_\eps^{n+1} \big\rangle_{L^2_{x,v}} \geq \tfrac{\lambda_0}{\eps^2} \| \partial^m_\alpha \mathbb{P}^\perp G_\eps^{n+1} \|^2_{L^2_{x,v} (\nu)} - \tfrac{\lambda_1}{\eps^2} \sum_{\alpha' < \alpha} \| \partial^m_{\alpha'} \mathbb{P}^\perp G_\eps^{n+1} \|^2_{L^2_{x,v}} \,.
	\end{aligned}
	\end{equation}
	
	Since $|m| + |\alpha| \leq s$ and $\alpha \neq 0$, $0 \leq |m| \leq s - 1$. Then the term $Z_2$ can be estimated as 
	\begin{equation}\label{Z2}
	\begin{aligned}
	Z_2 \leq \tfrac{C}{\eps} \| E_\eps^n \|_{H^{s-1}_x} \| \partial^m_\alpha \mathbb{P}^\perp G_\eps^{n+1} \|_{L^2_{x,v}(\nu)} \leq & \tfrac{C}{\eps} \mathbb{E}_s^\frac{1}{2} ( G_\eps^n, E_\eps^n , B_\eps^n ) \| \partial^m_\alpha \mathbb{P}^\perp G_\eps^{n+1} \|_{L^2_{x,v}(\nu)}
	\end{aligned}
	\end{equation}
	by using the H\"older inequality and the Sobolev embedding theory. For the term $Z_3$, one can decompose it as
	\begin{equation}
	\begin{aligned}
	Z_3 = & \underset{Z_{31}}{\underbrace{ \tfrac{1}{2} \sum_{|\alpha'| = 1} C_\alpha^{\alpha'} \big\langle \partial^m_x [ E_\eps^n \cdot ( \partial^{\alpha'}_v v \otimes \partial^{\alpha - \alpha'}_v G_\eps^{n+1} ) ] , \mathsf{q} \partial^m_\alpha \mathbb{P}^\perp G_\eps^{n+1} \big\rangle_{L^2_{x,v}} }} \\
	& + \underset{Z_{32}}{\underbrace{ \tfrac{1}{2} \big\langle \partial^m_x [ E_\eps^n \cdot ( v \otimes \partial^\alpha_v \mathbb{P} G_\eps^{n+1} ) ] , \mathsf{q} \partial^m_\alpha \mathbb{P}^\perp G_\eps^{n+1} \big\rangle_{L^2_{x,v}} }} \\
	& + \underset{Z_{33}}{\underbrace{ \tfrac{1}{2} \big\langle \partial^m_x [ E_\eps^n \cdot ( v \otimes \partial^\alpha_v \mathbb{P}^\perp G_\eps^{n+1} ) ] , \mathsf{q} \partial^m_\alpha \mathbb{P}^\perp G_\eps^{n+1} \big\rangle_{L^2_{x,v}} }}
	\end{aligned}
	\end{equation}
	by using $G_\eps^{n+1} = \mathbb{P} G_\eps^{n+1} + \mathbb{P}^\perp G_\eps^{n+1}$. Then the terms $Z_{31}$ and $Z_{32}$ can be bounded by
	\begin{equation}
	\begin{aligned}
	C \| E_\eps^n \|_{H^s_x} \| G_\eps^{n+1} \|_{H^s_{x,v}} \| \partial^m_\alpha \mathbb{P}^\perp G_\eps^{n+1} \|_{L^2_{x,v}}
	\end{aligned}
	\end{equation}
	and the term $Z_{33}$ can be bounded by
	\begin{equation}
	\begin{aligned}
	C \| E_\eps^n \|_{H^s_x} \| \mathbb{P}^\perp G_\eps^{n+1} \|_{\widetilde{H}^s_{x,v}(\nu)} \| \partial^m_\alpha \mathbb{P}^\perp G_\eps^{n+1} \|_{L^2_{x,v}(\nu)} \,.
	\end{aligned}
	\end{equation}
	We figure out that the factor $v$ in $Z_{33}$ can only be estimated by $\nu (v)$, so that it will be pushed to the $\nu$-weight of norms. However, the $v$ factor of $Z_{32}$ will be absorbed by the factor $\sqrt{M}$ in $\mathbb{P} G_\eps^{n+1}$ and  the term $Z_{31}$ does not involve the $v$ factor. We thereby have
	\begin{equation}\label{Z3}
	\begin{aligned}
	Z_3 \leq C \mathbb{E}_s^\frac{1}{2} (G_\eps^n , E_\eps^n, B_\eps^n) \Big[ \mathbb{E}_s ( G_\eps^{n+1} , E_\eps^{n+1} , B_\eps^{n+1} ) \\
	+ \| \mathbb{P}^\perp G_\eps^{n+1} \|_{\widetilde{H}^s_{x,v}(\nu)} \| \partial^m_\alpha \mathbb{P}^\perp G_\eps^{n+1} \|_{L^2_{x,v}(\nu)} \Big] \,.
	\end{aligned}
	\end{equation}
	Via Lemma \ref{Lmm-Gamma-Torus} and the decomposition $G_\eps^n = \mathbb{P} G_\eps^n + \mathbb{P}^\perp G_\eps^n$, we easily estimate that
	\begin{equation}\label{Z4}
	\begin{aligned}
	Z_4 \leq & \tfrac{C}{\eps} \| G_\eps^n \|_{H^s_{x,v}} \big( \| \mathbb{P} G_\eps^n \|_{H^s_x L^2_v} + \| \mathbb{P}^\perp G_\eps^n \|_{H^s_{x,v}(\nu)} \big) \| \partial^m_\alpha \mathbb{P}^\perp G_\eps^{n+1} \|_{L^2_{x,v}(\nu)} \\
	\leq & \tfrac{C}{\eps} \mathbb{E}_s (G_\eps^n , E_\eps^n, B_\eps^n) \| \partial^m_\alpha \mathbb{P}^\perp G_\eps^{n+1} \|_{L^2_{x,v}(\nu)} \\
	& + \tfrac{C}{\eps} \mathbb{E}_s^\frac{1}{2} (G_\eps^n, E_\eps^n , B_\eps^n) \| \mathbb{P}^\perp G_\eps^n \|_{H^s_{x,v}(\nu)} \| \partial^m_\alpha \mathbb{P}^\perp G_\eps^{n+1} \|_{L^2_{x,v}(\nu)} \,.
	\end{aligned}
	\end{equation}
	
	Next we control the term $Z_5$. Recalling the expression of $\mathbb{P} G$ in \eqref{VMB-Proj-Smp}, we have
	\begin{equation}
	\begin{aligned}
	\mathbb{P} G_\eps^{n+1} = \rho_\eps^{n+1,+} \phi_1 (v) + \rho_\eps^{n+1,-} \phi_2 (v) + \sum_{i=1}^3 u_{\eps,i}^{n+1} \phi_{i+2} (v) + \theta_\eps^{n+1} \phi_6 (v) \,.
	\end{aligned}
	\end{equation}
	Thus the term $Z_5$ can be estimated by
	\begin{equation}\label{Iter-Z5-rho-u-theta}
	\begin{aligned}
	Z_5 \leq C \big[ \| \partial_t \partial^m_x \rho_\eps^{n+1,+} \|_{L^2_x} + \| \partial_t \partial^m_x \rho_\eps^{n+1,-} \|_{L^2_x} + \| \partial_t \partial^m_x u_\eps^{n+1} \|_{L^2_x} \\
	+ \| \partial_t \partial^m_x \theta_\eps^{n+1} \|_{L^2_x} \big] \| \partial^m_\alpha \mathbb{P}^\perp G_\eps^{n+1} \|_{L^2_{x,v}} \,,
	\end{aligned}
	\end{equation}
	where $|m| \leq s - 1$. We further project the first $G_\eps^{n+1}$-equation of \eqref{Iter-Appro-Systm} into $\textrm{Ker} (\mathscr{L})$ by multiplying the vectors $\phi_1 (v)$, $\phi_2 (v)$, $\tfrac{1}{2} \phi_3 (v)$, $\tfrac{1}{2} \phi_4 (v)$, $\tfrac{1}{2} \phi_5 (v)$ and $\tfrac{2}{3} \phi_6 (v)$, respectively, and integrating over $v \in \R^3$. Thanks to $\Gamma (G_\eps^n,G_\eps^n) \in \textrm{Ker}^\perp (\mathscr{L})$, we deduce that
	\begin{equation}\label{Iter-Fluid-Part}
	\left\{
	\begin{array}{l}
	\partial_t \rho_\eps^{n+1,+} = \tfrac{1}{\eps} \big[ - \div_x \, u_\eps^{n+1} + \langle \Theta_n + \Psi_n , \phi_1 (v) \rangle_{L^2_v} \big] \,, \\[1.5mm]
	\partial_t \rho_\eps^{n+1,-} = \tfrac{1}{\eps} \big[ - \div_x \, u_\eps^{n+1} + \langle \Theta_n + \Psi_n , \phi_2 (v) \rangle_{L^2_v} \big] \,, \\[1.5mm]
	\partial_t u_\eps^{n+1} = \tfrac{1}{\eps} \big[ - \partial_i ( \tfrac{\rho_\eps^{n+1,+} + \rho_\eps^{n+1,-}}{2} + \theta_\eps^{n+1} ) + \langle \Theta_n + \Psi_n , \phi_{i+2} (v) \rangle_{L^2_v} \big] \quad (1 \leq i \leq 3) \,, \\[1.5mm]
	\partial_t \theta_\eps^{n+1} = \tfrac{1}{\eps} \big[ - \tfrac{2}{3} \div_x \, u_\eps^{n+1} + \tfrac{1}{3} \langle \Theta_n + \Psi_n , \phi_6 (v) \rangle_{L^2_v} \big] \,, 
	\end{array}
	\right.
	\end{equation}
	where
	\begin{equation}\label{Iter-Phi-Psi}
	\begin{aligned}
	\Theta_n = & - \Big( v \cdot \nabla_x + \mathsf{q} ( \eps E_\eps^n + v \times B_\eps^n ) \cdot \nabla_v - \tfrac{1}{2} \eps \mathsf{q} ( E_\eps^n \cdot v ) \Big) \mathbb{P}^\perp G_\eps^{n+1} \,, \\
	\Psi_n = & - \mathsf{q} ( \eps E_\eps^n + v \times B_\eps^n ) \cdot \nabla_v \mathbb{P} G_\eps^{n+1} + \tfrac{1}{2} \eps \mathsf{q} ( E_\eps^n \cdot v ) \mathbb{P} G_\eps^{n+1} \,.
	\end{aligned}
	\end{equation}
	For the quantity $\| \partial_t \partial^m_x \rho_\eps^{n+1,+} \|_{L^2_x}$, we derive from the first equation of \eqref{Iter-Fluid-Part} that for all $|m| \leq s - 1$
	\begin{equation}
	\begin{aligned}
	\| \partial_t \partial^m_x \rho_\eps^{n+1,+} \|_{L^2_x} \leq \tfrac{C}{\eps} \| \nabla_x & \mathbb{P} G_\eps^{n+1} \|_{H^{s-1}_x L^2_v} \\
	& + \tfrac{C}{\eps} \| \partial^m_x  \langle \Theta_n , \phi_1 (v) \rangle_{L^2_v} \|_{L^2_x} + \tfrac{C}{\eps} \| \partial^m_x  \langle \Psi_n , \phi_1 (v) \rangle_{L^2_v} \|_{L^2_x} \,.
	\end{aligned}
	\end{equation}
	Combining the definitions of $\Phi_n$ and $\Psi_n$ in \eqref{Iter-Phi-Psi}, the quantity $ \tfrac{C}{\eps} \| \partial^m_x  \langle \Theta_n , \phi_1 (v) \rangle_{L^2_v} \|_{L^2_x} $ can be controlled by
	\begin{equation}
	\begin{aligned}
	\tfrac{C}{\eps} \big( 1 + \eps \| E_\eps^n \|_{H^s_x} + \| B_\eps^n \|_{H^s_x} \big) \| G_\eps^{n+1} \|_{H^s_{x,v}} \,,
	\end{aligned}
	\end{equation}
	and the quantity $ \tfrac{C}{\eps} \| \partial^m_x  \langle \Psi_n , \phi_1 (v) \rangle_{L^2_v} \|_{L^2_x} $ can be bounded by
	\begin{equation}
	\begin{aligned}
	\tfrac{C}{\eps} \big( \eps \| E_\eps^n \|_{H^s_x} + \| B_\eps^n \|_{H^s_x} \big) \| \nabla_x \mathbb{P} G_\eps^{n+1} \|_{H^{s-1}_x L^2_v} \,.
	\end{aligned}
	\end{equation}
	Therefore, we have
	\begin{equation}\label{Iter-rho-n+1+}
	\begin{aligned}
	\| \partial_t \partial^m_x \rho_\eps^{n+1,+} \|_{L^2_x} \leq & \tfrac{C}{\eps} \big( 1 + \eps \| E_\eps^n \|_{H^s_x} + \| B_\eps^n \|_{H^s_x} \big) \| G_\eps^{n+1} \|_{H^s_{x,v}} \\
	\leq & \tfrac{C}{\eps} \Big[ 1 + \mathbb{E}_s^\frac{1}{2} (G_\eps^n, E_\eps^n , B_\eps^n) \Big] \mathbb{E}_s^\frac{1}{2} ( G_\eps^{n+1} , E_\eps^{n+1} , B_\eps^{n+1} ) \,.
	\end{aligned}
	\end{equation}
	Furthermore, by the analogous arguments in estimating the norm $\| \partial_t \partial^m_x \rho_\eps^{n+1,+} \|_{L^2_x}$ in \eqref{Iter-rho-n+1+}, we easily yield that
	\begin{equation}\label{Iter-rho-u-theta}
	\begin{aligned}
	\| \partial_t \partial^m_x \rho_\eps^{n+1,-} \|_{L^2_x} + \| \partial_t \partial^m_x u_\eps^{n+1} & \|_{L^2_x} + \| \partial_t \partial^m_x \theta_\eps^{n+1} \|_{L^2_x} \\
	& \leq \tfrac{C}{\eps} \Big[ 1 + \mathbb{E}_s^\frac{1}{2} (G_\eps^n, E_\eps^n , B_\eps^n) \Big] \mathbb{E}_s^\frac{1}{2} ( G_\eps^{n+1} , E_\eps^{n+1} , B_\eps^{n+1} ) \,.
	\end{aligned}
	\end{equation}
	Collecting the estimates \eqref{Iter-Z5-rho-u-theta}, \eqref{Iter-rho-n+1+} and \eqref{Iter-rho-u-theta} gives
	\begin{equation}\label{Z5}
	\begin{aligned}
	Z_5 \leq & \tfrac{C}{\eps} \Big[ 1 + \mathbb{E}_s^\frac{1}{2} (G_\eps^n, E_\eps^n , B_\eps^n) \Big] \mathbb{E}_s^\frac{1}{2} ( G_\eps^{n+1} , E_\eps^{n+1} , B_\eps^{n+1} ) \| \partial^m_\alpha \mathbb{P}^\perp G_\eps^{n+1} \|_{L^2_{x,v}(\nu)} \,.
	\end{aligned}
	\end{equation}
	
	It is easy to know that the term $Z_6$ is bounded by
	\begin{equation}\label{Z6}
	\begin{aligned}
	Z_6 \leq \tfrac{C}{\eps} \| G_\eps^{n+1} \|_{H^s_{x,v}} \| \partial^m_\alpha \mathbb{P}^\perp & G_\eps^{n+1} \|_{L^2_{x,v}(\nu)} \\
	& \leq \tfrac{C}{\eps} \mathbb{E}_s^\frac{1}{2} ( G_\eps^{n+1} , E_\eps^{n+1} , B_\eps^{n+1} ) \| \partial^m_\alpha \mathbb{P}^\perp G_\eps^{n+1} \|_{L^2_{x,v}(\nu)} 
	\end{aligned}
	\end{equation}
	and $Z_7$ is estimated by
	\begin{equation}\label{Z7}
	\begin{aligned}
	Z_7 \leq & \tfrac{C}{\eps} \| B_\eps^n \|_{H^s_x} \| G_\eps^{n+1} \|_{H^s_{x,v}} \| \partial^m_\alpha \mathbb{P}^\perp G_\eps^{n+1} \|_{L^2_{x,v}(\nu)} \\
	\leq & \tfrac{C}{\eps} \mathbb{E}_s^\frac{1}{2} ( G_\eps^n , E_\eps^n , B_\eps^n ) \mathbb{E}_s^\frac{1}{2} ( G_\eps^{n+1} , E_\eps^{n+1} , B_\eps^{n+1} ) \| \partial^m_\alpha \mathbb{P}^\perp G_\eps^{n+1} \|_{L^2_{x,v}(\nu)} \,.
	\end{aligned}
	\end{equation}
	Finally, one notices that the terms $Z_8$ and $Z_9$ will be dominated by
	\begin{equation}\label{Z8}
	\begin{aligned}
	Z_8 \leq & \tfrac{C}{\eps} \| B_\eps^n \|_{H^s_x} \| \mathbb{P}^\perp G_\eps^{n+1} \|_{\widetilde{H}^s_{x,v} (\nu)} \| \partial^m_\alpha \mathbb{P}^\perp G_\eps^{n+1} \|_{L^2_{x,v}(\nu)} \\
	\leq & \tfrac{C}{\eps} \mathbb{E}_s^\frac{1}{2} (G_\eps^n, E_\eps^n, B_\eps^n) \| \mathbb{P}^\perp G_\eps^{n+1} \|_{\widetilde{H}^s_{x,v} (\nu)} \| \partial^m_\alpha \mathbb{P}^\perp G_\eps^{n+1} \|_{L^2_{x,v}(\nu)} \,,
	\end{aligned}
	\end{equation}
	and
	\begin{equation}\label{Z9}
	\begin{aligned}
	Z_9 \leq & \tfrac{C}{\eps} \big( 1 + \eps \| E_\eps^n \|_{H^s_x} + \| B_\eps^n \|_{H^s_x} \big) \| \nabla_x \mathbb{P} G_\eps^{n+1} \|_{H^{s-1}_x L^2_v} \| \partial^m_\alpha \mathbb{P}^\perp G_\eps^{n+1} \|_{L^2_{x,v}(\nu)} \\
	\leq & \tfrac{C}{\eps} \Big[ 1 + \mathbb{E}_s^\frac{1}{2} (G_\eps^n, E_\eps^n, B_\eps^n) \Big] \mathbb{E}_s^\frac{1}{2} (G_\eps^{n+1}, E_\eps^{n+1}, B_\eps^{n+1}) \| \partial^m_\alpha \mathbb{P}^\perp G_\eps^{n+1} \|_{L^2_{x,v}(\nu)} \,.
	\end{aligned}
	\end{equation}
	
	Consequently, by substituting the bounds \eqref{Z1}, \eqref{Z2}, \eqref{Z3}, \eqref{Z4}, \eqref{Z5}, \eqref{Z6}, \eqref{Z7}, \eqref{Z8}, \eqref{Z9} into the equality \eqref{Z1-9} and employing the Young's inequality, we obtain
	\begin{equation}\label{Iter-Unclosed-Est}
	\begin{aligned}
	& \tfrac{\d}{\d t} \| \partial^m_\alpha \mathbb{P}^\perp G_\eps^{n+1} \|^2_{L^2_{x,v}} + \tfrac{\lambda_0}{\eps^2} \| \partial^m_\alpha \mathbb{P}^\perp G_\eps^{n+1} \|^2_{L^2_{x,v}(\nu)} \\
	\leq & \tfrac{\lambda_*}{\eps^2} \sum_{|\alpha'| < |\alpha|} \| \partial^m_{\alpha'} \mathbb{P}^\perp G_\eps^{n+1} \|^2_{L^2_{x,v}(\nu)} + C \mathbb{E}_s ( G_\eps^n, E_\eps^n, B_\eps^n ) + C \mathbb{E}_s ( G_\eps^{n+1}, E_\eps^{n+1}, B_\eps^{n+1} ) \\
	& + C \mathbb{E}_s ( G_\eps^n, E_\eps^n, B_\eps^n ) \Big[ \mathbb{E}_s ( G_\eps^{n+1}, E_\eps^{n+1}, B_\eps^{n+1} ) + \| \mathbb{P}^\perp G_\eps^n \|^2_{H^s_{x,v}(\nu)} + \| \mathbb{P}^\perp G_\eps^{n+1} \|^2_{\widetilde{H}^s_{x,v}(\nu)} \Big]
	\end{aligned}
	\end{equation}
	for all $|m| + |\alpha| \leq s$ with $\alpha \neq 0$ and $0 < \eps \leq 1$.
	
	We now claim that for any given $1 \leq k \leq s$ and $|\alpha| \leq s$, there are positive constants $a_{|\alpha|}$, $b_k$, $r_k$ and $C_k'$, independent of $\eps$ and $n$, such that
	\begin{equation}\label{Iter-Mixed-Bnd}
	\begin{aligned}
	\tfrac{\d}{\d t} & \bigg\{ \| G_\eps^{n+1} \|^2_{H^s_x L^2_v} + \| E_\eps^{n+1} \|^2_{H^s_x} + \| B_\eps^{n+1} \|^2_{H^s_x} + \sum_{\substack{|m| + |\alpha| \leq s \\ 1 \leq |\alpha| \leq k}} a_{|\alpha|} \| \partial^m_\alpha \mathbb{P}^\perp G_\eps^{n+1} \|^2_{L^2_{x,v}} \bigg\} \\
	& + \tfrac{b_k}{\eps^2} \| \mathbb{P}^\perp G_\eps^{n+1} \|^2_{H^s_x L^2_v (\nu)} + \tfrac{r_k}{\eps^2} \sum_{\substack{|m| + |\alpha| \leq s \\ 1 \leq |\alpha| \leq k}} \| \partial^m_\alpha \mathbb{P}^\perp G_\eps^{n+1} \|^2_{L^2_{x,v}(\nu)} \\
	& \leq C_k' \bigg\{  \Big[ 1 + \mathbb{E}_s (G_\eps^n, E_\eps^n, B_\eps^n) \Big] \Big[ \mathbb{E}_s (G_\eps^n , E_\eps^n , B_\eps^n) + \mathbb{E}_s (G_\eps^{n+1} , E_\eps^{n+1} , B_\eps^{n+1}) \Big] \\
	& + \Big[ \mathbb{E}_s^\frac{1}{2} (G_\eps^n , E_\eps^n , B_\eps^n) + \mathbb{E}_s ( G_\eps^n , E_\eps^n , B_\eps^n ) \Big] \Big( \| \mathbb{P}^\perp G_\eps^n \|^2_{H^s_x L^2_v} + \| \mathbb{P}^\perp G_\eps^{n+1} \|^2_{H^s_{x,v}(\nu)} \Big) \bigg\}
	\end{aligned}
	\end{equation}
	for all $0 < \eps \leq 1$.
	
	We will verify the bound \eqref{Iter-Mixed-Bnd} by induction. If $k = 1$, $|\alpha'| < |\alpha| = k$ implies $\alpha' = 0$. Then there is a constant $\lambda_*' > 0$, independent of $\eps$ and $n$, such that
	\begin{equation}
	\begin{aligned}
	\sum_{|m|+|\alpha|\leq s \,, |\alpha| = 1} \sum_{|\alpha'| < |\alpha|} \| \partial^m_{\alpha'} \mathbb{P}^\perp G_\eps^{n+1} \|^2_{L^2_{x,v}(\nu)} \leq \lambda_*' \| \mathbb{P}^\perp G_\eps^{n+1} \|^2_{H^s_x L^2_v (\nu)} \,.
	\end{aligned}
	\end{equation}
	Then, summing up for $|m| + |\alpha| \leq s$ with $|\alpha| = 1$ in \eqref{Iter-Unclosed-Est}, multiplying it by $\tfrac{\lambda}{2 \lambda_* \lambda_*'}$ and adding it to the inequality \eqref{Iter-Spatial-Bnd} imply that the inequality \eqref{Iter-Mixed-Bnd} holds for the case $k=1$, where $a_1 = \tfrac{\lambda}{2 \lambda_* \lambda_*'} > 0$, $b_1 = \tfrac{1}{2} \lambda > 0$, $r_1 = \tfrac{\lambda_0 \lambda}{2 \lambda_* \lambda_*'} > 0$ and $C_1' > 0$ are determined by the coefficients in left-hand of \eqref{Iter-Spatial-Bnd} and \eqref{Iter-Unclosed-Est}. Now we assume that the inequality \eqref{Iter-Mixed-Bnd} is valid for $k$. For $|\alpha| = k+1$, summing up for $|m|+|\alpha| \leq s$ with $|\alpha| = k + 1$ in the inequality \eqref{Iter-Unclosed-Est}, we obtain
	\begin{equation}\label{Iter-Induction}
	\begin{aligned}
	\tfrac{\d}{\d t} & \sum_{|m|+|\alpha|\leq s \,, |\alpha| = k +1} \| \partial^m_\alpha \mathbb{P}^\perp G_\eps^{n+1} \|^2_{L^2_{x,v}} + \tfrac{\lambda_0}{\eps^2} \sum_{|m|+|\alpha|\leq s \,, |\alpha| = k +1} \| \partial^m_\alpha \mathbb{P}^\perp G_\eps^{n+1} \|^2_{L^2_{x,v}(\nu)} \\
	& \leq \tfrac{N_{k+1} \lambda_*}{\eps^2} \sum_{|m|+|\alpha|\leq s \,, |\alpha| \leq k} \| \partial^m_\alpha \mathbb{P}^\perp G_\eps^{n+1} \|^2_{L^2_{x,v}(\nu)} + C N_{k+1} \bigg\{ \mathbb{E}_s (G_\eps^n , E_\eps^n, B_\eps^n) \\
	& + \Big[ 1 + \mathbb{E}_s ( G_\eps^n , E_\eps^n , B_\eps^n ) \Big] \mathbb{E}_s ( G_\eps^{n+1} , E_\eps^{n+1} , B_\eps^{n+1} ) \\
	& + \mathbb{E}_s (G_\eps^n , E_\eps^n, B_\eps^n) \Big( \| \mathbb{P}^\perp G_\eps^n \|^2_{H^s_{x,v}(\nu)} + \| \mathbb{P}^\perp G_\eps^{n+1} \|^2_{H^s_{x,v}(\nu)}  \Big) \bigg\} \,.
	\end{aligned}
	\end{equation}
	Here $N_{k+1} > 0$ denotes the number of all possible $(m,\alpha)$ such that $|m|+|\alpha| \leq s$ with $|\alpha| = k +1$. By the assumption of induction, \eqref{Iter-Mixed-Bnd} holds for the case $k$. In order to absorb the first term on the right-hand side of \eqref{Iter-Induction} by the last term on the left-hand side of \eqref{Iter-Mixed-Bnd}, we multiply \eqref{Iter-Induction} by $\tfrac{r_k}{2 \lambda_* N_{k+1}}$ and add it to \eqref{Iter-Mixed-Bnd}. We then have
	\begin{equation}
	\begin{aligned}
	\tfrac{\d}{\d t} & \bigg\{ \mathbb{E}_s ( G_\eps^{n+1}, E_\eps^{n+1}, B_\eps^{n+1} ) + \sum_{|m| + |\alpha| \leq s \,, 1 \leq |\alpha| \leq k+1} a_{|\alpha|} \| \partial^m_\alpha \mathbb{P}^\perp G_\eps^{n+1} \|^2_{L^2_{x,v}} \bigg\} \\
	& + \tfrac{b_k}{\eps^2} \| \mathbb{P}^\perp G_\eps^{n+1} \|^2_{H^s_x L^2_v (\nu)} + \tfrac{r_k}{\eps^2} \sum_{|m| + |\alpha| \leq s \,, 1 \leq |\alpha| \leq k+1} \| \partial^m_\alpha \mathbb{P}^\perp G_\eps^{n+1} \|^2_{L^2_{x,v}(\nu)} \\
	& \leq C_{k+1}' \bigg\{  \Big[ 1 + \mathbb{E}_s (G_\eps^n, E_\eps^n, B_\eps^n) \Big] \Big[ \mathbb{E}_s (G_\eps^n , E_\eps^n , B_\eps^n) + \mathbb{E}_s (G_\eps^{n+1} , E_\eps^{n+1} , B_\eps^{n+1}) \Big] \\
	& + \Big[ \mathbb{E}_s^\frac{1}{2} (G_\eps^n , E_\eps^n , B_\eps^n) + \mathbb{E}_s ( G_\eps^n , E_\eps^n , B_\eps^n ) \Big] \Big( \| \mathbb{P}^\perp G_\eps^n \|^2_{H^s_x L^2_v} + \| \mathbb{P}^\perp G_\eps^{n+1} \|^2_{H^s_{x,v}(\nu)} \Big) \bigg\} \,,
	\end{aligned}
	\end{equation}
	where $a_{k+1} = \tfrac{r_k}{2 N_{k+1}} > 0$, $b_{k+1} = b_k > 0$, $r_{k+1} = \min \big\{ \tfrac{r_k}{2} , \tfrac{\lambda_0 r_k}{2 \lambda_* N_{k+1}} \big\} > 0$ and $C_{k+1}' = C_k' + \tfrac{C r_k}{2 \lambda_*} > 0$. Then the induction principle implies that the bound \eqref{Iter-Mixed-Bnd} holds.
	
	We now introduce a so-called instant iterating energy of $\mathfrak{E}_n (t)$
	\begin{equation}
	\begin{aligned}
	\widetilde{\mathfrak{E}}_n (t) = & \| G_\eps^n \|^2_{H^s_x L^2_v} + \| E_\eps^n \|^2_{H^s_x} + \| B_\eps^n \|^2_{H^s_x} + \sum_{\substack{|m| + |\alpha| \leq s \\ 1 \leq |\alpha| \leq k}} a_{|\alpha|} \| \partial^m_\alpha \mathbb{P}^\perp G_\eps^n \|^2_{L^2_{x,v}} \\
	& + \tfrac{1}{\eps^2} \int_0^t \Big[ b_s \| \mathbb{P}^\perp G_\eps^n \|^2_{H^s_x L^2_v (\nu)} + r_s \sum_{\substack{|m| + |\alpha| \leq s \\ 1 \leq |\alpha| \leq k}} \| \partial^m_\alpha \mathbb{P}^\perp G_\eps^n \|^2_{L^2_{x,v}(\nu)} \Big] (\tau) \d \tau \,.
	\end{aligned}
	\end{equation}
	It is easy to verify that
	\begin{equation}
	\begin{aligned}
	c_* \mathfrak{E}_n (t) \leq \widetilde{\mathfrak{E}}_n (t) \leq C_* \mathfrak{E}_n (t) \,,
	\end{aligned}
	\end{equation}
	where $c_* = \min \{ 1, a_1, \cdots , a_s, b_s, r_s \} > 0$ and $C_* = \max \{ 1, a_1, \cdots , a_s, b_s, r_s \} > 0$. We take $k = s $ in the inequality \eqref{Iter-Mixed-Bnd} and integrate over $[0,t]$. Then we derive from the fact $\sup_{0 \leq t \leq T^*} \widetilde{\mathfrak{E}}_n (t) \leq 2 \ell$ that
	\begin{equation}
	\begin{aligned}
	\widetilde{\mathfrak{E}}_{n+1} (t) \leq \widetilde{\mathfrak{E}}_{n+1} (0) + C_s (1 + \ell) \Big[ \ell  t +  t \sup_{0 \leq s \leq t} \widetilde{\mathfrak{E}}_{n+1} (s) + \sqrt{\ell} \ \widetilde{\mathfrak{E}}_{n+1} (t) \Big] \,,
	\end{aligned}
	\end{equation}
	where $C_s = \tfrac{12 C_s'}{c_*} ( 1 + \tfrac{1}{\sqrt{c_*}} + \tfrac{1}{c_*} ) > 0$. Here $\widetilde{\mathfrak{E}}_{n+1} (0) = \widetilde{\mathfrak{E}} (0) \leq \ell $. It follows that for $t \leq T^*$
	\begin{equation}
	\begin{aligned}
	\Big[ 1 - C_s (1 + \ell) T^* - C_s (1 + \ell) \sqrt{\ell}  \Big] \sup_{0 \leq t \leq T^*} \widetilde{\mathfrak{E}}_{n+1} (t) \leq \widetilde{\mathfrak{E}}_{n+1} (0) + C_s ( 1 + \ell ) \ell T^* \\
	\leq \ell + C_s ( 1 + \ell ) \ell T^* \,.
	\end{aligned}
	\end{equation}
	If $0 < T^* \leq \sqrt{\ell}$ and $\ell > 0$ is small such that $ ( 1 + \ell ) \sqrt{\ell} \leq \tfrac{1}{5 C_s} $, then 
	\begin{equation}
	\begin{aligned}
	1 - C_s (1 + \ell) T^* - C_s (1 + \ell) \sqrt{\ell} \geq 1 - 2 C_s (1 + \ell) \sqrt{\ell} \,.
	\end{aligned}
	\end{equation}
	Consequently, we have
	\begin{equation}
	\begin{aligned}
	\sup_{0 \leq t \leq T^*} \widetilde{\mathfrak{E}}_{n+1} (t) \leq \Big[ 1 + \tfrac{3 C_s ( 1 + \ell ) \sqrt{\ell}}{1 - 2 C_s ( 1 + \ell) \sqrt{\ell}} \Big] \ell \leq 2 \ell \,.
	\end{aligned}
	\end{equation}
	The proof of Lemma \ref{Lmm-Iter-Appro-Bnd} is finished.
\end{proof}

%%%%%%%%%%%%%%%%%%%%%%%%%%%%%%%%%%%%%%%%%%%%%%%%%%%%%%%%%%%%%%%%%%%%%%%%%%%

\section*{Acknowledgment}

The first named author N. J. appreciates Ars\'enio and Saint-Raymond. This work was encouraged by them, and the communications between N. J. and them on the work \cite{Arsenio-SRM-2016} share lots of lights on the current project. This work was supported by the grants from the National Natural Foundation of
China under contract Nos. 11471181 and 11731008.

%%%%%%%%%%%%%%%%%%%%%%%%%%%%%%%%%%%%%%%%%%%%%%%%%%%%%%%%%%%%%%%%%%%%%%%%%%%
\bigskip
% \phantomsection
% \addcontentsline{toc}{section}{\refname}

%\bibliographystyle{unsrtnat}
%\nocite{*}
\bibliography{reference}

\end{document}